\documentclass[english,12pt]{smfart}
\usepackage[english]{babel}
\usepackage{amsmath,amsfonts,amssymb,amsthm,mathrsfs,bbm}
\usepackage[font=sf, labelfont={sf,bf}, margin=1cm]{caption}
\usepackage{graphicx,graphics}
\usepackage{epsfig}
\usepackage{latexsym}
\usepackage[applemac]{inputenc}
\usepackage{ae,aecompl}
\usepackage{pstricks}
\usepackage{enumerate}
\usepackage{xcolor}
\usepackage[pdfpagemode=UseNone,bookmarksopen=false,colorlinks=true,urlcolor=blue,citecolor=blue,citebordercolor=blue,linkcolor=blue]{hyperref}

\usepackage[normalem]{ulem}
\usepackage[a4paper,vmargin={3.5cm,3.5cm},hmargin={2.5cm,2.5cm}]{geometry}
\usepackage{comment}

\usepackage{microtype,stmaryrd}          
\hypersetup{pdfstartview=XYZ}
\usepackage{array}
\usepackage[ruled]{algorithm}
\usepackage{algorithmic}
\usepackage{caption}
\usepackage{subcaption}
\usepackage[all, knot, cmtip]{xy} 
\usepackage{wrapfig}

\usepackage{ltablex}

\usepackage{changepage}

\numberwithin{equation}{section}

\newcommand{\ndN}{\mathbb{N}}
\newcommand{\ndZ}{\mathbb{Z}}

\newcommand{\ndR}{\mathbb{R}}
\newcommand{\ndC}{\mathbb{C}}

\newcommand{\ndK}{\mathbb{K}}
\newcommand{\ndA}{\mathbb{A}}

\renewcommand{\Pr}[1]{\mathbb{P}(#1)}
\newcommand{\Prb}[1]{\mathbb{P}\left(#1\right)}
\newcommand{\Pb}[1]{\mathbb{P}\left(#1\right)}
\newcommand{\Ex}[1]{\mathbb{E}[#1]}

\newcommand{\Va}[1]{\mathbb{V}[#1]}
\newcommand{\one}{\mathbbm{1}}

\newcommand{\id}{\mathrm{id}}

\newcommand{\scM}{\mathscr{M}}

\newcommand{\scT}{\mathscr{T}}

\newcommand{\om}{\mathbf{w}}

\newcommand{\cO}{\mathcal{O}}
\newcommand{\cC}{\mathcal{C}}
\newcommand{\cL}{\mathcal{L}}
\newcommand{\cM}{\mathcal{M}}
\newcommand{\cF}{\mathcal{F}}
\newcommand{\cB}{\mathcal{B}}

\newcommand{\cN}{\mathcal{N}}

\newcommand{\cK}{\mathcal{K}}

\newcommand{\cS}{\mathcal{S}}
\newcommand{\cG}{\mathcal{G}}
\newcommand{\cT}{\mathcal{T}}
\newcommand{\cR}{\mathcal{R}}
\newcommand{\cV}{\mathcal{V}}
\newcommand{\cA}{\mathcal{A}}
\newcommand{\cH}{\mathcal{H}}
\newcommand{\cD}{\mathcal{D}}
\newcommand{\cX}{\mathcal{X}}

\newcommand{\cQ}{\mathcal{Q}}

\newcommand{\fmT}{\mathfrak{T}}

\newcommand{\fmA}{\mathfrak{A}}

\newcommand{\mA}{\mathsf{A}}
\newcommand{\mC}{\mathsf{C}}
\newcommand{\mD}{\mathsf{D}}
\newcommand{\mB}{\mathsf{B}}
\newcommand{\mF}{\mathsf{F}}
\newcommand{\mM}{\mathsf{M}}
\newcommand{\mO}{\mathsf{O}}
\newcommand{\mK}{\mathsf{K}}

\newcommand{\mG}{\mathsf{G}}

\newcommand{\mH}{\mathsf{H}}

\newcommand{\mS}{\mathsf{S}}
\newcommand{\mR}{\mathsf{R}}

\newcommand{\mQ}{\mathsf{Q}}

\newcommand{\mX}{\mathsf{X}}

\newcommand{\me}{\mathsf{e}}

\newcommand{\CRT}{\mathcal{T}_\me}

\newcommand{\mGWT}{\hat{\mathcal{T}}}

\newcommand{\cE}{\mathcal{E}}

\newcommand{\hei}{\mathrm{h}}
\newcommand{\he}{\mathrm{h}}
\newcommand{\He}{\textnormal{H}}
\newcommand{\Di}{\textnormal{D}}
\newcommand{\emb}{\textnormal{emb}}

\newcommand{\spa}{\mathsf{span}}
\newcommand{\supp}{\mathsf{supp}}

\newcommand{\UHT}{\mathcal{U}_{\infty}}
\newcommand{\VHT}{\mathcal{V}_{\infty}}

\newcommand{\eqdist}{\,{\buildrel d \over =}\,}

\newcommand{\convdis}{\,{\buildrel d \over \longrightarrow}\,}
\newcommand{\convp}{\,{\buildrel p \over \longrightarrow}\,}

\newcommand{\Set}{\textsc{SET}}
\newcommand{\Seq}{\textsc{SEQ}}
\newcommand{\cZ}{\mathcal{Z}}

\newtheorem{theorem}{Theorem}[section]
\newtheorem{conjecture}[theorem]{Conjecture}
\newtheorem{question}[theorem]{Question}
\newtheorem{corollary}[theorem]{Corollary}
\newtheorem{proposition}[theorem]{Proposition}
\newtheorem{lemma}[theorem]{Lemma}

\newtheorem{definition}[theorem]{Definition}
\newtheorem{remark}[theorem]{Remark}
\newtheorem{example}[theorem]{Example}

\numberwithin{equation}{section}


\title{Limits of random tree-like discrete structures}

\date{}

\author{Benedikt Stufler \thanks{The author acknowledges support by the German Research Foundation DFG, STU 679/1-1.}}
\address[Benedikt Stufler]{Institute of Mathematics, University of Zurich}
\email{benedikt.stufler@math.uzh.ch}

\begin{document}



\begin{abstract}
	We study a model of random $\mathcal{R}$-enriched trees that is based on weights on the $\mathcal{R}$-structures and allows for a unified treatment of a large family of random discrete structures. We establish novel distributional limits describing local convergence around fixed and random points in this general context, limit theorems for component sizes when $\mathcal{R}$ is a composite class, and a Gromov--Hausdorff scaling limit of random metric spaces patched together from independently drawn metrics on the $\mathcal{R}$-structures.  Our main applications treat a selection of  examples encompassed by this model. We consider random outerplanar maps sampled according to arbitrary weights assigned to their inner faces, and  classify in complete generality  distributional limits for both the asymptotic local behaviour near the root-edge and near a uniformly at random drawn vertex. We consider random connected graphs drawn according to weights assigned to their blocks and establish a Benjamini--Schramm limit. We also apply our framework to recover in a probabilistic way a central limit theorem for the size of the largest $2$-connected component in random graphs from planar-like classes. We prove Benjamini--Schramm convergence of random $k$-dimensional trees and establish both scaling limits and local weak limits for random planar maps drawn according to Boltzmann-weights assigned to their $2$-connected components.
\end{abstract}

\maketitle

\newpage

{ \begin{adjustwidth}{-1cm}{-1cm}
\tableofcontents
\end{adjustwidth}
}

\newpage

\section{Introduction and main results}
\label{sec:intro}

In recent years, there has been considerable progress in understanding the asymptotic shape of large random discrete structures. Some focus was put on local weak convergence, which describes the behaviour of neighbourhoods around random points \cite{Stephenson2016, MR3083919, MR2243873, MR1873300,   MR3256879,MR3183575,MR2013797}. Asymptotic global geometric properties are, on the other hand, better described by scaling limits with respect to the Gromov--Hausdorff metric \cite{MR3342658, MR3414449, 2015arXiv150306738A,  MR2336042, MR2778796, MR3112934, MR3070569}, and more recent works \cite{2016arXiv160801129B} combine both viewpoints in local Gromov--Hausdorff scaling limits. A very successful approach in this field is to make use of appropriate combinatorial bijections that relate the objects under consideration to simpler structures such as different kinds of trees. To name only a few examples, the  Ambj{\o}rn--Budd bijection \cite{MR3090757}, the Cori--Vauquelin--Schaeffer bijection \cite{MR638363,SPHD} and the Bouttier--di Francesco--Guitter bijection \cite{MR2097335} have become well-known for their usefulness in this regard.  

The main difference to the present work is that instead of presenting a specific example of a random structure and afterwards a suitable bijection for this model, we consider an abstract family of all discrete structures that admit a certain  {\em type} of bijective encoding. Specifically, we consider the family of all objects admitting an $\cR$-enriched tree encoding, with $\cR$ ranging over all combinatorial classes. This high level of generality allows for a unified approach for studying a large family of random structures. 

The concept of enriched trees goes back to Labelle \cite{MR642392} who used it to provide a combinatorial proof of the Lagrange inversion formula. Roughly speaking, given a class $\cR$ of combinatorial objects, an $\cR$-enriched tree is a rooted tree together with a function that assigns to each vertex an $\cR$-structure {\em on} its offspring. For example, the structure can be a linear or cyclic order, a graph structure, or any other combinatorial construction.
If we assign a non-negative weight to each $\cR$-structure, we may draw an $\cR$-enriched tree of a given size at random with probability proportional to the product of its weights.  The list of random structures that may be described by this model is long, and includes random graphs sampled according to weights assigned to its maximal $2$-connected components, random outerplanar maps sampled according to weights assigned to their inner faces, likewise random dissections sampled according to such face-weights,  random planar maps with a given number of edges and weights on the blocks, and  subclasses of random $k$-dimensional trees with a given number of vertices.

In analytic combinatorics, random structures involving some sort of composition scheme are usually classified into subcritical, critical and supercritical regimes, depending on how the behaviour of the singularities of the inner and outer structure combine in order to determine the behaviour of the compound structure \cite[Ch. VI]{MR2483235}. For example, random graphs from so called subcritical classes of graphs have received considerable attention in the literature in the past decade \cite{MR2873207,MR2534261,MR3184197}.  We are going to deviate from this classification and instead use notions originating from a  probabilistic context.
Our study commences with the observation, that any random discrete structure admitting an enriched-tree type encoding has a canonical coupling with a simply generated tree. Janson's survey \cite{MR2908619} on the subject classifies this model of random trees into three kinds I, II and III, with two further possible subdivisions of the first into I$\alpha$ and I$\beta$, or I$a$ and I$b$. We recall the details in the preliminary Section~\ref{sec:simplygen}. This allows us to use the same classification for the random enriched-tree type structures under consideration, and gives a more fine-grained terminology. 

The core of our study of random weighted enriched trees describes asymptotic global and local properties, such as convergence of extended enriched fringe subtrees and left-balls, limit theorems for component sizes and scaling limits of associated random metric spaces. We provide applications to prominent examples of random discrete structures encompassed by this framework. The main novel applications of the present work may be summarized as follows.

{\em Random outerplanar maps and dissections of polygons.}
We consider random outerplanar maps with $n$ vertices sampled according to the product of weights assigned to their inner faces. 
The case of uniform random outerplanar maps where each face receives weight $1$ has received some attention in the recent literature from both combinatorial and probabilistic viewpoints \cite{MR2185278, caraceni2016, MR3634279}.

As our first main application, we establish for arbitrary weight-sequences a distributional limit that encodes convergence of neighbourhoods of the origin of the root-edge as the size of the map tends to infinity, and also a  Benjamini--Schramm limit that describes the asymptotic local behaviour around a uniformly at random selected vertex. We compare and precisely describe the distributions of both limit objects in terms of weighted Boltzmann distributions. The limits admit a canonical embedding in the plane and the local convergence preserves the planar structure of the random maps, that is, we really obtain convergence of the neighbourhoods with their embedding in the plane.  The approaches for obtaining the two limits are different, as for the first we use the local convergence of simply generated trees with a fixed number of vertices or leaves, and for the second we consider extended fringe subtrees at randomly selected vertices. 

In the type I case, we exploit the fact that the weak limits of the enriched tree encoding with respect to both a fixed and random root are locally finite and correspond to actual outerplanar maps. In the subcase I$\alpha$, where the diameter of this model of planar maps has order $\sqrt{n}$, we even obtain convergence in total variation of arbitrary $o(\sqrt{n})$-diameter neighbourhoods of the fixed and random roots.  This is best possible in this context, as the convergence fails for $\epsilon \sqrt{n}$-neighbourhoods for any fixed positive constant $\epsilon$.

In the type II regime, we apply  the condensation phenomenon observed for large conditioned Galton--Watson trees \cite{MR2764126,MR2908619, MR3183575}, and also establish a similar result for extended enriched fringe-subtrees. In this way, we obtain qualitatively different and interesting distributional limits, which contrarily to the type I case contain doubly infinite paths.  We also obtain limit theorems for the sizes of the $k \ge 1$ largest blocks and faces, in particular a central limit theorem for $k=1$, if the face-weights may additionally be tilted to probability weight-sequences that lie in the domain of attraction of some stable law. One of the ingredients for treating outerplanar maps is to understand the Benjamini--Schramm limits of large dissections of polygons sampled according to the product of weights assigned to their inner faces, for which we provide a complete description of their limits in the same spirit as for loop-trees in \cite{MR3317412}. Random face-weighted dissections have sparked the interest of probabilists in recent works \cite{MR3178472,MR3245291,MR3382675}. We identify dissections as enriched trees using the Ehrenborg--M\'endez transformation, which allows us to study them in a unified way using the same framework as for general enriched trees. If such a random dissection has type I, then its Benjamini--Schramm limit is given by an infinite planar map whose dual-tree is distributed like a modified Kesten tree. In the type II regime, giant faces emerge and the local weak limit contains a doubly-infinite path corresponding to the boundary of the large face nearest to the random root. Random dissections with type III converge in the local weak sense toward a deterministic doubly-infinite path. As for random outerplanar in the type II regime, we may locate a submap given by an ordered sequence of dissections whose random size (typically) becomes large. This is a special case  of a Gibbs partition, a general model of random partitions of sets which appear naturally in combinatorial stochastic processes~\cite{MR2245368}. Using recent results for convergent type  Gibbs partitions  \cite{doi:10.1002/rsa.20771}, we identify a giant component in this sequence. 
Roughly speaking, this implies that random outerplanar maps in this setting  contain "large" and "small"  dissections, and if we look close the root-edge of the map, we typically see at most one that is large. A priori, it would be possible that these "dissection-cores" have type I and hence  converge toward the Kesten-tree-like limit object. However, we check that if the map has type II, then so do the dissections. Thus the large dissections in type II outerplanar maps also have large faces. This allows us to deduce local convergence of random outerplanar maps toward limit objects containing a doubly infinite path that corresponds to the frontier of a large face.  We detail the explicit distribution of the limits in terms of weighted Boltzmann-distributions. If the random outerplanar map has type III, then it's local behaviour is typically determined by single large $2$-connected submap. In this case, the local weak limit for both the fixed and random root is given by a deterministic doubly-finite path, and hence agrees with the behaviour of type III dissections. Thus, our methods allow us to  classify the local behaviour of random face-weighted outerplanar maps and face-weighted dissections of polygons for both the vicinity of the root-edge and the neighbourhood of a uniformly at random selected vertex.

{\em Random graphs.}
The main example of random graphs in our setting is drawing a connected $n$-vertex graph with probability proportional to weights assigned to its maximal $2$-connected subgraphs. This generalizes the model of uniform random graphs from addable minor-closed graphs and also that of uniform random graphs from block-stable classes, which have received growing attention in recent literature, see in particular McDiarmid~\cite{MR2507738},  McDiarmid and Scott~\cite{MR3530623}, and Noy~\cite{noysurvey}. It  encompasses in particular the model of random graphs from planar-like classes introduced by Gim\'enez, Noy and Ru\'e \cite{MR3068033}, and so called subcritical graph classes as studied by Drmota, Fusy, Kang, Kraus and Ru\'e \cite{MR2873207}. 

It is not a restriction to treat connected graphs. If we draw a random possibly disconnected graph in the same way, then a giant component emerges with a stochastically bounded remainder, and hence  properties for the connected case carry over automatically to the disconnected case. This has been observed by McDiarmid~\cite{MR2507738} for uniform random graphs from proper addable minor-closed classes, then recovered and extended by probabilistic methods in Stufler~\cite[Thm. 4.2 and Section 5]{doi:10.1002/rsa.20771} to random block-weighted classes with analytic generating functions. In the present work we additionally establish results for Gibbs partitions with superexponential weights and apply these to complete the picture, showing in complete generality that random block-weighted graphs exhibit a giant component with a stochastically bounded remainder.



Our results for random enriched trees readily yield Benjamini--Schramm convergence in the type I setting, and the strong $o(\sqrt{n})$-neighbourhood convergence in the type I$\alpha$ setting. The limit object has a natural coupling with Kesten's modified Galton--Watson tree, which is reflected in the fact that it admits only one-sided infinite paths. In the less general type I$a$ setting, which roughly corresponds to a weighted version of random graphs from subcritical graph classes, this also yields laws of large numbers for the number of spanning trees and subgraph counts by results due to Lyons~\cite{MR2160416} and Kurauskas~\cite{2015arXiv150408103K}. The $o(\sqrt{n})$-neighbourhood convergence is best possible, as the diameter of these graphs has order $\sqrt{n}$. In the I$\beta$ setting, there are examples with a polynomially smaller expected diameter. So the asymptotic global geometric properties differ greatly, but interestingly we still obtain Benjamini--Schramm convergence toward a similar limit object. 

For random graphs of type II, such as the uniform $n$-vertex planar graphs or random graphs from planar-like classes, we obtain convergence toward a limit enriched tree that contains a vertex with infinite degree and hence does not correspond directly to a random graph. We still obtain convergence of the probability for the block-neighbourhood of a random vertex to be of a specific shape, but this does not amount to Benjamini--Schramm convergence, as it describes the asymptotic behaviour of neighbourhoods away from all large $2$-connected subgraphs. However, by combining results for the asymptotic behaviour of Gibbs-partitions, the convergence toward the limit tree, and projective limits of probability spaces, we show that there is sequence of random numbers $K_n \convdis \infty$ such that the random connected graph with $n$-vertices converges in the Benjamini--Schramm sense \emph{if and only if} the random $2$-connected graph drawn with probability proportional to its weight among all $K_n$-sized $2$-connected does. We detail the distribution of the limit of the connected graph in this case in terms of weighted Boltzmann-distributions and the $2$-connected limit. This is particularly interesting, when considering random weighted graphs and not just uniform choices from fixed graph classes. Apart from the class of planar graphs and related families, "most" graph classes in combinatorics are subcritical, and hence uniform graphs from such classes have the described behaviour of type I$a$ random weighted graphs. But from a probabilistic perspective it is natural to not only consider the uniform measure and we may easily force random weighted graphs from subcritical classes into the type II or critical regime, by adjusting the weights. For example, the uniform random outerplanar graph has type I$a$, but if we adjust the block-weights to the nongeneric type II regime, we obtain a new qualitatively different limit object, as $2$-connected outerplanar graphs behave like random dissections of polygons. This example also illustrates nicely the differences and similarities in the asymptotic behaviour of outerplanar maps and graphs. Likewise, we may force many other examples of subcritical graph classes such as cacti graphs into the type II regime, yielding a whole family of qualitatively different Benjamini--Schramm limits. As for uniform random graphs from addable minor-closed graph classes, it is known that these belong either to the type I or type II regime. In the type I case we immediately obtain distributional convergence, and in the type II case our results fully describe the relation to the  $2$-connected case. As we detail in Section~\ref{sec:blockcl}, this seems to be a first step in a promising direction for establishing and describing the Benjamini--Schramm limit of uniform random  planar graphs. 

As a further main result, we obtain in a purely probabilistic way limits for the extremal block-sizes  of random graphs from planar-like classes, which encompasses the uniform $n$-vertex planar graph. The limit laws for the size of the $i$-th largest blocks in this setting appear to be new for $i \ge 2$ and the central limit theorem for the size of the largest block has previously been observed by Gim\'enez, Noy and Ru\'e \cite{MR3068033}, who even showed a stronger local limit theorem by means of singularity analysis and the saddle-point method. The main contribution of the present paper in this regard is, however, the simple probabilistic approach, which shows that everything known about the extremal degree behaviour of simply generated trees may be transferred to the setting of random graphs. 
As a byproduct, the framework of enriched trees also yields results for the block-diameter of random graphs. McDiarmid and Scott \cite[Thm. 1.2]{MR3530623} showed using interesting combinatorial methods that with high probability any path in the random $n$-vertex graph from a block-class passes through at most $5 \sqrt{n \log(n)}$ blocks. They conjectured, that the extra factor $\sqrt{\log(n)}$ may be replaced by any sequence tending to infinity. In the tree-like representation of graphs considered here, the block-diameter corresponds up to an additive constant to the diameter of a simply generated tree, and hence we may support this conjecture by verifying it for various families of classes.  We also observe that the conjecture would be entirely verified, if one could affirm a question by Janson~\cite[Problem 21.8]{MR2908619}, who asked whether in general the diameter of any type of simply generated trees has no larger order than $\sqrt{n}$.

{\em Random $k$-dimensional trees.}
The notion of $k$-trees generalizes the graph-theoretic concept of trees.  A $k$-tree consists either of a complete graph with $k$ vertices, or is obtained from a smaller $k$-tree by adding a vertex and connecting it with $k$ distinct vertices of the smaller $k$-tree.  Such objects are interesting from a combinatorial point of view, as their enumeration problem has a long history, see \cite{MR0237382,MR3007180,MR1888841,MR0299535,MR2577935,MR0234868,MR0357214}. They are also interesting from an algorithmic point of view, as many NP-hard problems on graphs have polynomial algorithms when restricted to $k$-trees \cite{MR985145,MR2484635}. We apply results for extended fringe subtrees of random enriched trees to provide a Benjamini--Schramm limit of random $k$-trees. Even more ambitiously, we verify total variational convergence of $o(\sqrt{n})$-neighbourhoods, which is the strongest possible form of convergence in this context, as the diameter of random $k$-trees has order $\sqrt{n}$  \cite{2016arXiv160505191D}. We compare the limit graph with a local limit established in \cite{2016arXiv160505191D} that encodes convergence of neighbourhoods around a random $k$-clique. The limit objects are distinct, which is already evident from the different behaviour of the degree of a random vertex and a  vertex of a random front. Interestingly, we may however verify that the two limits are identically distributed as random  unrooted  graphs.

{\em Random planar maps.}
The study of random planar maps as their number of edges becomes large has been the driving force for numerous discoveries in the past decade, and their scaling limit and local limit are interesting objects in their own right. Tutte's core decomposition shows that planar maps are special cases of $\cR$-enriched trees, if we let $\cR$ denote the class of non-separable maps. Hence our results for random weighted enriched trees apply to random planar maps sampled according to the product of weights assigned to their maximal non-separable submaps. This includes the case of uniform $n$-vertex bipartite maps, loop-less maps, and many other natural classes of maps, whose constraints may be expressed in terms of constraints for the $2$-connected components. We establish a local weak limit for type I random block-weighted planar maps, and a scaling limit in the type I$a$ regime with respect to the first-passage percolation metric, for which we also strengthen the local convergence to total variational convergence of $o(\sqrt{n})$-neighbourhoods. In the type II case, which encompasses the mentioned examples of uniform planar maps with constraints, we apply the condensation phenomenon to establish a general principle stating that whenever random weighted non-separable maps converge in the local weak sense, then so does the corresponding random block-weighted planar map. The enriched tree corresponding to a random planar map is simply generated and its outdegrees correspond to the number of half-edges in the maximal non-separable submaps. Hence available limit theorems and bounds for extremal outdegrees in simply generated trees also hold for random block-weighted planar maps. A similar connection to simply generated trees has  been observed  by Addario-Berry~\cite{2015arXiv150308159A}. Specifically, the coupling with a simply generated tree in \cite[Prop.~1]{2015arXiv150308159A} is encompassed by Lemma~\ref{le:coupling} for the special case where $\cR$ is the species of non-separable planar maps.

Random enriched trees may also be considered up to symmetry. The combinatorial techniques necessary for this  task are not required for the present exposition concerning random labelled or asymmetric structures. For this reason, we undertake this endeavour in~\cite{StEJC2018}.

\subsection*{Plan of the paper}
Section~\ref{sec:intro} gives an informal introduction and overview of the main applications.
Section~\ref{sec:background} recalls basic notions related to graphs, trees and planar maps, and discusses the concepts of local weak convergence and Gromov--Hausdorff convergence. Section~\ref{sec:simplygen} fixes notation  regarding simply generated trees and their limits. Section~\ref{sec:combspec} discusses an algebraic formalization of weighted combinatorial structures and associated Boltzmann probability measures. Section~\ref{sec:prob} briefly recalls probabilistic tools that we will apply in our proofs, in particular projective limits of probability spaces. Section~\ref{sec:mainstudy} presents the contributions of the present paper in detail. Specifically, Subsection~\ref{sec:intro1} introduces our model of random weighted $\cR$-enriched trees and discusses how this encompasses various models of random graphs, dissections of polygons, outerplanar maps, planar maps and $k$-trees. Subsection~\ref{sec:partB} provides general results for the convergence of trimmings, left-balls and extended fringe-subtrees in enriched trees. In Subsection~\ref{sec:schroeder} we establish similar results for Schr\"oder enriched parenthesizations. Subsection~\ref{sec:gibbs} discusses the limits of Gibbs-partitions, which will be crucial in the application to type II and III random structures. In Subsection~\ref{sec:compsize} we provide limits for the extremal sizes of components for random $\cR$-enriched trees when $\cR$ is a composite structure. Subsection~\ref{sec:appcomb} discusses applications to prominent examples of random enriched trees and establishes further main results, such as the classification of local limits of face-weighted outerplanar maps and dissections. Subsection~\ref{sec:partA} introduces a general model of random semi-metric spaces  patched together from random semi-metrics associated to the $\cR$-structures. A scaling limit and a tail-bound for the diameter are established and applied to random block-weighted planar maps. In Section~\ref{sec:allproofs} we present the proofs of our main results. 

\subsection*{Notation}
Throughout, we set
\[
\ndN=\{1,2,\ldots\}, \qquad \ndN_0 = \{0\} \cup \ndN, \qquad [n]=\{1,2,\ldots, n\}, \qquad n \in \ndN_0.
\]
We usually assume that all considered random variables are defined on a common probability space whose measure we denote by $\mathbb{P}$, and let $\mathbb{L}_p$ denote the corresponding space of $p$-integrable real-valued functions.  All unspecified limits are taken as $n$ becomes large, possibly taking only values in a subset of the natural numbers.  We write $\convdis$ and $\convp$ for convergence in distribution and probability, and $\eqdist$ for equality in distribution. An event holds with high probability, if its probability tends to $1$ as $n \to \infty$.
 We let $O_p(1)$ denote an unspecified random variable $X_n$ of a stochastically bounded sequence $(X_n)_n$, and write $o_p(1)$ for a random variable $X_n$ with $X_n \convp 0$. We write $\cL(X)$ to denote the law of a random variable $X$. The total variation distance of measures and random variables is denoted by $d_{\textsc{TV}}$.

\section{Background on graph limits and combinatorial structures}
\label{sec:background}

\subsection{Graphs, trees and planar maps}

\subsubsection{Graphs}
A {\em graph} $G=(V(G),E(G))$ consists of a set of {\em labels} or {\em vertices} $V(G)$ and a set of {\em edges} $E(G)$ which are $2$-element subsets of the vertex set. Instead of writing $v \in V(G)$ we will often just write $v \in G$. We say an edge $e=\{v,w\}$ is {\em incident} to its {\em ends} $v$ and $w$, and will shortly denote by $e=vw$. The number of edges incident to a vertex $v$ is its degree \label{pp:deg}$d_G(v)$. A graph is {\em locally finite} if every vertex has finite degree, and {\em finite} if it has only finitely many vertices. Graphs $H$ with $V(H) \subset V(G)$ are {\em subgraphs} of $G$. We denote this by $H \subset G$. The graph $H$ is an {\em induced} subgraph, if additionally any edge of $G$ with both ends in $V(H)$ also belongs to $E(H)$. A {\em path} $v_0, v_1, \ldots, v_\ell$ in $G$ is a subgraph $P$ of the form
\[
V(P) = \{v_i \mid 0 \le i \le \ell \}, \quad E(P) = \{ v_iv_{i+1} \mid 0 \le i \le \ell -1\},
\]
with $v_i \ne v_j$ for all $i \ne j$.
The non-negative integer $\ell$ is the {\em length} of the path. We say $P$ {\em joins} or {\em connects} its {\em endvertices} $v_0$ and $v_\ell$. We will also encounter one-sided directed infinite paths $v_0, v_1, \ldots$ which {\em start} at the vertex $v_0$. Two-sided infinite paths $\ldots, v_{-1}, v_0, v_1, \ldots$ are defined analogously. A graph $G$ is \emph{connected}, if any two vertices may be joined by a path.  \label{pp:dg}The graph distance is a metric on the vertex set $V(G)$. The corresponding metric space is, by abuse of notation, usually denoted by $(G, d_G)$ and we write $v \in G$ instead of $v \in V(G)$. We let \label{pp:di}$\Di(G)= \sup_{x,y \in G} d_G(x,y)$ denote the {\em diameter} of $G$. A {\em cutvertex} is a vertex whose removal disconnects the graph. A connected graph is {\em $k$-connected}, if it has at least $k+1$-vertices and removing any $k$ vertices does not disconnect the graph. The {\em complete} graph \label{pp:com}$K_m$ with $m$ vertices has vertex set $[m]$ and any two distinct vertices are joined by an edge. 

A subgraph $B$ of a connected graph $G$ is a {\em block}, if it is $2$-connected or isomorphic to $K_2$, and if it is maximal with this property. That is, any subgraph $B \subsetneq B' \subset G$ must have a  cutvertex. Connected graphs have a tree-like block-structure, whose details are explicitly given in Diestel's book \cite[Ch. 3.1]{MR2744811}. We mention a few properties, that we are going to use. Any two blocks of $G$ overlap in at most one vertex. The cutvertices of $G$ are precisely the vertices that belong to more than one block. Many properties of $G$ are evident from looking at its blocks. For example, the graph $G$ is termed {\em bipartite}, if its vertex set may be partitioned into two disjoint sets $A$ and $B$, such that no edge with both ends in $A$ or both ends in $B$ exist. This is equivalent to requiring that every block of $G$ is bipartite.

A {\em graph isomorphism} between graphs $G$ and $H$ is a bijection $\phi: V(G) \to V(H)$ such that any two vertices $v,w \in V(G)$ are joined by an edge in $G$ if and only if their images $\phi(v), \phi(w)$ are joined by an edge in $H$. The graphs $G$ and $H$ are {\em structurally equivalent} or {\em isomorphic}, denoted by $G \simeq H$, if there exists at least one {\em graph isomorphism} between them. If we distinguish a vertex $v_G \in V(G)$, then the pair $G^\bullet = (G,v_G)$ is a {\em rooted} graph with {\em root vertex} $v_G$. We let \label{pp:he}$\He(G^\bullet)= \sup_{x \in G}d_G(v_G,x)$ denote the {\em height} of $G^\bullet$. For any vertex $x \in G$, we let $\he_{G^\bullet}(x) = d_G(v_G,x)$ the height of $x$ in $G$. A graph isomorphism between rooted graphs $G^\bullet$ and $H^\bullet$ is a graph isomorphism $\phi$ between the unrooted graphs $G$ and $H$ that satisfies $\phi(v_G) = v_H$. A graph considered up to isomorphism is an {\em unlabelled graph}. That is, any two unlabelled graphs are distinct if they are not isomorphic. Formally, unlabelled graphs are defined as isomorphism classes of graphs. Unlabelled rooted graphs are defined analogously.

\subsubsection{Trees} A {\em tree} $T$ is a  graph in which any two vertices are joined by a unique path. A rooted tree $T^\bullet$ has a natural partial order $\preccurlyeq$ on its vertex set, with $v \preccurlyeq w$ if the unique path from the root vertex to $w$ passes through $v$. If additionally $w \ne v$ and no vertex $u \notin \{v,w\}$ with $v \preccurlyeq u \preccurlyeq w$ exists, then $w$ is a {\em direct successor} or an {\em offspring} of $v$. The {\em offspring set} of a vertex $v$ is the collection of all its direct successors. Its cardinality is the {\em outdegree} \label{pp:dout}$d_T^+(v)$. 

Unlabelled rooted trees are also called {\em P\'olya trees}. Besides the four types of {\em unordered trees} that may be labelled or unlabelled, rooted or unrooted, there are also {\em ordered trees}. These trees are always rooted, but may be labelled or unlabelled. An ordered tree is a rooted labelled tree in which each offspring set is endowed with a linear order. That is, each vertex may have a first offspring, second offspring, and so on. Unlabelled ordered trees are usually called {\em plane trees}.

\begin{figure}[t]
	\centering
	\begin{minipage}{1.0\textwidth}
		\centering
		\includegraphics[width=0.2\textwidth]{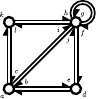}
	\caption{Corners and half-edges of planar maps.}
	\label{fi:halfedges}
	\end{minipage}
\end{figure}  

\subsubsection{Planar maps} A {\em multigraph} is a graph which may have multiple edges between vertices and in which an edge may of identical endpoints. Regular graphs are also often called {\em simple} graphs in order to distinguish the two notions. A graph or multigraph is {\em planar} if it may be embedded in the sphere or plane such that edges may only intersect at their endpoints. {\em Planar maps} are embeddings of connected planar multigraphs in the sphere, considered up to  orientation-preserving homeomorphism. We will not go into details and refer the reader to the book by Mohar and Thomassen \cite{MR1844449} for a complete exposition. Usually one studies {\em rooted} maps, in which one of the edges is distinguished and given an orientation. This oriented edge is called the {\em root edge} of the map and its origin is termed the {\em root vertex}. The complement of a map is divided into disjoint connected components, its {\em faces}. The face to the left of the root edge is termed the {\em root face} and the face to the right the {\em outer face}. The outer face is taken as the infinite face in plane representations. It is notationally convenient to also consider the map consisting of a single vertex as rooted, although it has no edges to be rooted at.

Many algorithms in computational geometry use a {\em half-edge} data structure in order to represent planar maps. Here any edge of the map is split into two directed half-edges that point in opposite directions. The half-edges correspond bijectively to the {\em corners} of the map, see Figure~\ref{fi:halfedges} for an illustration where corners are denoted by letters and half-edges by directed arrows. Formally, a corner incident to a vertex $v$ may be defined as a pair of consecutive (not necessarily distinct) elements in the cyclically ordered list of ends of edges incident to $v$. 

A map is termed {\em separable}, if its edge set may be partitioned into two non-empty subsets $S$ and $T$ such that there is precisely one vertex $v$ incident with both a member of $S$ and of $T$. In this case, $v$ is termed a {\em cutvertex} of the map. Note that this notion is more general than cutvertices of graphs. For example, an isolated vertex with two loops attached to it is a cutvertex of the map but not of the corresponding graph. A planar map that is not separable is termed {\em non-separable}. Note that a non-separable map with less than three vertices consists either of two vertices with an arbitrary positive number of edges between them, or a single vertex with at most one loop-edge attached to it. A simple rooted map is termed {\em outerplanar} if every vertex lies on the boundary of the outer face. Finally, a map is termed \emph{bipartite}, if the corresponding graph is bipartite.

\subsection{Local weak convergence}
\label{sec:lowe}
Let $G^\bullet = (G, v_G)$ and $
H^\bullet = (H, v_H)$ be two connected, rooted, and locally finite graphs. For any non-negative integer $k$ we may consider the {\em $k$-neighbourhoods} \label{pp:vk}$V_k(G^\bullet)$ and $V_k(H^\bullet)$ which are the subgraphs induced by all vertices with distance $k$ from the roots. The $k$-neighbourhoods are considered as rooted at $v_G$ and $v_H$, respectively. We may consider the distance
\begin{align}
\label{eq:thebsdistance}
d_{\textsc{BS}}(G^\bullet, H^\bullet) =  2^{-\sup \{k \in \ndN_0 \,\mid\, V_k(G^\bullet) \simeq V_k(H^\bullet) \}}
\end{align}
with  $V_k(G^\bullet) \simeq V_k(H^\bullet)$ denoting isomorphism of rooted graphs, that is, the existence of a graph isomorphism $\phi: V_k(G^\bullet) \to V_k(H^\bullet)$ satisfying 
$\phi(v_G) = v_H$.  This defines a premetric on the collection of all rooted locally finite connected graphs.  Two such graphs have distance zero, if and only if they are isomorphic. Hence this defines a metric on the collection $\mathbb{B}$ of all unlabelled, connected, rooted, locally finite graphs. The space $(\mathbb{B}, d_{\textsc{BS}})$ is complete and separable, that is, a Polish space. We refer the reader to the lecture notes \cite{clecture} for a detailed proof. 

A random rooted graph $\mG^\bullet$ from  $\mathbb{B}$ is the the {\em local weak limit} of a sequence $\mG_n^\bullet = (\mG_n, v_n)$, $n\in \ndN$  of random elements of $\mathbb{B}$, if it is the weak limit with respect to this metric. That is, if 
\begin{align}
\label{eq:cone1}
\lim_{n \to \infty} \Ex{f(G_n^\bullet)} = \Ex{f(G^\bullet)}
\end{align} for any bounded continuous function $f: \mathbb{B} \to \ndR$. This is equivalent to stating
\begin{align}
\label{eq:cone2}
\lim_{n \to \infty} \Pr{V_k(\mG_n^\bullet) \simeq G^\bullet} = \Pr{V_k(\mG^\bullet) \simeq G^\bullet}.
\end{align}
for any rooted graph $G^\bullet$. If the conditional distribution of $v_n$ given the graph $\mG_n$ is uniform on the vertex set $V(\mG_n)$, then the limit $G^\bullet$ is often also called the Benjamini--Schramm limit of the sequence $(\mG_n)_n$.

We are also going to consider the block-metric \label{pp:db}$d_{\textsc{block}}$ on the graph $G$ defined as follows.  Given vertices $u,v \in V(G)$, consider any shortest path $P$ connecting $u$ and $v$ in $G$, and let $d_{\textsc{block}}(u,v) \in \ndN_0$ denote the minimum number of blocks of $C$ required to cover the edges of $P$. Given a non-negative integer $k$, we let \label{pp:uk}$U_k(G^\bullet)$ denote the subgraph induced by all vertices with block-distance at most $k$. This graph may be considered as rooted at the vertex $v$. As $V_k(G^\bullet) \subset U_k(G^\bullet)$, verifying
\begin{align}
\label{eq:cone3}
\lim_{n \to \infty} \Pr{U_k(\mG_n^\bullet) \simeq G^\bullet} = \Pr{U_k(\mG^\bullet) \simeq G^\bullet}
\end{align}
for {\em all} rooted graphs $G^\bullet$ verifies \eqref{eq:cone2}, and hence implies distributional convergence of $\mG_n^\bullet$ to the graph $\mG^\bullet$.

\subsection{Gromov--Hausdorff convergence}
Let $X^\bullet =  (X,d_X, x_0)$ and $Y^\bullet = (Y,d_Y,y_0)$ be pointed compact metric spaces. A {\em correspondence} between $X^\bullet$ and $Y^\bullet$ is a subset $R \subset X \times Y$ containing $(x_0,y_0)$ such that for any $x \in X$ there is a $y \in Y$ with $(x,y) \in R$, and conversely for any $y \in Y$ there is a $x \in X$ with $(x,y) \in R$. The {\em distortion} of the correspondence is defined as the supremum \[\text{dis}(R) = \sup \{|d_X(x_1, x_2) - d_Y(y_1, y_2) \mid (x_1, y_1), (x_2, y_2) \in R \}.\] The {\em Gromov--Hausdorff distance} between the pointed spaces $X^\bullet$ and $Y^\bullet$ is given by \[d_{\textsc{GH}}(X,Y) = \frac{1}{2} \inf_R \text{dis}(R)\] with the index $R$ ranging over all correspondences between $X^\bullet$ and $Y^\bullet$. The factor $1/2$ is only required in order to stay consistent with an alternative definition of the Gromov--Hausdorff distance via the Hausdorff distance of embeddings of $X^\bullet$ and $Y^\bullet$ into common metric spaces, see \cite[Prop.\ 3.6]{MR3025391} and \cite[Thm.\ 7.3.25]{MR1835418}. This distance satisfies the axioms of a premetric on the collection of all compact rooted metric spaces. Two such spaces have distance zero from each other, if and only if they are isometric. That is, if there is a distance preserving bijection between the two that also preserves the root vertices. Hence we obtain a metric on the collection $\ndK^\bullet$ of isometry classes of pointed compact metric spaces. The space $(\ndK^\bullet, d_{\text{GH}})$ is known to be Polish (complete and separable), see \cite[Thm.\ 3.5]{MR3025391} and \cite[Thm. 7.3.30 and 7.4.15]{MR1835418}.

\section{Convergence of simply generated trees}
\label{sec:simplygen}
Simply generated trees are a model of random trees that generalize the concept of Galton--Watson trees conditioned on having a specific number of vertices. We recall relevant notions and results that we are going to use later in our study of combinatorial objects satisfying bijective encodings as enriched trees. This exposition follows parts of Janson's survey \cite{MR2908619}.
\subsection{Simply generated trees}
\label{sec:pretree}
\subsubsection{Random plane trees} Let $\om = (\omega_k)_{k \in \ndN_0}$ with $\omega_k \in \ndR_{\ge 0}$ for all $k$ denote a {\em weight sequence} satisfying $\omega_0 >0$ and $\omega_k > 0 $ for some $k \ge 2$. Then to each plane tree $T$ we assign its corresponding {\em weight}
\[
\omega(T) = \prod_{v \in T} \omega_{d_T^+(v)}.
\]
Let $\fmT_n$ denote the set of plane trees with $n$ vertices. The {\em partition function} is defined by
\[
Z_n = \sum_{T \in \fmT_n} \omega(T).
\]
The {\em support} of $\om$ is defined by \[\supp(\om) = \{k \mid \omega_k > 0 \}\] and the span $\spa(\om)$ is the greatest common divisor of the support. If the partition function $Z_n$ is positive, then $n \equiv 1 \mod \spa(\om)$. Conversely, if $n$ is large enough, then $n \equiv 1 \mod \spa(\omega)$ also implies $Z_n>0$, see \cite[Cor. 15.6]{MR2908619}. For any integer $n$ with $Z_n > 0$ we may draw a random tree $\cT_n$ from $\fmT_n$ with distribution given by
\[
\Pr{\cT_n = T} = \omega(T) / Z_n.
\]
Prominent examples of such {\em simply generated trees} are Galton--Watson trees conditioned on having $n$ vertices. 

\subsubsection{Types of weight sequences}
\label{sec:types}
It is convenient to partition the set of weight sequences into the three cases I, II, and III, as weight sequence having the same type share similar properties. 

In order to define these types, consider the power series \label{mark:p1}\[\phi(z) = \sum_{k \ge 0} \omega_k z^k\] and let \label{mark:p2}$\rho_\phi$ denote its radius of convergence. If $\rho_\phi>0$, then by \cite[Lem. 3.1]{MR2908619} the function \label{mark:p3}\[\psi(t) := \frac{t \phi'(t)}{\phi(t)}\] defined on $[0, \rho_\phi[$ is finite, continuous and strictly increasing. \label{mark:p5}
If $\rho_\phi >0$, set \[\nu := \psi(\rho_\phi) := \lim_{t \nearrow \rho_\phi} \psi(t) \in ]0, \infty].\] Otherwise, if $\rho_\phi=0$, set $\nu := \psi(0) := 0$. 

The constant $\nu$ has a natural interpretation. Unless $\rho_\phi=0$ (which is equivalent to $\nu=0$), $\nu$ is the supremum of the means of all probability weight sequences equivalent to $\mathbf{w}$. See Section 4 and in particular Remark 4.3 of Janson's survey \cite{MR2908619} for details.

\label{mark:p4}
We define the number $0 \le \tau < \infty$ as follows. If $\nu \ge 1$, let $\tau \in [0, \rho_\phi]$ be the unique number satisfying $\psi(\tau) = 1$. Otherwise, let $\tau := \rho_\phi$. \label{mark:p6}Define the probability distribution $(\pi_k)_k$ on $\ndN_0$ by \begin{align} \label{eq:tt} \pi_k = \tau^k \omega_k / \phi(\tau).\end{align} The mean and variance of the distribution $(\pi_k)_k$ are given by \begin{align}\label{eq:mu}\mu = \psi(\tau) = \min(\nu,1)\end{align} and \begin{align}\label{eq:variance} \sigma^2 = \tau \psi'(\tau) \le \infty.\end{align}

\label{mark:we}We define the cases I) $\nu \ge 1$, II) $0 < \nu < 1$ and III) $\nu=0$. The case I) may be subdivided into mutually exclusive cases by either I$a$) $\nu > 1$ and I$b$) $\nu=1$, or I$\alpha$) $\nu \ge 1$ and $\sigma < \infty$ and I$\beta$) $\nu=1$ and $\sigma=\infty$. In the cases I) and II) the simply generated tree with $n$ vertices is distributed like a Galton--Watson tree conditioned on having size $n$ with offspring distribution $(\pi_k)_k$. In the case III) the weight sequence does not correspond to any offspring distribution.

The generating function \[\cZ(z) = \sum_{n=0}^\infty Z_n z^n,\] with $Z_n$ denoting the partition function, is the unique power series satisfying \[\cZ(z) = z \phi(\cZ(z)).\] This follows from \cite[Rem. 3.2]{MR2908619} and the Lagrange inversion formula. Let $\rho_{\cZ}$ denote its radius of convergence. By \cite[Ch. 7]{MR2908619} we have that $0 \le \rho_\cZ < \infty$  and \begin{align}\label{eq:z1}\rho_\cZ = \tau / \phi(\tau).\end{align} Moreover, $\rho_\cZ = 0 \Leftrightarrow \rho_\phi =0 \Leftrightarrow \tau = 0$ and it holds that \begin{align}\label{eq:z2} \tau = \cZ(\rho_\cZ).\end{align}

\subsection{Local convergence of simply generated trees}
\label{se:locosi}
Simply generated trees convergence weakly toward an infinite limit tree, which, depending on the weight sequence, need not be locally finite.

\subsubsection{The modified Galton--Watson tree}
\label{se:modgwt}
Let $\xi$ be a random non-negative integer with average value $\mu := \Ex{\xi} \le 1$ and let $(\pi_k)_{k \ge 0}$ denote its distribution. The {\em modified Galton--Watson tree} $\mGWT$ is defined in \cite[Ch. 5]{MR2908619} as follows.   Any vertex is either {\em normal} or {\em special} and we start with a root vertex that is declared special. Normal vertices have offspring according to an independent copy of $\xi$ and special vertices have offspring (outdegree) according to an independent copy of the random variable $\hat{\xi}$ with distribution given by
\[
\Pr{\hat{\xi} = k} = \begin{cases} k \pi_k, & k \in \ndN_0,\\ 1 - \mu, & k=\infty. \end{cases}
\] 
All children of a normal vertex are declared normal and if a special node gets an infinite number of children all are declared normal as well. When a special vertex gets finitely many children all are declared normal with one uniformly at random chosen exception which is declared special. The special vertices form a path which is called the {\em spine} of the tree $\mGWT$. Note that if $\mu < 1$ (the subcritical case) then $\mGWT$ has almost surely a finite spine ending with an explosion. The length of the spine follows a geometric distribution. If $\mu = 1$ then $\mGWT$ is almost surely locally finite and has an infinite spine.

\subsubsection{Local convergence}
\label{se:loconv}
The {\em Ulam--Harris tree} $\UHT$ is an infinite plane tree in which each vertex has countably infinitely many offspring. Its vertex set \begin{align}
\label{pp:vht}
\VHT = \{\emptyset\} \cup \ndN^{1} \cup \ndN^{2} \cup \cdots
\end{align} is the set of all finite strings of positive integers. Its root is given by the empty string $\emptyset$, and any string $v = (v_1, \ldots, v_\ell)$ has ordered offspring $(v, 1), (v,2), \ldots$. 

Any plane tree $T$ can be viewed as subtree of the Ulam--Harris tree $\UHT$ and is uniquely determined by its sequence of outdegrees $(d_T^+(v))_{v \in \VHT} \in \overline{\ndN}_0^{\VHT}$ with $\overline{\ndN}_0 = \ndN_0 \cup \{\infty\}$. We endow the set $\overline{\ndN}_0$ with a compact topology as the one-point compactification of the discrete space $\ndN_0$. The space $\overline{\ndN}_0^{\VHT}$ is a compact Polish-space since it is the product of countably many such spaces.
We let $\fmT \subset \overline{\ndN}_0^{\VHT}$ denote the subspace of trees, allowing nodes with infinite outdegree. The subset $\fmT$ is closed and hence also compact.

Let $\om=(\omega_k)_k$ be a weight sequence with $\omega_0 > 0$ and $\omega_k > 0$ for some $k \ge 2$. Let $\cT_n$ denote the simply generated random tree with $n$ vertices. Let $\mGWT$ denote the modified Galton--Watson tree corresponding to the distribution $(\pi_k)_k$ defined in Section~\ref{sec:types}. 

\begin{theorem}[{Local limit of simply generated trees, \cite[Thm. 7.1]{MR2908619}}]
	\label{te:convsigen}
	 It holds that $\cT_n \convdis \mGWT$ in the metric space $\fmT$ as $n \equiv 1 \mod \spa(\om)$ tends to infinity.
\end{theorem}
Note that the limit object $\mGWT$ is almost surely locally finite if and only if the weight sequence has type I. In this case, convergence in $\fmT$ implies convergence in the local weak sense of Section~\ref{sec:lowe}.

\subsection{Scaling limits of simply generated trees}

\subsubsection{The continuum random tree}
The (Brownian) continuum random tree (CRT) is a random metric space constructed by Aldous in his pioneering papers \cite{MR1085326, MR1166406, MR1207226}. Its construction is as follows. To any continuous function $f: [0,1] \to [0, \infty[$  satisfying $f(0) = f(1) = 0$ we may associate a premetric $d$ on the unit interval $[0,1]$ given by
\[
d(u,v) = f(u) + f(v) - 2 \inf_{u \le s \le v} f(s)
\]
for $u \le v$. The corresponding quotient space $(\cT_f, d_{\cT_f}) = ([0,1]/ \mathord{\sim}, \bar{d})$, in which points with distance zero from each other are identified, is considered as rooted at the coset $\bar{0}$ of the point zero. This pointed metric space is an $\ndR$-tree, see  \cite{MR2147221, MR3025391} for the definition of $\ndR$-trees and further details. The CRT may be defined as the random pointed metric space \label{pp:crt}$(\CRT, d_{\CRT}, \bar{0})$ corresponding to Brownian excursion $\me = (\me_t)_{0 \le t \le 1}$ of duration one.

\subsubsection{Convergence toward the continuum random tree}
Depending on the weight sequence, the simply generated tree $\cT_n$ may or may not admit a scaling limit with respect to the Gromov--Hausdorff metric. In the case I$\alpha$), the tree $\cT_n$ is distributed like a critical Galton--Watson tree conditioned on having $n$ vertices, with the offspring distribution having finite non-zero variance.

\begin{theorem}[Scaling limit of simply generated trees, \cite{MR1207226}]
	If the weight-sequence $\mathbf{w}$ has type I$\alpha$, then
	\[
	(\cT_n, \frac{\sigma}{2} n^{-1/2} d_{\cT_n}, \emptyset) \convdis (\CRT, d_{\CRT}, \bar{0})
	\]
	in the Gromov--Hausdorff sense, with $\sigma$  given  in Equation~\eqref{eq:variance}.
\end{theorem}
This invariance principle is due to Aldous \cite{MR1207226} and there exist various extensions, see for example Duquesne \cite{MR1964956}, Duquesne and Le Gall \cite{MR2147221}, Haas and Miermont \cite{MR3050512}.

\subsubsection{Depth-first-search, height and width} \label{sec:dfs}
Suppose that the weight sequence $\mathbf{w}$ has type I$\alpha$. We are going to list a few known results that we are going to use frequently in our proofs later on. Addario-Berry, Devroye and Janson {\cite[Thm. 1.2]{MR3077536}} showed that there are constants $C,c>0$ such that for all $n$ and $h \ge 0$
\begin{align}
	\label{eq:gwttail}
	\Pr{\He(\cT_n) \ge h)} \le C \exp(-c h^2/n).
\end{align}	
Janson \cite[Problem 21.9]{MR2908619} posed the question, whether such a bound holds for all types of weight sequences. While this question has not been answered fully yet, significant progress was made in recent work by Kortchemski \cite{MR3651047, MR3335012}. A corresponding left-tail upper bound of the form
\begin{align}
	\label{eq:gwtlefttail}
\Pr{\He(\cT_n) \le h} \le C \exp(-c(n-2)/h^2)
\end{align}
for all $n$ and $h \ge 0$ is given in \cite[p. 6]{MR3077536}. The first moment of the number \label{pp:lk}$L_k(\cT_n)$ of all vertices $v$ with height $h_{\cT_n}(v)= k$ admits a bound of the form
\begin{align}
	\label{eq:ineqstr1}
	\Ex{L_k(\cT_n)} \le C k \exp(-ck^2/n).
\end{align}
for all $n$ and $k \ge 1$. See \cite[Thm. 1.5]{MR3077536}.

\begin{figure}[t]
	\centering
	\begin{minipage}{1.0\textwidth}
		\centering
		\includegraphics[width=0.22\textwidth]{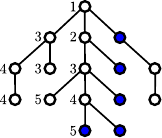}
	\caption{The lexicographic DFS-queue.}
	\label{fi:dfs}
	\end{minipage}
\end{figure}

Recall that the \emph{lexicographic depth-first-search (DFS)} of the plane tree $\cT_n$ is defined by listing the vertices in lexicographic order $v_0, v_1, \ldots, v_{n-1}$ and defining the queue $(Q_i)_{0 \le i \le n}$ by $Q_0=1$ and the recursion 
\[
Q_i = Q_{i-1} -1 + d_{\cT_n}^+(v_{i-1}).
\] Compare with Figure \ref{fi:dfs}, in which the numbers $Q_i$ are adjacent to the vertices $v_i$. We may also consider the \emph{reverse DFS} $(Q'_i)_{0 \le i \le n}$ as the DFS of the tree obtained from $\cT_n$ by reversing the ordering on each offspring set. Then $(Q_i)_i$ and $(Q'_i)_i$ agree in distribution and by \cite[Ineq.\ (4.4)]{MR3077536} there are constants $C,c>0$ such that
\begin{align}
\label{eq:dfsbound}
\Pr{\max_{j} Q_j \ge x} \le C \exp(-c x^2 /n)
\end{align}
for all $n$ and $x \ge 0$. Given a vertex $v$ of $\cT_n$ let $j$ and $k$ denote the corresponding indices in the DFS and reverse DFS. In particular, $v = v_j$ in the lexicographic ordering. Then
\begin{align}
\label{eq:queues}
Q_j + Q'_k = 2 + \sum_{u} d_{\cT_n}^+(u) - h_{\cT_n}(v)
\end{align} 
with the index $u$ ranging over all ancestors of the vertex $v$.

\section{Combinatorial species and weighted Boltzmann distributions}
\label{sec:combspec}
The language of combinatorial species was developed by Joyal \cite{MR633783} as a unified way to describe combinatorial structures and their symmetries. It provides a clean and powerful framework in which complex combinatorial bijection may be stated using simple algebraic terms. Rota predicted its rise in importance in various mathematical disciplines in the foreword of the book by Bergeron, Labelle and Leroux \cite{MR1629341}. The present work aims to make a contribution by showing its usefulness in combinatorial probability theory.

\subsection{Weighted combinatorial species}
\label{sec:preop}
We take a gentle approach in introducing the required notions, following \cite{MR633783,MR1629341}.
A {\em combinatorial species} $\cF$ is a rule that produces for each finite set $U$ a finite set $\cF[U]$ of $\cF$-objects and for each bijection $\sigma: U \to V$ a bijective map $\cF[\sigma]: \cF[U] \to \cF[V]$ such that the following properties hold.
\begin{enumerate}[1)]
	\item $\cF$ preserves identity maps, that is for any finite set $U$ it holds that \[\cF[\text{id}_U] = \text{id}_{\cF[U]}.\]
	\item $\cF$ preserves composition of maps, i.e. for any bijections of finite sets $\sigma:U \to V$ and $\sigma': V \to W$ we require that \[\cF[\sigma' \sigma] = \cF[\sigma']\cF[\sigma].\]
\end{enumerate}

A combinatorial species $\cF$ maps any finite set $U$ of {\em labels} to the finite set $\cF[U]$ of {\em $\cF$-objects}  and any bijection $\sigma: U \to V$ to the {\em transport function} $\cF[\sigma]$. 
For example, we may consider the species of finite graphs that maps any finite set $U$ to the set of graphs with vertex set $U$. In this context, the size of a graph is its number of vertices. Any bijection of finite sets is mapped to the relabelling bijection between the corresponding sets of graphs.

We are going to study random labelled $\cF$-objects over a fixed set, drawn with probability proportional to certain weights. To this end, we require the notion of a weighting of a species. Letting $\ndA=\ndR_{\ge 0}$ denote the non-negative real numbers, an {\em $\ndA$-weighted species} $\cF^\omega$ consists of a species $\cF$ and a {\em weighting} $\omega$ that produces for any finite set $U$ a map \[\omega_U: \cF[U] \to \ndA\] such that $\omega_U = \omega_V \circ \cF[\sigma]$ for any bijection $\sigma: U \to V$. Any object $F \in \cF[U]$ has {\em weight} $\omega_U(F)$ and we may form the {\em inventory} \[|\cF[U]|_\omega = \sum_{F \in \cF[U]} \omega_U(F).\] By abuse of notation we will often drop the index and write $\omega(F)$ instead of $\omega_U(F)$. Isomorphic structures have the same weight, hence we may define the weight of an unlabelled $\cF$-object  to be the weight of any representative. The inventory $|\tilde{\cF}[n]|_\omega$ is defined as the sum of weights of all unlabelled $\cF$-objects of size $n$. Any species may be considered as a weighted species by assigning weight $1$ to each structure, and in this case the inventory counts the number of $\cF$-objects. To any weighted species $\cF^\omega$ we associate its {\em exponential generating series} \[\cF^\omega(z) = \sum_{n\ge 0} |\cF[n]|_\omega z^n/n!. \]Two species $\cF$ and $\cG$ are termed {\em isomorphic}, denoted by $\cF \simeq \cG$, if there is a family $(\alpha_U)_U$ of bijections $\alpha_U: \cF[U] \to \cG[U]$, with the index $U$ ranging over all finite sets, such that the following diagram commutes for any bijection $\sigma: U \to V$ of finite sets.
\[
\xymatrix{ \cF[U] \ar[d]^{\alpha_U} \ar[r]^{\cF[\sigma]} &\cF[V]\ar[d]^{\alpha_V}\\
	\cG[U] \ar[r]^{\cG[\sigma]} 		    &\cG[V]}
\]
We say the family $(\alpha_U)_U$ is a {\em species isomorphism} from $\cF$ to $\cG$.

Two weighted species $\cF^{\omega}$ and $\cG^\nu$ are called isomorphic, if there exists a species isomorphism $(\alpha_U)_U$ from $\cF$ to $\cG$ that preserves the weights, that is, with $\nu(\alpha_U(F)) = \omega(F)$ for each finite set $U$ and $\cF$-object $F \in \cF[U]$. 

There are some natural examples of species that we are going to encounter frequently. The species \label{pp:set}$\Set$ with $\Set[U] = \{U\}$ has only one structure of each size and its exponential generating series is given by 
\[
	\Set(z) = \exp(z).
\]
The species \label{pp:seq}$\Seq$ of linear orders assigns to each finite set $U$ the set $\Seq[U]$ of tuples $(u_1, \ldots, u_t)$ of distinct elements with  $U=\{u_1, \ldots, u_t\}$. Its exponential generating series is given by
\[
	\Seq(z) = 1/(1-z).
\]
Finally, the species \label{pp:x}$\cX$ is given by $\cX[U] = \emptyset$ if $|U| \ne 1$ and $\cX[U] = \{U\}$ if $U$ is a singleton.

\subsection{Operations on species}
\label{sec:opspe}
Species may be combined in several ways to form new species. We discuss the the relevant operations following \cite{MR633783, MR1629341}.

\subsubsection{Products}
The {\em product} $\cF \cdot \cG$ of two species $\cF$ and $\cG$ is the species given by \[(\cF \cdot \cG)[U] = \bigsqcup_{(U_1, U_2)} \cF[U_1] \times \cG[U_2]\] with the index ranging over all ordered 2-partitions of $U$, that is, ordered pairs of (possibly empty) disjoint sets whose union equals $U$. The transport of the product along a bijection is defined componentwise. Given weightings $\omega$ on $\cF$ and $\nu$ on $\cG$, there is a canonical weighting on the product given by \[\mu(F,G) = \omega(F)\nu(G).\] This defines the product of weighted species
\[(\cF \cdot \cG)^\mu = \cF^\omega \cdot \cG^\nu.\]
The corresponding generating sums  satisfy \[(\cF \cdot \cG)^\mu(z) = \cF^\omega(z) \cG^\nu(z).\]

\subsubsection{Sums} Let $(\cF_i)_{i \in I}$ be a family of species such that for any finite set $U$ only finitely many indices $i$ with $\cF_i[U] \ne \emptyset$ exist. Then the {\em sum} $\sum_{i \in I} \cF_i$ is a species defined by \[(\sum_{i \in I} \cF_i)[U] = \bigsqcup_{i \in I} \cF_i[U].\] Given weightings $\omega_i$ on $\cF_i$, there is a canonical weighting $\mu$ on the sum given by \[\mu(F) = \omega_i(F)\] for any $i$ and $F \in \cF_i[U]$. This defines the sum of the weighted species
\[
(\sum_{i \in I} \cF_i)^\mu = \sum_{i \in I} \cF_i^{\omega_i}.
\]
The corresponding exponential generating series is given by \[
(\sum_i \cF_i^{\omega_i})(z) = \sum_i \cF_i^{\omega_i}(z).
\]

\subsubsection{Derived species}
Given a species $\cF$, the corresponding {\em derived} species $\cF'$ is given by
\[
\cF'[U] = \cF[U \cup \{*_U\}]
\]
with $*_U$ referring to an arbitrary fixed element not contained in the set~$U$. (For example, we could set $*_U = \{U\}$.) Any weighting $\omega$ on $\cF$ may also be viewed as a weighting on $\cF'$, by letting the weight of a derived object $F \in \cF'[U]$ be given by $\omega_{U \cup \{*_U\}}(F)$. The transport along a bijection $\sigma: U \to V$ is done by applying the transport $\cF[\sigma']$ of the bijection $\sigma': U \cup \{*_U\} \to V \cup \{*_V\}$ with $\sigma'|U = \sigma$. The generating series of the weighted derived species $(\cF^\omega)'$ is satisfies \[(\cF^\omega)'(z) = \frac{d}{dz} \cF^\omega(z).\] 

\subsubsection{Pointing} For any species $\cF$ we may form the {\em pointed} species $\cF^\bullet$. It is given by the product of species \[\cF^\bullet = \cX \cdot \cF'\] with $\cX$ denoting the species consisting of single object of size $1$. In other words, an $\cF^\bullet$-object is pair $(m, v)$ of an $\cF$-object $m$ and a distinguished label $v$ which we call the root of the object. Any weighting $\omega$ on $\cF$ may also be considered as a weighting on $\cF^\bullet$, by letting the weight of $(m,v)$ be given by $\omega(m)$. This choice of weighting is consistent with the natural weighting given by the product and derivation operation $\cX \cdot \cF'$, if we assign weight $1$ to the unique object of $\cX$. The corresponding exponential generating series is consequently given by
\[
(\cF^\bullet)^\omega(z) = z \frac{d}{dz} \cF^\omega(z).
\]
\subsubsection{Substitution} Given species $\cF$ and $\cG$ with  $\cG[\emptyset] = \emptyset$, we may form the composition $\cF \circ \cG$ as the species with object sets \[(\cF \circ \cG)[U] = \bigcup_\pi \left ( \{\pi \} \times \cF[\pi] \times \prod_{Q \in \pi} \cG[Q] \right ),\] with the index $\pi$ ranging over all unordered partitions of the set $U$. Here the transport $(\cF \circ \cG)[\sigma]$ along a bijection $\sigma: U \to V$ is done as follows. For any object $(\pi, F, (G_Q)_{Q \in \pi})$ in $(\cF \circ \cG)[U]$ define the partition
\[
\hat{\pi} = \{\sigma(Q) \mid Q \in \pi \},
\]
and let
\[
\hat{\sigma}: \pi \to \hat{\pi}
\]
denote the induced bijection betweenthe partitions. Then set
\[
(\cF \circ \cG)[\sigma](\pi, F, (G_Q)_{Q \in \pi}) = (\hat{\pi}, \cF[\hat{\sigma}](F), (\cG[\sigma|_Q](g_Q))_{ \sigma(Q) \in \hat{\pi}}).
\]
That is, the transport along the induced bijection of partitions gets applied to the $\cF$-object and the transports along the restrictions $\sigma|_Q$, $Q \in \pi$ get applied to the $\cG$-objects. Often, we are going to write $\cF(\cG)$ instead of $\cF \circ \cG$. Given a weighting $\omega$ on $\cF$ and a weighting $\nu$ on $\cG$, there is a canonical weighting $\mu$ on the composition given by \[\mu(\pi, F, (G_Q)_{Q \in \pi}) = \omega(F) \prod_{Q \in \pi} \nu(Q).\] This defines the composition of weighted species
\[
(\cF \circ \cG)^\mu = \cF^\omega \circ \cG^\nu.
\]
The corresponding generating series is given by
\begin{align}
\label{eq:cyccomp}
(\cF \circ \cG)^\mu(z) = \cF^\omega(\cG^\nu(z)).
\end{align}
\subsubsection{Restriction}
For any subset $\Omega \subset \ndN_0$ we may restrict a weighted species $\cF^\omega$ to objects whose size lies in $\Omega$. The result is denoted by $\cF^\omega_{\Omega}$. For convenience, we are also going to use the notation  $\cF^\omega_{\ge k}$ for the special case $\Omega = \{k, k+1, \ldots\}$, and define $\cF^\omega_{> k}$, $\cF^\omega_{< k}$, and $\cF^\omega_{\le k}$ analogously.

\subsubsection{Interplay between the  operators}
There are many natural isomorphisms that describe the interplay of the operations discussed in this section. The two most important are the product rule and the chain rule, which we are going to use frequently.
\begin{proposition}[Product rule and chain rule, \cite{MR633783}]
	\label{pro:chainprod}
	Let $\cF^\omega$ and $\cG^\nu$ be weighted species.
	\begin{enumerate}
		\item There is a canonical choice for an isomorphism
	\[
	(\cF^\omega \cdot \cG^\nu)' \simeq (\cF^\omega)'\cdot \cG^\nu + \cF^\omega \cdot(\cG^\nu)'.
	\]
		\item Suppose that $\cG[\emptyset]=\emptyset$. Then there is also a canonical isomorphism
	\[
		(\cF^\omega \circ \cG^\nu)' \simeq ((\cF^\omega)' \circ \cG^\nu) \cdot (\cG^\nu)'.
	\]
	\end{enumerate}
\end{proposition}
The product rule is easily verified, as the $*$-label in $(\cF^\omega \cdot \cG^\nu)'$ may either belong the $\cF$-structure, accounting for the summand $(\cF^\omega)' \cdot \cG^\nu$, or to the $\cG$-structure, accounting for the second summand. The idea behind the chain rule is that the partition class containing the $*$-label in an $(\cF^\omega \circ \cG^\nu)'$-structure distinguishes an atom of the $\cF$-structure. We refer the reader to the cited literature for details and further properties.

\subsection{Weighted Boltzmann distributions and samplers}
Boltzmann distributions appear naturally in the local limit of random discrete structures and in the limit of certain convergent Gibbs partitions. A {\em Boltzmann sampler} is a  procedure involving random choices that generates a structure according a Boltzmann distribution.

\subsubsection{Boltzmann distributions}
\label{sec:bodistr}
Let $\cF^\omega$ be a weighted species. For any parameter $y>0$ with $0< \cF^\omega(y) < \infty$ we may consider the  {\em Boltzmann distribution for labelled $\cF$-objects with parameter $y$}, given by
\begin{align}
\label{pp:bol}
\mathbb{P}_{\cF^\omega, y}(F) = \cF^\omega({y})^{-1}  \omega(F) \frac{ {y}^{|F|}}{|F|!}, \quad F \in  \bigsqcup_{m \ge 0} \cF[m].
\end{align}

\subsubsection{Boltzmann samplers}
\label{sec:WeBoSa}
The following lemma allows us to construct Boltzmann distributed random variables for the sum, product and composition of species. The results in this subsection are a straight-forward generalizations of corresponding results in a setting without weights, see for example Duchon, Flajolet, Louchard, and Schaeffer \cite{MR2095975} and Bodirsky, Fusy, Kang and Vigerske \cite[Prop. 38]{MR2810913}.

\begin{lemma}[Weigthed Boltzmann distributions and operations on species]
	\label{le:bole}
	\hspace{1pt}
\begin{enumerate} 
	\item Let $\cF^\omega$ and $\cG^\nu$ be weighted species, and let $X$ and $Y$ be independent random variables with distributions $\mathcal{L}(X) = \mathbb{P}_{\cF^\omega, y}$ and $\mathcal{L}(Y) = \mathbb{P}_{\cG^\nu, y}$. Then $(X,Y)$ may be interpreted as an $\cF \cdot \cG$-structure over the set $[|X|] \sqcup [|Y|]$. If $\alpha$ denotes a uniformly at random drawn bijection from this set to $[|X| + |Y|]$, then
	\[
		\cL \left ( (\cF \cdot \cG)[\alpha](X,Y) \right ) = \mathbb{P}_{\cF^\omega \cdot \cG^\nu, y}.
	\]
	\item Let $(\cF_i^{\omega_i})_{i \in I}$ be a family of weighted species, $y \ge 0$ a parameter with $\sum_i \cF_i^{\omega_i}(y) < \infty$,  and $(X_i)_{i \in I}$ a family of independent random variables with distributions $\cL(X_i) = \mathbb{P}_{\cF_i^{\omega_i}, y}$. If $K \in I$ gets drawn at random with probability proportional to $\cF_K^{\omega_K}(y)$, that is \[\Pr{K=k} = \cF_k^{\omega_k}(y) / \sum_i \cF_i^{\omega_i}(y),\] then 
	\[
		\cL(X_K) = \mathbb{P}_{\sum_i \cF_i^{\omega_i}, y}.
	\] 
	\item Let $\cF^\omega$ and $\cG^\nu$ be species such that $\cG^\nu(0) = 0$ and let $y > 0$ be parameter with $0< \cG^\nu(y) < \infty$ and $0 < \cF^\omega(\cG^\nu(y)) < \infty$. Let $X$ be a $\mathbb{P}_{\cF^\omega, \cG^\nu(y)}$-distributed random $\cF$-object and $(Y_i)_{i \in \ndN}$ a family of independent $\mathbb{P}_{\cG^\nu, y}$-distributed random $\cG$-objects, that are also independent of $X$. Then $(X, Y_1, \ldots, Y_{|X|})$ may be interpreted as an $\cF \circ \cG$-object with partition $\{ [|Y_i|] \times \{i\} \mid 1 \le i \le |X|\}$. Let $\alpha$ denote a uniformly at random drawn bijection from the underlying set to the set $[|Y_1| + \ldots + |Y_{|X|}|]$. Then
	\[
	\cL( (\cF \circ \cG)[\alpha](X, Y_1, \ldots, Y_{|X|})) = \mathbb{P}_{\cF^\omega \circ \cG^\nu, y}.
	\]
\end{enumerate}
\end{lemma}

\section{Probabilistic tools}
\label{sec:prob}
For ease of reference, we explicitly state a selection of classical results that we are going to use in our proofs.

\subsection{Projective limits of  probability spaces}

Let $(I, \preccurlyeq)$ be a \emph{directed} non-empty set. That is, we assume that the relation $\preccurlyeq$ is reflexive and transitive, and every pair of elements in $I$ has an upper bound. Let $(X_i, \scT_i)_{i \in I}$ be a family of topological spaces. Suppose that for each pair $i, j \in I$ with $i \preccurlyeq j$ we are given a continuous map \[
f_{i,j}: X_j \to X_i,\]
such that $f_{i,i} = \id_{X_i}$ for all $i$, and for all $i \preccurlyeq j \preccurlyeq k$ the diagram
\[
\xymatrix{ X_k  \ar[r]^{f_{j,k}} \ar[dr]^{f_{i,k}} 
	&X_j \ar[d]^{f_{i,j}}\\
	&X_i}
\]
commutes. The system $((X_i, \scT_i)_{i \in I}, (f_{i,j})_{i \preccurlyeq j})$ is termed a \emph{projective system} of topological spaces.

Let $(X, \scT_X)$ be a topological space, and for each $i \in I$ let  $f_i: X \to X_i$ be a continuous map. Suppose that for all $i \preccurlyeq j$ the diagram
\begin{align}
\label{dia:property}
\xymatrix{ X  \ar[r]^{f_{j}} \ar[dr]^{f_{i}} 
	&X_j \ar[d]^{f_{i,j}}\\
	&X_i}
\end{align}
commutes. The space $(X, \scT_X)$ is termed a \emph{projective limit} of the system $((X_i, \scT_i)_{i \in I}, (f_{i,j})_{i \preccurlyeq j})$, if for any topological space $(Y, \scT_Y)$ and any family of continuous maps $g_i: Y \to Y_i$ that also satisfy~\eqref{dia:property} there is a unique continuous map $f: Y \to X$ such that for all $i \preccurlyeq j$ the diagram
\begin{align}
\label{dia:univ}
	\xymatrix{
			& Y \ar[d]^>>>>{f} 	\ar[ddl]_{g_j} \ar[ddr]^{g_i}	& 			\\
			& X \ar[dl]^{f_j} \ar[dr]_{f_i}				&			\\
		X_j \ar[rr]^{f_{i,j}}  	&				& X_i
}
\end{align}
commutes. In particular, between any two projective limits there is a canonical homeomorphism that is compatible with the projections of the system. 

The projective limit always exist. We may define the space $X$ as the subset $X \subset \prod_{i \in I} X_i$ of all families $x = (x_i)_i$ that satisfy $x_i = f_{i,j}(x_j)$ for all $i \preccurlyeq j$. For each $i$ we let  $f_i: X \to X_i$ denote the projection to the $i$th coordinate. Let $\scT$ denote the smallest topology on $X$ that makes all projections $f_i$ continuous. Then the space $(X, \scT)$ together with $(f_i)_{i \in I}$ is a projective limit of the system $((X_i, \scT_i)_{i \in I}, (f_{i,j})_{i \preccurlyeq j})$. 


If we equip each of the topological spaces $(X_i, \scT_i)$ with its Borel $\sigma$-algebra $\sigma(\scT_i)$, then the maps $f_{i,j}$ become measurable. The smallest $\sigma$-algebra on the projective limit $(X, \scT)$ that makes all  projections $f_i$ measurable coincides with its Borel $\sigma$-algebra $\sigma(\scT)$.

Suppose that for each $i \in I$ we are given a probability measure $\mu_i$ on $\sigma(\scT_i)$. We say $(\mu_i)_{i \in I}$ is a \emph{projective family}, if for all $i \preccurlyeq j$ the measure $\mu_i$ is the image measure of $\mu_j$ under $f_{i,j}$. That is, for each event $A \in \sigma(\scT_i)$ we require  that $\mu_i(A) = \mu_j(f^{-1}_{i,j}(A))$. 

\begin{lemma}[{\cite[Ch. 9, \S 4, No. 3, Theorem 2]{MR0276436}}]
	\label{le:project}
	Let $(X_i, \scT_i)_{i \in I}$, $(f_{i,j})_{i \preccurlyeq j}$ be a projective system of topological spaces, and $(\mu_i)_{i \in I}$ a projective family of probability measures on the Borel $\sigma$-algebras $\sigma(\scT_i)$, $i \in I$. If the index set $I$ is countable, then there exists a  probability measure $\mu$ on the projective limit $(X, \scT)$ such that for all $i \in I$ the measure $\mu_i$ is the image of $\mu$ under the projection $f_i$.
\end{lemma}

\subsection{A central local limit theorem}
The following lattice version of the local limit theorem for sums of independent random variables is taken from Durrett's book.

\begin{lemma}[{\cite[Ch. 3.5]{MR2722836}}]
	\label{le:llt1dim}
	Let $(X_n)_n$ a family of independent identically distributed random integers with first moment $\mu = \Ex{X_1}$ and finite non-zero variance $\sigma^2 = \Va{X_1}$. Let $d \ge 1$ denote the smallest integer such that the support
	$
	\{k \mid \Pr{X_1=k}\}
	$
	is contained in a lattice of the form $a + d \ndZ$ for some $a\in \ndZ$. Then the sum $S_n = X_1 + \ldots X_n$ satisfies the local limit theorem
	\[\lim_{n \to \infty} \sup_{x \in  a + d\ndZ} \left | \sqrt{n}\Pr{S_n = x} - \frac{d}{\sqrt{2 \pi \sigma^2}}\exp(-\frac{(n \mu -x)^2}{2n\sigma^2}) \right | = 0.
	\]
\end{lemma}

\subsection{A deviation inequality}
\label{sec:deviation}
The following deviation inequality is found in most textbooks on the subject.

\begin{lemma}[Medium deviation inequality for one-dimensional random walk]
	\label{le:deviation}
	Let $(X_i)_{i \in \ndN}$ be an i.i.d. family of real-valued random variables with $\Ex{X_1} = 0$ and $\Ex{e^{t X_1}} < \infty$ for all $t$ in some open interval containing zero. Then there are constants $\delta, c>0$ such that for all $n\in \ndN$, $x \ge 0$ and $0 \le\lambda\le\delta$ it holds that \[\Pr{|X_1 + \ldots + X_n| \ge x} \le 2 \exp(c n \lambda^2 - \lambda x).\]
\end{lemma}
The proof is by observing that $\Ex{e^{\lambda |X_1|}} \le 1 + c\lambda^2$ for some constant $c$ and sufficiently small $\lambda$, and applying Markov's inequality to the random variable $\exp(\lambda(|X_1| + \ldots + |X_n|))$.

\section{A probabilistic study of tree-like discrete structures}
\label{sec:mainstudy}

In this section, we develop a framework for random enriched trees and present our main results as well as their applications to specific models of random discrete structures. 

\vspace{10pt}

\paragraph*{\em Index of notation} 

The following list summarizes frequently used terminology in this section.

\begin{center}

\begin{tabularx}{\linewidth}{ll}
$\cR^\kappa$ & $\kappa$-weighted species of $\cR$-structures, page \pageref{mark:rk}\\[2pt]
$\cA_\cR^\omega$ & $\omega$-weighted species of $\cR$-enriched trees, page \pageref{mark:ao}\\[2pt]
$\mA_n^\cR$ & random $n$-sized $\cR$-enriched tree, page \pageref{mark:an}\\[2pt]
$(\cT_n, \beta_n)$ & random $n$-vertex $\cR$-enriched plane tree coupled to $\mA_n^\cR$, page \pageref{le:coupling}\\[2pt]
$(\hat{\cT}, \hat{\beta})$ & random modified $\cR$-enriched plane tree with a spine, page \pageref{te:local}\\[2pt]
$(\cT^*, {\beta}^*)$ & another random modified $\cR$-enriched plane tree with a spine  that\\& grows backwards, page \pageref{de:tinf}\\[2pt]
$f(A,x)$ & the enriched fringe subtree of a vertex $x$ in an enriched \\& tree $A$, page \pageref{mark:fr} \\[2pt]
$(T,\alpha)^{[k]}$ & the enriched tree pruned at height $k$, page \pageref{pp:trim} \\[2pt]
 $\mathbf{w}=(\omega_k)_k$& weight-sequence associated to $\cR^\kappa$ and $\cT_n$, page \pageref{sec:partB}\\[2pt]
 $\phi(z)$ & generating series of $\mathbf{w}$, page \pageref{mark:p1}\\[2pt]
 $\rho_\phi$ & radius of convergence of $\phi(z)$, page \pageref{mark:p2}\\[2pt]
 $\psi$ & series $\psi(t) = t\phi'(t)/\phi(t)$, page \pageref{mark:p3}\\[2pt]
 $\tau$ & limit $\tau = \lim_{t \uparrow \rho_\phi} \psi(t)$, page \pageref{mark:p4}\\[2pt]
 $\nu$ & maximal first moment of probability weight sequences \\&equivalent to $\mathbf{w}$, page \pageref{mark:p5}\\[2pt]
 $(\pi_k)_k$ & canonical probability weight sequence equivalent to $\mathbf{w}$, page \pageref{mark:p6}\\[2pt]
  $\xi$ & random variable with distribution $(\pi_k)_k$, page \pageref{pp:xi}\\[2pt]
    $\hat{\xi}$ & size-biased version of $\xi$, page \pageref{pp:xih}\\[2pt]
      $\VHT$ & vertex set of the Ulam--Harris tree, page \pageref{pp:vht}\\[2pt]
 I, I$a$, I$\alpha$, I$\beta$, II, III& types of weight sequences, page \pageref{mark:we}\\[2pt]
    $d_T^+(x)$ & outdegree of a vertex $x$ in a rooted tree $T$, page \pageref{pp:dout}\\[2pt]
    $d_G(x)$ & degree of a vertex $x$ in a graph $G$, page \pageref{pp:deg}\\[2pt]
    $d(G^\bullet)$ & degree of the root-vertex in a rooted graph $G^\bullet$, page \pageref{pp:deg}\\[2pt]
   $d_G(x,y)$ & graph-distance between $x,y \in G$, page \pageref{pp:dg}\\[2pt]
   $d_{\textsc{BLOCK}}$ & block-metric, page \pageref{pp:db}\\[2pt]
   $d_{\textsc{FPP}}$ & first-passage-percolation metric, page \pageref{pp:fpp}\\[2pt] 
  $V_k(\cdot)$ & graph metric $k$-neighbourhood, page \pageref{pp:uk}\\[2pt]
  $U_k(\cdot)$ & block metric $k$-neighbourhood, page \pageref{pp:vk}\\[2pt]
 $\mathbf{Ex}(\cdot)$& graph class defined by excluded minors, page \pageref{pp:ex}\\[2pt]
 $\cC$ & species of connected graph, page \pageref{pp:c}\\[2pt]
  $\mC_n^\omega$ & random $n$-vertex connected graph, page \pageref{pp:cn}\\[2pt]
 $\cB$ & species of $2$-connected graph, page \pageref{pp:b}\\[2pt]
  $\cB_n^\gamma$ & random $n$-vertex $2$-connected graph, page \pageref{pp:bn}\\[2pt]
 $\Set$ & exponential species, page \pageref{pp:set}\\[2pt]
 $\Seq$ & species of linear orders, page \pageref{pp:seq}\\[2pt]
 $\cX$ & single point species, page \pageref{pp:x}\\[2pt]
 $\cD$ & species of edge-rooted dissections of polygons, page \pageref{pp:d}\\[2pt]
 $\mD^\omega_n$ & random dissection of an $n$-gon, page \pageref{pp:dn}\\[2pt]
 $\cO$ & species of simple outerplanar maps, page \pageref{pp:o}\\[2pt]
  $\mO_n^\omega$ & random $n$-vertex outerplanar map, page \pageref{pp:on}\\[2pt]
 $\cM$ & species of planar maps, page \pageref{pp:m}\\[2pt]
  $\mM_n^\omega$ & random planar map with $n$ edges, page \pageref{pp:mn}\\[2pt]
 $\cQ$ & species of $2$-connected planar maps, page \pageref{pp:q}\\[2pt]
  $\mQ_n^\kappa$ & random $2$-connected planar map with $n$ edges, page \pageref{pp:qn}\\[2pt]
 $\cK$ & species of $k$-trees, page \pageref{pp:k}\\[2pt]
$\mK_n$ & uniform random $k$-tree with $n$ hedra, page \pageref{pp:kn}\\[2pt]
 $\cK^\circ$ & species of front-rooted $k$-trees, page \pageref{pp:kc}\\[2pt]
 $\cK_1^\circ$ & species of front-rooted $k$-trees where the root-front is contained in \\& a unique hedra, page \pageref{pp:kc1}\\[2pt]
  $K_k$& the complete graph with $k$ vertices, page \pageref{pp:com}\\[2pt]
    $\mathbb{P}_{\cF^\omega, y}$& Boltzmann distribution for a weighted species $\cF^\omega$ with \\& parameter $y$, page \pageref{pp:bol}\\[2pt]
  $\hat{\mC}$ & Benjamini--Schramm limit of random graphs, pages \pageref{pp:ch},\pageref{pp:ch2}\\[2pt]
  $\hat{\mO}$ & distributional limit of outerplanar maps, pages \pageref{pp:oh},\pageref{pp:oh2}\\[2pt]
  $\hat{\mO}_*$ & Benjamini--Schramm limit of outerplanar maps, page \pageref{pp:ohs}\\[2pt]
  $\hat{\mD}$ & Benjamini--Schramm limit of dissections, page \pageref{pp:dh}\\[2pt]
  $\hat{\mK}$ & Benjamini--Schramm limit of $k$-trees, page \pageref{pp:hk}\\[2pt]
  $L_k(\cdot)$& number of vertices with height $k$ in a rooted graph, page \pageref{pp:lk}\\[2pt]
\end{tabularx}

\end{center}


\subsection{Prominent examples of weighted $\cR$-enriched trees}
\label{sec:intro1}

In this section, we state the formal definition of  $\mathcal{R}$-enriched which were introduced by Labelle~\cite{MR642392} using the language of combinatorial species by Joyal \cite{MR633783}. These notions allow for a unified treatment of a large class of combinatorial objects. We introduce a model of random enriched trees and explain how this generalizes many well-known models of random discrete structures.

As a motivation, consider the species $\cA$ of rooted unordered trees. Any such tree consists of a root vertex together with an unordered list of rooted trees attached to it. This may be expressed in the grammar of Section~\ref{sec:opspe} by an isomorphism
\begin{align}
\label{eq:tmp1}
\cA \simeq \cX \cdot \Set(\cA),
\end{align}
with $\cX$ denoting the species consisting of a single object of size $1$, and $\Set$ the species having a single object of size $k$ for each $k \in \ndN_0$. Enriched trees are rooted trees where the offspring set of each vertex is decorated with an additional structure. They are characterized by a similar isomorphism as \eqref{eq:tmp1}. Let $\cR$ be a combinatorial species. The species of {\em $\cR$-enriched trees} $\cA_\cR$ is constructed as follows. For each finite set $U$ let $\cA_\cR[U]$ be the set of all pairs $(A, \alpha)$ with $A \in \cA[U]$ a rooted unordered tree with labels in $U$, and $\alpha$ a function that assigns to each vertex $v$ of $A$ with offspring set $M_v$ an $\cR$-structure $\alpha(v) \in \cR[M_v]$. The transport along a bijection $\sigma: U \to V$ relabels the vertices of the tree and the $\cR$-structures on the offspring sets accordingly. That is, $\cA_\cR[\sigma]$ maps the enriched tree $(A, \alpha)$ to the tree $(B, \beta)$ with $B=\cA[\sigma](A)$ and $\beta(\sigma(v)) = \cR[\sigma|_{M_{v}}](\alpha(v))$ for each $v \in A$. Analogous to \eqref{eq:tmp1}, the species of $\cR$-enriched trees satisfies an isomorphism
\begin{align}
\label{eq:renr}
\cA_\cR \simeq \cX \cdot \cR(\cA_\cR),
\end{align}
as any $\cR$-enriched tree consists of a root vertex (corresponding to the factor $\cX$) together with an $\cR$-structure, in which each atom is identified with the root of a further $\cR$-enriched tree. Conversely, Joyal's theorem of implicit species \cite[Thm. 6]{MR633783} ensures that given any species $\cF$ with an isomorphism $\cF \simeq \cX \cdot \cR(\cF)$, there is a natural choice of an isomorphism
$
\cF \simeq \cA_\cR.
$ As the examples below show, many classes of combinatorial objects that have been studied by both combinatorialists and probabilists admit a decomposition as in \eqref{eq:renr}, and may hence be treated in a unified way by working with enriched trees.

We consider weightings on the enriched trees that are based on weights on the $\cR$-structures. Let $\kappa$ be a weighting on the species $\cR$. \label{mark:rk} Then we obtain a weighting $\omega$ on the species $\cA_\cR$ given by \label{mark:ao}
\begin{align}
\label{eq:omw}
\omega(A, \alpha) = \prod_{v \in A} \kappa(\alpha(v)).
\end{align}
This weighting is consistent with the isomorphism in \eqref{eq:renr}, that is,
\begin{align}
\label{eq:weighting}
\cA_\cR^\omega \simeq \cX \cdot \cR^\kappa(\cA_\cR^\omega).
\end{align}

\label{mark:an} For each integer $n$ with  $|\cA_\cR[n]|_\omega >0$ we may consider the random labelled enriched tree $\mA_n^\cR$, drawn with probability  proportional to its weight among all labelled objects from $\cA_\cR[n]$. In the following, we illustrate how this models of random enriched trees generalizes a large variety of random combinatorial objects.

\subsubsection{Simply generated (plane) trees}
A natural example of a random enriched tree is the simply generated plane tree $\cT_n$ discussed in Section~\ref{sec:simplygen}. Given a weight sequence $\om = (\omega_k)_{k \in \ndN_0}$ of non-negative real numbers with $\omega_0>0$ and $\omega_k >0$ for at least one $k \ge 2$, we may consider the weighting $\kappa$ on the species $\cR = \Seq$ of linear orders that assigns weight $\omega_k$ to each linear order on a $k$-element set. Then $\cA_\cR$ is the species of ordered rooted trees and $\mA_n^\cR$ is the random simply generated plane tree $\cT_n$ with $n$ vertices. Strictly speaking, the vertices of $\mA_n^\cR$ are additionally labelled from $1$ to $n$, but as any plane tree with $n$ vertices has $n!$ different labellings, this does not make a difference. 

\subsubsection{Random block-weighted graphs}
\label{sec:bijblock}

Let \label{pp:c}$\cC$ denote the species of connected graphs and \label{pp:b}$\cB$ the subspecies of graphs that are $2$-connected or consist of two distinct vertices joined by an edge.
There is a well-known decomposition 
\begin{align}
\label{eq:blode}
\cC^\bullet \simeq \cX \cdot \Set(\cB'(\cC^\bullet))
\end{align}
illustrated in Figure~\ref{fi:enrblockgraph}, that allows us to identify the species  $\cC^\bullet$  of rooted connected graphs with $\Set \circ \cB'$-enriched trees. That is, rooted trees, in which each offspring set gets partitioned, and each partition class $Q$ carries a $\cB'$-structure, that has $|Q| +1$ vertices, as the $*$-vertex receives no label. The isomorphism \eqref{eq:blode} can be found for example in Harary and Palmer \cite[1.3.3, 8.7.1]{MR0357214},  Robinson \cite[Thm. 4]{MR0284380}, and Labelle \cite[2.10]{MR699986}. 

The idea behind \eqref{eq:blode} is the block-decomposition of connected graphs. A block of a graph $G$ is a maximal connected subgraph that does not contain a cutvertex of itself, that is, deleting any vertex does not disconnect the block. Any edge of the graph lies in precisely one block and any two blocks may intersect in at most one vertex. The cutvertices of $G$ are precisely the vertices that belong to more than one block, see for example Diestel's book on graph theory \cite[Ch.3]{MR2744811}. Hence any rooted graph consists of the root-vertex (accounting for the factor $\cX$ in \eqref{eq:blode}), and an unordered list of blocks incident to the root vertex, where at each non-root vertex a further rooted graph is inserted (accounting for the factor $\Set \circ \cB' \circ \cC^\bullet$).

\begin{figure}[t]
	\centering
	\begin{minipage}{1.0\textwidth}
		\centering
		\includegraphics[width=1.0\textwidth]{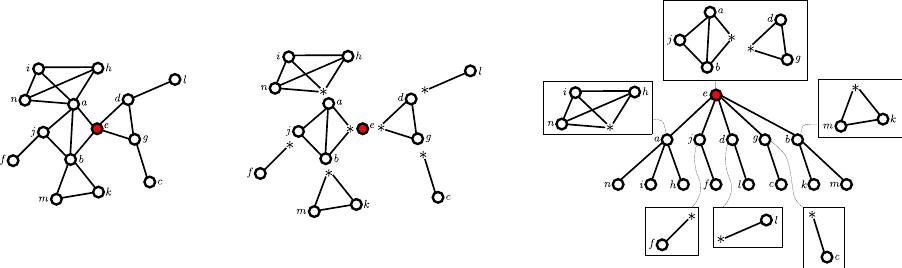}
		\caption{Correspondence of rooted connected graphs and enriched trees.}
		\label{fi:enrblockgraph}
	\end{minipage}
\end{figure}

If we fix a weighting $\gamma$ on $\cB$, we may consider the weighting $\omega$ on $\cC$ that assigns weight $\omega(C) = \prod_B \gamma(B)$ to any graph $C$, with the index $B$ ranging over the blocks of $C$. The random graph $\mC_n^\omega$ drawn from $\cC[n]$ with probability proportional to its $\omega$-weight is distributed like the random enriched tree $\mA_n^\cR$ for the weighted species $\cR^\kappa=(\Set \circ \cB')^\kappa$, with $\kappa$ assigning the product of the $\gamma$-weights of the individual classes to any assembly of $\cB'$-structures. Note that formally $\mA_n^\cR$ is a random rooted graph from $\cC^\bullet[n]$, but we may simply drop the root in order to obtain \label{pp:cn} $\mC_n^\omega$.

If we set the $\gamma$-weights of blocks to zero, we obtain random connected graphs from so called  {\em block-stable} classes, that is, classes of graphs defined by placing constraints on the allowed blocks. A well-known example is the class of planar graphs, where each graph (equivalently, each block of the graph) is required to admit an embedding in the complex plane, such that any two distinct edges may only intersect at their endpoints. More generally, any class of graphs \label{pp:ex}$\mathbf{Ex}(M)$ that may be defined by excluding a set $M$ of $2$-connected minors is also block-stable. Here a {\em minor} of a graph $G$ refers to any graph that may be obtained from $G$ by repeated deletion and contraction of edges. Kuratowski's theorem \cite[Thm. 4.4.6]{MR2744811} states that any graph is planar if and only if it does not admit the complete graph $K_5$ or the complete bipartite graph $K_{3,3}$ as minor, identifying the class of planar graphs with $\mathbf{Ex}(K_5, K_{3,3})$. Further prominent examples are outerplanar graphs ($\mathbf{Ex}(K_4, K_{2,3})$), that may be drawn in the plane such that each vertex lies on the frontier of the infinite face, and series-parallel graphs ($\mathbf{Ex}(K_4)$), that may be constructed similar to electric networks in terms of repeated serial and parallel composition. These two classes fall under the more general setting of random graphs from {\em subcritical} block-classes in the sense of  Drmota, Fusy, Kang, Kraus and Ru\'e \cite{MR2873207}, which also are special cases of the random graph $\mC_n^\omega$.

\subsubsection{Random dissections of polygons and  Schr\"oder enriched parenthesizations}

\label{sec:diss}

\begin{figure}[t]
	\centering
	\begin{minipage}{1.0\textwidth}
		\centering
		\includegraphics[width=0.4\textwidth]{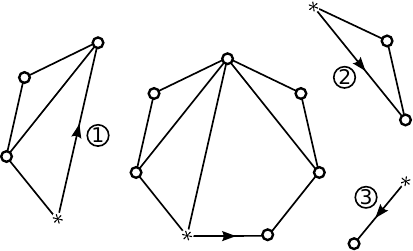}
	\caption{Decomposition of edge-rooted dissections of polygons.}
	\label{fi:disspol}
	\end{minipage}
\end{figure}

Consider a convex polygon $P$ in the complex plane, whose corners are the $n$-th roots of unity. If we add an arbitrary number of diagonals to $P$ in such a way, that different diagonals may only intersect at their endpoints, we obtain a {\em dissection} of $P$. We may interpret dissections of polygons as simple rooted planar maps, by distinguishing the edge from $1$ to $\exp(2\pi i/n)$. Let \label{pp:d}$\cD$ denote the class of edge-rooted dissections of polygons. It will be convenient to define the size of any $\cD$-object to be the number of non-root vertices (that is, vertices different from the origin of the root-edge), and allow a "degenerate" dissection consisting of a single root edge to be an element of $\cD$.

Any dissection consists of a root face, where each non-root edge  is identified with the root edge of a smaller dissection. Hence any element $D \in \cD$ where the root-face has degree $k$, may be interpreted as an ordered sequence $(D_1,\ldots, D_{k-1})$ of $k-1$ smaller $\cD$-objects. Since we do not count root vertices, the  size of $D$ agrees with the sum of the sizes of the $D_i$. This yields an isomorphism
\begin{align}
\label{eq:decompdis}
\cD \simeq \cX + \Seq_{\ge 2} \circ \cD,
\end{align}
with the summand $\cX$ corresponding to a single root-edge, and $\Seq_{\ge \ell}$ denoting the species of linear orders with length at least $\ell$. Compare with Figure~\ref{fi:disspol}, where the root vertices are depicted as a $*$-placeholders, in order to illustrate that they do not count as regular vertices. The isomorphism in \eqref{eq:decompdis} is a slight modification of a decomposition established by Bernasconi, Pangiotou and Steger \cite[Eq. (3.1)]{MR2789731}.

Given a species $\cN$ with no structures of size zero or one, we may consider the species $\cS_{\cN}$ of {\em Schr\"oder $\cN$-enriched parenthesizations}. For any finite set $U$, a structure in $\cS_{\cN}[U]$ can be described as a rooted tree whose leaves are labelled with elements of $U$, such that to each (unlabelled) internal vertex $v$ with offspring set $M_v$ an $\cN$-structure from $\cN[M_v]$ is assigned. The species $\cS_{\cN}$ satisfies an isomorphism of the form
\begin{align}
\label{eq:schr}
\cS_{\cN} \simeq \cX + \cN \circ \cS_\cN,
\end{align}
see Ehrenborg and M\'endez \cite[Def. 2.1]{MR1284403}. Joyal's theorem of implicit species \cite[Thm. 6]{MR633783} ensures that given any species $\cS$ with an isomorphism $\cS \simeq \cX + \cN(\cS)$, there is a natural choice of an isomorphism
$
\cS \simeq \cS_\cN.
$
In particular, \eqref{eq:decompdis} allows us to identify the class $\cD$ of edge-rooted dissections of polygons with $\Seq_{\ge 2}$-enriched Schr\"oder parenthesizations.

\begin{figure}[t]
	\centering
	\begin{minipage}{1.0\textwidth}
		\centering
		\includegraphics[width=0.95\textwidth]{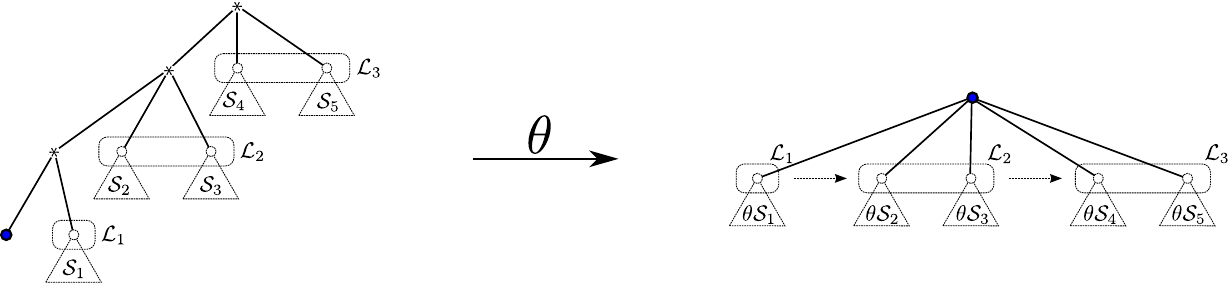}
		\caption{The Ehrenborg--M\'endez transformation of Schr\"oder ($\cX \cdot \cL$)-enriched parenthesizations into $\Seq\circ\cL$-enriched trees.}
		\label{fi:ehrenb}
	\end{minipage}
\end{figure}

Suppose that each object of the species $\cN$ admits a canonical point of reference, that is, $\cN \simeq \cX \cdot \cL$ for some species $\cL$. Ehrenborg and M\'endez \cite[Prop. 2.1]{MR1284403} showed that there is an isomorphism
\begin{align}
	\label{eq:schrenr}
	\cS_{\cN} \simeq \cA_{\Seq \circ \cL}
\end{align}
which identifies Schr\"oder $\cN$-enriched parenthesizations with $\Seq \circ \cL$-enriched trees. The idea is that any $\cS_{\cN}$-object  consists of a leaf-labelled tree where each unlabelled internal vertex $v$ with offspring set $M_v$ has a preferred son $v_0 \in M_v$ and an $\cL$-structure on $M_v \setminus \{v_0\}$. Starting at the root, we may follow the preferred sons until reaching a leaf, and the $\cL$-structures along that path form a $\Seq\circ \cL$-structure that we assign to the label of the leaf. Compare with Figure~\ref{fi:ehrenb}. 
Each atom of the $\Seq \circ \cL$-structure is the root of a smaller Schr\"oder $\cN$-enriched parenthesization, hence we may continue in this way until the whole parenthesization got explored, yielding a $\Seq \circ \cL$-enriched tree.

We may choose the last element of any $\Seq_{\ge2}$-structure as its point of reference, yielding an isomorphism between $\Seq_{\ge2}$ and $\cX \cdot \Seq_{\ge 1}$. Hence the isomorphism \eqref{eq:schrenr} allows us to identify the species $\cD$ of edge-rooted dissections of polygons with $\Seq \circ \Seq_{\ge 1}$-enriched trees. That is,
\begin{align}
	\label{eq:isodis}
	\cD \simeq \cX \cdot \Seq(\Seq_{\ge 1}(\cD)).
\end{align}
Here any $\Seq \circ \Seq_{\ge 1}$-structure corresponds to a dissection of a polygon, where each diagonal must be incident with the destination of the root-edge, and the vertices incident to the root-edge do not count as regular vertices. See  Figure~\ref{fi:decdis} for an illustration of the correspondence \eqref{eq:isodis}.

\begin{figure}[t]
	\centering
	\begin{minipage}{1.0\textwidth}
		\centering
		\includegraphics[width=1.0\textwidth]{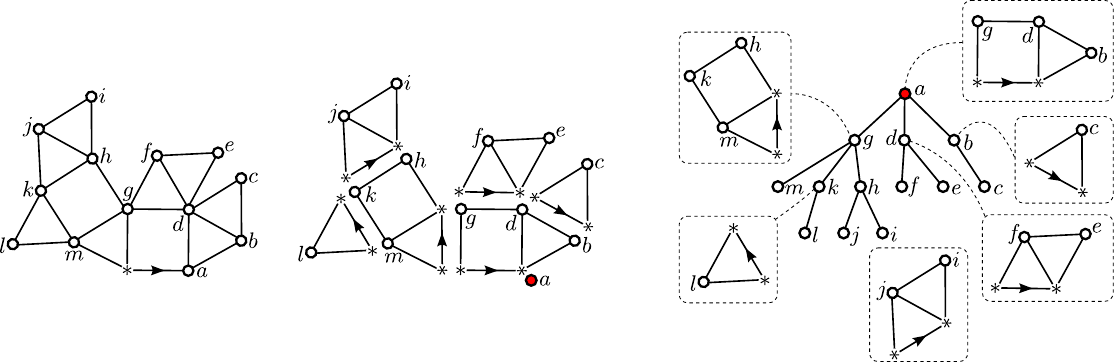}
		\caption{Correspondence of edge-rooted dissections of polygons and enriched trees}
		\label{fi:decdis}
	\end{minipage}
\end{figure}

Given a sequence of non-negative weights $\gamma_3, \gamma_4 \ldots$ with $\gamma_k >0$ for at least one $k$, we may assign to each dissection  \label{pp:dn}$D$ of a polygon the weight \[\omega(D) = \prod_F \gamma_{|F|},\] with the index $F$ ranging over the inner faces of $D$, and $|F|$ denoting the face-degree. The random dissection $\mD_n^\omega$ of an $n$-gon that gets drawn with probability proportional to its $\omega$-weight is distributed like the random enriched tree $\mA_{n-1}^\cR$ for the weighted species $\cR^\kappa = \Seq \circ \Seq_{\ge 1}^\gamma$ with the weighting $\gamma$ given by $\Seq_{\ge 1}^\gamma(z) = \sum_{k=1}^\infty \gamma_{k+2} z^k$. This model of a random plane graph has received some attention in recent literature. A particular highlight is the work by Curien, Haas and Kortchemski \cite{MR3382675}, who established the continuum random tree as the scaling limit of $\mD_n^\omega$, if the weight-sequence $(\gamma_k)_k$ satisfies certain conditions.

\subsubsection{Random outerplanar maps with face weights or block weights}
\label{sec:decompouter}

Half-edge rooted planar maps are so called {\em asymmetric} objects. That is, any object with $n$ vertices may be labelled in precisely $n!$ ways, using a fixed $n$-element set of labels. Hence it makes no difference, whether we treat random labelled or unlabelled maps. In the following we are going to work with classes of labelled maps, in order to stay consistent with the framework of the present paper.

Let \label{pp:o}$\cO$ denote the class of rooted simple outerplanar maps with vertices as atoms. Moreover, let $\cD$ denote the class of rooted non-separable simple outerplanar maps, in which the origin of the root-edge is replaced by a $*$-vertex that does not contribute to the size of the maps. Any non-separable simple outerplanar map with at least $3$ vertices has a unique Hamilton cycle given by the boundary of the outer face. Hence $\cD$ is the class of dissections of edge-rooted  polygons.

The class of simple outerplanar maps admits a tree-like decomposition according to the blocks, which was established in Stufler \cite{MR3634279}. Any such map can be constructed in a unique manner as follows. Start with a root vertex, then take an ordered (possibly empty) sequence of dissections and glue them together at the root vertex in a counter-clockwise way. The root edge of the first map in the sequence becomes the root edge of the resulting map, and we declare root vertex as marked. For each unmarked vertex left, take another ordered (possibly empty) sequence of dissections, glue them together in a counterclockwise way at that vertex, and finally declare that vertex as marked. Repeat the last step, until no unmarked vertices are left.

\begin{figure}[t]
	\centering
	\begin{minipage}{1.0\textwidth}
		\centering
		\includegraphics[width=1.0\textwidth]{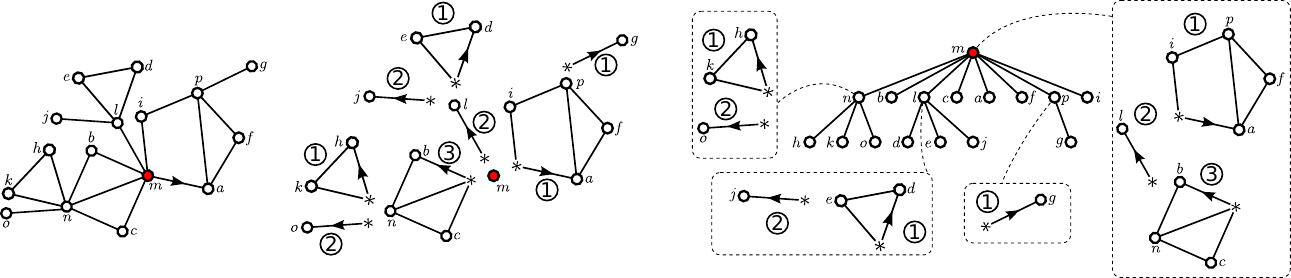}
		\caption{Correspondence of rooted simple outerplanar maps to labelled enriched trees.}
		\label{fi:decompout}
	\end{minipage}
\end{figure}

This may be expressed in the language of species as follows. Let $\Seq$ denote the species of linear orders. Hence $\Seq \circ \cD$ is the class of ordered sequences of dissections, in which the root vertices of the dissections do not contribute to the total size of the objects. If for each vertex $v$ of an outerplanar map we let $\alpha(v)$ denote the  $\Seq \circ \cD$-object corresponding to $v$ in the above decomposition, and declare each non-$*$-vertex of $\alpha(v)$ as the offspring of $v$, then we end up with an encoding of this map as an $\Seq \circ \cD$-enriched tree $(T,\alpha)$. This yields an isomorphism between $\cM^\text{out}$ and the species of $\Seq \circ \cD$-enriched trees. The corresponding recursive isomorphism as in \eqref{eq:renr} reads as follows:
\begin{align}
	\label{eq:suboutl}
	\cO \simeq \cX \cdot \Seq(\cD(\cO)).
\end{align}

If we fix a weighting $\gamma$ on $\cD$, for example the weighting considered in Section~\ref{sec:diss}, then we may consider the weighting
$\omega$ on $\cO$ that assigns weight $\omega(M) = \prod_D \gamma(D)$ to any outerplanar map $M$, with the index $D$ ranging over the blocks of $M$. The random map \label{pp:on}$\mO_n^\omega$ drawn from $\cO[n]$ with probability proportional to its $\omega$-weight is distributed like the random enriched tree $\mA_n^\cR$ for the weighted species $\cR^\kappa=(\Seq \circ \cD)^\kappa$, with $\kappa$ assigning the product of the $\gamma$-weights of the individual dissections to any ordered sequence of $\cD$-structures. This encompasses the uniform outerplanar map, which received some attention  in recent literature, particularly due to the work by Caraceni~\cite{caraceni2016}, who established the continuum random tree as its scaling limit. A further natural example of $\mO_n^\omega$ is that of random bipartite outerplanar maps, which is obtained by  setting the $\gamma$-weights of unwanted (that is, not bipartite) dissections to zero.

The Ehrenborg--M\'endez isomorphism discussed in Section~\ref{sec:diss} yields a weight-preserving isomorphism
\begin{align}
	\label{eq:outerschroeder}
	\cO^\omega \simeq \cX + (\cX \cdot \cD^\gamma)(\cO^\omega)
\end{align}
which identifies weighted outerplanar maps as Schr\"oder $(\cX \cdot \cD^\gamma)$-enriched parenthesizations. The combinatorial interpretation of Equation~\eqref{eq:outerschroeder} is that any outerplanar map is either a single vertex (accounting for the summand $\cX$) or an edge-rooted dissection of a polygon, where each vertex (including the origin of the root-edge, which is why we multiply $\cD^\gamma$ by $\cX$) gets identified with the origin of the root-edge of another outerplanar map.

\begin{figure}[t]
	\centering
	\begin{minipage}{1.0\textwidth}
		\centering
		\includegraphics[width=1.0\textwidth]{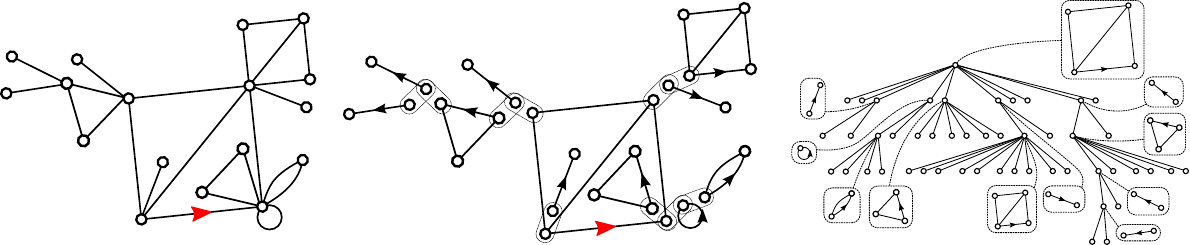}
		\caption{Correspondence of rooted planar maps to trees enriched with non-separable maps.}
		\label{fi:decplanar}
	\end{minipage}
\end{figure}

\subsubsection{Random planar maps with block-weights}
\label{sec:decmaps}
Let \label{pp:m}$\cM$ denote the species of rooted planar maps whose atoms are corners, or equivalently half-edges. Let $\cQ$ denote the subclass of all non-separable maps. Tutte's "substitution decomposition" (see for example Banderier, Flajolet, Schaeffer, 
and Soria \cite{MR1871555} and Flajolet and Sedgewick  \cite[Ex. IX.42]{MR2483235}) states that any rooted planar map consists of a non-separable block or core $Q$ that contains the root-edge, where for each vertex $v$ of $Q$ and each corner $c$ incident to $v$ an arbitrary  rooted map $M_c$ is attached to $v$  by drawing $M_c$ in the face corresponding to $c$ and identifying the root-vertex of $M_c$ with the vertex $v$. Hence
\begin{align}
	\cM \simeq \cQ(\cX \cdot \cM).
\end{align}
This identifies the species $\cX \cdot \cM$ as $\cQ$-enriched trees. The canonical isomorphism is illustrated in Figure~\ref{fi:decplanar}. Given a weighting $\kappa$ on the species \label{pp:q}$\cQ$, we may assign the weight 
\[
	 \omega(M) = \prod_Q \kappa(Q)
\]
to any map $M$, with the index $Q$ ranging over all maximal non-separable submaps of $M$. Let \label{pp:mn}$\mM^\omega_n$ denote the random planar map with $n$ edges drawn with probability proportional to its $\omega$-weight. Then $\mM^\omega_n$ is distributed like the map corresponding to the random enriched tree $\mA_{2n+1}^{\cQ^\kappa}$.

A planar map is simple if and only if all its maximal non-separable submaps are simple. The same holds for many other properties, such as being bipartite, loopless, or bridgeless. We may set the $\kappa$-weight of unwanted blocks to zero in order for the random map $\mM_n^\omega$ to satisfy any subset of these constraints.

\subsubsection{Random $k$-dimensional trees}
\label{sec:decktrees}
A $k$-tree is a simple graph obtained by starting with a $k$-clique $K_k$ and adding in each step a vertex and $k$ distinct edges from the vertex to the graph. For example, $1$-trees are simply unordered trees. Any $k$-clique in a $k$-tree is termed a front, a $(k+1)$-clique a hedron. Throughout we fix $k$ and let \label{pp:k}$\cK$ denote the species of $k$-trees. Let \label{pp:kn}$\mK_n$ denote the uniform random $k$-tree with $n+k$ vertices, or equivalently $n$ hedra. Any $k$-tree with $n$ hedra may be rooted at $\binom{n+k}{k}$ different fronts. So if  \label{pp:kc}$\cK^\circ$ denotes the species of $k$-trees that are rooted at a front consisting of distinct $*$-placeholder vertices, then $\mK_n$ may be sampled by taking a uniform random element from $\cK^\circ[n]$.
This reduces the study of labelled $k$-trees to the study of front-rooted $k$-trees, for which a decomposition is available \cite{MR2577935}.

Let \label{pp:kc1} $\cK_1^\circ$ denote the subspecies of $\cK^\circ$ where the root-front is contained in precisely one hedron. Clearly any element from $\cK^\circ$ may be obtained in a unique way by  glueing an arbitrary unordered collection of $\cK_{1}^\circ$-objects together at their root-fronts. Hence
\begin{align}
	\label{eq:lab1}
	\cK^\circ \simeq \Set(\cK_{1}^\circ).
\end{align}

\begin{figure}[t]
	\centering
	\begin{minipage}{1.0\textwidth}
		\centering
		\includegraphics[width=0.27\textwidth]{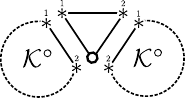}
	\caption{Decomposition of the class $\cK_1^\circ$ for $k=2$.}
	\label{fi:ktree}
	\end{minipage}
\end{figure}

Any $\cK_{1}^\circ$-object may be constructed in a unique way as illustrated in Figure~\ref{fi:ktree}, by starting with a hedron $H$ consisting of the root-front and a vertex $v$, and then choosing, for each front $M$ of $H$ that contains $v$, a $k$-tree from $\cK^\circ$ whose root-front gets identified in a canonical way with $M$. Hence
\begin{align}
	\label{eq:lab2}
	\cK_{1}^\circ \simeq \cX \cdot \Seq_{\{k\}}(\cK^\circ).
\end{align}
Combining the isomorphisms in \eqref{eq:lab1} and \eqref{eq:lab2} yields
\begin{align}
	\label{eq:lab3}
	\cK_{1}^\circ \simeq \cX \cdot (\Seq_{\{k\}}\circ \Set)(\cK_{1}^\circ).
\end{align}
This identifies the species $\cK_{1}^\circ$ as $\Seq_{\{k\}}\circ \Set$-enriched trees, and the species $\cK^\circ$ as unordered forest of enriched trees.

\subsubsection{Simply generated trees with leaves as atoms}
\label{sec:lgwt}

We may consider the species $\cT_\ell$ of plane trees with leaves as atoms, such that no vertex is allowed to have outdegree $1$. This way, only finitely many trees correspond to any given finite set of atoms. 
The species $\cT_\ell$ admits the decomposition
\[
\cT_\ell \simeq \cX + \Seq_{\ge 2}(\cT_\ell),
\]
as any such tree is either a single root-vertex or an internal vertex, that does not contribute to the total size, with an ordered sequence of at least two such trees dangling from it. 

In combinatorial terminology, the class $\cT_\ell$ is a so called Schr\"oder-enriched parenthesization. We may write \[\Seq_{\ge 2} \simeq \cX \cdot \Seq_{\ge 1}\] by distinguishing any canonical element of the order, for example the left-most or the right-most.  The Ehrenborg--M\'endez transformation \eqref{eq:schr}  illustrated in Figure~\ref{fi:ehrenb} now yields
\begin{align}
\label{eq:emiso}
\cT_\ell \simeq \cX \cdot (\Seq \circ \Seq_{\ge 1})(\cT_\ell).
\end{align}
Given a weight-sequence $(\gamma_k)_{k \ge 2}$ of non-negative real numbers with $\gamma_k >0$ for at least one $k$, we may assign the weight
\[
\omega(T) = \prod_{v} \gamma_{d_T^+(v)}
\]
to any plane tree $T$, with the index $v$ ranging over all internal vertices of $T$. That is, vertices that are not leaves. 
We define a weighting $\gamma$ on $\Seq_{\ge 1}$, such that
\[
\Seq^\gamma_{\ge 1}(z) = \sum_{k \ge 1} \gamma_{k+1} z^{k}.
\]
As the Ehrenborg--M\'endez isomorphism \eqref{eq:emiso} is compatible with these weightings, we obtain
\[
\cT_\ell^\omega \simeq \cX \cdot (\Seq \circ \Seq_{\ge 1}^\gamma)(\cT_\ell^\omega).
\]
Thus, for $\cR^\kappa = \Seq \circ \Seq^\gamma_{\ge1}$, the random enriched tree $\mA_n^\cR$ corresponds to a random plane tree with $n$ leaves drawn with probability proportional to its $\omega$-weight.

\subsection{Local convergence of random enriched trees near the root node}
\label{sec:partB}
{\em Throughout this  section, let $\cR^\kappa$ be a weighted species such that the weight sequence $\mathbf{w} = (\omega_k)_k$ with $\omega_k = |\cR[k]|_\kappa / k!$ satisfies $\omega_0>0$ and $\omega_k > 0$ for some $k \ge 2$. Moreover, let $\omega$ be the corresponding weighting on the species $\cA_\cR$ of $\cR$-enriched trees, as given in Equation~\eqref{eq:weighting}. }

\label{sec:convnearroot}
In order to formalize local convergence, it is convenient to work with objects that we will call {\em $\cR$-enriched plane trees} in the following, that is, pairs $(T, \beta)$ of a plane tree $T$ and a map $\beta$ that maps each vertex $v$ of $T$  to an $\cR$-structure $\beta(v) \in \cR[d_T^+(v)]$ with $d_T^+(v)$ denoting the outdegree.

Recall that, as discussed in Section~\ref{se:loconv}, any plane tree may be viewed as a subtree of the infinite {\em Ulam--Harris tree} $\UHT$ whose vertex set $\VHT$ consists of the finite sequences of positive integers.  As there is a  canonical bijection between the set of numbers $[d_T^+(v)] = \{1, 2, \ldots,  d_T^+(v)\}$ and the offspring set $v1, v2, \ldots, vd_T^+(v)$ for any vertex $v \in T$, this allows us to interpret an enriched plane tree $(T, \beta)$ as an enriched tree.

The following lemma provides a coupling that allows us to make use of the wealth of results for simply generated trees in order to study random enriched trees.

\begin{lemma}[A coupling of random $\cR$-enriched trees with simply generated trees]
	\label{le:coupling}
	Let $n \in \ndN$ with $|\cA_\cR[n]|_\omega > 0$ be given.
	The outcome $\mA^\cR_n = (\mA_n, \alpha_n)$ of the following procedure draws a random enriched tree from the set $\cA_\cR[n]$ with probability proportional to its $\omega$-weight.
	\begin{enumerate}
		\item Draw a simply generated plane tree $\cT_n$ of size $n$ according to the weight sequence $\om$.
		\item For each vertex $v \in V(\cT_n)$ choose an $\cR$-structure \[\beta_n(v) \in \cR[d_{\cT_n}^+(v)]\] at random with conditional distribution given by
		\[
		\Pr{\beta_n(v) = R \mid \cT_n} = \kappa(R) / |\cR[d_{\cT_n}^+(v)]|_\kappa
		\]
		for all $R \in \cR[d_{\cT_n}^+(v)]$.
		\item Choose a bijection \[\sigma: V(\cT_n) \to [n]\] between the vertex set of $\cT_n$ and the set $[n]$ uniformly at random, and distribute labels by applying the transport function: \[(\mA_n, \alpha_n) = \cA_\cR[\sigma](\cT_n, \beta_n).\]
	\end{enumerate}
\end{lemma}
By corresponding results for simply generated trees recalled in Section~\ref{sec:pretree}, we know that $|\cA_\cR[n]|>0$ implies  $n \equiv 1 \mod \spa(\mathbf{w})$ and conversely, if $n \equiv 1 \mod \spa(\mathbf{w})$ is large enough, then $|\cA_\cR[n]|>0$.  The random enriched plane tree $(\cT_n, \beta_n)$ encodes all information about the enriched tree $\mA_n$ apart from the labeling. The vertices of enriched plane trees have unique coordinates which allow us to encode these objects as elements of a product space as follows.

If the maximum size of an $\cR$-object is finite, we equip the finite set
\begin{align}
	\label{eq:spacex0}
	X := \{*\} \sqcup \bigsqcup_{n \ge 0} \cR[n]
\end{align}
 with the discrete metric. Here $*$ denotes some placeholder value. Otherwise, if the sizes of $\cR$-objects are unbounded, we instead let
\begin{align}
\label{eq:spacex}
X := \{*, \infty\} \sqcup \bigsqcup_{n \ge 0} \cR[n]
\end{align} such that the set $X \setminus \{\infty\}$ is equipped with the discrete topology and the space $X$ is the corresponding one-point compactification. Clearly $X$ is a compact Polish space in both cases, and so is the product $X^{\VHT}$ with countably many factors.

An $\cR$-enriched plane tree $(T, \beta)$ may be encoded as an element of $X^{\VHT}$ by setting $\beta(v) := *$ for all vertices $v \in \VHT \setminus V(T)$. (There could be various $\cR$-structures of size $0$, which is why we make use of the $*$-placeholder.) 
Let $\fmA \subset X^{\VHT}$ denote the subset of all $\cR$-enriched plane trees that may have vertices with infinite degree. We do not require the offspring set of such a vertex $v$ to be endowed with an additional structure and set $\beta(v) := \infty$. The subset $\fmA$ is closed and hence also a compact Polish space, see the proof of Theorem \ref{te:local} in Section~\ref{sec:propartA} for details.

We may now state our first main theorem which ensures the local convergence of our model of random enriched trees. Janson \cite[Thm. 7.1]{MR2908619} showed the local convergence of simply generated trees (i.e. the case $\cR = \Seq$) in this generality, and our proof builds on this result. Various subcases of simply generated trees were treated separately earlier, see Kennedy \cite{MR0386042}, Aldous and Pitman \cite{MR1641670}, Jonsson, and Stef{\'a}nsson \cite{MR2764126}, and Janson, Jonsson, and Stef{\'a}nsson \cite{MR2860856}. 
\begin{theorem}[Local convergence of random $\cR$-enriched trees]
	\label{te:local}

	Let $(\cT_n, \beta_n)$ denote the random $\cR$-enriched plane tree from Lemma~\ref{le:coupling}. 
	We define the random modified $\cR$-enriched plane tree $(\hat{\cT}, \hat{\beta})$ as follows.
	\begin{enumerate}
		\item Let ${\hat{\cT}} \in \fmT$ be the modified Galton--Watson tree defined in Theorem~\ref{te:convsigen} that corresponds to the weight sequence $(\omega_k)_k$.
		\item For each vertex $v \in V(\hat{\cT})$ with finite outdegree $d_{\hat{\cT}}^+(v) < \infty$ choose \[\hat{\beta}(v) \in \cR[d_{\hat{\cT}}^+(v)]\] at random with conditional distribution
		\[
		\Pr{\hat{\beta}(v) = R \mid {\hat{\cT}}} = \kappa(R) / |\cR[d_{\hat{\cT}}^+(v)]|_\kappa
		\]
		for all $\cR$-structures $R \in \cR[d_{\hat{\cT}}^+(v)]$.
		For each vertex $v \in V({\hat{\cT}})$ with $d_{\hat{\cT}}^+(v) = \infty$ set $\hat{\beta}(v) = \infty$.
	\end{enumerate}
	Then $(\cT_n, \beta_n)$ converges in distribution toward $(\hat{\cT}, \hat{\beta})$ in the metric space $\fmA$.
\end{theorem}

In particular, the $\cR$-structure $\beta_n(o)$ of the root converges in distribution to $\hat{\beta}(o)$ in the space $X$. 
Recall that, as discussed in Section~\ref{sec:types}, the weight sequence $\mathbf{w}$ may be classified into certain types according the supremum of the means of all possible equivalent probability weight sequences. Depending on the weight sequence, we may strengthen the form of convergence in Theorem~ \ref{te:local} as follows.

For any enriched tree $(T, \alpha)$ and any integer $k$, let \label{pp:trim}$(T,\alpha)^{[k]} = (T^{[k]}, \alpha^{[k-1]})$ denote the corresponding tree that gets trimmed at height $k$. That is, $T^{[k]}$ is obtained from $T$ by deleting all vertices with height greater than $k$, and $\alpha^{[k-1]} = (\alpha(v))_{v \in T^{[k-1]}}$.  Hence $(T, \alpha)^{[k]}$ is a tree where each vertex with height less than $k$ is enriched with an $\cR$-structure on its offspring set. If the weight-sequence $\mathbf{w}$ has type I, then Theorem~\ref{te:local} implies that \[
(\cT_n, \beta_n)^{[k]} \convdis (\hat{\cT}, \hat{\beta})^{[k]}
\] for every fixed $k \ge 1$. If  $\mathbf{w}$ has type I$\alpha$, this may be strengthened to convergence of trees pruned at height $o(\sqrt{n})$:

\begin{theorem}[A stronger form of local convergence of random $\cR$-enriched trees]
	\label{te:strengthend}
	Suppose that the weight sequence $\mathbf{w}$ has type $I\alpha$. Then for any sequence of positive integers $k_n = o(n^{1/2})$ it holds that
	\[
	d_{\textsc{TV}}((\cT_n, \beta_n)^{[k_n]}, (\hat{\cT}, \hat{\beta})^{[k_n]}) \to 0
	\]
	as $n \equiv 1 \mod \spa(\mathbf{w})$ becomes large. 
\end{theorem}

In Lemma~\ref{le:coupling} and Theorem~\ref{te:local} we observed a "natural" construction that crops up when working with enriched trees: First we sample a random plane tree, and then we draw for each vertex $v$ with finite degree $d^+(v)$ an $\cR$-structure from $\cR[d^+(v)]$ with probability proportional to its weight. The proof of Theorem~\ref{te:local} may easily be extended to the following result.

\begin{lemma}
	\label{le:extension}
	Let $(\tau_n)_{n \ge 1}$ be a sequence of random locally finite plane trees, that convergence weakly in the space $\fmT$ toward a random limit tree $\hat{\tau}$.  Then the naturally enriched tree $(\hat{\tau}, \hat{\beta})$ is the weak limit of the naturally enriched trees $(\tau_n, \beta_n)$ in the space $\fmA$.
\end{lemma}

Many models of random plane trees are known to converge weakly in $\fmT$. See in particular Abraham and Delmas~\cite{MR3227065, MR3164755} and Janson~\cite[Ch. 22]{MR2908619}. We are going to apply Lemma~\ref{le:extension} in Section~\ref{sec:schroeder} to Schr\"oder enriched parenthesizations such as random face-weighted outerplanar maps.

\subsection{Convergence of random enriched trees that are centered at a random vertex}
{\em As in the previous section, we let $\cR^\kappa$ denote a weighted species such that the weight sequence $\mathbf{w} = (\omega_k)_k$ with $\omega_k = |\cR[k]|_\kappa / k!$ satisfies $\omega_0>0$ and $\omega_k > 0$ for some $k \ge 2$. By Equation~\eqref{eq:weighting}, this induces a weighting on the species $\cA_\cR$ of $\cR$-enriched trees, which we denote by $\omega$.}

\label{sec:fringe}

The present section is dedicated to studying the  random $\cR$-enriched trees locally around a uniformly at random selected vertex. This is a natural question, as this behaviour may differ from the behaviour around the fixed root-vertex. Our main application will be to face-weighted random outerplanar maps in Section~\ref{sec:outerplanar}, for which we characterize different limit graphs depending on whether we look at the vicinity of the root-edge or of a uniformly at random drawn vertex.


\subsubsection{The space of pointed plane trees}

We start with the construction of an infinite plane tree $\UHT^\bullet$ having a spine $(u_i)_{i \ge 0}$ that grows backwards, that is, such that $u_{i}$ is a parent of $u_{i-1}$ for all $i \ge 1$. Any vertex $u_i$ with $i \ge 1$ has an infinite number of offspring vertices to the left and to the right of its distinguished offspring $u_{i-1}$, and each of these non-centered offspring vertices is the root of a copy of the Ulam--Harris tree $\UHT$.  To conclude the construction, the start-vertex $u_0$ of the spine also gets identified with the root of a copy of $\UHT$. We let $\VHT^\bullet$ denote the vertex-set of the tree $\UHT^\bullet$. 

Any  plane tree $T$ together with a distinguished vertex $v_0$ may be interpreted in a canonical way as a subtree of $\UHT^\bullet$. To do so, let $v_0, v_1, \ldots, v_k$ denote the path from $v_0$ to the root of $T$. This way, any vertex $v_i$ for $i \ge 1$ may have offspring to the left and to the right of $v_{i-1}$. Thus there is a unique order-preserving and outdegree preserving embedding of $T$ into $\UHT^\bullet$ such that $v_i$ corresponds to $u_i$ for all $0 \le i \le k$.

Similarly as for the encoding of plane trees, we may identify the pair $T^\bullet = (T,v_0)$ with the corresponding family of outdegrees $(\bar{d}_{T^\bullet}^+(v))_{v \in \VHT^\bullet}$, such that \[\bar{d}^+_{T^\bullet}(v) \in \bar{\ndN}_0 = \ndN_0 \cup \{ \infty\}\]
for $v \notin \{u_1, u_2, \ldots\}$, and
\[
	\bar{d}^+_{T^\bullet}(u_i) \in \{*\} \sqcup (\bar{\ndN}_0 \times  \bar{\ndN}_0), \qquad i \ge 1
\]
such that the two numbers represent the number of offspring vertices to the left and right of the distinguished son $u_{i-1}$, and the $*$-placeholder represents the fact that the vertex does not belong to the tree.


We may consider $\bar{\ndN}_0$ as a compact Polish space given by the one-point compactification of the discrete space $\ndN_0$. A metrization that is complete and separable is given by
\[
	d_{\bar{\ndN}_0}(a,b) = \begin{cases} 0, &a = b \\ (1 + \min(a,b))^{-1}, &a \ne b.\end{cases}
\]
Consequently, the product  $\bar{\ndN}_0 \times  \bar{\ndN}_0$ is also compact and  Polish. A possible metrization is given by
\[
	d_{\bar{\ndN}_0 \times \bar{\ndN}_0}( (a_1, a_2), (b_1, b_2) ) = \max( d_{\bar{\ndN}_0}(a_1,b_1), d_{\bar{\ndN}_0}(a_2, b_2)).
\]
The same goes for the disjoint union topology on $\{*\} \sqcup (\bar{\ndN}_0 \times  \bar{\ndN}_0)$, where we let the $*$-point  have distance $2$ from any other vertex.
Hence the space
\[
	\{ (\bar{d}^+(v))_{v \in \cV^\bullet_\infty} \mid d^+(v) \in \bar{\ndN}_0 \text{ for $v \notin \{u_1, u_2, \ldots\}$},\,\bar{d}^+(v) \in \{*\} \sqcup (\bar{\ndN}_0 \times  \bar{\ndN}_0) \text{ for $v \in \{u_1, u_2, \ldots\}$ } \}  
\]
is the product of countably many compact Polish spaces, and hence also compact and  Polish.  Any element of this space corresponds to a subgraph of the tree $\UHT^\bullet$, and we say it is a tree if this subgraph is connected. The subset $\fmT^\bullet$ of all elements that correspond to trees is closed, and hence also a compact Polish space with respect to the subspace topology. 

Any  enriched plane tree together with a distinguished vertex has a spine given by the unique directed path from the distinguished vertex to the root. This makes the $\cR$-structures along that path (apart from the $\cR$-structure of the distinguished vertex) actually $\cR^\bullet$-structures, because each contains a unique distinguished vertex from the spine. We could easily generalize the space $\fmT^\bullet$ to a space whose elements encode pointed $\cR$-enriched trees, but there is no need for this, as we may phrase our limit theorems in a way that avoids this construction.

\subsubsection{The limit objects}
\label{sec:rndlimit}

As discussed in Section~\ref{sec:types} there is a probability distribution $(\pi_k)_k$ associated to the weight sequence $\mathbf{w}$, with density given in \eqref{eq:tt}. Let $\xi$ \label{pp:xi} be distributed according to $(\pi_k)_k$ and let $\cT$ be a $\xi$-Galton--Watson tree. By Equation~\eqref{eq:mu} it holds that $\mu := \Ex{\xi}\le1$. We may consider the {\em size-biased} random variable $\hat{\xi} \in \ndN_0 \cup \{\infty\}$\label{pp:xih}, distributed according to
\[
\Pr{\hat{\xi} = k} = k \pi_k \quad \text{and} \quad \Pr{\hat{\xi} = \infty} = 1 - \mu.
\]

\paragraph*{The type I regime}
If the weight-sequence $\mathbf{w}$ has type I, then $\hat{\xi} < \infty$ almost surely, and we define the random tree $\cT^*$ in the space $\fmT^\bullet$ as follows.
Let $u_0$ be the root of an independent copy of the Galton--Watson tree $\cT$. For each $i \ge 1$, we let $u_i$ receive offspring according to an independent copy of $\hat{\xi}$. The vertex $u_{i-1}$ gets identified with an uniformly at random chosen offspring of $u_i$. All other offspring vertices of $u_i$ become roots of independent copies of the Galton--Watson tree $\cT$.  The construction of the $\cT^*$ goes back to Aldous \cite{MR1102319}, who established it as a limit of large critical Galton--Watson trees re-rooted at a random vertex.

For each vertex $v \in \cT^*$, we draw an $\cR$-structure $\beta^*(v) \in \cR[d^+_{\cT^*}(v)]$ with probability proportional to its weight. The atoms of the $\cR$-structure are matched in a canonical order preserving way with the offspring vertices of the vertex $v$. This yields an infinite locally finite $\cR$-enriched tree~$(\cT^*,\beta^*)$. \label{de:tinf} Note that for $i \ge 1$, the $\cR$-structure $\beta^*(u_i)$ becomes  an $\cR^\bullet$-structure by distinguishing the atom corresponding to the vertex $u_{i-1}$.

\paragraph*{The condensation regime}
Suppose that the weight-sequence $\mathbf{w}$ has type II or III, that is, $0 \le \nu < 1$.  Let $o$ denote the root vertex of the simply generated tree $\cT_n$.  Janson~\cite[Lem. 19.32, Lem. 15.7]{MR2908619} showed that there is a deterministic sequence $\Omega_n$ that tends to infinity sufficiently slowly, such that for any sequence $K_n \to \infty$ with $K_n \le \Omega_n$ it holds that 
\begin{align}
\label{eq:om}
\lim_{n \to \infty} \Pr{d_{\cT_n}^+(o) > K_n} = 1 - \nu.
\end{align}
Let $\tilde{D}_n$ be a sequence of random variables with distribution given by
\begin{align}
	\label{eq:thed}
	\tilde{D}_n \eqdist (d_{\cT_n}^+(o) \mid  d_{\cT_n}^+(o) > \Omega_n).
\end{align}
We construct the random pointed tree $\cT^*_n$ as follows. The center $u_0$ becomes the root of an independent copy of the Galton--Watson tree $\cT$.  For $i=1, 2, \ldots$ the vertex $u_i$ receives offspring according to an independent copy $\hat{\xi}_i$ of $\hat{\xi}$, where a randomly chosen son gets identified with $u_{i-1}$ and the rest become roots of independent copies of $\cT$. We proceed in this way for $i=1,2, \ldots$ \emph{until} it occurs for the first time $i_1$ that $\hat{\xi}_{i_1}=\infty$. When $\hat{\xi}_1, \ldots, \hat{\xi}_{i_1-1}< \infty$ and $\hat{\xi}_{i_1} = \infty$, then $u_{i_1}$ receives offspring according to $\tilde{D}_n$, among which we select one uniformly at random and identify it with $u_{i_1 -1}$. Each of these vertices (except $u_{i_1-1}$ of course) gets identified with the root of an independent copy of the Galton--Watson tree $\cT$. 

We form the finite pointed $\cR$-enriched tree $(\cT_n^*, \beta_n^*)$ by drawing for each vertex $v$ of the tree $\cT_n^*$  an $\cR$-structure $\beta^*_n(v)$ from $\cR[d^+_{\cT_n^*}(v)]$ with probability proportional to its weight.  The atoms of the $\cR$-structure are matched in a canonical order-preserving way with the offspring vertices of $v$. For $i \ge 1$, the $\cR$-structure $\beta^*(u_i)$ becomes  an $\cR^\bullet$-structure by distinguishing the atom corresponding to the vertex $u_{i-1}$.



\subsubsection{Convergence of the vicinity of a random node}
\label{sec:covic}

\label{mark:fr} Given a pointed $\cR$-enriched tree $A^\bullet=(A,x)$, we may consider the {\em enriched fringe subtree} $f(A^\bullet)$ which is the maximal enriched subtree of $A$ that is rooted at $x$. For all $k \ge 0$, we may also consider the pointed enriched tree $f_k(A^\bullet)$ given by the enriched fringe subtree at the $k$-th ancestor of $x$, that we consider as pointed at the vertex $x$. If the vertex $x$ has height  less than $k$ in $A^\bullet$, then we set $f_k(A^\bullet) = \diamond$ for some placeholder symbol $\diamond$.  The following theorem builds upon results by Janson~\cite{MR2908619,MR3432572} for simply generated trees:


\begin{theorem}[Convergence of type I re-rooted trees]
	\label{te:thmben}
	Let the $\cR$-enriched plane tree $(\cT_n, \beta_n)$ be pointed at a uniformly at random selected vertex $v_0$. For all $k \ge 0$ let
	\[
		\mH_k = f_k( (\cT_n, \beta_n), v_0)
	\]
	denote the pointed fringe-subtree at the $k$-th ancestor of $v_0$.
	\begin{enumerate}
		\item If the weight-sequence $\mathbf{w}$ has type I, then for each fixed $k \ge 0$ it holds that	
		\[
			\mH_k \convdis f_k(\cT^*, \beta^*)
		\]
		as  random elements of the countable set of finite pointed enriched trees, that we equip with the discrete metric. 
		
		\item If the weight-sequence $\mathbf{w}$ has type I$\alpha$, then for any sequence $k_n = o(\sqrt{n})$ it even holds that
		\[
		d_{\textsc{TV}}( \mH_{k_n}, f_{k_n}(\cT^*, \beta^*)) \to 0
		\]
		as $n$ becomes large. 
		\item  If the weight-sequence $\mathbf{w}$ has type I, then there is even a constant $c>0$ such that for any possible value $H$ of $\mH_k$ the number
		\[
			N_H := |\{v \in \cT_n \mid f_k((\cT_n, \beta_n),v) = H \}|
		\]
		satisfies
		\[
		\Prb{ \left| \frac{N_H}{n} -  \pi_H\right| > \epsilon} \le \exp(n(o(1) - c\epsilon^2))
		\]
		with $\pi_H= \Pr{f_k(\cT^*, \beta^*) = H}$ and an $o(1)$ term that does not depend on $H$ or $\epsilon$. Hence
		\[
		\frac{N_H }{n} \to \pi_H
		\]
		holds almost surely.
		\item If the weight-sequence $\mathbf{w}$ has type I$\alpha$ and $\Pr{f_k(\cT^*, \beta^*) = H}>0$, then there is a constant $\sigma_H>0$ with 
		\[
		\frac{N_H - n \pi_H}{\sqrt{n}} \convdis \cN(0, \sigma_H^2).
		\]

	\end{enumerate}
\end{theorem}

Let $m \ge 0$ be an integer and  $T^\bullet = (T,x_0)$ a pointed tree with spine $x_0 , x_1, \ldots, x_k$.  We consider the pointed subtree $P_m( T^\bullet)$ obtained by pruning away  the descendants of all siblings of $x_{k-1}$ that lie more than $m$ to the left or $m$ to the right of $x_{k-1}$. That is, all these siblings become leaves.

If $A^\bullet$ is a pointed $\cR$-enriched plane tree consisting of the pointed tree $T^\bullet$ and a family of $\cR$-structures $(\gamma(v))_{v \in T}$, we may likewise consider the pruned tree $P_m( A^\bullet)$, that is given by the pointed tree $P_m(T^\bullet)$ such that all vertices $v \in P_m(T^\bullet)$ keep their original $\cR$-structure, except for the siblings of the vertex $x_{k-1}$ that lie more than $m$ to the left or to the right of it, whose $\cR$-structure we set to some placeholder value.

\begin{theorem}[Convergence of re-centered trees in the condensation regime]
	\label{te:thmco}
	Let the $\cR$-enriched plane tree $(\cT_n, \beta_n)$ be pointed at a uniformly at random selected vertex $v_0$. Let $k_n \ge 0$ be minimal with the property, that the $k_n$th ancestor $v_{k_n}$ of $v_0$ in the tree $\cT_n$ has outdegree $d^+_{\cT_n}(v_{k_n}) > \Omega_n$. We set
	\[
		\mH_{k_n} = f_{k_n}( (\cT_n, \beta_n), v_0).
	\]
	If no such ancestor exists, we set $k_n = \infty$ and $\mH_{k_n} = \diamond$.  It holds that $1 \le k_n < \infty$ with probability tending to $1$ as $n$ becomes large, and for each $m \ge 0$ 
	\[
		d_{\textsc{TV}}(P_m(\mH_{k_n}), P_m(\cT_n^*, \beta_n^*)) \to 0.
	\]
\end{theorem}

\subsection{Schr\"oder $\cN$-enriched parenthesizations}
\label{sec:schroeder}
	
Given a weighted species $\cN^\gamma$ with no structures of size less than $2$ and at least one structure with positive $\gamma$-weight, we may consider the species $\cS^\upsilon_{\cN}$ of {\em Schr\"oder $\cN$-enriched parenthesizations} discussed in Section~\ref{sec:diss}, that satisfies a weight-preserving isomorphism
\begin{align}
	\cS_{\cN}^\upsilon \simeq \cX + \cN^\gamma \circ \cS_\cN^\upsilon.
\end{align}
We have seen in Section~\ref{sec:intro1} that the Ehrenborg--M\'endez isomorphism allows us to consider many classes of combinatorial objects such as face-weighted dissections of polygons and face-weighted outerplanar maps both as enriched trees and as Schr\"oder-enriched parenthesizations.

Similarly to random enriched trees, who have a canonical coupling with a simply generated tree (with vertices as atoms), random enriched parenthesizations have a natural coupling with a simply generated tree whose atoms are leaves. Each viewpoint has its own advantages and disadvantages:
For example, if we study a random face-weighted outerplanar map with $n$ vertices and analytic weights, we may interpret this map either as a Galton--Watson tree conditioned on having $n$ leaves that is enriched by dissections, or as a (different) Galton--Watson tree conditioned on having $n$ vertices that is enriched by ordered sequences of dissections. The first coupling is more convenient from a combinatorial viewpoint, because the outdegrees of the vertices in the tree then correspond precisely to the sizes of the $2$-connected components of the map, but there are less results available for Galton--Watson trees conditioned on their number of leaves than for simply generated trees. There are fairly recent additions though that are useful in this context, for example by Abraham and Delmas~\cite{MR3164755}, Kortchemski~\cite{MR2946438} and Curien and Kortchemski~\cite{MR3245291}. 





\begin{lemma}[A coupling of random Schr\"oder-enriched parenthesizations with simply generated trees that have leaves as atoms]
	\label{le:couplingschroeder}
	Let $n \in \ndN$ with $|\cS_\cN[n]|_\upsilon > 0$ be given.  Set $p_0=1$ and for each $k \ge 2$ set $p_k = |\cN[k]|_{\gamma}/k!$.
	The outcome $\mS^\cN_n$ of the following procedure draws a random enriched tree from the set $\cS_\cN^\upsilon[n]$ with probability proportional to its $\upsilon$-weight.
	\begin{enumerate}
		\item Sample a random plane tree $\tau_n$ with $n$ leaves according to 
		\[
			\Pr{\tau_n = T} = (\sum_{S} \prod_{v \in S} p_{d_S^+(v)})^{-1} \prod_{v \in T} p_{d_T^+(v)}.
		\]
		with the sum-index $S$ ranging over all plane trees with $n$ leaves.
		\item For each inner vertex $v$ of $\tau_n$ choose an $\cN$-structure \[\delta_n(v) \in \cN[d_{\tau_n}^+(v)]\] at random with conditional distribution given by
		\[
		\Pr{\delta_n(v) = N \mid \tau_n} = \gamma(N) / |\cN[d_{\tau_n}^+(v)]|_\gamma
		\]
		for all $N \in \cN[d_{\tau_n}^+(v)]$.
		\item Choose a bijection \[\sigma: V(\tau_n) \to [n]\] between the vertex set of $\tau_n$ and $[n]$ uniformly at random, and distribute labels by applying the transport function: \[\mS^\cN_n = \cS_\cN[\sigma](\tau_n, \delta_n).\]
	\end{enumerate}
\end{lemma}

Given parameters  $a,t>0$, we may tilt the weight-sequence $(p_k)_k$ by setting 
\[
p_0^{a,t} = a \qquad \text{and} \qquad p_k^{a,t} = p_k t^{k-1}
\]
 for $k \ge 1$. The modified weight of the tree is then given by
\[
\upsilon^{a,t}(T) = \prod_{v \in T} p^{a,t}_{d^+_T(v)} = a^n t^{n-1} \upsilon(T).
\]
Hence it makes no difference whether we draw a tree from $\cT_\ell^\upsilon[n]$ with probability proportional to its $\upsilon$-weight or to its $\upsilon^{a,t}$-weight. If the generating series \[p(z) := \sum_{k \ge 2} p_k z^k\] is analytic at $0$, then for each $t>0$ with $p(t)/t < 1$ there is a unique parameter $a = a(t)>0$ such that the tilted weights $p^{(t)} = (p^{a,t}_k)_{k \ge 0}$ form a probability weight-sequence. The expected value of the offspring distribution is given by
\[
\mu_t = \sum_{k \ge 2} k p_k^{a,t} = \sum_{k \ge 2} k p_k t^{k-1},
\]
which we may interpret as a strictly increasing function 
\[
\mu: [0, \rho_p] \to [0, \infty]
\] 
in $t$, with $\rho_p$ denoting the radius of convergence of $p(z)$. There is a canonical choice for the parameter $t$. If $\mu_{\rho_p} \ge 1$, we let $t_0>0$ be the unique parameter with $\mu_{t_0} = 1$, and say the weight-sequence $(p_k)_k$ has type I. If $0 < \mu_{\rho_p} < 1$, we set $t_0 = \rho_p$ and say $(p_k)_k$ has type II. Finally, if $\rho_p=0$, we say $(p_k)_k$ has type III and set $t_0=0$ and let $p^{(t_0)}$ be the probability weight-sequence with mass $1$ on the value $0$.

Let $\hat{\tau}$ denote the modified Galton--Watson tree from Section~\ref{se:modgwt} that corresponds to the offspring distribution $p^{(t_0)}$. Abraham and Delmas \cite[Thm. 1.2]{MR3227065} (see also \cite{MR3245291, MR3164755} for previous results in subcases) showed convergence of $\tau_n$ toward $\hat{\tau}$ in the cases I and II. We complete the picture by treating the non-analytic case III, in which we establish the same asymptotic behaviour as simply generated trees with vertices as atoms and super-exponential weights on the out-degrees \cite{MR2860856, MR2908619}. The idea of the proof is to  transform the tree $\tau_n$ in two different ways to simply generated trees and use the convergence of each. Recall that in  Section~\ref{se:loconv} we discussed the compact Polish space $\fmT$ of plane trees that may have vertices with infinite degree.

\begin{lemma}[Convergence of simply generated trees with leaves as atoms]
	\label{le:convll}
	The random tree $\tau_n$ converges in distribution toward the modified Galton--Watson tree $\hat{\tau}$ in the space $\fmT$ as $n$ becomes large.
\end{lemma}

The random enriched parenthesization $(\tau_n, \delta_n)$ may be viewed as a random point in the metric space $\fmA$ introduced in Section~\ref{sec:partB}. Lemma~\ref{le:convll} and Lemma~\ref{le:extension}  allow us to obtain convergence toward a limit object $(\hat{\tau}, \hat{\delta})$.

\begin{theorem}[Local convergence of random enriched parenthesizations]
	\label{te:localschroeder}
	
	Let $(\tau_n, \delta_n)$ denote the random $\cN^\gamma$-enriched parenthesization from Lemma~\ref{le:couplingschroeder}. 
	We define the random modified enriched parenthesization $(\hat{\tau}, \hat{\delta})$ as follows.
	\begin{enumerate}
		\item Let $\hat{\tau}$ denote the modified Galton--Watson tree from Section~\ref{se:modgwt} that corresponds to the offspring distribution $p^{(t_0)}$.
		\item For each vertex $v \in V(\hat{\tau})$ with finite outdegree $d_{\hat{\tau}}^+(v) < \infty$ choose \[\hat{\delta}(v) \in \cN[d_{\hat{\tau}}^+(v)]\] at random with conditional distribution given by
		\[
		\Pr{\hat{\delta}(v) = N \mid {\hat{\tau}}} = \gamma(N) / |\cN[d_{\hat{\tau}}^+(v)]|_\gamma
		\]
		for all $\cN$-structures $N \in \cN[d_{\hat{\tau}}^+(v)]$.
		For each vertex $v \in V({\hat{\tau}})$ with $d_{\hat{\tau}}^+(v) = \infty$ set $\hat{\delta}(v) = \infty$.
	\end{enumerate}
	Then $(\tau_n, \delta_n)$ converges in distribution toward $(\hat{\tau}, \hat{\delta})$ in the metric space $\fmA$.
\end{theorem}

We will also refer to the type of the weight-sequence $(p_k)_k$ as the type of $\cS^\upsilon$, $(\tau_n, \delta_n)$ and $(\hat{\tau}, \hat{\delta})$.  In a certain sense, the types are compatible with those of simply generated trees: Suppose that we are given a species $\cA^\omega_\cR$ of $\cR^\kappa$-enriched trees such that
\[
\cR^\kappa = \Seq \circ \cH^\kappa
\]
for some species $\cH^\kappa$, whose weighting we also denote by $\kappa$ (committing a slight abuse of notation). The Ehrenborg--M\'endez isomorphism from Section~\ref{sec:diss} allows us identify $\cR^\kappa$-enriched trees with $(\cX \cdot \cH^\kappa)$-enriched parenthesizations:
\[
	\cA_{\cR^\kappa}^\omega \simeq \cS^\upsilon_{\cX \cdot \cH^\kappa}.
\]
We distinguished three types for the weight-sequence $\mathbf{w} = (\omega_k)_k$ with $\omega_k = |\cR[k]|_\kappa$. The following observation states that the type of $\mathbf{w}$ agrees with the type of $(p_k)_k$. Recall the definition of the parameter $\tau$ from Section~\ref{sec:types}.

\begin{lemma}[The types of the two weight-sequences agree]
	\label{le:types}
	The weight-sequence $\mathbf{w}$ has type I if  $\mu_{t_0}\ge1$, type II if $0 < \mu_{t_0} < 1$, and type III if $\mu_{t_0} = 0$. Thus, the type of the weight-sequence $(p_k)_k$ agrees with the type of the weight-sequence $(\omega_k)_k$. Moreover, it holds that $\tau=t_0$.
\end{lemma}


\subsection{Giant components in Gibbs partitions}
\label{se:gibbs}
\label{sec:gibbs}

Suppose that we are given weighted species $\cF^\upsilon$ and $\cG^\gamma$ such that the composition $\cF^\upsilon \circ \cG^\gamma$ is well-defined. That is, we assume that $\cG^\gamma(0)=0$.  If we consider a random compound structure from $(\cF^\upsilon \circ \cG^\gamma)[n]$ that gets drawn with probability proportional to its weight, then the corresponding random partition of the set $[n]$ may be  called a {\em Gibbs partition}. Pitman~\cite{MR2245368} gives an extensive survey on this topic, and since then further additions to the theory have been made~\cite{MR2453776,doi:10.1002/rsa.20771}. We are interested in cases where typically a giant component emerges as the size total size of the composite structure becomes large and provide some new results for Gibbs partitions with superexponential weights.

\subsubsection{The convergent case}


In this section, we discuss a setting where the random composite structure asymptotically looks like a Boltzmann distributed $(\cF^\upsilon)' \circ \cG^\gamma$-object, where the marked $*$-placeholder atom in the $\cF$-structure is replaced with a giant $\cG$-component. The term is inspired by Barbour and Granovsky's terminology  for the convergent case of random partitions satisfying a conditioning relation \cite{MR2121024}. The class of partitions satisfying a conditioning has a non-trivial intersection with that of Gibbs partitions, but neither contains the other. See for example Arratia, Barbour and Tavar\'e's book \cite{MR2032426} for a detailed discussion of this model.

\begin{definition}[Convergent type substitution]
	\label{def:convergent}
	Let $\cF^\upsilon$ and $\cG^\gamma$ be weighted species such that $\cG^\gamma(z)$ is not a polynomial and satisfies $\cG^\gamma(0)=0$. For each $n$ with $[z^n]\cF^\upsilon(\cG^\gamma(z)) >0$ let $\mS_n$ denote the isomorphism type of a random labelled compound structure from the set $(\cF^\upsilon \circ \cG^\gamma)[n]$ sampled with probability proportional to its weight. Suppose that the radius of convergence $\rho_\cG$ of the series $\cG^\gamma(z)$ is positive and that \[0<((\cF')^\upsilon \circ \cG^\gamma)(\rho_\cG) < \infty.\]
	Let $\hat{\mS}_n$ denote the isomorphism type of the composite structure obtained by sampling a $\mathbb{P}_{(\cF')^\upsilon \circ \cG^\gamma, {\rho}_\cG}$-distributed object $\hat{\mS}$ and replacing the $*$-placeholder atom in the $\cF$-structure by a random $\cG$-structure sampled from $\cG[n - |\hat{\mS}|]$ with probability proportional to its $\gamma$-weight. This is only well-defined if $n - |\hat{\mS}|>0$ and $|\cG[n - |\hat{\mS}|]|_\gamma > 0$, otherwise we set $\hat{\mS}_n$ to some place-holder value. We say the composition $\cF^\upsilon \circ \cG^\gamma$ has convergent type, if
	\begin{align}
	\label{eq:convtype}
	\lim_{n \to \infty} d_{\textsc{TV}}(\mS_n, \hat{\mS}_n) = 0.
	\end{align}
\end{definition}

\begin{definition}
	\label{pp:sd}
	\label{def:subexp}
	Let $d \ge 1$ be an integer. We say the coefficients of a power series $g(z) = \sum_{n =0}^\infty g_n z^n$ with non-negative coefficients and radius of convergence $r >0$ belongs to the class $\mathscr{S}_d$ of subexponential sequences with span $d$, if $g_n=0$ whenever $n$ is not divisible by $d$, and
	\begin{align}
	\label{eq:condition}
	\frac{g_n}{g_{n+d}} \sim r^d, \qquad \frac{1}{g_n}\sum_{i+j=n}g_ig_j \sim 2 g(r) < \infty
	\end{align}
	as $n \equiv 0 \mod d$ becomes large.
\end{definition}

The broad scope of this setting is illustrated by the following easy observation, which has been noted in various places, see for example \cite{MR772907}.
\begin{proposition}
	\label{pro:yehaa}
	If $g_n = h(n) n^{-\beta} \rho^{-n}$ for some constants $\rho>0$, $ \beta > 1$ and a slowly varying function $h$, then the series $\sum_{n \in d\ndN} g_n z^n$ belongs to the class $\mathscr{S}_d$.
\end{proposition}

The notions introduced above are relevant to the study of enriched trees. The following result follows from \cite[Lem. 3.3]{doi:10.1002/rsa.20771} and shows that the coefficients of the exponential generating series of an arbitrary class of enriched trees form up to a constant shift of indices a subexponential sequence, as long as the series has positive radius of convergence.

\begin{lemma}[Subexponentiality of enriched trees]
	\label{le:subexp}
	If $\cR^\kappa$ is a weighted species, and $\cA_\cR^\omega$ the corresponding species of $\cR$-enriched trees such that the weight-sequence $\mathbf{w}=(\omega_k)_k$ with $\omega_k = |\cR[k]|_\kappa$ has type $I$ or $II$, then the series $\cA_\cR^\omega(z)/z$ belongs to the class $\mathscr{S}_d$ of subexponential sequences with  $d=\spa({\mathbf{w}})$.
\end{lemma}

The following general criterion follows from \cite[Thm. 3.4, Eq. (6.3)]{doi:10.1002/rsa.20771} and ensures that the composite structure $\cR = \cF^\upsilon \circ \cG^\gamma$ has convergent type in the settings we are interested in.

\begin{lemma}[Convergence of Gibbs partitions along subsequences]
	\label{le:gibbs}
	Suppose that there is an integer $0 \le m < d$ such that $\cG^\gamma(z) / z^m$ belongs to the class $\mathscr{S}_d$. Let $D = d/\gcd(m,d)$ and for each $0 \le a < D$, let  $\cF_a^\upsilon$ denote the restriction of $\cF^\upsilon$ to objects whose  size lies in $a + D\ndZ$. If the exponential generating series $\cF_a^\upsilon(z)$ is not constant, then the composition $\cF_a^\upsilon \circ \cG^\gamma$ has convergent type. That is,
	\[
		 d_{\textsc{TV}}(\mS_n^a, \hat{\mS}_n^a) \to 0, \qquad n \to \infty, \qquad n \equiv am \mod d,
	\]
	with $\mS_n^a$ denoting the isomorphism type of a random composite structure sampled from $(\cF_a^\upsilon \circ \cG^\gamma)[n]$ with probability proportional to its weight, and $\hat{\mS}_n^a$ as in Definition \ref{def:convergent}, but for the species $\cF_a^\upsilon$ and $\cG^\gamma$. If $\mS_n$ denotes the isomorphism type of a random  element from $(\cF^\upsilon \circ \cG^\gamma)[n]$ that is drawn with probability proportional to its weight, then it holds for $n \equiv am \mod d$ that
	\[
		\mS_n \eqdist \mS_n^a.
	\]
\end{lemma}

We are going to apply Lemma~\ref{le:gibbs} in many ways. For example, random face-weighted outerplanar maps correspond by the discussion in Section~\ref{sec:decompouter} to trees enriched with ordered sequences of dissections of polygons. We interpreted face-weighted dissections as classes of enriched trees in Section~\ref{sec:diss}, and hence their generating series is up to a constant shift subexponential by  Lemma~\ref{le:subexp}. This will allow us to apply Lemma~\ref{le:gibbs} to random type II outerplanar maps, where the condensation phenomenon yields large submaps given by ordered sequence of dissections, which may be interpreted as randomly-sized Gibbs partitions. The randomness of the size will even out the different behaviour observed along subsequences in Lemma~\ref{le:gibbs}, which allows us to establish  local weak convergence of {\em arbitary} face-weighted outerplanar maps in Theorem~\ref{te:facethm}. 

We will occasionally also make use of the following fact:
\begin{proposition}[{\cite[Thm. 1]{MR0348393}, \cite[Thm. C]{MR772907}}]
	\label{pro:sleeky}
	Suppose that the power series $g(z)$ belongs to $\mathscr{S}_d$ with radius of convergence $r$. Then for any complex function $f(z)$  that is analytic in an open set containing all $g(z)$ with $|z| \le r$, it holds that
	\[
	[z^n] f(g(z)) \sim f'(g(r)) [z^n]g(z), \qquad n \to \infty, \qquad n \equiv 0 \mod d.
	\]
\end{proposition}

\subsubsection{The superexponential case}

If the radius of convergence of the exponential generating series of the species under consideration equals zero, then the limit object studied in the  convergent case is not well-defined, as Boltzmann-distributions only make sense for analytic species. For such superexponential weights, we may establish a regime where typically the composite structure consists of a single component, or of a giant component with a deterministically bounded rest.

\begin{lemma}
	\label{le:supen}
	Suppose that the species $\cA_\cR^\omega$ of $\cR$-enriched trees has type III, and set 
	\[
		a_i = [z^i] \cA_\cR^\omega(z)
	\]
	for all $i \ge 0$.  Then for any integer $k \ge 2$ it holds that
	\begin{align}
	\label{eq:starter}
	\sum_{\substack{i_1 + \ldots + i_k = n \\ 1 \le i_1, \ldots, i_k < n - (k-1)}} a_{i_1} \cdots a_{i_k} = o(a_{n-(k-1)}).
	\end{align}
	Here both sides of the equation are equal to zero, unless $n \equiv k \mod \spa(\mathbf{w})$.
\end{lemma}

The probabilistic interpretation of Equation~\eqref{eq:starter} is that if we draw a $k$-tupel from $(\cA_\cR^\omega)^k[n]$ with probability proportional to its weight, then with high probability all trees in the forest have size $1$, except for a single giant tree which accounts for the total remaining mass.

\begin{theorem}
	\label{te:need}
	Let $\cF^\upsilon$ and $\cG^\gamma$ denote weighted combinatorial species, such that $\cG^\gamma(z)$ has radius of convergence zero and $\cF^\upsilon(z)$ has positive radius of convergence. Suppose that the coefficients
	\[
		a_i = [z^i] \cG^\gamma(z)
	\]
	are supported on a lattice $1 + d \ndN_0$ with $d \ge 1$, such that  there is an integer $I$ with $a_i >0$ for all $i \ge I$ with $i \equiv 1 \mod d$. Let $\mS_n$ be drawn from $(\cF^\upsilon \circ \cG^\gamma)[n]$ with probability proportional to its weight. If the coefficients $(a_i)_i$ satisfy Equation~\eqref{eq:starter} for all $k \ge 2$, then there is a constant $n_0$ such that with probability tending to $1$ as $n$ becomes large the composite structure $\mS_n$ has a giant component with size at least $n - n_0$. In particular, the total number of components in $\mS_n$ is at most $n_0 + 1$ with high probability.
\end{theorem}

Note that if we condition the largest component on having size $k$, then it is up relabelling distributed like drawing a $\cG$-structure from $\cG^\gamma[k]$ with probability proportional to its $\gamma$-weight. This follows easily by applying Lemma~\ref{le:bole} to the analytic species $\cF_{\le n}^\upsilon \circ \cG^\gamma_{\le n}$. (We cannot apply Lemma~\ref{le:bole}  directly to $\cF^\upsilon \circ \cG^\gamma$, as Boltzmann distributions only make sense for analytic species, but there is no difference between $\cF_{\le n}^\upsilon \circ \cG^\gamma_{\le n}[n]$ and $\cF^\upsilon \circ \cG^\gamma[n]$.) 

We also obtain a sufficient criterion for the  Gibbs partition to typically consist of a single component.

\begin{corollary}
	\label{co:ne}
	If additionally to the requirements of Theorem~\ref{te:need} it holds that \[
	[z^1] \cF^\upsilon(z) >0 \qquad \text{and} \qquad a_1, a_2 >0,
	\] then $\mS_n$ consists with high probability of a single component.
\end{corollary}
A result similar to Corollary~\ref{co:ne} has been obtained by Wright~\cite[Thm. 3]{MR0229546}, who studied under which circumstances the $\Set \circ \cG^\gamma$ Gibbs partition typically consists of a single component.  In the setting considered there, the generating series $\cG^\gamma(z)$ has radius of convergence zero and the coefficients $a_i = [z^i]\cG^\gamma(z)$ satisfy
\begin{align}
\label{eq:pro}
\sum_{i=1}^{n-1} a_i a_{n-i} = o(a_{n}).
\end{align}
It is easy to adapt the proof of Theorem~\ref{te:need} to see that in Corollary~\ref{co:ne} we may replace the assumption \eqref{eq:starter} by Equation~\eqref{eq:pro}.  We leave the details to the inclined reader.

\subsection{Extremal component sizes}
\label{sec:compsize}

The sizes of the $\cR$-structures in the random enriched tree $\mA_n^\cR$ correspond to the outdegrees of the vertices in the coupled simply generated tree $\cT_n$. In many cases, such as for random connected graphs or outerplanar maps, we consider compound $\cR$-structures of the form $\cR^\kappa= \cF^{\upsilon} \circ \cG^{\gamma}$, and are interested in extremal sizes of the $\cG$-components, which in these examples correspond to the maximal $2$-connected subgraphs. Note that we may always assume that $\cR$ is a compound species, as there is the trivial isomorphism $\cR^\kappa \simeq \cX \circ \cR^\kappa$.

Recall the definition of the series $\phi$, the numbers $\nu$, $\tau$ and $\rho_\phi$, and the distribution $(\pi_k)_k$ from Section~\ref{sec:pretree}. Let $\xi$ be a random non-negative integer following the distribution $(\pi_k)_k$.  Combining Lemma~\ref{le:gibbs} with results for the largest degrees in simply generated trees \cite[Ch. 9 and Thm. 19.34]{MR2908619} and Jonsson and Stef\'ansson \cite{MR2764126}  we obtain the following asymptotic results for the extremal sizes of the $\cG$-structures in the random enriched tree $\mA_n^\cR$.

\begin{theorem}[Component size asymptotics]
	\label{te:compsize}
	Suppose that $\cR^\kappa = \cF^{\upsilon} \circ \cG^{\gamma}$ and set $\gamma_k = |\cG[k]|_\gamma / k!$ for all $k$.
	Let $B_{(1)} \ge  B_{(2)} \ge \ldots$ denote the descendlingy ordered list of the sizes of the $\cG$-components of the random enriched tree $\mA_n^\cR$. 
	\begin{enumerate}
		\item If $\mathbf{w}$ has type I$a$,  then 
		\[B_{(1)} \le \frac{\log n}{\log(\rho_\phi /\tau)} + o_p(\log n).\]
		In particular, if $\rho_\phi = \infty$, then
		\[
		B_{(1)} = o_p(\log n).
		\]
		\item If $\mathbf{w}$ has type $I\alpha$, then \[B_{(1)} = o_p(\sqrt{n}).\]
		If $\mathbf{w}$ has type $I\beta$, then \[
		B_{(1)} = o_p(n).
		\]
		In both cases there is a positive constant $C$ (that may depend on $\mathbf{w}$) such that for all $k \ge 1$
		\[
		\Pr{B_{(1)} \ge k} \le C n \Pr{\xi \ge k}.
		\]
		\item Suppose $\mathbf{w}$ has type II, $\gamma_k \sim c k^{-\beta} \rho_\cG^{-k}$ for some constants $\rho_\cG, c>0$ and $\beta>2$, and that  $\cF^\upsilon(z)$ is analytic at $\cG^\gamma(\rho_\cG)$. Set $\alpha = \min(2,\beta -1)$ and $c' = c / \cF^\upsilon(\cG^\gamma(\rho_\cG))$.
		\begin{enumerate}[a)]
			\item It holds that \[B_{(1)} = (1-\nu)n + O_p(n^{1/\alpha}) \quad \text{and} \quad n^{-1/\alpha}((1-\nu)n - B_{(1)}) \convdis X_\alpha,\] where $X_\alpha$ is an $\alpha$-stable random variable with Laplace transform \[\Ex{e^{-tX_\alpha}} = \exp(c' \Gamma(-\alpha)t^\alpha), \quad \operatorname{Re}\, t \ge 0.\]
			\item It holds that \[B_{(2)} = O_p(n^{1/\alpha}) \quad \text{and}\quad n^{-1/\alpha} B_{(2)} \convdis W,\] with $W$ satisfying the Fr\'echet distribution \[\Pr{W \le x} = \exp(- \frac{c'}{\alpha}x^{-\alpha}),\quad x \ge 0.\]
			\item For any $j \ge 2$, \[B_{(j)} = O_p(n^{1/\alpha}) \quad \text{and}\quad n^{-1/\alpha} B_{(j)} \convdis W_j,\] where $W_j$ has the density function \[c' x^{-\alpha -1}\frac{(c' \alpha^{-1} x^{-\alpha})^{j-2}}{(j-2)!}\exp(-c' \alpha^{-1} x^{-\alpha}),\quad x \ge 0,\] and \[c' \alpha^{-1} W_j^{-\alpha} \eqdist \Gamma(j-1,1).\] 
		\end{enumerate}
	\end{enumerate}
\end{theorem}

More recent results due to Kortchemski \cite{MR3335012} allow for the following generalization.
\begin{proposition}[Component size asymptotics in a subcase of case II]
	\label{pro:new}
	Suppose that the weight sequence $\mathbf{w}$ has type II and $\gamma_k = f(k) k^{-\beta} \rho_{\cG}^{-k}$ for some constants $\rho_{\cG}>0$, $\beta >2$ and a function $f$ that varies slowly at infinity. Set $\alpha=\min(2, \beta-1)$ and let $(Y_t)_{t\ge1}$ denote a spectrally positive L\'evy process with Laplace exponent
	\[
	\Ex{\exp(-\lambda Y_t)} = \exp(t \lambda^{\alpha}).
	\]
	Then there exists a slowly varying function $g$ such that
	\[
	\frac{(1-\nu)n - B_{(1)}}{g(n)n^{1/\alpha}} \convdis Y_1 \qquad \text{and} \qquad B_{(2)} = O_p(g(n)n^{1/\alpha}).
	\]
\end{proposition}

Theorem~\ref{te:compsize} and Proposition~\ref{pro:new} provide a novel and  simple proof for extremal component-size asymptotics of many examples of random discrete structures. As detailed in Section~\ref{sec:blocksizes}, they apply to random planar graphs and, more generally, random graphs from  so called planar-like classes. In this way, we may recover the central limit theorem for the largest block size in random planar graphs, and provide new results for the sizes of the $k$th-largest blocks for $k \ge 2$. For random outerplanar maps drawn with probability proportional to weights corresponding to the degrees of their faces we obtain novel results for the block-size limits in various settings.

\subsection{Applications to random discrete structures}
\label{sec:appcomb}
In this section, we apply the general results of Sections~\ref{sec:partB} to $\ref{sec:compsize}$ to the families of random discrete structures discussed in Section~\ref{sec:intro1}, and provide further main results such as the classification of local limits of random outerplanar maps.

\subsubsection{Applications to random weighted outerplanar maps}
\label{sec:outerplanar}
As discussed in Section~\ref{sec:decompouter}, the species $\cO$ of simple outerplanar maps is isomorphic to the species of $\Seq \circ \cD$-enriched trees, with $\Seq$ denoting the species of linear orders, and $\cD$ the species of dissections of edge-rooted polygons in which the root vertex does not contribute to the total size of the dissection.

Let $\gamma$ be a weighting on $\cD$, and $\kappa$ the corresponding weighting on $\Seq \circ \cD$ that assigns to each sequence of dissections $D_1, \ldots, D_t$ the weight $\gamma(D_1) \cdot \ldots \cdot \gamma(D_t)$. The weighting $\omega$ on the species $\cO$ that corresponds to $\kappa$ as in Equation~\eqref{eq:omw} is given by
\[
\omega(M) = \prod_D \gamma(D)
\]
for each simple outerplanar map $M$, with the index $D$ ranging over the blocks of $M$. Let $\mathbf{w} = (\omega_k)_k$ be the weight-sequence with $\omega_k = |(\Seq \circ \cD)[k]|_\kappa / k!$. Note that in the bijection of Section~\ref{sec:decompouter} each block of $M$ has a canonical root edge, depending on the location of the root edge of $M$. The weight $\gamma(D)$ may depend on the location of its root edge.

In the following we study the random map $\mO_n^\omega$ drawn from the set $\cO[n]$ with probability proportional to its $\omega$-weight. Recall that we let $d_{\textsc{block}}$ denote the block-distance on the vertex set of any connected graph, and that for any rooted graph $C^\bullet$, we let $V_k(C^\bullet)$ and $U_k(C^\bullet)$ denote the $k$-neighbourhoods with respect to the graph metric and the block-metric, respectively.

Note that if we set $\gamma(D)=1$ for each dissection $D$, then $\mO_n^\omega$ is the uniform outerplanar map with $n$ vertices. If we set $\gamma(D)=\one_{D \text{ is bipartite}}$, then $\mO_n^\omega$ is the uniform bipartite outerplanar map.  If $(\iota_k)_{k \ge 3}$ denotes a weight sequence, then we may set $\gamma(D) = \prod_F \iota_{|F|}$ with the index $F$ ranging over all inner faces of the dissection $D$, and $|F|$ denoting the number of bounding edges of the face $F$. So the model $\mO_n^\omega$ encompasses the case of random outerplanar map drawn with probability proportional to a product of weights corresponding to the degrees of its faces. The highlight of this section will be  that for {\em arbitrary} weight-sequences $(\iota_k)_{k \ge 3}$ with $\iota_\ell \ne 0$ for at least one $\ell$ the random map $\mO_n^\omega$ admits both a Benjamini--Schramm limit and a local weak limit near its root-edge.

\paragraph*{Local convergence - the infinite spine case}

\begin{figure}[t]
	\centering
	\begin{minipage}{1.0\textwidth}
		\centering
		\includegraphics[width=0.43\textwidth]{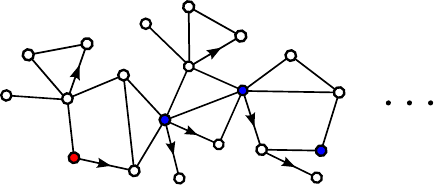}
		\caption{The local weak limit of random outerplanar maps with type I.}
		\label{fi:limouter}
	\end{minipage}
\end{figure}

We are going to apply Theorems~\ref{te:local}, \ref{te:strengthend} and \ref{te:thmben} to the random outerplanar map $\mO_n^\omega$. If the weight sequence $\mathbf{w}$ has type I, then the random $\Seq\circ\cD$-enriched tree $(\hat{\cT}, \hat{\beta})$ corresponds to a random locally finite outerplanar map \label{pp:oh}$\hat{\mO}$ according to the bijection in Section~\ref{sec:decompouter}. Likewise, the $\Seq\circ\cD$-enriched tree $(\cT^*, \beta^*)$ corresponds to a random locally finite outerplanar graph \label{pp:ohs}$\hat{\mO}_*$. The two limit objects are, in general, not identically distributed. See Remarks~\ref{re:mapdistr} and \ref{re:mapdistr2} below for detailed descriptions of their distributions. Theorems~\ref{te:local} and \ref{te:thmben}  yield that $\hat{\mO}$ is the local weak limit of $\mO_n^\omega$ with respect to convergence of neighbourhood around the origin of the root-edge, and $\hat{\mO}_*$ is the Benjamini--Schramm limit of $\mO_n^\omega$. See Figures~\ref{fi:limouter} and \ref{fi:limouter2} for illustrations of these limit objects.

\begin{theorem}[Local weak convergence and Benjamini--Schramm convergence of random outerplanar maps with type I] \text{ } \label{te:convouter}
	\begin{enumerate}
		\item
		If the weight-sequence $\mathbf{w}$ has type I, then the random outerplanar map $\mO_n^\omega$ converges in the local weak sense toward $\hat{\mO}$ as $n$ becomes large. A bit stronger, the $k$-neighbourhoods with respect to the block distance $d_{\textsc{BLOCK}}$ around the root vertex converge, that is
		\[
		\lim_{n \to \infty} \Pr{U_k(\mO_n^\omega) \in \cE} = \Pr{U_k(\hat{\mO}) \in \cE}
		\]
		for any set $\cE$ of finite unlabelled rooted graphs. Similarly, $\hat{\mO}_*$ is the Benjamini--Schramm limit of $\mO_n^\omega$, and 
		\[
		\lim_{n \to \infty} \Pr{U_k(\mO_n^\omega, v_n) \in \cE} = \Pr{U_k(\hat{\mO}_*) \in \cE}
		\]
		with $v_n$ denoting a uniformly at random drawn vertex of $\mO_n^\omega$.
		\item If the weight-sequence $\mathbf{w}$ has type I$\alpha$, then we obtain a stronger form of convergence. For any sequence of non-negative integers $k_n = o(n^{1/2})$, the total variation distance of the $k_n$-block-neighbourhoods converges to zero:
		\[
		d_{\textsc{TV}}(U_{k_n}(\mC_n^\omega,v_n), U_{k_n}(\hat{\mC})) \to 0.
		\]
		Moreover, $\hat{\mO}_*$ is the Benjamini--Schramm limit of $\mO_n^\omega$. If $v_n$ denotes a uniformly at random chosen vertex of $\mO_n^\omega$, then
		\[
		d_{\textsc{TV}}(U_{k_n}(\mO_n^\omega,v_n), U_{k_n}(\hat{\mO}_*)) \to 0.
		\]
	\end{enumerate}
\end{theorem}

As detailed in Remarks~\ref{re:mapdistr} and \ref{re:mapdistr2} below, the limit objects $\hat{\mO}$ and $\hat{\mO}_*$ have a canonical embedding in the plane and are hence not just graphs but infinite planar maps. The convergence in Theorem~\ref{te:convouter} respects the the planar structures, that is, we actually obtain the stronger form of convergence of neighbourhoods embedded into the plane.

\begin{figure}[t]
	\centering
	\begin{minipage}{1.0\textwidth}
		\centering
		\includegraphics[width=0.5\textwidth]{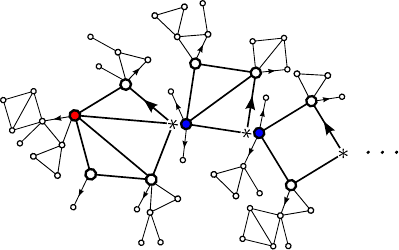}
		\caption{The Benjamini--Schramm limit of outerplanar maps with type I.}
		\label{fi:limouter2}
	\end{minipage}
\end{figure}

\begin{remark}
	\label{re:mapdistr}
	The distribution of the local weak limit $\hat{\mO}$ has the following description.
	\begin{enumerate}
		\item Let $(\mD_i^\bullet)_{i \ge 1}$ be a family of independent identically distributed $\cD^\bullet$-objects following a weighted Boltzmann distribution 
		$\mathbb{P}_{(\cD^\bullet)^\gamma, \tau}$. Concatenate the $\mD_i^\bullet$ by identifying the pointed vertex of $\mD_i^\bullet$ with the root $*$-vertex of $\mD_{i+1}^\bullet$ for all $i$. The resulting chain of dissections $C$ has an infinite spine given by the $*$-vertices of the $\mD_i^\bullet$. We consider this chain as rooted at the start of the spine, which is the root $*$-vertex of $\mD_1^\bullet$. 
		\item Let $\mO$ denote a random outerplanar map that follows a Boltzmann distribution $\mathbb{P}_{\cO^\omega, \tau/\phi(\tau)}$. For each non-spine vertex $v$ of $C$ take a fresh independent copy of $\mO$ and identify its root with $v$, such that the map is attached from the outside. For each spine-vertex $v$ of $C$ take two fresh independent copies of $\mO$ and identify their roots with $v$ by glueing one from each side.
		\item The resulting graph consisting of $C$ with one Boltzmann map glued to each non-spine vertex and two such maps glued to each spine vertex is distributed like $\hat{\mO}$.
	\end{enumerate}			
\end{remark}

\begin{remark}
	\label{re:mapdistr2}
	The distribution of the Benjamini--Schramm limit $\hat{\mO}_*$ may be described as follows. 
	\begin{enumerate}
		\item Similarly as for $\hat{\mO}$ we start with an i.i.d. family $(\mD_i^\bullet)_{i \ge 1}$ of $\mathbb{P}_{(\cD^\bullet)^\gamma, \tau}$-distributed pointed dissections. But this time, we form a chain $C_*$ by glueing together the pointed vertex of $\mD_{i+1}^\bullet$ with the root $*$-vertex of $\mD_i^\bullet$ for all $i$. $C_*$ has a spine consisting of the marked vertices of the $\mD_i^\bullet$ that we consider as rooted at the marked vertex of $\mD_1^\bullet$.
		\item Let $\mO$ denote a random outerplanar map that follows a Boltzmann distribution $\mathbb{P}_{\cO^\omega, \tau/\phi(\tau)}$. For each non-spine vertex $v$ of $C_*$ take a fresh independent copy of $\mO$ and identify its root with $v$. We do the same for the root of the chain, that is the first spine-vertex. For each non-root spine-vertex $v$ of $C_*$ take two fresh independent copies of $\mO$ and identify their roots with $v$ by glueing one from each side.
		\item The resulting graph is distributed like $\hat{\mO}_*$.	It consists of $C_*$ with one Boltzmann map glued to each non-spine vertex and the first spine-vertex, and two such maps glued to all the other spine-vertices.
	\end{enumerate}
\end{remark}

If the weight sequence $\mathbf{w}$ has type I$a$, then the root-degrees in $\hat{\mO}$ and $\hat{\mO}_*$ have finite exponential moments. Hence in this case both are almost surely recurrent by a general result~\cite[Thm. 1.1]{MR3010812} on distributional limits of random planar graphs.


\paragraph*{Local convergence - the finite spine case}

Our first result reduces the study of random outerplanar maps to the study of random dissections.

\begin{figure}[t]
	\centering
	\begin{minipage}{0.8\textwidth}
		\centering
		\includegraphics[width=1\textwidth]{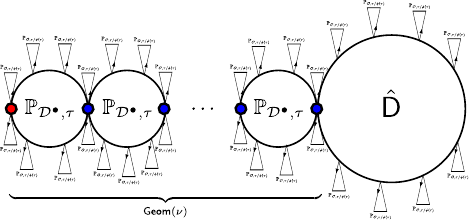}
		\caption{The local weak limit of type II random outerplanar maps.}
		\label{fi:limoutfresh}
	\end{minipage}
\end{figure}

\begin{theorem}[Local weak convergence of outerplanar maps with type II]
	\label{te:convout}
	Suppose that $\mathbf{w}$ has type II and that the composition $\Seq \circ \cD^\gamma$ has convergent type in the sense of Definition~\ref{def:convergent}.  Let $\mD_n^\gamma$ denote the random random dissection that gets drawn from $\cD[n]$ with probability proportional to its $\gamma$-weight. If $\mD_n^\gamma$ converges in the local weak sense toward a limit object $\hat{\mD}^\bullet$, then the random outerplanar map $\mO_n^\omega$ converges as well.  The limit object \label{pp:oh2}$\hat{\mO}$ may be constructed as illustrated in Figure~\ref{fi:limoutfresh}. That is:
	\begin{enumerate}
		\item Draw a random integer $L\ge0$ that follows the geometric distribution
		\[
		\Pr{L=\ell} = \nu^\ell(1-\nu).
		\]
		\item Let $(\mD_i^\bullet)_{1 \le i \le L}$ be a family of independent identically distributed $\cD^\bullet$-objects following a weighted Boltzmann distribution 
		$\mathbb{P}_{(\cD^\bullet)^\gamma, \tau}$. Concatenate the $\mD_i^\bullet$ by identifying the pointed vertex of $\mD_i^\bullet$ with the root $*$-vertex of $\mD_{i+1}^\bullet$ for all $i \le L-1$. Identify the root of the limit object $\hat{\mD}$ with the marked vertex of $\mD_L^\bullet$. The resulting chain $C$ has a finite spine with $L+1$ vertices given by the root-vertices of the $\mD_i^\bullet$ and $\hat{\mD}$.
		\item Let $\mO$ denote a random outerplanar map that follows a Boltzmann distribution $\mathbb{P}_{\cO^\omega, \tau/\phi(\tau)}$. For each non-spine vertex $v$ of $C$ take a fresh independent copy of $\mO$ and identify its root with $v$. For each spine-vertex $v$ of $C$ take two independent copies of $\mO$ and identify their roots with $v$ by glueing one from each side.  The resulting outerplanar map follows the distribution of $\hat{\mO}$.
	\end{enumerate}
\end{theorem}

In Section~\ref{sec:diss22} we obtain concrete limit theorems for random dissections in two settings, where the limit either has an  infinite spine of circles glued together at edges if the dissection, or a finite random-length spine of this type with a doubly infinite path attached to its end. We observe that if $\mathbf{w}$ has type II and if the $\gamma$-weights correspond an enriched tree weighting on the species $\cD$ of dissections, then $\mD_n^\gamma$ also has type II. This means that in this setting the type I limit of $\mD_n^\gamma$ cannot appear as core of $\mO_n^\omega$.


\begin{lemma}
	\label{le:light}
	Suppose that the weighting $\gamma$ on the species of dissections $\cD$ is of the form
	\begin{align}
	\label{eq:needmorenames}
	\cD^\gamma = \cX \cdot (\Seq \circ \Seq_{\ge 1})^\upsilon (\cD^\gamma)
	\end{align}
	for an arbitrary weighting $(\Seq \circ \Seq_{\ge 1})^\upsilon$.
	This includes the case where we draw the random outerplanar map $\mO_n^\omega$ according to weights corresponding to its face-degrees. Recall that $\mO_n^\omega$ has type I, II or III depending on whether $\nu\ge 1$, $0< \nu < 1$ or $\nu=0$.
	\begin{enumerate}
		\item If $\cD^\gamma$ has type I, then $\mO_n^\omega$ has type I$a$ with $\nu= \infty$.
		\item Suppose that $\cD^\gamma$ has type II. Let $\tau_\cD$ denote the  radius of convergence of \[\phi_\cD(z) := (\Seq \circ \Seq_{\ge 1})^\upsilon(z),\] and set \[\nu_\cD := \lim_{t \nnearrow \tau_\cD} \phi_\cD'(t)t/\phi_\cD(t) \in ]0,1[.\]
		\subitem a) If $\tau_\cD < 1$, then 
		\[
		\nu = \frac{\tau_\cD}{(1- \tau_\cD)(1 - \nu_\cD)} \in ]0, \infty[.
		\]
		\subitem b) If $\tau_\cD \ge 1$, then $\mO_n^\omega$ has type I$a$ with $\nu=\infty$.
		\item If $\cD^\gamma$ has type III, then $\nu=0$ and $\mO_n^\omega$ has type III.
	\end{enumerate}
\end{lemma}

\begin{figure}[t]
	\centering
	\begin{minipage}{1.0\textwidth}
		\centering
		\includegraphics[width=1.0\textwidth]{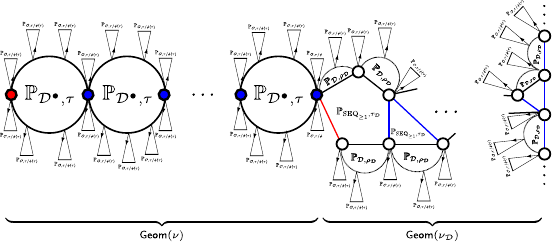}
		\caption{The local weak limit of type II random face-weighted outerplanar maps.
		}
		\label{fi:limoutice}
	\end{minipage}
\end{figure}

In light of Lemma~\ref{le:light} and Theorem~\ref{te:convout}, we obtain the following classification of the local weak limits of random outerplanar maps drawn with probability proportional to {\em arbitrary} weights assigned to their faces.

\begin{theorem}(Classification of local weak limits of random face-weighted outerplanar maps)
	\label{te:facethm}
	Let $(\iota_k)_{k \ge 3}$ be a sequence of non-negative weights, such that at least one weight is positive. Consider the case where $\mO_n^\omega$ is the random outerplanar map with $n$ vertices drawn with probability proportional to the product of $\iota$-weights corresponding to the degrees of its inner faces. That is, we define the weighting $\gamma$ on the class of dissections $\cD$  such that
	\[
	\cD^\gamma \simeq \cX \cdot (\Seq \circ \Seq_{\ge 1}^\iota)( \cD^\gamma) \qquad \text{with} \qquad \Seq_{\ge 1}^\iota(z) = \sum_{k \ge 1} \iota_{k+2} z^k.
	\]
	\begin{enumerate}
		\item  If $\mathbf{w}$ has type $I$, then $\mO_n^\omega$ convergences in the local weak sense by Theorem~\ref{te:convouter} toward a limit graph that contains no doubly-infinite paths.
		\item If $\mathbf{w}$ has type $II$, then the local weak limit of the  random outerplanar map $\mO_n^\omega$ is the random map $\hat{\mO}$ given in Theorem~\ref{te:convout} for the type II dissection limit $\hat{\mD}$ given in Theorem~\ref{te:bsconvdis} for random dissections with $\iota$-face-weights. The distribution of the limit planar map is illustrated in Figure~\ref{fi:limoutice}, where we use the notation \[\rho_\cD = \tau_\cD (1 - \Seq_{\ge 1}^\iota(\tau_\cD))\] with  $\tau_\cD$ denoting the radius of convergence of the generating series $\Seq_{\ge 1}^\iota(z)$. The limit almost surely contains  doubly-infinite paths and differs in this aspect from the limit of type I outerplanar maps.
		\item If $\mathbf{w}$ has type III, then $\mO_n^\omega$ converges in the local weak sense toward a single doubly-infinite path.

	\end{enumerate}
\end{theorem}

In the following example we provide an explicit weight sequence $(\iota_k)_{k \ge 3}$ for which the random outerplanar map $\mO_n^\omega$ exhibits the interesting type II limit with doubly infinite paths.

\begin{example}[Convergent type II outerplanar maps]
	\label{ex:cool}
	Consider the weight sequence $(\iota_k)_{k \ge 3}$ with \[\Seq_{\ge 1}^\iota(z) =  \label{eq:ttttemp} \sum_{k\ge 1} \iota_{k+2} z^k = z/2 + (1-4z)^{3/2}/12.\] Then it holds that
	\[
	\iota_k \sim (16 \sqrt{\pi})^{-1} k^{-5/2} 4^k, \qquad \nu_\cD =3/23 <1, \qquad \nu = 23/60<1.
	\]
	Thus the random $\iota$-face-weighted outerplanar map $\mO_n^\omega$ converges to the limit illustrated in Figure~\ref{fi:limoutice}.
\end{example}


The following result establishes the Benjamini--Schramm limit $\hat{\mO}_*$ of arbitrary  random face-weighted outerplanar maps. In general, this limit is not identical to the local weak limit~$\hat{\mO}$. Using a rerooting invariance of face-weighted dissections and the explicit description of the limit $\hat{\mO}_*$  obtained in the proof, we show that there is a natural coupling such that the map $\hat{\mO}_*$ is an induced subgraph of the map $\hat{\mO}$. In particular, the two limits are identically distributed if and only if the weight-sequence $\mathbf{w}$ has type III, in which case they both are deterministic doubly-infinite paths.

\begin{theorem}(Classification of Benjamini--Schramm limits of random face-weighted outerplanar maps)
	\label{te:bsfaces}
	Let $(\iota_k)_{k \ge 3}$ be a sequence of non-negative weights, such that at least one weight is positive, and let $\mO_n^\omega$ be the random outerplanar map with $n$ vertices drawn with probability proportional to the product of $\iota$-weights corresponding to the degrees of its inner faces as in Theorem~\ref{te:facethm}. That is, we consider the case of a weighting $\gamma$ on the class of dissections $\cD$  such that
	\[
	\cD^\gamma \simeq \cX \cdot (\Seq \circ \Seq_{\ge 1}^\iota)( \cD^\gamma) \qquad \text{with} \qquad \Seq_{\ge 1}^\iota(z) = \sum_{k \ge 1} \iota_{k+2} z^k.
	\]
	
	\begin{enumerate}
		\item  If $\mathbf{w}$ has type $I$, then $\mO_n^\omega$ convergences in the Benjamini--Schramm sense by Theorem~\ref{te:convouter} toward a limit graph $\hat{\mO}_*$ that contains no doubly-infinite paths. 
		\item If $\mathbf{w}$ has type $II$, then the  random outerplanar map $\mO_n^\omega$ converges in the Benjamini--Schramm limit toward a limit $\hat{\mO}_*$ that almost surely contains  doubly-infinite paths.
		\item If $\mathbf{w}$ has type III, then $\mO_n^\omega$ converges in the Benjamini--Schramm sense toward a single doubly-infinite path.
	\end{enumerate}
	For the cases I and II, the construction of $\hat{\mO}_*$ is almost identical to the construction of $\hat{\mO}$ in Remark~\ref{re:mapdistr} for the type I case and to the construction of $\hat{\mO}$ in Theorem~\ref{te:convout} for the type II case, with the only difference being that in both cases we identify the root-vertex with the root of only one $\mathbb{P}_{\cO^\omega, \tau/\phi(\tau)}$-distributed outerplanar map instead of two. In particular, $\hat{\mO}$ may be obtained from $\hat{\mO}_*$ by identifying the root vertex of $\hat{\mO}_*$ with the root vertex of an independent $\mathbb{P}_{\cO^\omega, \tau/\phi(\tau)}$-distributed map. 
\end{theorem}

\paragraph*{Limit theorems for the largest blocks and faces}

As for the sizes of the largest blocks in the random outerplanar map $\mO_n^\omega$, we immediately obtain the following result.

\begin{corollary}
	\label{eq:cosize}
	In the setting of Theorem~\ref{te:facethm}, we may apply Theorem~\ref{te:compsize} and Proposition~\ref{pro:new} directly to obtain bounds and limits  for the sizes $B_{(i)}$ of the $i$th largest block in the random outerplanar map $\mO_n^\omega$.
\end{corollary}

Studying the face-degrees requires some extra work, which essentially amounts to applying the general results for enriched trees twice, once for the outerplanar maps and once for the dissections.


\begin{corollary}
	\label{co:gottaname} In the setting of Theorem~\ref{te:facethm}, suppose that the weight-sequence $\mathbf{w}$ has type II and 
	\[
	[z^k] \Seq_{\ge 1}^\iota(z) = f(k) k^{-\beta} r^{-\beta}
	\]
	for some constants $r>0$ and $\beta>2$, and a slowly varying function $f$. Set $\alpha=\min(2, \beta-1)$. Then there is a slowly varying function $g$ such that the size $F_{(1)}$ of the largest face satisfies the central limit theorem
	\[
	\frac{ (1-\nu)(1- \nu_\cD)n - F_{(1)}}{g(n) n^{1/\alpha}} \convdis X_\alpha,
	\] where $X_\alpha$ is an $\alpha$-stable random variable with Laplace transform \[\Ex{e^{-tX_\alpha}} = \exp(\Gamma(-\alpha)t^\alpha), \quad \operatorname{Re}\, t \ge 0.\]
	The size $F_{(2)}$ of the second largest face admits the bound $F_{(2)} = O_p(\bar{g}(n)n^{1/\alpha})$ for some sequence $\bar{g}(n)$ satisfying $\bar{g}(n)= o(n^{\epsilon})$ for all $\epsilon>0$.
\end{corollary}

In the proof of Theorem~\ref{te:facethm} we applied Lemma~\ref{le:couplingschroeder} to also construct a coupling of $\mO_n^\omega$ with an $n$-leaves simply generated tree $\tau_n$ such that the out-degrees in $\tau_n$ correspond precisely to the sizes of the $\cD$-objects in $\mO_n^\omega$. Likewise, there is a coupling of the random face-weighted dissection $\mD_n^\gamma$ with such a tree, but with a different weight-sequence.

This yields an alternative approach to Corollaries~\ref{eq:cosize} and \ref{co:gottaname}, as all results (future and present) regarding the vertex outdegrees in a Galton--Watson tree conditioned on having $n$ leaves directly translate to results for the block-sizes and face-sizes in the random outerplanar map $\mO_n^\omega$, and the face-sizes in the random dissection $\mD_n^\gamma$. Conversely, we may also interpret the results from Corollary~\ref{eq:cosize} as statistics for the vertex outdegrees for this model of random trees.



\subsubsection{Applications to random block-weighted graphs} 
\label{sec:blockcl}

As discussed in Section~\ref{sec:bijblock}, if $\cC$ denotes the class of connected graphs and $\cB$ its nonempty subclass of graphs that are two-connected or  a single edge with its ends, then the corresponding class of rooted graphs $\cC^\bullet$ is isomorphic to the class of $\Set \circ \cB'$-enriched trees.

Let $\gamma$ be a weighting on the species $\cB$, and $\kappa$ the weighting on the species $\Set \circ \cB'$ that assigns to any collection of derived blocks the product of the individual $\gamma$-weights. The weighting $\omega$ on the species $\cC^\bullet$ and $\cC$, that corresponds to $\kappa$ as in Equation~\ref{eq:omw}, assigns to any connected graph $C$ from this class the weight
\[
\omega(C) = \prod_B \gamma(B)
\]
with the index $B$ ranging over all blocks of the graph $C$. We let $\mathbf{w} =(\omega_k)_k$ denote the weight sequence given by 
\[
\omega_k = |(\Set \circ \cB')[k]|_\kappa / k!.
\] 

In the following we are going to focus on the random connected graph $\mC_n^\omega$ drawn from the set $\cC[n]$ with probability proportional to its $\omega$-weight. It is no loss of generality to consider connected graphs, as we are going to observe in Theorem~\ref{te:discon} below.

\paragraph*{Giant connected components in disconnected graphs}

Let $\cG^\upsilon$ denote the weighted species of all (possibly disconnected) graphs such that the weight of each graph is the product of $\omega$-weights of its connected components. Similarly to the connected case, we may consider the random graph $\mG_n^\upsilon$ that gets sampled from the set $\cG[n]$ with probability proportional to its weight. The following result shows in complete generality that it is not a restriction to focus our studies on connected graphs, as "almost all" properties of the random graph $\mC_n^\omega$ carry over automatically to the random graph $\mG_n^\upsilon$.

\begin{theorem}
	\label{te:discon}
	The largest connected component of the random graph $\mG_n^\upsilon$ has size \[K_n = n + O_p(1)\] as $n$ becomes large. Up to relabelling, it is distributed like the randomly sized random connected graph $\mC_{K_n}^\omega$.
\end{theorem}

McDiarmid~\cite{MR2507738} observed such a result for uniform random graphs from proper addable minor-closed classes, together with convergence of the small fragments toward a Boltzmann limit graph. These results where later recovered and extended in Stufler~\cite[Thm. 4.2 and Section 5]{doi:10.1002/rsa.20771} to random block-weighted classes with analytic generating functions. In order to complete the picture, we have to treat the case with superexponential weights. Our proof goes by applying  general results on Gibbs partitions that we established in Theorem~\ref{te:need} and Lemma~\ref{le:supen}. Corollary~\ref{co:ne} also allows us to make the following observation.
\begin{corollary}
	\label{co:discon}
	If the weight-sequence $\mathbf{w}$ has type III, then there is an integer $n_0\ge 1$ such that the largest connected component in the random graph $\mG_n^\upsilon$ has with high probability at least $n-n_0$ vertices. If the complete graph with $2$ vertices has positive $\omega$-weight, then the random graph $\mG_n^\omega$ is connected with probability tending to $1$ as $n$ becomes large.
\end{corollary}
This generalizes \cite[Ex. II.15]{MR2483235}, where a direct counting argument was used to show that almost all labelled graphs are connected.

\paragraph*{Local convergence - the infinite spine case} 
We are going to apply Theorems~\ref{te:local} and \ref{te:strengthend} to the random graph $\mC_n^{\omega}$. If the weight-sequence $\mathbf{w}$ has type I, then the tree $\hat{\cT}$ is locally finite and the decorated limit tree $(\hat{\cT}, \hat{\beta})$ corresponds, according to the bijection discussed in Section~\ref{sec:bijblock}, to a locally finite graph \label{pp:ch}$\hat{\mC}$. For each pair $(u_i, u_{i+1})$ of consecutive spine vertices of $\hat{\cT}$ there is a unique block in that graph containing $u_i$ and $u_{i+1}$. Hence, in the type I case, the limit $\hat{\mC}$ is shaped like an infinite sequence of doubly rooted blocks, where at each vertex a random finite graph is inserted. See Figure~\ref{fi:infgraph}. The doubly rooted blocks are independent and identically distributed. Conditionally on this infinite sequence of blocks, the random finite graphs inserted at each vertex are also independent and identically distributed. See Remark~\ref{re:distribution} below for a detailed justification and description of the individual distributions.

\begin{figure}[t]
	\centering
	\begin{minipage}{1.0\textwidth}
		\centering
		\includegraphics[width=0.5\textwidth]{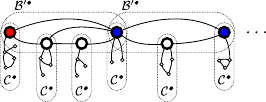}
		\caption{Illustration of the infinite limit graph  $\hat{\mC}$ for weight sequences with type I.}
		\label{fi:infgraph}
	\end{minipage}
\end{figure}

\begin{theorem}[Benjamini--Schramm convergence of random connected graphs]
	\label{te:loclunl} \text{ }
	\begin{enumerate}
		\item
	If the weight-sequence $\mathbf{w}$ has type I, then the random graph $\mC_n^\omega$ converges as $n$ becomes large in the Benjamini--Schramm sense toward the limit graph $\hat{\mC}$. Slightly stronger, the $k$-neighbourhoods with respect to the block distance $d_{\textsc{BLOCK}}$ around a uniformly at random drawn vertex $v_n$ converge. That is,
	\[
		\lim_{n \to \infty} \Pr{U_k(\mC_n^\omega,v_n) \in \cE} = \Pr{U_k(\hat{\mC}) \in \cE}
	\]
	for any set $\cE$ of finite unlabelled rooted graphs.
	\item If the weight-sequence $\mathbf{w}$ has type I$\alpha$, then we establish a significantly stronger form of convergence. For any sequence of non-negative integers $k_n = o(n^{1/2})$, the total variation distance of the $k_n$-block-neighbourhoods converges to zero:
	\[
		d_{\textsc{TV}}(U_{k_n}(\mC_n^\omega,v_n), U_{k_n}(\hat{\mC})) \to 0.
	\]
	\end{enumerate}
\end{theorem}
Item {\em (2)} of Theorem~\ref{te:loclunl} applies in particular to the special case of a uniform random labelled graph $\mC_n$ from a subcritical graph class in the sense of Drmota, Fusy, Kang, Kraus, and Ru{\'e} \cite[Ch. 4.1]{MR2873207}. In our setting this corresponds to  a subcase of type I$a$,  which is itself a proper subcase of type I$\alpha$. 

Hence, Theorem \ref{te:loclunl} recovers the asymptotic degree distribution of random vertices in random graphs from subcritical classes, which was established by Bernasconi, Panagiotou and Steger \cite{MR2538796,MR2534261} using different methods. Our contribution in this regard is the probabilistic description of the limit distribution as the root degree of a limit graph constructed from the classical Kesten tree.

We emphasize that the present work focus on random labelled graphs. As stated in the introduction, random unlabelled enriched trees and special cases are treated in the author's work~\cite{StEJC2018}.

Local weak convergence of the uniform random labelled graph $\mC_n$ from a subcritical graph class was verified independently in~\cite[Lem 3.3]{2015arXiv151203572G}. The proof given there is by a generalization of an enumerative argument  by Aldous~\cite[Thm. 2, Eq. (11)]{MR1085326} for the special case where $\mC_n$ is the uniform labelled tree. 


Care has to be taken when trying to strengthen the form of convergence in  Theorem~\ref{te:loclunl}. First, it is clear that the random graph $\mC_n$  does \emph{not} converge almost surely in the Benjamini--Schramm sense. The sequence of random graphs $(\mC_n)_{n \ge 1}$ is independent. A well-known application of Kolmogorov's 0-1 law shows that any almost surely convergent sequence of independent random variables with values in a Polish space converges to a constant. In the context of local convergence, the target metric space is given by the neighbourhood metric on the collection of isomorphism classes of rooted locally finite connected graphs. Hence if the independent sequence of random graphs $(\mC_n)_{n \ge 1}$ converges almost surely, then the limit graph must follow a degenerate distribution. But the distributional limit $\hat{\mC}$ of (the uniformly rooted version of) $\mC_n$ is clearly not a constant graph.  There are examples of almost surely convergent \emph{dependent} sequences $(G_n)_{n \ge 1}$ of random trees~\cite{MR2060629} and random maps~\cite{MR3254733}.  There the goal is to establish a growth procedure that tells us explicitly (not just a mere proof of existence of a coupling as in Skorokhod's representation theorem) how to construct $G_{n+1}$ from $G_n$  in a clever way to get almost sure convergence.

A strengthened form of convergence may be stated in terms of random probability measures. Recall from Section~\ref{sec:lowe} that $\mathbb{B}$ denotes the Polish space of rooted locally finite connected graphs. We may consider the corresponding Polish space $\mathscr{M}(\mathbb{B})$ of probability measures on the Borel $\sigma$-field $\mathscr{F}$ of $\mathbb{B}$, equipped with the weak convergence topology. Let $\mathbb{P}_n$ denote  the random  probability measure (r.p.m. for short) on $\mathscr{F}$ that corresponds to the $n$ equally likely rooted versions of the graph $\mC_n$. That is, $\mathbb{P}_n$ is a \emph{random} element of $\mathscr{M}(\mathbb{B})$. It is not hard to see that \[
\text{$\mathbb{P}_n$ converges almost surely to a \emph{deterministic} probability measure $\hat{\mathbb{P}} \in \mathscr{M}(\mathbb{B})$},\] given by the law of the Benjamini--Schramm limit $\hat{\mC}$ of $(\mC_n)_{n \ge 1}$: As stated in~\eqref{eq:cone2}, this is equivalent to requiring that the random variable
\[
	\mathbb{P}_n(V_k^{-1}(G^\bullet)) = \frac{|\{v \in \mC_n \mid V_k(\mC_n, v) \simeq G^\bullet\}|}{n}
\]
converges almost surely to the constant $\hat{\mathbb{P}}(V_k^{-1}(G^\bullet)) = \Pr{V_k(\hat{\mC})=G^\bullet}$ for each integer $k \ge 1$ and each of the countably many finite connected rooted graphs $G^\bullet$ with radius at most $k$. And this follows from Theorem~\ref{te:thmben}, which transfers corresponding  bounds for the number of fringe subtrees in the tree $\cT_n$.

\begin{remark}[The distribution of the limit object] 
	\label{re:distribution}
	Recall the notation from Sections~\ref{sec:pretree} and \ref{se:locosi}. The distribution of the limit graph $\hat{\mC}$ may equivalently be described as follows. Compare with Figure~\ref{fi:infgraph}.
	\begin{enumerate}
		\item Let $(\mB'^\bullet_i)_{i \ge 1}$ denote an independent identically distributed sequence of random $\cB'^\bullet$-objects following a weighted Boltzmann distribution with parameter $\tau$.
		For each $i$, identify the pointed vertex of $\mB'^\bullet_i$ with the $*$-vertex of $\mB'^\bullet_{i+1}$ in order to form an infinite block-chain $\mB$. We consider $\mB$ as rooted at the $*$-vertex of $\mB'^\bullet_1$. 
		\item Let $\mC^\bullet$ be a random $\cC^\bullet$-object following a weighted Boltzmann distribution with parameter $\tau / \phi(\tau)$.
		For each vertex $v$ of the block-chain $\mB$ (including the root-vertex), take an independent copy of $\mC^\bullet$ and identify its root with the vertex $v$.
	\end{enumerate}
	The block neighbourhood $U_1(\hat{\mC})$ of the limit graph is distributed as follows.
	\begin{enumerate}
		\item Let $\mB'^\bullet$ denote an independent identically distributed sequence of random $\cB'^\bullet$-objects following a weighted Boltzmann distribution with parameter $\tau$.
		\item Draw a random integer $K \ge 0$ according to a Poisson distribution with parameter $(\cB'^\bullet)^\gamma(\tau)$.
		\item Let $\mB'_1, \ldots, \mB'_K$ be independent (conditionally on $K$) random $\cB'$-object that follow a weighted Boltzmann distribution with parameter $\tau$.
		\item Glue the $K+1$ blocks together at their $*$-vertices, and declare the resulting vertex as its root.
	\end{enumerate}
	In particular, the distribution of the degree $d(\hat{\mC})$ is given by
	\begin{align}
	\label{eq:rootdegree}
	d(\hat{\mC}) \eqdist d(\mB'^\bullet) + \sum_{i=1}^K d(\mB'_i).
	\end{align}
\end{remark}

The description of the root degree of the limit graph in \eqref{eq:rootdegree} is very useful, as the involved constants and distributions may be determined explicitly for many graph classes. For the specific example of random outerplanar graphs, we may recover the  expressions for the limit probabilities calculated by Bernasconi and Panagiotou \cite[Cor. 1.2]{MR2538796} in this way.

\begin{corollary}[Combinatorial applications]
	\label{co:cntemb}
	If $\mathbf{w}$ has type I$a$, then the degree $d(\mC_n^\omega,v_n)$ of a uniformly at random chosen point $v_n$ in $\mC_n^\omega$ is arbitrarily high uniformly integrable. Consequently, results by Kurauskas~\cite[Thm. 2.1]{2015arXiv150408103K} and Lyons~\cite[Thm. 3.2]{MR2160416} yield laws of large numbers for subgraph counts and spanning tree counts, where the limiting constant is expressed in terms of the limit graph $\hat{\mC}$.
\end{corollary}

We close the section on the infinite spine case with the following remark regarding a more general model of random graphs, where the rerooting symmetry may break.

\begin{remark}
	We defined the weights governing the distribution of $\mC_n^\omega$ in such a way, that there is no difference between taking an $n$-vertex rooted graph sampled with probability proportional to its $\omega$-weight, and $\mC_n^\omega$ rooted at a uniformly at random chosen location. We may break this rerooting symmetry easily, by considering more general choices of $\kappa$-weights on $\cR = \Set \circ \cB'$. If $\mathbf{w}$ has type I, then the random rooted graph corresponding to $\mA_n^\cR$ has a local weak limit by  Theorem~\ref{te:local} and a Benjamini--Schramm limit by Theorem~\ref{te:thmben}, and the distributions of the two limits may very well differ.
\end{remark}

\paragraph*{Local convergence - the finite spine case}

If the weight-sequence $\mathbf{w}$ has type II, then $\hat{\cT}$ has an almost surely finite spine ending in a vertex with infinite degree. Each offspring of this vertex becomes the root of an independent copy of a subcritical Galton--Watson tree. The random length of the spine follows a geometric distribution, according to the discussion in Section~\ref{sec:types}. As the infinite offspring set of the tip of the spine carries no additional structures, there is a priori no natural interpretation in terms of sets of blocks. If we delete all enriched fringe subtrees dangling from the tip of the spine in the enriched limit tree $(\hat{\cT}, \beta)$, then we are left with an almost surely finite enriched tree that may be interpreted as a rooted random graph $\hat{\mH}^\bullet$ according to the bijection discussed in Section~\ref{sec:bijblock}. If the spine has length zero, then there is no $\cR$-structure assigned to the root of $\hat{\cT}$, and rather than setting $\hat{\mH}^\bullet$ to a single vertex in this case, we let it assume some placeholder value (for example, the empty set) that is different from all graphs. Theorem~\ref{te:local} shows how this random object encodes {\em some} information about asymptotic local properties of the random graph $\mC_n^\omega$.

\begin{corollary}[Convergence of block neighbourhoods in the cases II and III]
	\label{co:prel}
	Let $v_n$ be a uniformly at random drawn vertex of the graph $\mC_n^\omega$. Then for any rooted graph $G^\bullet$ and any integer $k \ge 1$ the following holds.
	\begin{enumerate}
		\item If the weight-sequence $\mathbf{w}$ has type II, then
		\begin{align}
			\label{eq:limit}
			\lim_{n \to \infty} \Pr{U_k(\mC_n^{\omega}, v_n) \simeq G^\bullet} = \Pr{U_k(\hat{\mH}^\bullet) \simeq G^\bullet}.
		\end{align}
		\item If the weight-sequence $\mathbf{w}$ has type III, then
		\begin{align}
			\lim_{n \to \infty} \Pr{U_k(\mC_n^{\omega}, v_n) \simeq G^\bullet} = 0.
		\end{align}
	\end{enumerate}
\end{corollary}

Note that the limit probabilities in Equation~\eqref{eq:limit} sum up to the probability that the spine of $\hat{\cT}$ has non-zero length, which is given by $\nu < 1$. Hence, contrary to the infinite spine case, the convergence in Corollary~\ref{co:prel} is not sufficient to imply Benjamini--Schramm convergence. It does, however, provide information on how parts of the limit object must look, {\em if} such a limit exists.  Our next main theorem  strengthens this greatly, by showing that Benjamini--Schramm convergence of $\mC_n^\omega$ is in fact  \emph{equivalent} to Benjamini--Schramm convergence of a \emph{randomly} sized $2$-connected graph. In particular, it is \emph{sufficient}, if the random graph $\mB_n^\gamma$ drawn from $\cB[n]$ with probability proportional to its weight converges as $n$ deterministically tends to infinity.

\begin{figure}[t]
	\centering
	\begin{minipage}{1.0\textwidth}
		\centering
		\includegraphics[width=1.0\textwidth]{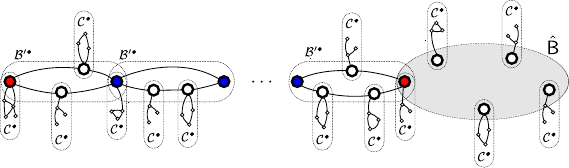}
		\caption{If the $2$-connected random graph $\mB_n^\gamma$  admits a distributional limit $\hat{\mathsf{B}}$, and the generating series $(\cB')^\gamma(z)$ belongs to the class $\mathscr{S}_1$ of subexponential sequences, then the random connected graph $\mC_n^\omega$ converges in the Benjamini--Schramm sense toward a limit $\hat{\mC}$. The limit consists of a finite chain of independent Boltzmann-distributed doubly-rooted blocks, with the limit $\hat{\mathsf{B}}$ attached to the tip of the chain. The random number of blocks in the chain follows a geometric distribution, and at each vertex of the chain and of $\hat{\mathsf{B}}$ an independent random Boltzmann distributed rooted connected graph is inserted.}
		\label{fi:conjlim}
	\end{minipage}
\end{figure}

\begin{theorem}[Characterization of Benjamini--Schramm convergence of $\mC_n^\omega$]
	\label{te:convpl}
	Suppose that $\mathbf{w}$ has type II. \label{pp:bn} Let $\mB_n^\gamma$ denote the random $2$-connected graph that gets drawn from $\cB[n]$ with probability proportional to its $\gamma$-weight. Then there is a random integer $K_n$ that is independent from the family $(\cB_n^\gamma)_n$  and satisfies $K_n \convdis \infty$, such that the random connected graph $\mC_n^\omega$ converges in the Benjamini--Schramm sense if and only if the graph $\mB_{K_n}^\gamma$ does. In this case, the limit \label{pp:ch2}$\hat{\mC}$ of $\mC_n^\omega$ contains the limit $\hat{\mB}$ as a subgraph via a coupling as illustrated in Figure~\ref{fi:conjlim}. The precise distribution of $\hat{\mC}$ is given  as follows.
	\begin{enumerate}
		\item Draw a random integer $K\ge0$ that follows the geometric distribution
		\[
		\Pr{K=k} = \nu^k(1-\nu).
		\]
		\item Let $(\mB'^\bullet_i)_{1 \le i \le K}$ denote a (conditionally) independent identically distributed sequence of random $\cB'^\bullet$-objects following a weighted Boltzmann distribution with parameter $\tau$.
		For each $i < K$, identify the pointed vertex of $\mB'^\bullet_i$ with the $*$-vertex of $\mB'^\bullet_{i+1}$ in order to form a finite chain of blocks. We consider this chain as rooted at the $*$-vertex of $\mB'^\bullet_1$. 
		\item Identify the root of the Benjamini--Schramm limit $\hat{\mB}$ with the pointed vertex of $\mB'^\bullet_{K}$ in the chain of blocks, and let $\mD$ denote the result.
		\item Let $\mC^\bullet$ be a random $\cC^\bullet$-object following a weighted Boltzmann distribution with parameter $\tau / \phi(\tau)$. For each vertex $v$ of $\mD$ take a fresh independent copy of $\mC^\bullet$ and identify its root with $v$. 
	\end{enumerate}

\end{theorem}


Note that Benjamini--Schramm convergence of $\mB_n^\gamma$ implies Benjamini--Schramm convergence of $\mB_{K_n}^\gamma$, but the converse need not hold. The first contribution of Theorem~\ref{te:convpl} is that, if the generating function of the weight-sequence has positive radius of convergence, then the random connected graph $\mC_n^\omega$ either has type I and converges by Theorem~\ref{te:loclunl} to a limit object with small blocks, or it has type II and then its limit, \emph{if} it exists, \emph{must} have precisely the shape as described in Theorem~\ref{te:convpl}. The second contribution is, that if we manage to deduce Benjamini--Schramm convergence of $\mB_n^\gamma$ in the type II regime, then Benjamini--Schramm convergence of $\mC_n^\omega$ follows automatically.

The degree distribution of random planar graphs has been studied by Drmota, Gim\'enez and Noy \cite{MR2802191} and Panagiotou and Steger \cite{MR2858393} by means of analytic combinatorics and Boltzmann samplers. In both papers the authors make use of the decomposition of graphs into components having higher connectivity. Hence we deem this a promising approach to establish and explicitly describe the Benjamini--Schramm limit of random planar graphs (or more generally uniform random graphs from classes defined by given $3$-connected components), and Theorem~\ref{te:convpl} is a first step in this direction. In the light of Corollary~\ref{co:prel}, Theorem~\ref{te:convpl} and Theorem~\ref{te:loclunl}, we pose the following question.
\begin{question}
	If the generating series $\cC^\omega(z)$ (or equivalently, $\cB^\gamma(z)$) has positive radius of convergence, does the random graph $\mC_n^\omega$ converges in the Benjamini--Schramm sense?
\end{question}
This includes random graphs from block-stable classes of graphs with analytic weights, and in particular the case of uniform connected graphs from proper minor-closed addable classes. Such a graph class is of the form $\mathbf{Ex}(M)$ for some non-empty set $M$ of $2$-connected graphs. See Section~\ref{sec:bijblock} for an explanation of this notation. By results of Bernardi, Noy and Welsh \cite{MR2644234} we know that uniform random graphs from minor-closed classes have either type I or II. In the type I case, Theorem~\ref{te:loclunl} immediately yields Benjamini--Schramm convergence. If the random graph has type II, then Theorem~\ref{te:convpl} may be used to reduce the Benjamini--Schramm convergence to the $2$-connected case.  It was established in \cite{MR2507738} furthermore that random graphs from proper minor-closed addable classes typically admit a giant component, and that the remaining fragments converge in total variation toward a limit called the Boltzmann Poisson random graph of the class. In \cite{doi:10.1002/rsa.20771}, these results where recovered using fundamentally different methods, and generalized to the random weighted model $\mC_n^\omega$.  Thus the weighted graph $\mC_n^\omega$ is a sensible model that allows for a unified study of  many other classes of graphs, and subjectively has the correct level of generality to discover the "abstract" reasons behind the limiting behaviour of these special cases.

The distribution of $K_n$ mentioned in Theorem~\ref{te:convpl} is made explicit in the proof. There is a deterministic sequence $\Omega_n$ that tends to infinity, such that if $d^+_{\cT_n}(o)$ denotes the root-degree of the simply generated tree $\cT_n$, then
\[
	K_n \eqdist  (d^+_{\cT_n}(o) \mid d^+_{\cT_n}(o) \ge \Omega_n) - R_n
\]
for a sequence of random variables $(R_n)_n$ that converge in total variation to the size of a $\mathbb{P}_{\Set \circ (\cB')^\gamma, \tau}$-distributed collection of blocks.  

The conditioned root-degree distribution crops up in other places of the present work too. For example, in Equations~\eqref{eq:yomei}, \eqref{eq:boltfin} we determined for certain classes of weights related to outerplanar maps the asymptotic probability for $(d^+_{\cT_n}(o) \mid d^+_{\cT_n}(o) \ge \Omega_n)$ to lie in lattices of the form $a + d \ndZ$. It might be interesting to investigate, whether the probability for  $K_n$ to lie in a given subset of $\ndN$ converges. The idea behind this is that for some interesting graph classes such as uniform cubic planar graphs the sequence $(\mB_n^\gamma)_n$ lies in the compact subspace of $(\mathbb{B}, d_{\textsc{BS}})$ of all graphs with a fixed upper bound for their vertex degrees, and hence the natural numbers may be partitioned into subsequences such that $\mB_n^\gamma$ converges weakly along each. If the probability for $K_n$ to lie in any of those sequences converges (uniformly), it follows that $\mB_{K_n}^\gamma$ converges weakly to a mixture of the limits of $\mB_n^\gamma$ along the subsequences.



Since we consider random weighted graphs, Theorem~\ref{te:convpl} also applies to other types of graph classes. We may easily force well-known subcritical graph classes (where the uniform random graph has type I$a$) into the type II regime, by adjusting the weights on the blocks. We demonstrate this for the well-known example of outerplanar graphs which have been studied individually in various contexts \cite{MR2350456, MR2538796}, but there are many further examples. 

\begin{example}[Type II outerplanar graphs]
	\label{ex:outerplanar}
	An outerplanar graph is a planar graph that admits at least one embedding in the plane such that every vertex may be reached from the outside. We consider the case where $\mC_n^\omega$ is a random $n$-vertex outerplanar graph that is drawn according to $\gamma$-weights assigned to its blocks.
	
	Outerplanar graphs that are $2$-connected have a unique Hamilton cycle which in any "outerplanar" drawing may be oriented in two directions. Hence any labelled edge-rooted dissection of a polygon may be obtained in a unique way by taking a vertex-rooted $2$-connected outerplanar graph and marking one of the two edges of the Hamilton cycle that are incident to the root vertex. This means that the random $\gamma$-weighted $n$-vertex $2$-connected outerplanar graph $\mB_n^\gamma$ is distributed like a random $\gamma$-weighted dissection $\mD_n^\gamma$ of an $n$-gon.
	
	We identify in Section~\ref{sec:diss22} three qualitatively distinct Benjamini--Schramm limits of random dissections where the $\gamma$-weights are products of $\iota$-weights assigned to their inner faces. That is, when \[\cD^\gamma \simeq \cX \cdot (\Seq \circ \Seq_{\ge 1}^\iota)(\cD^\gamma),\] see Section~\ref{sec:diss22} for details. The limits are illustrated in Figure~\ref{fi:limdis3}. One of them only admits one-sided infinite paths, while the others contain doubly-infinite paths. Let us assume that the $\gamma$-weighting on the blocks is given by the corresponding products of these face-weights.
	
	It is tempting to expect Theorem~\ref{te:convpl} to yield  different types of Benjamini--Schramm limits for $\mO_n^\omega$, depending on the limit in the $2$-connected case. However, the types for the connected and $2$-connected case are related, which can be shown analogously to the corresponding result for outerplanar maps in Lemma~\ref{le:light}. If  $\mC_n^\omega$ has type $I$, then the Benjamini--Schramm limit of $\mC_n^\omega$ is given as in Theorem~\ref{te:loclunl}. If $\mC_n^\omega$ has type $II$ and the weight-sequence for the face-weights is aperiodic, then Theorem~\ref{te:convpl} applies and $\mC_n^\omega$ converges toward the limit of Theorem~\ref{te:convpl} where $\hat{\mB}$ is the type II dissection limit of Theorem~\ref{te:bsconvdis}. If the weight-sequence is periodic, we still obtain convergence analogously to the proof of Theorem~\ref{te:facethm} for type II outerplanar maps. If $\mC_n^\omega$ has type III, then the Benjamini--Schramm limit is a deterministic doubly-infinite path, which may be justified analogously as in Theorem~\ref{te:bsfaces} for type III outerplanar maps.
\end{example}

\paragraph*{Block sizes}
\label{sec:blocksizes}
The maximum block-size of the random graph $\mC_n^\omega$ is an important parameter which influences its geometric shape. The coupling of $\mC_n^\omega$ with the simply generated tree $\cT_n$ has the property, that the extremal outdegrees of $\cT_n$ are upper bounds of the extremal block-sizes in $\mC_n^\omega$. Theorem~\ref{te:compsize} ensures that in many cases the $k$th largest block has the order of the $k$th largest outdegree, by providing a corresponding lower bound.

We detail the results obtained in this way. The extremal block size of (uniform) random graphs from block classes has been studied by various authors before, which is why we do not claim novelty of all subcases of Theorem~\ref{te:compsize}. We emphasize its applications to random planar graphs, and more generally to random graphs from planar-like classes:

\begin{corollary}
	\label{co:blocksize}
	If $\mC_n^\omega$ is the uniform random planar graph with $n$ vertices, then $\mathbf{w}$ has type II and $|\cB'[k]|_\gamma / k! \sim c \rho_{\cB}^{-k} k^{-5/2}$ for some constants $c, \rho_{\cB}>0$. Hence part {\em (3)} of Theorem~\ref{te:compsize} applies and yields limit theorems for the size of the $j$th largest block for all fixed $j \ge 1$.  More generally, these limit theorems also hold if the random weighted graph $\mC_n^\omega$ has type II and  satisfies $|\cB'[k]|_\gamma/k! \sim c \rho_{\cB}^{-k} k^{-\beta}$ for some $\beta > 2$.  This encompasses random graphs from so called planar-like classes introduced by Gim\'enez, Noy and Ru\'e \cite{MR3068033}. In order to obtain the central limit theorem for the size of the largest block, we may even replace the constant $c$ by any function of $k$ that varies slowly at infinity.
\end{corollary}

The asymptotics for the size of the  $j$th largest blocks are new for $j \ge 2$. The central limit theorem for the size of the largest block has been known for some years for random graphs from planar-like classes, but was obtained using different, analytic methods such as singularity analysis and the saddle-point method. We give a short comparison of Corollary~\ref{co:blocksize} to previous results.

Panagiotou and Steger \cite[Thm. 1.2]{MR2675698} obtained by a detailed study of Boltzmann samplers  that for any $\epsilon >0$ and sequence $t_n \to \infty$ the uniform planar graph with $n$ vertices has with high probability a unique giant block with $(1\pm \epsilon) cn$ vertices and the next largest block has size between $n^{2/3}/\log(n)$ and $n^{2/3} t_n$. They also provide similar bounds \cite[Thm. 1.4, ii)]{MR2675698} in a slightly less general setting as in {\em (1)} of Theorem~\ref{te:compsize}.

Gim\'enez, Noy and Ru\'e \cite[Thm. 5.4]{MR3068033} applied an elaborated analytical framework to  obtain the strong result that the largest block size $X_n$ of the uniform random planar graph with $n$ vertices satisfies $\Pr{X_n =k} \sim n^{-2/3} f(k)$ uniformly for $k = (1-\nu)n + xn^{2/3}$ an integer and $x$ in a fixed compact interval. Here $f$ denotes the density function of the distribution of the random variable $X_{3/2}$ from Theorem~\ref{te:compsize}. They obtained that the next largest block has with high probability at most $O(n^{2/3})$ vertices, and provided extensions of this result to random graphs from planar-like classes, which are encompassed by the case a) of Theorem~\ref{te:compsize}. 

Drmota and Noy \cite[Thm. 3.1]{MR3184197} used analytic methods to show that if $\mathbf{w}$ has  type I$a$ (and, for simplicity, $\spa(\mathbf{w}) =1$) then $\Ex{B_{(1)}} = O(\log(n))$ and if additionally the series $\cB(z)$ satisfies the ratio test, then $\Pr{B_{(1)} \le k} \sim \exp(-\exp(\log(n)-g(k)))$ uniformly for $n,k \to \infty$ where $g(k)$ is a function with $g(k) \sim Ck$ for some constant $C>0$.

\paragraph*{The diameter of the block-tree} Let $\mC_n$ denote an instance of the random graph $\mC_n^\omega$ where each block receives either weight $0$ or $1$. Such random graphs are also called block-stable graphs or simply block-graphs. McDiarmid and Scott \cite[Thm. 1.2]{MR3530623} showed that  with high probability any path in the random graph $\mC_n$ passes through at most $5 \sqrt{n \log(n)}$ blocks. They conjectured, that the extra factor $\sqrt{\log(n)}$ can be replaced by any sequence tending to infinity.
\begin{conjecture}[McDiarmid and Scott]
	\label{con:mcscott}
	If $t_n$ denotes a sequence tending to infinity, then with high probability any path in the uniform random graph $\mC_n$ from a block class passes through at most $t_n \sqrt{n}$ blocks.
\end{conjecture}
By Lemma~\ref{le:coupling} there is a coupling of the random graph $\mC_n^\omega$ with a simply generated tree $\cT_n$ such that the diameter $\Di(\cT_n)$ and the maximum number $D_n$ of blocks along a path in $\mC_n^\omega$ satisfy
\begin{align}
\label{eq:teeemp}
 D_n \in \{\Di(\cT_n), \Di(\cT_n) - 1\}.
\end{align} Addario-Berry, Devroye and Janson \cite{MR3077536} established that if $\mathbf{w}$ has type I$\alpha$, then there are constants $C,c>0$ depending on $\mathbf{w}$ such that for all $h \ge 0$ and $n \ge 1$ it holds that \begin{align}\label{eq:tempref}\Pr{\Di(\cT_n) \ge h} \le C \exp(-c h^2/n).\end{align} By \eqref{eq:teeemp}, this yields an equivalent tail bound for $D_n$ and Conjecture~\ref{con:mcscott} holds in this case. Janson~\cite[Problem 21.8]{MR2908619} posed the question, whether a tail-bound of the form \eqref{eq:tempref} can be obtained for {\em any} sensible weight sequence.
{
	\begin{question}[Janson]
		\label{pro:jans}
		Given a weight sequence $(\omega_k)_k$ such that $\omega_0>0$ and $\omega_k >0$ for at least one $k \ge 2$, are there constants $C,c>0$ such that the tail-bound \eqref{eq:tempref} holds for the corresponding simply generated tree $\cT_n$?
	\end{question}
}
The coupling between $\mC_n^\omega$ and $\cT_n$ allows us to relate the two questions as follows.
\begin{corollary}[Relating the block-diameter with the diameter of simply generated trees]
	If Question \ref{pro:jans} can be answered in the affirmative, then a corresponding tail-bound also holds for the block diameter of the random graph $\mC_n^\omega$ and Conjecture \ref{con:mcscott} follows.
\end{corollary}

	If $\mathbf{w}$ has type I$\beta$, it is reasonable to expect that $\Di(\cT_n) / \sqrt{n}$ even converges in probability to zero. This is still an open conjecture \cite[Conj. 21.5]{MR2908619}, which however has been confirmed for an important class of weight sequences. Kortchemski \cite[Thm. 1.2]{MR3651047} showed that if the offspring distribution $(\pi_k)_k$ from Section~\ref{sec:pretree} has mean $1$ and belongs to the domain of attraction of a stable law of index $\alpha \in ]1,2]$, then there exists sequence $b_n$ of order $n^{1/\alpha}$ (more precisely, $b_n / n^{1/\alpha}$ is slowly varying) such that for every $\delta \in ]0, \alpha[$ there exist constants $C,c>0$ such that for all $u \ge 0$ and $n\ge1$
	\[
		\Pr{\Di(\cT_n) \ge u n /b_n} \le C \exp(-c u^{\delta}).
	\]
	Of course, this directly translates into an equivalent bound for $D_n$. See also  Duquesne~\cite[Thm. 3.1]{MR1964956},  Haas and Miermont~\cite[Thm. 8]{MR3050512} and Kortchemski~\cite{MR3185928} for related results regarding the metric properties of $\cT_n$ in this setting.

In the following nongeneric case, we may build on results due to Kortchemski \cite[Thm. 4 and Prop. 2.11]{MR3335012} to obtain more precise information. Recall the definition of the constant $\nu$ given in Section~\ref{sec:pretree}.

\begin{corollary}[Block radius asymptotics in a subcase of case II]
	\label{cor:planarpath}
	Suppose that the weight-sequence $\mathbf{w}$ has type II and $|\cB'[k]|_\gamma/k! = f(k) k^{-\beta} \rho_{\cB}^{-k}$ for constants $\rho_\cB>0$, $\beta>2$ and a slowly varying function $f$. Choose a vertex of $\mC_n^\omega$ uniformly at random and let $h_n$ denote the maximum number of blocks along a path starting in that vertex. Then for each function $t_n \to \infty$ it holds that
	\[
		|h_n - \log(n) / \log(1/\nu)| \le t_n
	\]
	with probability tending to one as $n$ becomes large. Moreover, $h_n / \log(n)$ converges to the constant $\log(1/\nu)$ in the space $\mathbb{L}_p$ for $p \ge 1$. In particular, this applies to the uniform random planar graph for which we have $\beta = 5/2$, and more generally to random graphs from planar-like classes in the sense of Gim\'enez, Noy and Ru\'e \cite{MR3068033}.
\end{corollary}

\subsubsection{Applications to random dissections} 
\label{sec:diss22}

Given a sequence of non-negative weights $(\gamma_k)_{k \ge 3}$ with $\gamma_k >0$ for at least one $k$, we may assign to each dissection  $D$ of a polygon the weight \[\omega(D) = \prod_F \gamma_{|F|},\] with the index $F$ ranging over the inner faces of $D$. Here $|F|$ denotes the degree of the face $F$. As discussed in Section~\ref{sec:diss},  the random dissection $\mD_n^\omega$ of an $n$-gon that gets drawn with probability proportional to its $\omega$-weight is distributed like the random enriched tree $\mA_{n-1}^\cR$ for the weighted species $\cR^\kappa = \Seq \circ \Seq_{\ge 1}^\gamma$ with the weighting $\gamma$ given by $\Seq_{\ge 1}^\gamma(z) = \sum_{k=1}^\infty \gamma_{k+2} z^k$. Let $\mathbf{w} = (\omega_k)_k$ be the weight-sequence with $\omega_k = |\cR[k]|_\kappa$. 

The connection of random type I dissections and critical conditioned Galton--Watson trees has been fruitfully applied in work by Kortchemski~\cite{MR3178472}, Curien and Kortchemski \cite{MR3245291}, and Curien,  Haas and Kortchemski~\cite{MR3382675}, who provide both scaling limits and combinatorial applications. In particular, vertex and face degrees where studied in \cite[Sec. 4.2]{MR3245291}, and the approach taken in this section is  similar, although  additional ideas are required to treat the type II and III cases. We summarize the main results in this section and refer to Section~\ref{sec:proofsdis} for details and proofs.


\paragraph*{Type I dissections}

The random dissection $\mD_n^\omega$ possesses the rerooting invariance property. That is, its distribution as rooted plane graph does not change if we reroot at a uniformly at random chosen vertex. Hence we may apply Theorems~\ref{te:local} and \ref{te:strengthend}  to obtain Benjamini--Schramm convergence if $\mathbf{w}$ has type I.

\begin{theorem}[Benjamini--Schramm convergence, type I]
	\label{te:bsconvdist1}
	Suppose that the weight sequence $\mathbf{w}$ has type I. Then the limit object $(\hat{\cT}, \hat{\beta})$ corresponds to a random infinite planar map \label{pp:dh}$\hat{\mD}$. The random dissection $\mD_n^\omega$ converges toward $\hat{\mD}$ in the Benjamini--Schramm sense. If the $\mathbf{w}$ has type I$a$,  then it even holds that for any sequence of non-negative integers $k_n = o(n^{1/2})$
	\[
	\lim_{n \to \infty} d_{\textsc{TV}}(U_{k_n}(\mD_n^\omega), U_{k_n}(\hat{\mD})) \to 0.
	\]
\end{theorem}

The limit admits an easy description. Recall the parameter $\tau$ and the series $\phi(z)$ from Section~\ref{sec:types}.

\begin{remark}
		\label{re:disdist}
		The distribution of $\hat{\mD}$ is given as follows.
		\begin{enumerate}
			\item Let $(F_i)_{i \ge 1}$ a family of independent copies of a random number $F\ge 3$ with distribution given by
			\begin{align}
			\label{eq:distsbf}
			\Pr{F=k} = \gamma_k (k-1) \tau^{k-2}.
			\end{align}
			\item Let $(D_i)_{i \ge 1}$ be random polygons such that $D_i$ has degree $F_i$ for all $i$. For each polygon, we distinguish an arbitrary edge as its root-edge and we orient that edge in counter-clockwise direction. We form an infinite planar map by identifying the root-edge of $D_{i+1}$ with an uniformly at random chosen non-root-edge of $D_i$ for all $i$. The resulting object $D$ is a planar map that we consider as rooted at the root-edge of $D_1$. The sequence of root-edges forms the spine of $D$.
			\item Let $\mD$ denote a random dissection following the Boltzmann distribution $\mathbb{P}_{\cD, \tau / \phi(\tau)}$. We identify each non-spine edge of $D$ with a fresh independent copy of $\mD$ (attached from the outside) and let $\hat{\mD}$ denote the result.
		\end{enumerate}
\end{remark}

\begin{figure}[t]
	\centering
	\begin{minipage}{1.0\textwidth}
		\centering
		\includegraphics[width=0.8\textwidth]{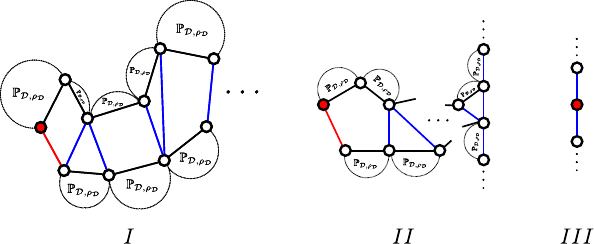}
		\caption{The three different types of Benjamini--Schramm limits of random face-weighted dissections of polygons.}
		\label{fi:limdis3}
	\end{minipage}
\end{figure}

Uniform dissections and triangulations of polygons where also studied by Bernasconi, Panagiotou and Steger \cite{MR2789731}. Using a different approach, they showed concentration results that imply laws of large numbers for the number of induced copies of (necessarily $2$-connected) subgraphs \cite[Thm. 1.4, 1.5]{MR2789731} and the degree of a random root \cite[Thm. 1.1]{MR2789731}.



\paragraph*{\em Type II dissections}

If $\mathbf{w}$ has type II, then we obtain convergence toward a limit object having a finite spine with path at its tip that grows to infinity in both directions, and corresponds to a face with large degree in $\mD_n^\omega$. The limit graph is illustrated in Figure~\ref{fi:limdis3}. Recall the parameter $\nu$ from Section~\ref{sec:types}.

\begin{theorem}[Benjamini--Schramm convergence, type II]
		\label{te:bsconvdis}
If $\mathbf{w}$ has type II, then the Benjamini--Schramm limit \label{pp:dh2} $\hat{\mD}$ of $\mD_n^\omega$ is given as follows.
\begin{enumerate}
	\item Draw a random integer $L\ge0$ that follows the geometric distribution
	\[
	\Pr{L=\ell} = \nu^\ell(1-\nu).
	\]
	\item Let $F$ denote the random integer with distribution given by 
	\[
		\Pr{F = k} = \gamma_k(k-1) \tau^{k-2} / \nu.
	\]
	\item Let $D_1, \ldots, D_L$ be random polygons such that the degree of $D_i$ is an independent copy of $F$ for all $i$. We consider each polygon as rooted at a directed edge in counter-clockwise direction. Form a planar map $D$ by identifying the root-edge of $D_{i+1}$ with a uniformly at random drawn non-root-edge of $D_{i}$ for all $1 \le i \le L-1$. Choose a random non-root-edge $e$ of $D_L$. We call the sequence of root-edges together with $e$ the spine of $D$.
	\item Identify the edge $e$ with an arbitrary edge of a path that grows to infinity in both directions. Let $D_P$ denote the result.
	\item Again let $\mD$ denote a random dissection following the Boltzmann distribution $\mathbb{P}_{\cD, \tau / \phi(\tau)}$. We identify each non-spine edge of $D_P$ with a fresh independent copy of $\mD$ (attached from the outside) and let $\hat{\mD}$ denote the result.
\end{enumerate}
\end{theorem}


\begin{corollary}
	\label{co:itsclear}
	Theorem~\ref{te:compsize} and Proposition~\ref{pro:new} yield bounds and limit laws for the sizes $B_{(i)}$ of the $i$th largest faces in the random dissection $\mD_n^\omega$ for various cases of weight sequences. In the non-generic type II setting, we hence obtain a central limit theorem for the size of the largest face.
\end{corollary}


\paragraph*{\em Type III dissections }

In the type III regime, the limit enriched tree consists of a  single vertex with infinitely many offspring, all of which are leaves.  The root corresponds to a vertex with large degree in the enriched tree representation of $\mD_n^\omega$. That is, a $\Seq \circ \Seq_{\ge 1}^\gamma$ structure with a large random total size that depends on $n$. The number of components equals the number of inner faces to which the root-vertex of $\mD_n^\omega$ is incident. The idea of the following result is that with high probability the root vertex is incident to only a single inner face, and any fixed-sized neighbourhood of the root looks with high probability like a straight line.

\begin{theorem}[Benjamini--Schramm convergence of type III dissections of polygons]
	\label{te:gibbst3}
	If the weight sequence $\mathbf{w}$ has type III, then the Benjamini--Schramm limit of $\mD_n^\omega$ is a doubly-infinite path as illustrated in Figure~\ref{fi:limdis3}.
\end{theorem}

Considered from a Gibbs-partition viewpoint, Theorem~\ref{te:gibbst3} is actually a small surprise. It is a priori not clear at all why the $\Seq \circ \Seq_{\ge 1}^\gamma$-structure belonging to the root of $\mD_n^\omega$ with a large {\em random} size should consist with high probability of a single component, when such a behaviour does not need to hold for random $\Seq \circ \Seq_{\ge 1}^\gamma$-structures with a large {\em deterministic} size:

\begin{remark}[Gibbs partitions: the non-analytic case]
\label{re:non-analytic}
 Being in the type III regime means that $\Seq_{\ge 1}^\gamma(z)$ is not analytic at the origin. If we take a random compound structure $\mS_k$ from $(\Seq \circ \Seq_{\ge 1}^\gamma)[k]$ with probability proportional to its weight for a {\em deterministic} $k$, and let $k$ tend to infinity, then the probability 
\[
	r_k = \frac{[z^k] \Seq_{\ge 1}^\gamma(z)}{[z^k] \Seq \circ \Seq_{\ge 1}^\gamma(z)}
\]
that $\mS_k$ consists of a single component may not converge at all. More precisely, the fact that $\Seq_{\ge 1}^\gamma(z)$ has radius of convergence $0$ ensures that
\begin{align}
	\label{eq:mayshow}
	\limsup_{k \to \infty} r_k = 1,
\end{align}
but there are concrete examples for which the limes inferior belongs to $[0, 1[$. This has been shown by Bell~\cite{MR1771616} for  composition schemes of the form $\Set \circ \cF^\kappa$, and his arguments may easily be adapted to verify \eqref{eq:mayshow}. It is also not hard to construct aperiodic weightings for which $r_k = 0$ for infinitely many $k$.

The $\Seq \circ \Seq_{\ge 1}^\gamma$-structure belonging to the root of $\mD_n^\omega$ is distributed like $\mS_{K_n}$ for a random integer $K_n$ independent from $(\mS_k)_k$, that satisfies $\Pr{\mK_n \ge \Omega_n} \to 1$ for some deterministic sequence $\Omega_n \to \infty$. If we look at the distribution of $\mK_n$ directly, then it is not clear why it behaves so nicely such that $\mS_{\mK_n}$ has with high probability only a single component, as $n$ becomes large. We circumvent this issue in the proof of Theorem~\ref{te:gibbst3}, by using the  flexibility of the Ehrenborg--M\'endez isomorphism to treat arbitrary non-analytic $\gamma$-weights.
\end{remark}

Theorem~\ref{te:gibbst3} may be reformulated to show that the random plane tree $\cT_\ell^\omega$ with $\ell$ leaves from Section~\ref{sec:lgwt} converges in the type III regime toward the same infinite star as simply generated trees do. See  Lemma~\ref{le:cordt} and Lemma~\ref{le:convll}  for a precise statement of this fact.

\subsubsection{Applications to random $k$-trees}
\label{sec:ktrees}
We are interested in obtaining the Benjamini--Schramm limit of the random $k$-dimensional tree $\mK_n$ with $n$ hedra as $n$ becomes large. A local limit around (a fixed vertex of) a uniformly at random drawn front was established in  \cite{2016arXiv160505191D}. As we shall see, the two limits are distinct, which is already evident from the fact that the degrees of the root vertices tends to different limit distributions. Interestingly, the two limits are however identically distributed as unrooted graphs.

We use the notation of Section~\ref{sec:decktrees}. As discussed there,  $\mK_{n}$ is distributed like the uniform front-rooted $k$-tree $\mK_n^\circ$ from the front-rooted class $\cK^\circ$, which by Equations~\eqref{eq:lab1} and \eqref{eq:lab2} admits the decomposition
\[
\cK^\circ \simeq \Set(\cK_1^\circ), \qquad \cK_1^\circ \simeq \cX \cdot \Set^k(\cK_1^\circ).
\]
Here $\cK_1^\circ$ denotes the class of front-rooted $k$-trees where the root-front is contained in precisely one hedron and consists of $k$ distinct $*$-place-holder vertices that do not contribute to the total size of the object. The decomposition shows that the uniform random $n$-sized element $\mK_{1,n}^\circ$ from the class $\cK_1^\circ$ corresponds to the random enriched tree $\mA_n^\cR$ for $\cR = \Seq_{\{k\}} \circ \Set$. The weight-sequence $\mathbf{w} = (\omega_k)_k$ is defined accordingly by $\omega_k = |\cR[k]|$.

 As $\phi(z)= \cR(z) = \exp(k z)$ is infinite when evaluated at its radius of convergence $\infty$, a general criterion given in \cite[Lem. 3.1]{MR2908619} applies and ensures that $\mathbf{w}$ has type I$a$. Thus \cite[Thm. 18.11]{MR2908619} yields that
\[
[z^n] \cK_1^\circ(z) \sim \sqrt{\frac{\phi(\tau)}{2 \pi \phi''(\tau)}} \left ( \frac{\phi(\tau)}{\tau} \right )^n n^{-3/2} = \frac{1}{k \sqrt{2 \pi}} (e k)^n n^{-3/2},
\]
as $\tau$ is defined by $\phi'(\tau)\tau = \phi(\tau)$ which yields $\tau = 1/k$.
Hence Lemma~\ref{le:gibbs} ensures that the decomposition $\Set \circ \cK_1^\circ$ has convergent type, that is,
\begin{align}
\label{eq:totalv}
\lim_{n \to \infty} d_{\textsc{TV}}(\mK_n, \mK^\circ + \mK_{1,n-|\mK|}^\circ) =0.
\end{align}
Here $\mK^\circ$ denotes a random $k$-tree that follows the Boltzmann-distribution $\mathbb{P}_{\cK^\circ,\frac{1}{ek}}$ and is independent from $(\mK_{1,n}^\circ)_n$. That is, $\mK^\circ$ is obtained  by glueing $\text{Poisson}( \cK_1^\circ(\frac{1}{ek}))$-many (with $\cK_1^\circ(\frac{1}{ek})=\tau=1/k$ by Equation \eqref{eq:z2}) independent $\mathbb{P}_{\cK_1^\circ,\frac{1}{ek}}$-distributed $\cK_1^\circ$-objects together at their root-fronts. If the Poisson-distributed number equals zero, then we define it to be just a single root-front.  Moreover, $\mK^\circ + \mK_{1,n -|\mK|}^\circ$ denotes the graph obtained by canonically identifying the root-fronts of the two $k$-trees $\mK^\circ$ and $\mK_{1,n -|\mK|}^\circ$ with each other. This is only well-defined if $|\mK^\circ| < n$, but the probability for this event tends to one as $n$ becomes large.

Theorem~\ref{te:thmben} that describes the asymptotic behaviour of  extended fringe subtrees in enriched trees implies that $\mK_{n,1}^\circ$ converges in the Benjamini--Schramm sense toward the limit graph $\hat{\mK}$ that corresponds to the  enriched tree $(\cT^*, \beta^*)$ according to the bijection in Section~\ref{sec:decktrees}, and we may use Equation~\eqref{eq:totalv} to deduce that $\mK_n$ also converges in the Benjamini--Schramm sense. 

Even more ambitiously, we may establish total variational convergence of arbitrary $o(\sqrt{n})$-neighbourhoods. This is best possible, as the diameter of $\mK_n$ has order $\sqrt{n}$. We also provide a simple description of the distribution of the limit $\hat{\mK}$  in Remark~\ref{re:limktreedis} below, as the interpretation of $(\cT^*, \beta^*)$ as a graph requires some thought.

\begin{theorem}
	\label{te:anamesname}
	Let \label{pp:hk}$\hat{\mK}$ denote the infinite random enriched tree corresponding to the limit $\cR$-enriched tree $(\cT^*, \beta^*)$ according to the bijection in Section~\ref{sec:decktrees}. Let $v_n$ be a uniformly at random drawn vertex from the random $n$-vertex $k$-tree $\mK_n$ and let $t_n= o(\sqrt{n})$ be a series of positive integers. Then the total variational distance between the graph-distance neighbourhoods $V_{t_n}(\cdot)$ of $(\mK_n, v_n)$ and $\hat{\mK}$ converges to zero, that is
	\[
		\lim_{n \to \infty} d_{\textsc{TV}}(V_{t_n}(\mK_n, v_n), V_{t_n}(\hat{\mK})) = 0.
	\]
	This establishes $\hat{\mK}$ as the Benjamini--Schramm limit of the random $k$-tree $\mK_n$.
\end{theorem}

We are going to describe the distribution of the Benjamini--Schramm limit $\hat{\mK}$ in detail using a random walk and family of independent Boltzmann distributions as illustrated in Figure~\ref{fi:liktree}.

\begin{figure}[t]
	\centering
	\begin{minipage}{1.0\textwidth}
		\centering
		\includegraphics[width=0.3\textwidth]{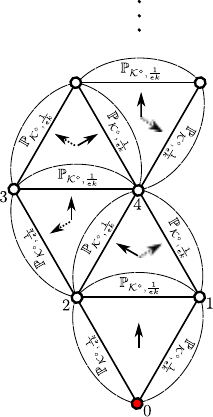}
	\caption{The Benjamini--Schramm limit of random $k$-trees.}
	\label{fi:liktree}
	\end{minipage}
\end{figure}

\begin{remark}
	\label{re:limktreedis}
	The limit $\hat{\mK}$ corresponding to $(\cT^*, \beta^*)$ may be described as follows.
	\begin{enumerate}
		\item Imagine a random walker that sits inside of a hedron, which we may interpret as a convex subset in $k$-dimensional space. For example, for $k=2$ the walker sits inside of a triangle. He leaves the hedron by passing through any of the $k+1$ fronts. We label the unique vertex not contained in this front as $u_0$. We take another hedron and glue one of its fronts from the outside to the front through which the walker just passed, so that the walker finds himself trapped again. That is, he is now trapped in the second hedron of a $k$-tree with two hedra in total.
		\item In the $i$th step (we start with $i=1$), the random walker  chooses uniformly at random one of the $k$ fronts of his prison hedra through which he has {\em not} passed before. We label the unique vertex not contained in this front with $u_i$. He leaves his prison $k$-tree by passing through the chosen front, and we trap him as before by attaching a new hedra from the outside.
		\item This yields a $k$-tree $\mG$  consisting of an infinite ordered sequence of glued-together hedra. The vertices of the $k$-tree are labelled by $u_0, u_1, \ldots$. As in Equation~\eqref{eq:totalv}, let $\mK^\circ$ denote a random front-rooted $k$-tree that follows the Boltzmann-distribution $\mathbb{P}_{\cK^\circ,\frac{1}{ek}}$. For each front $F$ of the infinite $k$-tree $\mG$ take a fresh independent copy $\mK_F^\circ$ of $\mK^\circ$ and identify $F$ with the root-front of $\mK_F^\circ$ in a canonical way. Let $\hat{\mK}$ denote the result and root it at the vertex $u_0$.
	\end{enumerate}
	We included  the vertices $u_0, u_1, \ldots$ in the description of $\hat{\mK}$ as they correspond to the backwards growing spine of the infinite $\cR$-enriched tree  $(\cT^*, \beta^*)$.  The limit is illustrated in Figure~\ref{fi:liktree}, where the red vertex corresponding to $u_0$ being the root. 
\end{remark}

From this, we may deduce the precise limit distribution of a random vertex in the random $k$-tree $\mK_n$.

\begin{remark}
	\label{re:ktreesrem}
	The  root-degree $d_{\hat{\mK}}(u_0)$ has finite exponential moments and its probability generating function $u(z)$ given by 
	\[
		u(z) = z^k v(z)^k, \qquad v(z) = exp((z v(z)^k -1)/k).
	\]
	That is, for $k=2$ it is given by the number of vertices  of a forest of $k$ subcritical Galton--Watson trees with offspring distribution $\text{Poisson}(1/2)$. In general, the distribution is given by the number of type $A$ vertices in a 2-type Galton--Watson tree, where  type $A$ vertices have no offspring, the root always has type $B$, and the $(A,B)$-offspring of a type $B$ vertex is given by $(Z, (k-1)Z)$, for a $\text{Poisson}(1/k)$-distributed random variable $Z$.
\end{remark}

Since $2$-trees are planar, it follows by a general result \cite[Thm. 1]{MR3010812} on distributional limits of planar graphs that the limit $\hat{\mK}$ is almost surely recurrent for $k=2$. In \cite[Thm. 2]{2016arXiv160505191D} a local weak limit $\hat{\mK}^\circ$ was established that describes the convergence of neighbourhoods of a fixed vertex of the root-front in a random front-rooted $k$-tree as its number of hedra tends to infinity. The two limits are distinct as rooted graphs, since their root-degrees follow different distributions. However, they are identically distributed as {\em unrooted} graphs:

\begin{remark}
	\label{re:comparison}
	We may compare the Benjamini--Schramm limit $\hat{\mK}$ of the random $k$-tree $\mK_n$ with the limit $\hat{\mK}^\circ$ that describes the convergence of neighbourhood of a uniform random vertex of a uniform random front in $\mK_n$. Our main observation is that their root-degrees have different distributions, but the two graphs are identically distributed as {\em unrooted} graphs. 
	
	The limit   $\hat{\mK}^\circ$ of \cite[Thm. 2]{2016arXiv160505191D} admits a description similar as in Remark~\ref{re:limktreedis} with the following difference. For $\hat{\mK}$ we started with an initial hedra that the random walker leaves through a front and we label the opposite vertex as the root $u_0$. So in the construction of $\hat{\mK}$, $u_0$ is always incident to precisely $k$ independent $\mathbb{P}_{\cK^\circ, \frac{1}{ek}}$-distributed objects. For $\hat{\mK}^\circ$ we have to start with an initial hedra that is rooted at a front and one of the vertices of that front is distinguished as the root-front. The random walker then leaves the hedra through a uniformly at random drawn non-root-front and proceeds afterwards as in step {\em (2)} of Remark~\ref{re:limktreedis}, that is, avoiding fronts through which he has passed before. So in the resulting limit $\hat{\mK}^\circ$, the root is incident to a  random number of independent $\mathbb{P}_{\cK^\circ, \frac{1}{ek}}$-distributed objects and that number is with probability $(k-1)/k$ strictly larger than $2$. Now if we simply reroot $\hat{\mK}^\circ$ at the vertex opposite of the front through which the  random walker leaves the initial hedra, then the result is distributed like $\hat{\mK}$. So $\hat{\mK}$ may be obtained from $\hat{\mK}^\circ$ by rerooting at a random location depending on the first move of the random walker. This verifies that $\hat{\mK}$ and $\hat{\mK}^\circ$ are identically distributed as {\em unrooted} graphs.
\end{remark}

\subsubsection{Applications to random planar maps} 
\label{sec:appmaps}

We study the random block-weighted planar map $\mM_n^\omega$ that corresponds as described in Section~\ref{sec:decmaps} to the random enriched tree $\mA_{2n +1}^{\cQ^\kappa}$ with $\cQ^\kappa$ denoting the $\kappa$-weighted class of non-separable planar maps. So in this section, the weight-sequence $\mathbf{w}=(\omega_k)_k$ is defined by $\omega_k = |\cQ[k]|_\kappa$ for all $k\ge0$.

\paragraph*{\em Local convergence - the infinite spine case} 

\begin{figure}[t]
	\centering
	\begin{minipage}{1.0\textwidth}
		\centering
		\includegraphics[width=0.3\textwidth]{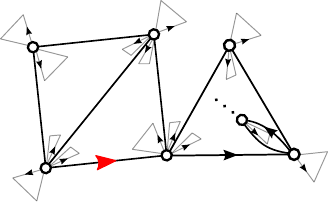}
		\caption{The local weak limit of random block-weighted planar maps with type I.}
		\label{fi:plalimit}
	\end{minipage}
\end{figure}

In the type I regime, our framework yields the following distributional limit.

\begin{theorem}
	\label{te:plant1}
	Suppose that the weight-sequence $\mathbf{w}$ has type I. Let $\hat{\mM}$ denote the planar map corresponding to the infinite $\cQ^\kappa$-enriched tree $(\hat{\cT}, \hat{\beta})$.
	\begin{enumerate}
	\item The random block-weighted planar map $\mM_n^\omega$ converges in the local weak sense toward $\hat{\mM}$. The limit respects the planar embedding of the map. That is, for each fixed $k$ the neighbourhood $V_k(\mM_n^\omega)$ of the origin of the root-edge of $\mM_n^\omega$ converges weakly as plane graph or edge-rooted planar map toward the neighbourhood $V_k(\hat{\mM}^\omega)$.
	\item If $\mathbf{w}$ has type I$a$, then it even holds for any sequence $k_n = o(\sqrt{n})$ that 
	\[
		\lim_{n \to \infty} d_{\textsc{TV}}( V_{k_n}(\mM_n^\omega), V_{k_n}(\hat{\mM})) =0.
	\]
	\end{enumerate}
\end{theorem}

The distribution of the limit planar map $\hat{\mM}$ admits an easy description, which is illustrated in Figure~\ref{fi:plalimit}.

\begin{remark}
	\label{re:distrmap}
	The distribution of $\hat{\mM}$ may be described as follows.
	\begin{enumerate}
	\item	 Let $\mQ$ denote a random non-separable map following a Boltzmann distribution $\mathbb{P}_{\cQ^\kappa, \tau}$. Likewise, let $\mQ^\bullet$ denote random corner-rooted map following a $\mathbb{P}_{(\cQ^\bullet)^\kappa, \tau}$-distribution. 
	\item We start with an independent copy of $\mQ^\bullet$. The root-edge of this object will be the root-edge of $\hat{\mM}$. We declare all corners as unvisited.
	\item In each step, pick an arbitrary unvisited corner $c$. If it is a root-corner, then attach an independent copy $Q$ of $\mQ^\bullet$ at the corner $c$. Otherwise use an independent copy $Q$ of $\mQ$ instead. 
	\item If $Q$ is the empty map with no edges, then just declare the corner $c$ as visited. Otherwise, the attaching $Q$ replaces the old corner $c$ by two new corners. Declare the one incident with the root-edge of $Q$ as visited.
	\end{enumerate}
\end{remark}


\paragraph*{\em Local convergence - the finite spine case}
In the type II regime, which encompasses the case of a uniform $n$-edge planar map, we obtain the following result (that is similar to Theorem~\ref{te:convpl}).

\begin{theorem}
	\label{te:thmplan2}
	Suppose that the weight-sequence $\mathbf{w}$ has type II. Let \label{pp:qn}$\mQ_n^\kappa$ denote a random $n$-edge non-separable map that is drawn with probability proportional to its $\kappa$-weight. 
	
	There is a random integer $K_n$ that is independent from the family $(\mQ_n^\kappa)_n$  and satisfies $K_n \convdis \infty$, such that the random planar map $\mM_n^\omega$ converges in the local weak sense if and only if the random map $\mQ_{K_n}^\gamma$ does.
	
	
	In this case the limit $\hat{\mM}$  of $\mM_n^\omega$ is a modification of the type $I$ limit described in Remark~\ref{re:distrmap}, with the only difference being that instead of an infinite spine of $\mathbb{P}_{(\cQ^\bullet)^\kappa, \tau}$-distributed corner-rooted non-separable maps we take a spine having geometric length $L$ with
	\[
		\Pr{L=\ell} = \nu^\ell (1 - \nu)
	\]
	and then attach $\hat{\mQ}$ at the root-corner of the map at the tip of spine. 
	
	That is, in the description of Remark~\ref{re:distrmap} we 
may use exactly $L$ independent copies of the corner-rooted map. As soon as we would have to use the $(L+1)$th independent copy, we attach the limit $\hat{\mQ}$ instead, whose corners are of course all declared unvisited. In particular for $L=0$ we directly start the recursive description with $\hat{\mQ}$.
\end{theorem}

\paragraph*{\em Block sizes and block diameter}

The diameter of the coupled tree $\cT_{2n+1}$ is an upper bound for the block-diameter of the random $n$-edge block-weighted planar map $\mM_n^\omega$. The outdegrees in $\cT_{2n+1}$ correspond precisely to the number of half-edges or corners in the blocks of $\mM_n^\omega$. Hence all known bounds and limit theorems regarding the extremal out-degrees in simply generated trees also apply to the block-sizes in $\mM_n^\omega$. See Janson~\cite[Chapters 9, 19]{MR2908619} for an overview of such results, and Kortchemski \cite[Thm. 1]{MR3335012} for a recent addition. Formally,  Theorem~\ref{te:compsize} and Proposition~\ref{pro:new}  apply, as we may express $\cR=\cQ^\kappa$ trivially as a compound structure $\cR \simeq \cX \circ \cQ^\kappa$.

Banderier, Flajolet, Schaeffer, and Soria \cite{MR1871555} studied many natural models such as uniform (bipartite) planar maps, which in our setting correspond to type II block-weighted maps with $\omega_k \sim c k^{-5/2} \tau^{-k}$ for some $c>0$. Using analytic methods, they also established a local limit law \cite[Thm. 3]{MR1871555} for the size of the largest block, which is a stronger result than the central limit theorem obtained by the probabilistic approach in this setting.

Addario-Berry~\cite{2015arXiv150308159A} used a similar probabilistic approach as in the present paper to relate the block-sizes of random planar maps to out-degrees in simply generated trees.

\subsection{Scaling limits of metric spaces based on $\cR$-enriched trees}
\label{sec:partA}

\begin{figure}[t]
	\centering
	\begin{minipage}{1.0\textwidth}
		\centering
		\includegraphics[width=0.35\textwidth]{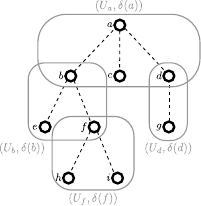}
	\caption{Patching together discrete semi-metric spaces.}
	\label{fi:patch}
	\end{minipage}
\end{figure}

\subsubsection{Patching together discrete semi-metric spaces}
\label{sec:patch}

We describe a model of semi-metric spaces patched together from metrics associated to the vertices of a tree. Let $A$ be a rooted tree with vertex set $V(A)$ and for each vertex $v$ let $M_v$ denote its offspring set. Let $\delta$ be a map that assigns to each vertex $v$ of $A$ a semi-metric $\delta(v)$ on the set $U_v := M_v \cup \{v\}$. We may define a semi-metric $d$ on the vertex vertex set $V(A)$ that extends the semi-metrics $\delta(v)$ by patching together as illustrated in Figure~\ref{fi:patch}. Formally, this semi-metric is defined as follows. Consider the graph $G$ on $V(A)$ obtained by connecting any two vertices $x \ne y$ if and only if there is some vertex $v$ of the tree $A$ with $x,y \in U_v$ and assigning the weight $\delta(v)(x,y)$ to the edge. The resulting graph is connected and the distance of any two vertices $a$ and $b$ is defined by the minimum of all sums of edge-weights along paths joining $a$ and $b$ in the graph $G$.

We now introduce our model of random semi-metric spaces associated to random enriched trees. Let $\cR^\kappa$ be a weighted species such that the weight sequence $\om = (\omega_k)_k$ given by $\omega_k = |\cR[k]|_\kappa / k!$ satisfies $\omega_0 > 0$ and $\omega_k > 0$ for some $k \ge 2$. 
Consider the weighting $\omega$ on the species $\cA_\cR$ introduced in Section~\ref{sec:intro1}, that is $\omega(A, \alpha) = \prod_{v \in [n]} \kappa(\alpha(v))$ for all $(A, \alpha) \in \cA_\cR[n]$.
For any integer $n\ge0$ with $|\cA_\cR[n]|_\omega > 0$ we form the random $\cR$-enriched tree $\mA_n^\cR=(\mA_n, \alpha_n)$ drawn from the set $\cA_\cR[n]$ with probability proportional to its $\omega$-weight. Suppose that for each finite subset $U \subset \ndN$ and each $\cR$-structure $R \in \cR[U]$ we are given a random semi-metric $\delta_R$ on the set $U \cup \{*_U\}$ with $*_U$ denoting an arbitrary fixed element not contained in $U$. For example, we could set $*_U := \{ U\}$.  We may form the random semi-metric space \label{pp:xn}$\mX_n =([n], d_{\mX_n})$ as follows. For each vertex $v$ of $\mA_n$ with offspring set $M_v$ let $\delta_n(v)$ be the semi-metric on the set $M_v \cup \{v\}$ obtained by taking an independent copy of $\delta_{\alpha_n(v)}$ and identifying $*_{M_v}$ with $v$. Let $d_{\mX_n}$ denote the semi-metric patched together from the family $(\delta_n(v))_v$ as described in the preceding paragraph.

In order for this to be a sensible model of a random tree-like structure we require the following three assumptions. 
\begin{enumerate}
\item We assume that there is a real-valued random variable $\chi\ge0$ such that for any $\cR$-structure $R$ the diameter of the semi-metric $\delta_R$ is stochastically bounded by the sum of $|R|$ independent copies $\chi_1^R, \ldots, \chi_{|R|}^R$ of $\chi$.
\item For any bijection $\sigma: U \to V$ of finite subsets of $\ndN$ and for any $\cR$-structure $R \in \cR[U]$ we require that the semi-metric $\delta_{\cR[\sigma](R)}$ is identically distributed to the push-forward of the semi-metric $\delta_R$ by the bijection $\bar{\sigma}: U \cup \{*_U\} \to V \cup \{*_V\}$ with $\bar{\sigma}|_U = \sigma$. 
\item We assume that there is at least one $\cR$-structure $R$ having positive $\kappa$-weight $\kappa(R)>0$ such that the diameter of $\delta_R$ is not almost surely zero.
\end{enumerate}
The first requirement ensures that the semi-metric space $\mX_n$ maintains a tree-like structure and the second that the symmetries of the enriched tree do not influence the choice of the random semi-metrics. The third requirement ensures that $\mX_n$ is not a degenerate space.

\subsubsection{Scaling limits and a diameter tail-bound}
\label{sec:scali}
In the following, we interpret $\mX_n$ in a canonical way as a metric space by passing to the quotient space, in which points with distance zero from each other are glued together.

\begin{theorem}[Scaling limit of the enriched tree based model of random metric spaces]
\label{te:main1}
Suppose that the weight sequence $\mathbf{w}$ has type I$a$. Then the rescaled space $(\mX_n, n^{-1/2}d_{\mX_n})$ converges weakly to a constant multiple of the (Brownian) continuum random tree $\CRT$ with respect to the Gromov--Hausdorff metric as $n \equiv 1 \mod \spa(\mathbf{w})$ tends to infinity. 
\end{theorem}
The scaling factor is made explicit in the corresponding proof in Section~\ref{sec:propartA}.
In order to ensure that extremal parameters of the rescaled (pointed) metric space such as the height and diameter converge not only in distribution, but also in higher moments, we also show a diameter tail-bound in Lemma~\ref{le:tail1} below.
	
\begin{lemma}[Tail bound for the diameter]
\label{le:tail1} 
Suppose that the weight sequence $\mathbf{w}$ has type I$\alpha$.
Then there are positive constants $C$ and $c$ such that for all $n$ and $x \ge 0$ it holds that
\[
\Pr{\Di(\mX_n) \ge x} \le C ( \exp(-c x^2/n) + \exp(-c x)).
\]
\end{lemma}
Note that if the random variable $\chi$ is bounded, then $\Pr{\Di(\mX_n) \ge x} = 0$ for all $x$ larger than a constant multiple of $n$ and hence it follows that there are constants $C,c>0$ with \[\Pr{\Di(\mX_n) \ge x} \le C \exp(-c x^2/n)\] for all $n$ and $x \ge 0$. The requirements of Lemma~\ref{le:tail1} are slightly weaker than in Theorem~\ref{te:main1}, since we did not assume exponential moments. The main ingredient in the proof is a similar tail-bound \eqref{eq:gwttail} for the height of Galton--Watson trees.

Theorem~\ref{te:main1} and Lemma~\ref{le:tail1} are inspired by previous work on scaling limits and diameter tail-bounds in \cite{MR3382675, inannals, MR3634279}, but the author felt it would nevertheless be interesting to complement the results on the local convergence in the general setting of random enriched trees by examples describing global geometric properties. We limit the scope of this paper to the "Brownian" case where the global geometric properties are similar to a large critical Galton--Watson trees whose offspring distribution has finite variance, but extensions to arbitrary enriched trees in more involved settings of weight-sequences could be very interesting.

\subsubsection{Applications}
\label{sec:app1}

The main application of the results in Section~\ref{sec:scali} is to the model $\mM_n^\omega$ of random  planar maps with block-weights introduced in Section~\ref{sec:decmaps}.  Furthermore, Theorem~\ref{te:main1} and Lemma~\ref{le:tail1} mildly generalize results for uniform random graphs from subcritical graph classes \cite[Thm. 5.1, Thm. 7.1]{inannals} and uniform outerplanar maps \cite[Thm. 1.2]{MR3634279} to the setting of weighted random structures, that is, block-weighted random graphs and block-weighted or face-weighted outerplanar maps.

Let $\iota > 0$ denote a random variable which has finite exponential moments.
Given a connected graph $G$ we may consider the first-passage percolation metric $d_{\textsc{FPP}}$ on $G$ by assigning an independent copy of $\iota$ to each edge of $G$, letting for any two vertices $x,y$ the distance \label{pp:fpp}$d_{\textsc{FPP}}(x,y)$ be given by the minimum of all sums of weights along paths joining $x$ and $y$. We let ${\Di}_{\textsc{FPP}}(G)$ denote the diameter with respect to the $d_{\textsc{FPP}}$-distance. 

We may construct the first-passage percolation metric $(\mM_n^\omega, d_{\textsc{FPP}})$ as the quotient space of a random semi-metric space $(\mX_n, d_{\mX_n})$ as in Section~\ref{sec:patch}. Here we  make good use of the freedom to assign semi-metric spaces to biconnected maps, as the non-root-vertices of the enriched tree $\mA_{2n+1}^{\cQ^\kappa}$ correspond to the corners of the random planar map $\mM_n^\omega$ and not to the vertices. For each $\cQ$-structure $Q$ we let $\delta_Q$ denote the semi-metric space whose points are the corners of $Q$ with the distance of two corners $x_1, x_2$ being defined as the first-passage percolation distance in $(Q, d_{\textsc{FPP}})$ of the vertices incident to $x_1$ and $x_2$. So the quotient space of $(\mX_n, d_{\mX_n})$ is, as a random metric space, distributed like $(\mM_n^\omega, d_{\textsc{FPP}})$. Theorem~\ref{te:main1} and Lemma~\ref{le:tail1} now immediately yield the following result.

\begin{theorem}
	\label{te:convplan}
	Suppose that the weight sequence $\mathbf{w}$ has type I$a$. Then the rescaled space $(\mM_n^\omega, n^{-1/2}d_{\textsc{FPP}})$ converges weakly to a constant multiple of the (Brownian) continuum random tree $\CRT$ with respect to the Gromov--Hausdorff metric as $n \equiv 1 \mod \spa(\mathbf{w})$ tends to infinity. Moreover, there are constants $C,c>0$ such that for all $n$ and $x \ge 0$ it holds that
	\[
	\Pr{\Di_{\textsc{FPP}}(\mM_n^\omega) \ge x} \le C ( \exp(-c x^2/n) + \exp(-c x)).
	\]
\end{theorem}

\section{Proofs}
\label{sec:allproofs}
\subsection{Proofs of the results on random enriched trees in  Section~\ref{sec:partB}}
\label{sec:propartB}

We start with a proof for the coupling of random enriched trees with simply generated trees.
\begin{proof}[Proof of Lemma~\ref{le:coupling}]
Let $\fmT_n$ denote the set of plane trees with $n$ vertices and \[Z_n = \sum_{T \in \fmT_n} \omega(T)\] the partition function of the weight sequence $(\omega_k)_k$. Let $\cA$ denote the species of rooted unordered trees. Every unordered rooted tree $A \in \cA[n]$ corresponds to $\prod_{v \in V(A)} d^+_A(v)!$ ordered trees (with labels in the set $[n]$) and every plane tree corresponds to $n!$ ordered (labelled) trees. Hence
\begin{align}
\label{eq:zuerst}
|\cA_\cR[n]|_\omega   = \sum_{A \in \cA[n]} \prod_{v \in V(A)} |\cR[d^+_A(v)]|_\kappa = n! \sum_{T \in \fmT_n} \omega(T) = n! Z_n.
\end{align}
We may view the tree $\cA_\cR[\sigma][(\cT_n, \beta_n)]$ as a labelled ordered enriched tree by keeping the orderings on the offspring sets of $\cT_n$. Each enriched tree $(A, \alpha) \in \cA_\cR[n]$ may be ordered in $\prod_{v \in V(A)} d^+_A(v)!$ many ways, so the event $\mA_n^\cR = (A, \alpha)$ as unordered tree corresponds to $\prod_{v \in V(A)} d^+_A(v)!$ many outcomes for $\mA_n^\cR$ as ordered tree. Each of these outcomes corresponds to a unique value $(T, \beta)$ for the random enriched plane tree $(\cT_n, \beta_n)$ and a unique value for the random partition $\sigma$. Hence
\begin{align}
\label{eq:zuerst1}
\Pr{(\mA_n,\alpha_n) = (A, \alpha)} = \frac{1}{n!} \sum_{(T, \beta)} \Pr{(\cT_n, \beta_n) = (T, \beta)}
\end{align}
with the sum index $(T, \beta)$ ranging over all enriched plane trees corresponding to the enriched tree $(A, \alpha)$ as described. It holds that $\Pr{\mA_n^\cR = (A, \alpha)} > 0$ if and only if  $\Pr{(\cT_n, \beta_n) = (T, \beta)}$ for all $(T, \beta)$ corresponding to $(A, \alpha)$, and in this case it holds that
\begin{align*}
\Pr{(\cT_n, \beta_n) = (T, \beta)} &= \Pr{(\cT_n, \beta_n) = (T, \beta) \mid \cT_n = T} \, \Pr{\cT_n = T} \\
&= \left(\prod_{v \in V(A)} \frac{\kappa(\beta(v))}{|\cR[M_v]|_\kappa} \right) \left(\frac{1}{Z_n}\prod_{v \in V(A)} \omega_{d^+_A(v)} \right ) \\
&= \frac{1}{Z_n} \prod_{v \in V(A)} \frac{\kappa(\beta(v))}{ d^+_{A}(v)!}.
\end{align*}
As there are $\prod_{v \in V(A)} d^+_A(v)!$ many choices for $(T, \beta)$, it follows from Equations~\eqref{eq:zuerst} and \eqref{eq:zuerst1} that
\[
\Pr{(\mA_n,\alpha_n) = (A, \alpha)} = \frac{1}{n! Z_n} \omega(A, \alpha) = \frac{1}{|\cA_\cR[n]|_\omega}\omega(A, \alpha).
\]
This concludes the proof.
\end{proof}

Next, we are going to prove the local convergence of our model of random enriched trees.

\begin{proof}[Proof of Theorem \ref{te:local}]
By construction, the set of enriched trees $\fmA$ is a subset of the compact product space $X^{\VHT}$ with $X = \{*, \infty\} \sqcup \bigcup_n \cR[n]$. We are going to argue that $\fmA$ is  also compact. For any vertex $v \in \VHT$ and integers $i>k\ge0$ set
\[
U_{v,k,i} = \{f \in X^{\VHT} \mid |f(v)| = k, f(vi) \ne *\}.
\]
Then each subset $U_{v,k,i}$ is open. Thus the subspace \[\fmA = \bigcap_{v,k,i} (X^{\VHT} \setminus U_{v,k,i})\]  is closed and hence compact.

Since $\fmA$ is compact, any sequence of random enriched plane trees has a convergent subsequence. In particular, the sequence $(\cT_n, \beta_n)$ of random enriched plane trees converges towards a limit object $(\bar{\cT}, \bar{\beta})$ along a subsequence $(n_k)_k$. We are going to show that $(\bar{\cT}, \bar{\beta}) \eqdist (\hat{\cT}, \hat{\beta})$ regardless of the subsequence. By standard methods \cite[Thm. 2.2]{MR0310933} this implies that $(\cT_n, \beta_n) \convdis (\hat{\cT}, \hat{\beta})$.

Consider the continuous map 
\begin{align}
\label{eq:varphi}
\varphi: X \to \bar{\ndN}_0, \, X \mapsto |X|
\end{align}
 where we set $|*| := 0$ and $|\infty| := \infty$. The induced continuous map $\varphi^{\VHT}: X^{\VHT} \to \bar{\ndN}_0^{\VHT}$ may be interpreted as the projection that sends an enriched plane tree $(T, \beta)$ to the plane tree $T$. Thus $\varphi^{\VHT}(\fmA) \subset \fmT$.
Hence it holds that \[\cT_{n_k} \eqdist \varphi^{\VHT}(\cT_{n_k}, \beta_{n_k}) \convdis \varphi^{\VHT}(\bar{\cT}) \eqdist \bar{\cT}.\]
 But $\cT_n \convdis \hat{\cT}$ by Theorem~\ref{te:convsigen} and thus $\bar{\cT} \eqdist \hat{\cT}$. 

Moreover, in order to show $(\bar{\cT}, \bar{\beta}) \eqdist (\hat{\cT}, \hat{\beta})$ it suffices to show that for any finite set $V \subset \VHT$ we have that $\bar{\beta}(v) \eqdist \hat{\beta}(v)$ for all $v \in V$. The set $X^V$ is countable, hence this is equivalent to 
\begin{align}
\label{eq:1star}
\Pr{\bar{\beta}(v) = R_v \text{ for all $v \in V$}} = \Pr{\hat{\beta}(v) = R_v \text{ for all $v \in V$}}
\end{align}
for all $(R_v)_v \in X^V$. Since $\bar{\cT} \eqdist \hat{\cT}$ it suffices to consider the case
\[
\Pr{|\hat{\beta}(v)| = |R_v| \text{ for all $v \in V$}} \ne 0
\] because otherwise both sides of Equation~\eqref{eq:1star} equal zero. 
 In particular, since the tree $\hat{\cT}$ has almost surely at most one vertex with infinite degree, we may assume that at the number of vertices $v \in V$, with $R_v = \infty$, equals one or zero.

{\em Case a): $R_v \neq \infty$ for all $v \in V$.} Then 
\begin{align}
\label{eq:orff1}
\Pr{\bar{\beta}(v) = R_v \text{ for all $v \in V$}} &= \lim_{k \to \infty} \Pr{\beta_{n_k}(v) = R_v \text{ for all $v \in V$}}.
\end{align}
Set $V_* = \{v \in V \mid R_v = *\}$. Since $\bar{\cT} \eqdist \hat{\cT}$ we have that
\begin{multline}
\label{eq:orff2}
\lim_{k \to \infty} \Pr{|\beta_{n_k}(v)| = |R_v| \text{ for all $v \in V \setminus V_*$}, \, \beta_{n_k}(v) = * \text{ for all $v \in  V_*$}} \\ = \Pr{|\hat{\beta}(v)| = |R_v| \text{ for all $v \in V \setminus V_*$}, \, \hat{\beta}(v) = * \text{ for all $v \in  V_*$} }.
\end{multline}
Moreover, by Lemma~\ref{le:coupling} and the definition of $(\hat{\cT}, \hat{\beta})$ it follows that
\begin{align}
\label{eq:orff3}
\Pr{&\beta_{n_k}(v) = R_v \text{ for all $v \in V$} \mid \, |\beta_{n_k}(v)| = |R_v| \text{ for all $v \in V \setminus V_*$}, \, |\beta_{n_k}(v)| = * \text{ for all $v \in V_*$}} \nonumber \\ 
&=  \prod_{v \in V \setminus V_*} \kappa(R_v) / |\cR[|R_v|]|_\kappa \nonumber  \\
&=  \Pr{\hat{\beta}(v) =  R_v \text{ for all $v \in V$} \mid |\hat{\beta}(v)| = |R_v| \text{ for all $v \in V \setminus V_*$}, \, |\hat{\beta}(v)| = * \text{ for all $v \in V_*$}}.
\end{align}
Hence Equation~\eqref{eq:1star} holds in this case.

{\em Case b): $R_u = \infty$ for precisely one $u \in V$.} By a similar argument as in case a) it follows that for all $K \ge 1$
\[
\Pr{\bar{\beta}(v) = R_v \text{ for all $v\ne u$}, |\bar{\beta}(u)| \ge K} = \Pr{\hat{\beta}(v) = R_v \text{ for all $v\ne u$}, |\hat{\beta}(u)| \ge K}.
\]
Letting $K$ tend to infinity yields Equation~\eqref{eq:1star}.
\end{proof}

Next, we are going to prove the strong type of convergence, if the weight sequence has type I$\alpha$.

\begin{proof}[Proof of Theorem~\ref{te:strengthend}]
Let $\cE$ denote the countably infinite set of all $\cR$-enriched plane trees and set
\[
\cE_k = \{(T,\beta)^{[k]} \mid (T,\beta) \in \cE\}.
\]
We have to show that
\begin{align}
\label{eq:toshowstr}
\lim_{n \to \infty}	\sup_{\cH \subset \cE_{k_n}} |\Pr{(\cT_n, \beta_n)^{[k_n]} \in \cH} - \Pr{(\hat{\cT}, \hat{\beta})^{[k_n]} \in \cH}| =0.
\end{align}
By assumption the weight sequence $\mathbf{w}$ has type $I\alpha$. Hence the random tree $\cT_n$ is distributed like a critical Galton--Watson tree conditioned on having $n$ vertices, with offspring distribution $\xi$ given in \eqref{eq:tt}. In particular, $\xi$ has finite non-zero variance.  For any $k$ and enriched plane tree $(T, \beta) \in \cE$ with height $\He(T) \ge k$ it holds that the probability $\Pr{(\cT_n,\beta_n)^{[k]}=(T,\beta)^{[k]}}$ is zero if and only if $\Pr{(\hat{\cT}, \hat{\beta})^{[k]} = (T, \beta)^{[k]}}$ is zero. Let \[
\cE_k' =\{ (\tau,\gamma) \in \cE_k \mid \He(\tau) = k, \Pr{(\cT_n,\beta_n)^{[k]} = (\tau,\gamma)} >0 \}
\] denote the subset of all enriched plane trees satisfying this property. By assumption, there is a sequence $t_n \to 0$ with $k_n = n^{1/2} t_n$. The left-tail upper bound ~\eqref{eq:gwtlefttail} for the height of $\cT_n$ implies that $\He(\cT_n) \ge k_n$ with probability tending to $1$ as $n$ becomes large. Thus, as $\He(\hat{\cT})=\infty$, it suffices to verify \eqref{eq:toshowstr} with the index $\cH$ ranging over all subsets of $\cE'_{k_n}$ instead of $\cE_{k_n}$.

Recall that for any tree $T$ and non-negative integer $\ell$ we let $L_\ell(T)$ denote the number of vertices with height $\ell$ in $T$. Let $\hat{\xi}$ denote the size-biased version of $\xi$ with distribution given by $\Pr{\hat{\xi}=i}=i\Pr{\xi=i}$. Then $L_0(\hat{\cT})=1$ and, as $\Ex{\xi}=1$, for all integers $k \ge 1$
\begin{align}
\label{eq:ineqstr2}
\Ex{L_{k}(\hat{\cT})} = \Ex{L_{k-1}(\hat{\cT})} + \Ex{\hat{\xi}} -1 = 1 + k(\Ex{\hat{\xi}} -1).
\end{align}
 For any $C>0$ and all $k$ and $n$ we define with foresight the subset $\cE_{C,k,n} \subset \cE'_k$ by
\begin{align}
\label{eq:defbige}
\cE_{C,k,n} = \{(T,\beta)^{[k]} \mid (T,\beta) \in \cE, L_k(\cT) \le  C (nt_n)^{1/2}, \sum_{i=0}^{k} L_i(T) \le C n t_n\} \cap \cE'_k.
\end{align}
Using Markov's inequality, Inequality \eqref{eq:ineqstr2} and the similar Inequality \eqref{eq:ineqstr1} for the number $L_k(\cT_n)$, we may easily check that there is a constant $C>0$ such that $(\cT_n, \beta_n)^{[k_n]}$ and $(\hat{\cT}, \hat{\beta})^{[k_n]}$ belong to $\cE_{C,k_n,n}$ with probability tending to $1$ as $n$ becomes large. Hence it suffices to verify \eqref{eq:toshowstr} with the index $\cH$ ranging only over all subsets of $\cE_{C,k_n,n}$ instead of $\cE_{k_n}$. 

In order to check this modification of \eqref{eq:toshowstr}, we are going to show
\begin{align}
	\label{eq:stuff}
	\lim_{n \to \infty} \sup_{(\tau,\gamma) \in \cE_{C,k_n,n}} |\Pr{(\cT_n, \beta_n)^{[k_n]} =(\tau,\gamma)} / \Pr{(\hat{\cT}, \hat{\beta})^{[k_n]} = (\tau,\gamma)} -1| =0.
\end{align}
Note that since $\cE_{C, k_n, n} \subset \cE_{k_n}'$ by definition, we do not divide by zero in \eqref{eq:stuff}. The crucial point in the following argument is that for any $(\tau,\gamma) \in \cE_{C, k_n, n}$ 
\[
\Pr{(\cT_n, \beta_n)^{[k_n]} = (\tau,\gamma) \mid \cT_n^{[k_n]} = \tau} = \Pr{(\hat{\cT}, \hat{\beta})^{[k_n]} =(\tau,\gamma) \mid \hat{\cT}^{[k_n]} = \tau},
\]
and hence
\begin{align}
	\label{eq:stuff2}
	\Pr{(\cT_n, \beta_n)^{[k_n]} =(\tau,\gamma)} / \Pr{(\hat{\cT}, \hat{\beta})^{[k_n]} = (\tau,\gamma)} = \Pr{\cT_n^{[k_n]} = \tau} / \Pr{\hat{\cT}^{[k_n]}=\tau}.
\end{align}
Let $D(\tau)$ denote the number of edges of $\tau$ and $\ell(\tau)$ the number of vertices of the restriction $\tau^{[k_n-1]}$. The probability for a trimmed simply generated tree to assume a given tree is known: Let $(\xi_i)_i$ be a family of independent copies of $\xi$ and set $S_m = \sum_{j=1}^m (\xi_j -1)$ for all $m$.  By Equation (15.11) in Janson's survey \cite{MR2908619} it holds that
\[
\Pr{\cT_n^{[k_n]} = \tau} = \frac{n}{n-\ell(\tau)}(D(\tau)- \ell(\tau) + 1)  \frac{\Pr{S_{n-\ell(\tau)} = \ell(\tau) -D(\tau) -1 }}{\Pr{S_n = -1}} \prod_{v \in \tau^{[k_n-1]}} \Pr{\xi = d^+_\tau(v)}.
\]
The probability for $\hat{\cT}^{[k_n]}=\tau$ is easily verified to be
\[
\Pr{\hat{\cT}^{[k_n]}=\tau} = L_{k_n}(\tau) \prod_{v \in \tau^{[k_n-1]}} \Pr{\xi = d^+_\tau(v)},
\]
as there are 
\[
L_{k_n}(\tau) = D(\tau)- \ell(\tau) + 1
\] possible places for the tip of the spine of $\hat{\cT}^{[k_n]}$ to end up, and the probability for the offspring of a spine-vertex to have size $i$ with the successor being at a fixed position $j \le i$ is given by $\Pr{\hat{\xi}=i}/i = \Pr{\xi=i}$. By Definition \eqref{eq:defbige} it holds that 
\begin{align}
\label{eq:klein}
\ell(\tau) \le Cn t_n \quad \text{and} \quad
\ell(\tau) - D(\tau) -1 = -L_{k_n}(\tau) \in [-C(nt_n)^{1/2},0] 
\end{align}
uniformly for all $(\tau,\gamma) \in \cE_{C,k_n,n}$.  Hence \eqref{eq:stuff2} simplifies to
\[
\Pr{\cT_n^{[k_n]} = \tau} / \Pr{\hat{\cT}^{[k_n]}=\tau} = (1+o(1)) \frac{\Pr{S_{n-\ell(\tau)} = \ell(\tau) -D(\tau) -1 }}{\Pr{S_n = -1}}. 
\]
with a uniform  term $o(1)$. Since $\Ex{\xi} =1$ and $\xi$ has finite variance $\sigma^2$, the local limit theorem in Lemma~\ref{le:llt1dim} for lattice distributed random variables ensures that \[\sup_{x \in  \spa(\mathbf{w}) \ndZ }| \sqrt{m}\Pr{S_m = x} - \frac{\spa(\mathbf{w})}{\sqrt{2 \pi \sigma^2}}\exp(-\frac{x^2}{2m\sigma^2})| \to 0\] as $m$ becomes large. 
Inequality \eqref{eq:klein} implies that $l(\tau) - D(\tau) -1 = o(n^{1/2})$. Hence it follows that uniformly in $(\tau,\gamma) \in \cE_{C,k_n,n}$
\[
\frac{\Pr{S_{n-\ell(\tau)} = \ell(\tau) -D(\tau) -1 }}{\Pr{S_n = -1}} = 1 + o(1).
\]
This verifies the limit in \eqref{eq:stuff} and hence completes the proof.
\end{proof}

\begin{proof}[Proof of Lemma~\ref{le:extension}]
	The proof is analogous to the proof of Proof of Theorem \ref{te:local}. As $\fmA$ is compact, the sequence $(\tau_n, \beta_n)$ converges toward a limit object $(\bar{\tau}, \bar{\beta})$ along a subsequence $(n_k)_k$, and we need to show that $(\bar{\tau}, \bar{\beta}) \eqdist (\hat{\tau}, \hat{\beta})$. As the projection  $\varphi: \fmA \to \fmT$ from \eqref{eq:varphi} is continuous, it follows that \[\hat{\tau} \eqdist \bar{\tau}.\] Given an arbitrary  finite subset $V \subset \VHT$, it remains to show that \[(\bar{\beta}(v))_{v \in V} \eqdist (\hat{\beta}(v))_{v \in V}\] in the countable space $X^V$. Let $(R_v)_v \in X^V$ and set \[V_\infty := \{v \in V \mid R_v = \infty\}. \] We need to show that for all positive integers $K$
	\begin{multline*}
		\Pr{\bar{\beta}(v) = R_v \text{ for } v \in V \setminus V_\infty, \,\, |\bar{\beta}(v)| \ge K \text{ for } v \in V_\infty} = \\ \Pr{\hat{\beta}(v) = R_v \text{ for } v \in V \setminus V_\infty, \,\, |\hat{\beta}(v)| \ge K \text{ for } v \in V_\infty}.
	\end{multline*}
	But this follows from $\hat{\tau} \eqdist \bar{\tau}$ entirely analogous as in Equations~\eqref{eq:orff1}, \eqref{eq:orff2} and \eqref{eq:orff3}.
\end{proof}

\subsection{Proofs for the limits of re-rooted enriched trees in Section~\ref{sec:fringe}}

\begin{proof}[Proof of Theorem~\ref{te:thmben}] We start with the first statement of the theorem. Suppose that the weight-sequence $\mathbf{w}$ has type I.
	It was shown in \cite[Thm. 5.1]{inaihpb} that
	\begin{align}
		\label{eq:ccoo}
		(\cT_n, v_0) \convdis \cT^*
	\end{align}
	in the space $\fmT^\bullet$. For any finite $\cR$-enriched plane tree $A = (T, \gamma)$ that is pointed at a vertex $v$ it holds that if $n$ is large enough, then the probability for $f_k(\cT_n, v_0) = (T,v)$ is positive, if and only if the probability for $f_k(\cT^*) = (T,v)$ is positive. We may hence assume that both probabilities are positive. Let $V \subset \VHT^\bullet$ denote the set of vertices that correspond to the tree $T$. Then
	\begin{align}
	\label{eq:tti}
	&\Pr{f_k( (\cT_n, \beta_n), v_0) = (A, v)} \nonumber \\
	&\qquad= \Pr{f_k(\cT_n, v_0) = (T,v)} \Pr{ \beta_n(v) = \gamma(v) \text{ for all $v \in V$} \mid f_k(\cT_n, v_0) = (T,v)} \nonumber \\
	&\qquad=  \Pr{f_k(\cT_n, v_0) = (T,v)} \prod_{v \in V} \frac{\kappa(\gamma(v))}{|\cR[d^+_T(v)]|_\kappa } \nonumber \\
	&\qquad= \Pr{f_k(\cT_n, v_0) = (T,v)} \Pr{ \beta^*(v) = \gamma(v) \text{ for all $v \in V$} \mid f_k(\cT_n, v_0) = (T,v)}.
	\end{align}
	The limit in \eqref{eq:ccoo} implies that
	\[
	\Pr{f_k(\cT_n, v_0) = (T,v)} \to \Pr{f_k(\cT^*) = (T,v)}.
	\]
	It follows that
	\[
		\Pr{\mH_k = (A,v)} \to \Pr{f_k(\cT^*, \beta^*) = (A,v) }.
	\]
	As $f_k(\cT^*)$ is almost surely finite, this yields
	\[
		\mH_k \convdis f_k(\cT^*, \beta^*).
	\]
	
	We proceed with the second statement of the theorem. In the subcase where the weight-sequence $\mathbf{w}$ has type I$\alpha$, it was shown in \cite[Thm. 5.2]{inaihpb} that for any sequence $k_n = o(\sqrt{n})$ of non-negative integers it holds that
	\[
	d_{\textsc{TV}}(f_{k_n}(\cT_n,v_0), f_{k_n}(\cT^*)) \to 0.
	\]
	It follows by Equation~\eqref{eq:tti}  that
	\begin{align*}
	d_{\textsc{TV}}( \mH_{k_n},  f_{k_n}(\cT^*, \beta^*) ) &= \frac{1}{2} \sum_{ ((T, \gamma),v)} \left| \Pr{\mH_{k_n} = ((T, \gamma),v)} - \Pr{f_{k_n}(\cT^*, \beta^*) = ((T, \gamma),v)} \right| \\
	&\le d_{\textsc{TV}}(f_{k_n}(\cT_n,v_0), f_{k_n}(\cT^*)) \to 0.
	\end{align*}
	
	We now show the third statement of the theorem.  Let $N_T$ and  denote the number of vertices $x \in \cT_n$ with $f_k(\cT_n, x) = (T,v)$,  and likewise let $N_H$ denote the number of vertices $y \in \cT_n$ with  $f_k((\cT_n, \beta_n), y) = H$ for $H := (A,v)$. The extended fringe subtrees satisfying $f_k(\cT_n, x) = (T,x)$ are disjoint when considered as subtrees of $\cT_n$. Consequently, each  has a conditionally independent chance with probability $p = \prod_{v \in V} \frac{\kappa(\gamma(v))}{|\cR[d^+_T(v)]|_\kappa }$ that it additionally satisfies $f_k((\cT_n, \beta_n), x) = (A,v)$. Hence, 
	\begin{align}
	\label{eq:na444}
	N_H \eqdist \sum_{i=1}^{N_T} X_i
	\end{align}
	for an i.i.d. family $(X_i)_{i \ge 1}$ of Bernoulli distributed random variables with $\Pr{X_i = 1} = p$ and $\Pr{X_i = 0} = 1- p$.  
	
	Let $(\xi_i)_{1 \le i \le n}$ be a family of independent copies of $\xi$ and set $S_n = \sum_{j=1}^n \xi_j$. Let $d_1, \ldots, d_r$ denote the depth-first-search ordered outdegree sequence of the tree $T$. In the following we assume that $r < n$. Janson~\cite[Eq. (17.1)]{MR3432572} argued that
	\[
		N_T \eqdist \left( f(\xi_1, \ldots, \xi_n) \, \Big|\, S_n = n-1 \right)
	\]
	with $f: \ndN_0^n \to \ndR$ a function so that $f(x_1, \ldots, x_n)$ counts the number of indices $1 \le i_0 \le n$ with $x_{i_0+j} = d_j$ for all $1 \le j \le r$. (Here we set $x_{i_0 + j} := x_{i_0 +j -n}$ in case $i_0 + j >n$.) Note that for any two such indices $i_0 \ne i'_0$ it must hold that $|i_0 - i'_0| \ge r$ because the sequence $d_1, \ldots, d_r$ is the depth-first-search ordered outdegree sequence of a tree. Hence for any $1 \le i \le n$ and any integer $\bar{x}_i \in \ndN_0$ it holds that
	\[
		| f(x_1, \ldots, x_n) - f(x_1, \ldots, x_{i-1}, \bar{x}_i, x_{i+1}, \ldots, x_n)| \le 2.
	\]
	Setting $\pi_T = \prod_{j=1}^r \Pr{\xi=d_j}$, McDiarmid's inequality yields
	\begin{align*}
		\Pr{ |N_T - n \pi_T| > n\epsilon} &\le \Pr{S_n = n-1}^{-1} \Pr{| f(\xi_1, \ldots, \xi_n) -n \pi_T| > n\epsilon} \\
		&\le 2 \Pr{S_n = n-1}^{-1} \exp(-n \epsilon^2 / 2).
	\end{align*}
	By~\cite[Thm. 18.1]{MR3432572} it holds that $\Pr{S_n = n-1} = \exp(o(n))$, hence
	\begin{align}
	\Pr{ |N_T - n \pi_T| > n\epsilon} \le \exp((o(1) - \epsilon^2 / 2)n).
	\end{align}
	Combining this with Equation~\eqref{eq:na444} and the well-known Chernoff bounds yields
	\[
		\Pr{ |N_H - n p\pi_T| > n\epsilon} \le \exp((o(1) - \epsilon^2 c)n)
	\]
	for some constant $c>0$ that does not depend on $n$ or $A$. As $\pi_H= p\pi_T$ this concludes the proof of the deviation bound for $N_H$. By the Borel--Cantelli Lemma it follows that
	\[
		N_H / n \to \pi_H
	\]
	holds almost surely.
		
	It remains to verify the fourth part. Suppose that $\pi_H > 0$.	If $\mathbf{w}$ even has type I$\alpha$, then known limits for fringe subtrees~\cite[Cor. 1.8]{MR3432572} imply 
	\[
	\frac{N_T - n \pi_T}{\sqrt{n}} \convdis \cN(0, \sigma_T^2)
	\]
	for a constant $\sigma_T > 0$. As $p>0$, the central limit theorem states
	\[
		\frac{\sum_{i=1}^n X_i - n p}{\sqrt{n}} \convdis \cN(0, \sigma_p^2)
	\]
	for some $\sigma_p>0$. Since $(X_i)_{i\ge 1}$ is independent from the random variable $N_T$, Skorokhod's representation theorem allows us to assume that there are independent random variables $Y \sim \cN(0, \sigma_T^2)$ and $Z \sim \cN(0, \sigma_p^2)$ such that  $\lim_{n \to \infty} \frac{\tilde{N}_T - n \pi_T}{\sqrt{n}} = Y$ and $\lim_{k \to \infty} \frac{Z_k - k p}{\sqrt{k}} = Z$ hold almost surely for some random variables $\tilde{N}_T \eqdist N_T$ (for all $n \ge 1$, as the dependence on $n$ is implicit) and $Z_k \sim \text{Bin}(k,p)$, $k \ge 1$ such that $(\tilde{N}_T)_{n \ge 1}$ is independent from $(Z_k)_{k \ge 1}$. By~\eqref{eq:na444} it follows that
	\[
	\frac{N_A - n p \pi_T}{\sqrt{n}}  \eqdist \frac{Z_{\tilde{N}_T} - n p \pi_T}{\sqrt{n}},
	\]
	and
	\begin{align*}
	\frac{Z_{\tilde{N}_T} - n p \pi_T}{\sqrt{n}} = \frac{Z_{\tilde{N}_T} - \tilde{N}_T \pi_T}{\sqrt{\tilde{N}_T}} \sqrt{\frac{\tilde{N}_T}{n}} + \pi_T\frac{\tilde{N}_T - np}{\sqrt{n}} 
	\to \pi_T(Y + Z)
	\end{align*}
	almost surely. Since $Y$ and $Z$ are independent, it follows that
	\[
	\frac{N_A - n \pi_H}{\sqrt{n}} \convdis \cN(0, \pi_T^2(\sigma_T^2 + \sigma_p^2)).
	\]
\end{proof}

\begin{proof}[Proof of Theorem~\ref{te:thmco}]
	In \cite[Thm. 7.1]{inaihpb} it was shown that for any finite set $x_1, \ldots, x_r \in \VHT^\bullet$ of vertices it holds that
	\begin{align}
		\label{eq:yd}
		d_{\textsc{TV}}( (\bar{d}_{ (\cT_n, v_0)}(x_i))_{1 \le i \le r}, (\bar{d}_{ \cT_n^*}(x_i))_{1 \le i \le r}) \to 0.
	\end{align}
	The height of the specified vertex in $\cT_n^*$ is stochastically bounded, hence it follows that for each fixed $m \ge 0$ the size of the pruned tree $P_m(\cT_n^*)$ is stochastically bounded. Thus, for each $\epsilon >0$, we may choose a sufficiently large set of vertices $V \subset \VHT^\bullet$ such that for all $n$ it holds with probability at least $1 - \epsilon$ that the vertex set of the pruned tree $P_m(\cT_n^*)$ is a subset of $V$. Hence by the limit in \eqref{eq:yd} it follows that
	\[
		d_{\textsc{TV}}(P_m(\cT_n^*), P_m(\cT_n, v_0)) \le 2 \epsilon
	\]
	for sufficiently large $n$. As $\epsilon>0$ was arbitrary, it follows that
	\begin{align}
		\label{eq:yyd}
			d_{\textsc{TV}}(P_m(\cT_n^*), P_m(\cT_n, v_0)) \to 0.
	\end{align}
	Let $T^\bullet=(T,x)$ be a finite pointed plane tree where $x$ has height at least $1$ and  $P_m(T,x) = (T,x)$. Let $V \subset \fmT^\bullet$ denote the subset of vertices that correspond to the vertices of $T$, except for the sons of the root of $T$ that lie more than $m$ to the left or more than $m$ to the right of the unique spine successor. 
	
	Given $P_m(\cT_n^*) = (T,x)$, the family of $(\beta_n^*(v))_{v \in V}$ of $\cR$-structures is (conditionally) independent, and for each $v \in V$ the $\cR$-structure $\beta_n^*(v)$ gets drawn from the set $\cR[ d_T^+(v)]$ with probability proportional to its $\kappa$-weight. The same holds for the conditional distribution of the family $(\beta_n(v))_{v \in V}$ given $P_m(\cT_n,v_0) = (T,x)$. Thus, for any finite pointed plane tree $T^\bullet$ it holds that
	\[
		( P_m( \cT_n^*, \beta_n^*) \mid P_m(\cT_n^*) = T^\bullet) \eqdist ( P_m( (\cT_n, \beta_n), v_0) \mid P_m(\cT_n) = T^\bullet).
	\]
	By the limit in \eqref{eq:yyd}, it follows that
	\[
		d_{\textsc{TV}}( P_m( \cT_n^*, \beta_n^*),  P_m( (\cT_n, \beta_n), v_0)) \to 0.
	\]
\end{proof}

\subsection{Proofs for the results on Schr\"oder enriched parenthesizations in Section~\ref{sec:schroeder}}

\begin{proof}[Proof of Lemma~\ref{le:couplingschroeder}]
	Let $M_n$ denote the set of plane trees with $n$ leaves and no vertices with outdegree $1$ and set \[L_n = \sum_{T \in M_n} \prod_{v \in T} p_{d^+_T(v)}.\] 
	
	Let $\cA$ denote the species of unordered rooted trees with unlabelled internal vertices and leaves as atoms.  Note that although the internal vertices of an $\cA$-object are unlabelled, they are nevertheless distinguishable, as each may be identified with the set of labels of the leaves of the fringe-subtree at that vertex. Hence every element $A \in \cA[n]$ corresponds to $\prod_{v \in A} d^+_A(v)!$ ordered trees with labelled leaves. Moreover, for any plane tree with $n$ leaves, there are $n!$ ways to label the leaves from $1$ to $n$. Consequently,
	\begin{align}
	\label{eq:etzad}
	|\cS_\cN[n]|_\upsilon = \sum_{A \in \cS_\cN[n]} \prod_{v \in A} |\cN[d^+_A(v)]|_\gamma = \sum_{T \in M_n}  n! \prod_{v \in T} \frac{|\cN[d^+_T(v)]|_\gamma}{d^+_T(v)!} =  n! L_n.
	\end{align}
	An element of $\cS_\cN[n]$ consists of a tree $A \in \cA[n]$ together with a function $f$ that assigns to each vertex $v$ of $A$ with offspring set $O_v$ an $\cN$-structure $f(v) \in \cN[O_v]$. There are $\prod_{v \in A} d^+_A(v)!$ many ways to order the offspring sets. We may consider the tree $\cS_\cN[\sigma](\tau_n, \delta_n)$ as ordered by keeping the ordering of $(\tau_n, \delta_n)$, so the event $\mS_n^\cN=(A,f)$ as unordered trees would correspond to $\prod_{v \in A} d^+_A(v)!$ different outcomes for $\mS_n^\cN$ as ordered tree. Each of these outcomes corresponds to a unique value $(T, \delta)$ for $(\tau_n, \delta_n)$ and a unique value for the permutation $\sigma$. Hence the probability for the event that $\mS_n^\cN = (A, f)$ as unordered trees is given by
	\begin{align}
	\label{eq:etzad1}
	\Pr{ \mS_n^\cN = (A, f)} =   \frac{1}{n!} \sum_{(T, \delta)} \Pr{(\tau_n, \delta_n) = (T, \delta)}
	\end{align}
	with the sum index $(T, \delta)$ ranging over all $\prod_{v \in A} d^+_A(v)!$ enriched plane trees corresponding to $(A,f)$ as described. For each such enriched plane tree $(T, \delta)$ it holds that $\Pr{(\tau_n,\delta_n) = (T, \delta)} > 0$ if and only if  $\Pr{\mS_n^\cN = (A, f)} > 0$.  If this is the case, then
	\begin{align*}
	\Pr{(\tau_n, \delta_n) = (T, \delta)} &= \Pr{(\tau_n, \delta_n) = (T, \delta) \mid \tau_n = T} \, \Pr{\tau_n = T} \\
	&= \left(\prod_{v \in V(T)} \frac{\gamma(\delta(v))}{|\cN[d^+_T(v)|_\gamma} \right) \left(\frac{1}{L_n}\prod_{v \in T} p_{d^+_T(v)} \right ) \\
	&= \frac{1}{L_n} \prod_{v \in V(A)} \frac{\gamma(f(v))}{ d^+_{A}(v)!} \\
	&= \frac{1}{L_n} \left (\prod_{v \in V(A)} d^+_A(v)! \right)^{-1} \upsilon(A,f).
	\end{align*}
	Since there are $\prod_{v \in A} d^+_A(v)!$ choices for $(T, \delta)$,  it follows from Equations~\eqref{eq:etzad} and \eqref{eq:etzad1} that
	\[
	\Pr{\mS_n^\cN = (A, f)} = \frac{1}{n! L_n} \upsilon(A, f) = \frac{1}{|\cS_\cN[n]|_\upsilon}\upsilon(A, f).
	\]
	This concludes the proof.
\end{proof}

\begin{figure}[t]
	\centering
	\begin{minipage}{1.0\textwidth}
		\centering
		\includegraphics[width=0.6\textwidth]{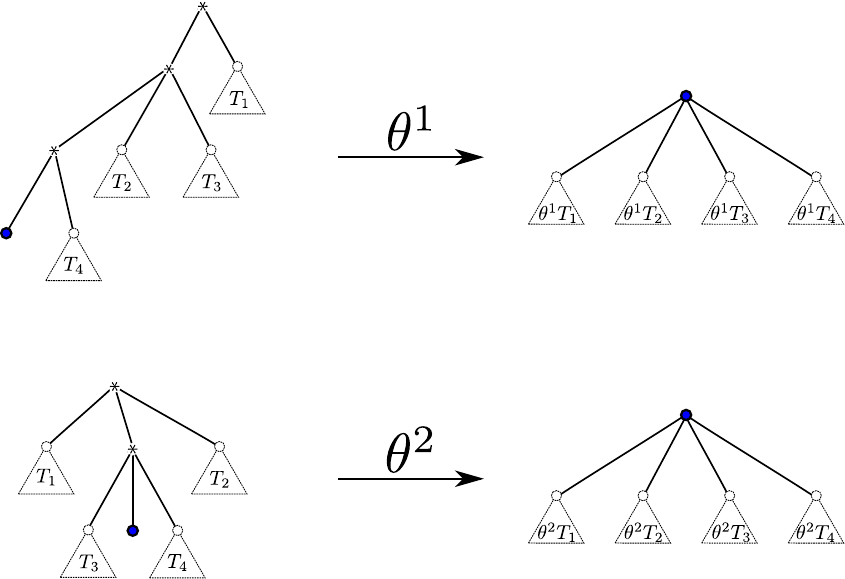}
		\caption{Variants of the Ehrenborg--M\'endez transformation.}
		\label{fi:ehrenb2}
	\end{minipage}
\end{figure}  

\begin{proof}[Proof of Lemma~\ref{le:convll}]
	It remains to treat the case $III$, where $\hat{\tau}$ consists of a root-vertex with infinitely many offspring, all of which are leaves. In order to show $\tau_n \convdis \hat{\tau}$, we need to show that for arbitrarily large $K \ge 1$ the root of $\tau_n$ has with high probability at least $K$ children, with the first $K$ of them being leaves.
	
	Let $\cT_\ell^\upsilon$ denote the species of plane trees with leaves as atoms as discussed in Section~\ref{sec:lgwt}, where each tree $T \in \cT_\ell^\upsilon[n]$ receives the weight 
	\[
	\upsilon(T) = \prod_{v \in T} p_{d^+_T(v)}.
	\]
	As discussed in Section~\ref{sec:lgwt}, $\cT_\ell^\upsilon$ is a Schr\"oder enriched parenthesization, since it satisfies an isomorphism
	\[
		\cT_\ell^\upsilon \simeq \cX + \Seq_{\ge 2}^p \circ \cT_\ell^\upsilon \qquad \text{with} \qquad \Seq_{\ge 2}^p(z) = \sum_{k=2}^\infty p_k z^k.
	\]

	The random tree $\tau_n$ is distributed like the outcome of sampling an element from $\cT_\ell^\upsilon[n]$ with probability proportional to its weight, and dropping the labels. 
	
	Recall the Ehrenborg--M\'endez transformation of Schr\"oder enriched parenthesizations that is given in Equation~\eqref{eq:schrenr} in Section~\ref{sec:diss} and is illustrated in Figure~\ref{fi:ehrenb}.  Any choice of a natural root-vertex of $\Seq_{\ge 2}^{p}$ yields a different isomorphism by Equation~\eqref{eq:schrenr}. We may choose the left-most element of sequence as distinguished point, but just as well the second from left. Moreover, as illustrated in Figure~\ref{fi:ehrenb}, when the descending along the distinguished offspring vertices, the transformation orders the encountered fringe trees starting with those of maximal height, then those just below, and so on. This is not essential at all, we may just as well reverse the order and put the fringe trees of the non-distinguished offspring of the root first, then those of height $2$, and so on, to obtain an isomorphism as in Equation~\eqref{eq:schrenr}. 
	
	Thus we may define two variants $\theta^1$ and $\theta^2$ of the Ehrenborg--M\'endez transformation as illustrated in  Figure~\ref{fi:ehrenb2}, where in $\theta^1$ we descend along the first offspring, in $\theta^2$ along the second, and in both the order the fringe subtrees by first placing those at height $1$, then those at height $2$, and so on.
	
	By Lemma~\ref{le:coupling} it follows that the transformed trees $\theta^1(\tau_n)$ and $\theta^2(\tau_n)$ are both distributed like simply generated trees with a common weight-sequence given by its generating series $\Seq ( \Seq_{\ge 2}^p(z)/z)$. Since we assumed $(p_k)_k$ to have type III, it follows that this series has radius of convergence $0$. Hence the simply generated trees $\theta^1(\tau_n)$ and $\theta^2(\tau_n)$ also have type III and  converge in distribution toward the star $\hat{\tau}$ by Theorem~\ref{te:convsigen}.
	
	Let $K \ge 1$ be an arbitrary fixed integer. The convergence of $\theta^2(\tau_n)$ toward $\hat{\tau}$ implies that with probability tending to $1$ as $n$ becomes large the first offspring of the root in $\tau_n$ is a leaf.  If we apply the transformation $\theta^1$ to any plane tree $T$ (that has no vertex with out-degree $1$) where the first offspring of the root is a leaf, then for all $k \ge 2$ the $k$th offspring of the root in $\theta^1(T)$ corresponds precisely to the $k$th offspring of the root in $T$, and it is a leaf in $T$ if and only if it is a leaf in $\theta^1(T)$. Hence, since the first offspring in $\tau_n$ is with high probability a leaf, it follows from the convergence $\theta^1(\tau_n) \convdis \hat{\tau}$ that the root of $\tau_n$ has with high probability at least $K$ offspring with the first $K$ all being leaves. Since $K \ge 1$ was arbitrary, this confirms $\tau_n \convdis \hat{\tau}$.
\end{proof}

Having Lemma~\ref{le:convll} at hand, Theorem \ref{te:local} follows directly by applying  Lemma~\ref{le:extension}. We proceed with the proof of Lemma~\ref{le:types}.

\begin{proof}[Proof of Lemma~\ref{le:types}]
	Recall that
	\[
	\phi(z) = 1 / (1 -  \cH^\kappa(z)) \qquad \text{and} \qquad p(z) = z \cH^\kappa(z).
	\]
	It is clear that $\phi(z)$ is analytic at $0$ if and only if $\rho_p>0$. Hence the weight-sequence $\mathbf{w}$ has type III if only if $\mu_{t_0}=0$, and in this case we have $t_0= 0 = \tau$.
	
	Suppose now that $\rho_\phi, \rho_p >0$.
	We first show that 
	\begin{align}
	\label{eq:tetete}
	\mu_{\rho_p} \ge 1 \qquad \text{if and only if} \qquad  \psi(\rho_\phi) \ge 1.
	\end{align}
	To do so, we are going to distinguish two different cases. 
	
	First, suppose that $\cH^\kappa(\rho_p) \ge 1$. It holds for all $0< t < \rho_p$ that 
	\[
	\mu_{t} = \sum_{k \ge 2} k p_k t^{k-1} \ge \sum_{k \ge 2} p_k t^{k-1} = \cH^\kappa(t).
	\]
	Letting $t$ tend to $\rho_p$ from below yields
	$
	\mu_{\rho_p} \ge 1.
	$
	Moreover, $\cH^\kappa(\rho_p) \ge 1$ implies that $\cH^\kappa(\rho_\phi) = 1$ and hence $\phi(\rho_\phi) = \infty$. By a general principle given in Janson~\cite[Lem. 3.1]{MR2908619} this implies that $\psi(\rho_\phi) > 1$. Hence \eqref{eq:tetete} is valid if $\cH^\kappa(\rho_p) \ge 1$.
	
	To conclude the verification of \eqref{eq:tetete}, it remains to consider the case $\cH^\kappa(\rho_p) < 1$. In this case, it follows that $\rho_\phi = \rho_p < \infty$ and $p(\rho_\phi) < \rho_\phi$.
	A quick calculation shows that $\mu_z = p'(z)$ and 
	\begin{align}
	\label{eq:flyingcircus}
	\psi(z) = z \phi'(z) / \phi(z) = (p'(z)z - p(z)) / (z - p(z)).
	\end{align}
	Thus $\psi(\rho_\phi) \ge 1$ if and only if $\mu_{\rho_p} \ge 1$. This verifies $\eqref{eq:tetete}$ in the case $\cH^\kappa(\rho_p) < 1$.
	
	This shows that $\mathbf{w}$ has type $I$ if and only if $\mu_{t_0}>0$.  In this case it holds that $\mu_{t_0} = p'(t_0)=1$, and it follows from Equation~\eqref{eq:flyingcircus}  that $\psi(t_0) = 1$. Hence $t_0=\tau$ in the type $I$ regime.
	
	We have also shown that $\mathbf{w}$ has type $II$ if and only if $0<\mu_{\rho_p}<1$. In this case, it holds that $t_0 = \rho_p$ and $\tau = \rho_\phi$. As shown above, if $\cH^\kappa(\rho_p) \ge 1$, then we would be in the type $I$ regime, so it must hold that $\cH^\kappa(\rho_p) < 1$, and consequently $\rho_p = \rho_\phi$. This verifies $t_0 = \tau$ in the type $II$ regime.
\end{proof}

\subsection{Proofs for the results on Gibbs partitions in Section~\ref{sec:gibbs}}

Before starting with the proof of Lemma~\ref{le:supen}, let us recall the cycle lemma.
\begin{lemma}[\cite{MR0138139}]
	\label{le:cycle}
	For each sequence of integers \[k_1, \ldots, k_n \ge -1\] with \[\sum_{i=1}^n k_i = -r \le 0\] there exist precisely $r$ values of $0 \le j \le n-1$ such that the cyclic shift \[(k_{1,j}, \ldots, k_{n,j}) :=   (k_{1+j},\ldots, k_n, k_1, \ldots, k_{j})\] satisfies \[\sum_{i=1}^{u} k_{i,j} > -r\] for all $1 \le u \le n-1$.
\end{lemma}

\begin{proof}[Proof of Lemma~\ref{le:supen}]

	In order to verify Equation~\eqref{eq:starter}, we use the notation 
	\[
	Z(m,n) := \sum_{\substack{k_1, \ldots, k_n \ge 0 \\ k_1 + \ldots + k_n = m}} \omega_{k_1} \cdots \omega_{k_n}
	\]
	for all $m,n \ge 0$. Note that for each $n$ the coefficient $a_n$ is equal to the sum of all products \[\omega_{j_1} \cdots \omega_{j_{n}}\]
	such that
	\[
	j_1 + \ldots + j_n = n-1
	\]
	and
	\[
	j_1 + \ldots + j_s \ge s
	\]
	for all $s < n$. By Lemma~\ref{le:cycle} it follows that
	\begin{align}
	\label{eq:erst}
	a_n = \frac{1}{n} Z(n-1,n).
	\end{align}
	It follows also that the expression
	\[
	\sum_{ \substack{ i_1, \ldots, i_k \ge 1 \\ i_1 + \ldots + i_k = n}} a_{i_1} \cdots a_{i_k}
	\]
	is the sum of all products
	\[
	\omega_{j_1} \cdots \omega_{j_n}
	\]
	such that 
	\[
	j_1 + \ldots + j_n = n-k
	\]
	and 
	\[
	j_1 + \ldots + j_s \ge s-k+1
	\]
	for all $s < n$. Thus, Lemma~\ref{le:cycle} implies that
	\begin{align}
	\label{eq:zweit}
	\sum_{ \substack{ i_1, \ldots, i_k \ge 1 \\ i_1 + \ldots + i_k = n}} a_{i_1} \cdots a_{i_k} = \frac{k}{n} Z(n-k,n).
	\end{align}
	Let $(Y_1^{(n-1,n)}, \ldots, Y_n^{(n-1,n)})$ denote a random vector of non-negative integers satisfying
	\[
	Y_1^{(n-1,n)} + \ldots + Y_n^{(n-1,n)} = n-1
	\]
	and distribution given by
	\[
	\Pr{(Y_1^{(n-1,n)}, \ldots, Y_n^{(n-1,n)}) = (j_1, \ldots, j_n)} = \frac{\omega_{j_1} \cdots \omega_{j_n}}{Z(n-1,n)}.
	\]	
	By Equations~\eqref{eq:erst} and \eqref{eq:zweit} it holds that
	\begin{align*}
	\frac{1}{\omega_0^{k-1}a_{n-k+1}}\sum_{ \substack{ i_1, \ldots, i_k \ge 1 \\ i_1 + \ldots + i_k = n}} a_{i_1} \cdots a_{i_k} &= \frac{k Z(n-k,n)}{\omega_0^{k-1}Z(n-k, n-k+1)}(1 + O(n^{-1})) \\
	&= k \Pr{Y_1^{(n-1,n)} = 0, \ldots, Y_{k-1}^{(n-1,n)}= 0}^{-1} (1 + O(n^{-1})).
	\end{align*}
	By \cite[Thm. 11.7]{MR2908619} it holds that
	\[
	\lim_{n \to \infty}\Pr{Y_1^{(n-1,n)} = 0, \ldots, Y_{k-1}^{(n-1,n)}= 0}  = 1.
	\]
	Thus 
	\begin{align}
	\label{eq:quint}
	\lim_{n \to \infty} \frac{1}{\omega_0^{k-1}a_{n-k+1}}\sum_{ \substack{ i_1, \ldots, i_k \ge 1 \\ i_1 + \ldots + i_k = n}} a_{i_1} \cdots a_{i_k} = k.
	\end{align}
	Note that $a_1 =  \omega_0$. Hence 
	\[
	\frac{1}{\omega_0^{k-1}a_{n-k+1}}\sum_{ \substack{ i_1, \ldots, i_k \ge 1 \\ i_1 + \ldots + i_k = n}} a_{i_1} \cdots a_{i_k}  = k +  \frac{1}{\omega_0^{k-1}a_{n-k+1}}\sum_{\substack{1 \le i_1, \ldots, i_k < n - (k-1) \\ i_1 + \ldots + i_k = n }} a_{i_1} \cdots a_{i_k}.
	\]
	Thus Equation~\eqref{eq:starter} follows by Equation~\eqref{eq:quint}.
\end{proof}

\begin{proof}[Proof of Theorem~\ref{te:need}]
	Let us first argue that it suffices to show that there is a number $L$ such that with high probability the Gibbs partition $\mS_n$ has at most $L$ components.
	
	Applying Lemma~\ref{le:bole} to the composition $\cF_{\le n}^\upsilon \circ \cG^\gamma_{\le n}$ yields that we may sample $\mS_n$ up to relabelling by choosing an arbitrary number $x>0$ and then drawing a random $\cF$-object $\mF_n$ according to the Boltzmann distribution $\mathbb{P}_{\cF_{\le n}^\upsilon, \cG_{\le n}^\gamma(x)}$, then for each of its atoms $1 \le i \le |\mF_n|$ an independent $\cG$-object $\mG_i$ according to a $\mathbb{P}_{\cG_{\le n}^\gamma,x}$-distribution, and conditioning the resulting composite structure $(\mF_n, (\mG_i)_i)$ on having total size $n$. It follows that for any integer $k \ge 0$ with $[z^k] \cF^\upsilon(z) >0$ the conditional distribution $( (\mG_i)_{1 \le i \le k} \mid |\mF| =k)$ does not depend on the species $\cF^\upsilon$ any more In particular, conditioned on having $k$ components, the component sizes of the  Gibbs partition $\mS_n$ are identically distributed as the component sizes of an $n$-sized $\Seq_{\{k\}} \circ \cG^\gamma$ Gibbs partition. Thus Equation~\eqref{eq:starter} yields that the largest component of this conditioned Gibbs partition has with high probability size $n - (k-1)$. Thus, in order to verify  Theorem~\ref{te:need}, it suffices to show that for some fixed number $L$ the Gibbs partition $\mS_n$ has with high probability at most $L$ components.
	
	It remains to show that for sufficiently large $L$ it holds that
	\begin{align}
	\label{eq:tos1}
	[z^n] \cF^\upsilon_{\ge L} \circ \cG^\gamma (z) = o([z^n] \cF^\upsilon_{< L} \circ \cG^\gamma (z))
	\end{align}
	as $n$ tends to infinity. This is trivial if the generating series $\cF^\upsilon(z)$ is a polynomial. Hence throughout the remaining proof we only consider the case where the set 
	\[
	\Omega = \{ m \in \ndN \mid [z^m]\cF^\upsilon(z) >0 \}
	\]
	is infinite.
	
	By assumption  there is a constant $d \ge 1$ such that $a_i = 0$ for $i \notin 1 + d \ndZ$, and a constant $I \ge 1$ such that $a_i > 0$ for all $i \in 1 + d \ndZ$ with $i \ge I$. For each $1 \le b \le d$ it holds that
	\[
	(\cF^\upsilon \circ \cG^\gamma)[n] = (\cF^\upsilon_{b + d \ndZ} \circ \cG^\gamma)[n]
	\]
	whenever $n \equiv b \mod d$.  Thus it suffices to consider the case
	\begin{align}
	\label{eq:thas}
	\Omega \subset b + d \ndZ \qquad \text{and} \qquad n \equiv b \mod d.
	\end{align}
	
	By assumption, the series $\cF^\upsilon(z)$ has positive radius of convergence. Hence there is a constant $C>1$ such that
	\begin{align}
	[z^m] \cF^\upsilon(z) \le C^m
	\end{align}
	for all $m \ge 1$. It will be convenient to use the notation
	\[
	A_m^{\ge L}  := [z^m] \sum_{\ell \ge L} C^\ell \cG^\gamma(z)^\ell
	\]
	for all $L \ge 1$. Clearly it holds that
	\[
	[z^n] \cF^\upsilon_{\ge L} \circ \cG^\gamma (z) \le A_n^{\ge L}
	\]
	for all $n$ and $L$. Thus, in order to show Equation~\eqref{eq:tos1} it suffices to find an integer $L$ such that
	\begin{align}
	\label{eq:tos2}
	A_n^{\ge L} = o([z^n] \cF^\upsilon_{< L} \circ \cG^\gamma (z))
	\end{align}
	as $n \equiv b \mod d$ becomes large.

	We choose an integer $t \ge 1$ with $t \equiv b \mod d$ such that there exists an element $s \in \Omega$ and integers $j_1, \ldots, j_{s-1} \ge 1$ with 
	\[
	j_1 + \ldots + j_{s-1} = t-1 \qquad \text{and} \qquad a_{j_1}, \ldots, a_{j_{s-1}} > 0.
	\]
	This is possible as we assumed that $\Omega \subset b + d \ndZ$ is infinite and $a_i>0$ for all $i \in b + d\ndZ$ with $i \ge I$. Furthermore, we fix an  integer $L$ that satisfies
	\begin{align}
	\label{eq:assumption}
	L \ge 1 + I + t \qquad \text{and} \qquad L \equiv 1 \mod d.
	\end{align}
	It follows that there is a constant $c>0$ such that
	\begin{align}
	[z^n] \cF^\upsilon_{<L} \circ \cG^\gamma \ge c a_{n - t +1}
	\end{align}
	for all large enough integers $n$ with $n \equiv b \mod d$. Thus, in order to verify Equation~\eqref{eq:tos2}, it suffices to show
	\begin{align}
	\label{eq:tos3}
	A_n^{\ge L} = o(a_{n-t+1}), \qquad n \to \infty, \qquad n \equiv b \mod d.
	\end{align}

	For $k=2$ Equation~\eqref{eq:starter} yields
	\begin{align}
	\label{eq:one}
	\sum_{i=2}^{m-2} a_i a_{m-i} = o(a_{m-1}), \qquad m \to \infty
	\end{align}
	with both sides of the equation being equal to zero unless $m \equiv 2 \mod d$.
	In particular, it holds that
	\begin{align}
	\label{eq:two}
	a_{m-i} = o(a_{m}), \qquad m \to \infty, \qquad m \equiv 1 \mod d
	\end{align}
	for all $i$ with $a_{i+1} >0$. If $i$ is not a multiple of $d$, then $a_{m-i} = 0$  for all $m$ satisfying $m \equiv 1 \mod d$, and Equation~\eqref{eq:two} holds trivially.
	
	We define
	\[
	\chi_n = \max \{ A_k^{\ge L} / a_{k-t+1} \mid k \in b + d \ndZ,  I-1 +t \le k \le n\}.
	\]
	It holds that
	\begin{align}
	\label{eq:pp1}
	A_n^{\ge L} 
	&= \sum_{L \le \ell \le L + L -2} C^\ell \sum_{i_1 + \ldots + i_\ell = n} a_{i_1} \cdots a_{i_\ell} + C^{L -1} \sum_{i_1 + \ldots + i_{L} = n}  a_{i_1} \cdots a_{i_{L-1}} A_{i_{L}}^{\ge L} \nonumber \\
	&\le O(1) \left(  \sum_{L \le \ell \le L + L -2}  \sum_{i_1 + \ldots + i_\ell = n} a_{i_1} \cdots a_{i_\ell} + \sum_{i_1 + \ldots + i_{L} = n}  a_{i_1} \cdots a_{i_{L-1}} A_{i_{L}}^{\ge L} \right).
	\end{align}
	It follows by Equations~\eqref{eq:assumption}, \eqref{eq:starter}, and \eqref{eq:two} that
	\begin{align}
	\label{eq:pp2}
	\sum_{L \le \ell \le 2L-2} \sum_{i_1 + \ldots + i_\ell = n} a_{i_1} \cdots a_{i_\ell} &  = \sum_{\substack{L \le \ell \le L + L -2 \nonumber \\ \ell \equiv 1 \mod d}} O(a_{n-(\ell -1)}) \\
	&= o(a_{n-t+1})
	\end{align}
	as $n \equiv b \mod d$ becomes large. As for the second sum in Equation~\eqref{eq:pp1}, note that each summand \[a_{i_1} \cdots a_{i_{L-1}} A_{i_{L}}^{\ge L}\] is equal to zero, unless 
	\begin{align*}
	i_1, \ldots, i_{L-1} \equiv 1 \mod d \qquad \text{and} \qquad i_{L} \ge L.
	\end{align*}
	Using $L \equiv 1 \mod d$, $ n \equiv b \mod d$, and $i_1 + \ldots + i_{L} = n$ it follows that we only need to consider summands where 
	\[
	L \le i_{L} \le n-(L-1) \qquad \text{and} \qquad i_{L} \equiv b \mod d.
	\]
	For such indices, it holds that
	\[
	A_{i_{L}}^{\ge L} \le a_{i_{L}-t+1} \chi_{i_{L}} \le a_{i_{L}-t+1} \chi_{n-(L-1)}.
	\]
	It follows by Equations~ \eqref{eq:starter} and \eqref{eq:two} that
	\begin{align*}
	\sum_{i_1 + \ldots + i_{L} = n} a_{i_1} \cdots a_{i_{L-1}} A_{i_{L}}^{\ge L} &\le \chi_{n-L+1} \sum_{\substack{i_1 + \ldots + i_{L} = n \\ i_L \ge L}} a_{i_1} \cdots a_{i_{L}-t+1} \\
	&= \chi_{n-L+1} O(a_{n-(L-1) - (t-1)}) \\
	&= \chi_{n-L+1} o(a_{n-t+1}).
	\end{align*}
	By Equations~\eqref{eq:pp1}, \eqref{eq:pp2} and \eqref{eq:assumption} it follows that
	\begin{align}
	\label{eq:pp3}
	A_n^{\ge L} \le o(a_{n-t+1}) + \chi_{n-L+1} o(a_{n-t+1}), \qquad n \to \infty, \qquad n \equiv b \mod d.
	\end{align}
	In particular, there is some $n_0 \ge 1$ such that $n_0 \equiv b \mod d$ and for all $n \ge n_0$ with $n \equiv b \mod d$ it holds that
	\[
	A_n^{\ge l} /a_{n-t+1} \le \frac{1}{2} + \frac{1}{2} \chi_{n-L+1}. 
	\]
	Hence
	\begin{align*}
	\chi_n &= \max(A_n^{\ge L} / a_{n-t+1}, \chi_{n-d}) \\
	&\le \max \left(\frac{1}{2} + \frac{1}{2} \chi_{n-(L-1)}, \chi_{n-d} \right) \\
	&\le \max \left(\frac{1}{2} + \frac{1}{2} \chi_{n-d}, \chi_{n-d} \right) \\
	& \le \max(1, \chi_{n-d}).
	\end{align*}
	Iterating this inequality yields
	\begin{align*}
	\chi_n \le \max(1, \chi_{n_0}).
	\end{align*}
	Hence $\chi_n = O(1)$ and it follows from Equation~\eqref{eq:pp3} that
	\[
	A_n^{\ge L} = o(a_{n-t+1}).
	\]
	This verifies Equation~\eqref{eq:tos3} and hence completes the proof.
\end{proof}

\begin{proof}[Proof of Corollary~\ref{co:ne}]
	If $[z^1]\cF^\upsilon(z)>0$ and $a_1, a_2>0$, then Equation~\eqref{eq:two} yields
	\begin{align}
	\label{eq:yup}
	a_{n-1} = o(a_n).
	\end{align}
	Theorem~\ref{te:need} yields that there is some integer $L \ge 2$ such that the Gibbs partition $\mS_n$ has with high probability less than $L$ components. That is,
	\[
	[z^n] \cF^\upsilon_{\ge L} \circ \cG^\gamma(z) = o( [z^n] \cF^\upsilon_{< L} \circ \cG^\gamma(z)).
	\]
	Equations~\eqref{eq:starter} and \eqref{eq:yup} readily yield that
	\[
	[z^n] \cF^\upsilon_{\{2, \ldots, L-1\}} \circ \cG^\gamma(z) = o(a_n).
	\]
	Hence
	\[
	[z^n] \cF^\upsilon_{\ge 2} \circ \cG^\gamma(z) = o(a_n).
	\]
	In other words, the mass of all composite structures with at least two components is negligible compared to the mass of composite structures with a single component. Thus the Gibbs partition $\mS_n$ consists with high probability of a single component.
\end{proof}

\subsection{Proofs for the component size asymptotics in Section~\ref{sec:compsize}}

\begin{proof}[Proof of Theorem~\ref{te:compsize}] 

	Let $Y_{(1)} \ge Y_{(2)}  \ge \ldots \ge Y_{(n)}$ denote the descendingly ordered list of the outdegrees of the vertices in $\cT_n$, and $v_1, \ldots, v_n$ the corresponding vertices. Here we fix the ordering between vertices having the same outdegree in any canonical way, for example according to the lexicographic ordering of their location in $\cT_n$. 
	
	If we replace the maximum component-sizes $B_{(j)}$ with the maximum outdegrees $Y_{(j)}$ in Theorem~\ref{te:compsize} or Proposition~\ref{pro:new}, then all bounds and limit theorems hold by \cite[Thm. 19.34, Thm. 19.3, Equation (19.20) and Ch. 9]{MR2908619}. (See also \cite{MR2764126}.) As it always holds that $B_{(1)} \le Y_{(1)}$, this already concludes the proof for the points {\em (1)} and {\em (2)} of Theorem~\ref{te:compsize}.
	 
	In the setting of {\em (3)}, we intend to show weak convergence of the extremal $\cG$-component sizes. The limit theorems for the $Y_{(j)}$ imply that there is a deterministic function $\ell(n) \to \infty$ such that for each fixed $j \ge 1$ the probability for the event $Y_{(j)} \ge \ell(n)$ tends to one as $n$ becomes large.
	Let $B_{(j)}^+$ denote the size of the largest $\cG$-structure of the $\cR$-object of the vertex $v_j$. As the composition $\cF^\nu \circ \cG^\gamma$ has convergent type by Lemma~\ref{le:gibbs}, it follows that there is a random non-negative integer $X$ such that for each fixed $j$
	\begin{align}
		\label{eq:yeah}
		\lim_{n \to \infty} d_{\textsc{TV}}(Y_{(j)} - B_{(j)}^+,X) =0.
	\end{align}
	As $Y_{(1)} \ge B_{(1)} \ge B_{(1)}^+$, this immediately yields $B_{(1)} = B_{(1)}^+$ with high probability and hence $B_{(1)}$ satisfies the same central limit theorem as $Y_{(1)}$. This concludes the proof of part {\em a)} of Theorem~\ref{te:compsize}. 
	
	It remains to  establish limit theorems for $B_{(j)}$ if $j \ge 2$ in the setting of {\em (3)}. Equation~\eqref{eq:yeah} implies that for, let us say $t_n = \log n$, it holds with probability tending to one that $B_{(i)}^+ \ge Y_{(i)} - t_n$ for all $i \le j$. As $Y_{(i)}$ has polynomial order for any fixed $i$, it follows that the $j$ largest $\cG$-components lie with high probability in $\cR$-structures of different vertices. If this event takes place, then it must also hold that $Y_{(i)} \ge B_{(i)}$ for all $i \le j$. This is due to the fact that if $B_{(\ell)} > Y_{(\ell)}$ for some $\ell$, then the $\cG$-objects corresponding to the $B_{(i)}$ for $i \le \ell$  would have to  belong to the $\cR$-structures of the vertices $v_i$ for $i < \ell$, and then the pigeon hole principle tells us that at least two such objects must belong to the same vertex. Hence it holds with high probability that $B_{(i)} \le Y_{(i)}$ for all $i \le j$.  For the lower bound, we observe that if $B_{(i)}^+ \ge Y_{(i)} - t_n$ for all $i \le j$, then 
	\[
	B_{(i)} \ge \min(B_{(1)}^+, \ldots, B_{(i)}^+)  \ge \min(Y_{(1)}, \ldots, Y_{(i)}) - t_n = Y_{(i)} - t_n.
	\]
	So $Y_{(i)} - t_n \le B_{(i)} \le Y_{(i)}$ and hence the limit theorems for the $Y_{(i)}$ also hold for the $B_{(i)}$. This concludes the proof.
	\end{proof}

The proof of Proposition~\ref{pro:new} is entirely analogous to the proof of Theorem~\ref{te:compsize}, only instead of using the asymptotics given in \cite{MR2908619}, we build on results by Kortchemski~\cite[Thm. 1]{MR3335012} for the asymptotic behaviour of the largest and second largest degree in non-generic Galton--Watson trees.

\subsection{Proofs of the applications to outerplanar maps in Section~\ref{sec:outerplanar}}

Theorem~\ref{te:convouter} follows by a direct application of Theorems~\ref{te:local} and \ref{te:strengthend}, that guarantee convergence of the pruned enriched tree $(\cT_n, \beta_n)^{[k]}$ in the type I regime.

\begin{proof}[Proof of Remark~\ref{re:mapdistr}]
	We have to interpret the limit enriched tree $(\hat{\cT}, \hat{\beta})$ as a graph. Note that the fringe subtree at its second spine vertex follows the same distribution as the whole tree. So we are going to interpret $(\hat{\cT}, \hat{\beta})$ without this fringe subtree as a graph, and then use this recursion. 
	
	The root $o$ of $\hat{\cT}$ receives an $\cR$-object $\hat{\beta}(o)$ of which a uniformly at random drawn atom forms the second spine vertex. Taking a $\hat{\xi}$-sized $\cR^\kappa = \Seq \circ \cD^\gamma$ object with probability proportional to its $\kappa$-weight and marking a uniformly at random chosen non-$*$-vertex, is equivalent to taking a $\mathbb{P}_{(\cR^\bullet)^\kappa, \tau}$ object. So, $\hat{\beta}(o)$ with the marked atom given by the second spine-vertex follows a $\mathbb{P}_{(\cR^\bullet)^\kappa, \tau}$-distribution.
	
	The rules for the operations on species in Section~\ref{sec:opspe}  show in a purely algebraic fashion, that
	\[
	(\cR^\bullet)^\kappa \simeq (\Seq' \circ \cD^\gamma) \cdot (\cD^\bullet)^\gamma.
	\]
	Any $\Seq'$-object can be decomposed in a unique way into the two linear orders before and after the $*$-atom, yielding $\Seq' \simeq \Seq \cdot \Seq$. Consequently, 
	\[
	(\cR^\bullet)^\kappa \simeq (\Seq \circ \cD^\gamma) \cdot (\cD^\bullet)^\gamma \cdot (\Seq \circ \cD^\gamma).
	\]
	The interpretation is that, in counter-clockwise order, we first glue a sequence of dissections together at their $*$-vertices, then comes the marked dissection, and afterwards again a sequence of unmarked dissections. 
	
	Now, the rules in Lemma~\ref{le:bole} governing the relation between weighted Boltzmann distributions and operations on species yield that the sequences of dissections before the $\cD^\bullet$-object, the $\cD^\bullet$-object itself, and the sequence of dissections after the $\cD^\bullet$-object are independent, and follow Boltzmann distributions  $\mathbb{P}_{(\cD^\bullet)^\gamma, \tau}$ and $\mathbb{P}_{\Seq \circ \cD^\gamma, \tau}$. In the tree $(\hat{\cT}, \hat{\beta})$, each of the non-marked atoms of the $\cR$-object $\beta(o)$ becomes the root of an independent copy of $(\cT, \beta)$. As graph, $(\cT, \beta)$ follows a Boltzmann distribution $\mathbb{P}_{\cO^\omega, \tau/\phi(\tau)}$.
	
	Now, the isomorphism
	\[
	\cO^\omega \simeq \cX \cdot (\Seq \circ \cD^\gamma)(\cO^\omega)
	\]
	and the rules in Lemma~\ref{le:bole} state, that if we glue the $*$-vertices of a $\mathbb{P}_{\Seq \circ \cD^\gamma, \tau}$-distributed sequence of dissections together, and identify each non-$*$-vertex with the root of a fresh independent copy of a $\mathbb{P}_{\cO^\omega, \tau/\phi(\tau)}$-distributed outerplanar map, then the result again follows a $\mathbb{P}_{\cO^\omega, \tau/\phi(\tau)}$ Boltzmann distribution.
	
	Summing up, the graph corresponding to $(\hat{\cT}, \hat{\beta})$, without the fringe-subtree at the second spine vertex, corresponds to a pointed  $\mathbb{P}_{(\cD^\bullet)^\gamma, \tau}$-distributed dissection, with two independent copies of a $\mathbb{P}_{\cO^\omega, \tau/\phi(\tau)}$-distributed outerplanar map attached to its root, and one fresh independent copy attached to every other vertex, except for the marked vertex.
	
	As the fringe subtree at the second spine vertex is distributed like $(\hat{\cT}, \hat{\beta})$ itself, it follows that the graph $\hat{\mO}$ corresponding to $(\hat{\cT}, \hat{\beta})$ is distributed like an infinite chain of such triples, where the marked vertex of any triple is identified with the $*$-vertex of the subsequent one. So $\hat{\mO}$ has the distribution as described in the remark.
\end{proof}

\begin{proof}[Proof of Remark~\ref{re:mapdistr2}]
	We have to show that the random graph corresponding to the enriched tree $(\cT^*, \beta^*)$ follows the described distribution. Here we may build on the intermediate results in the proof of Remark~\ref{re:mapdistr}. The fringe-subtree of $(\cT^*, \beta^*)$ is distributed like $(\cT, \beta)$, and hence corresponds to a $\mathbb{P}_{\cO^\omega, \tau/\phi(\tau)}$-distributed random outerplanar map. If $u_0, u_1, \ldots$ denotes the spine of $(\cT^*, \beta^*)$, then for all $i \ge 1$ the enriched fringe subtree $f( (\cT^*, \beta^*), u_i)$ without the enriched fringe subtree $f( (\cT^*, \beta^*), u_{i-1})$ is distributed like $(\hat{\cT}, \hat{\beta})$ without the fringe subtree at its second spine-vertex. The graph corresponding to this object has been identified in the proof of Remark~\ref{re:mapdistr} as a  $\mathbb{P}_{(\cD^\bullet)^\gamma, \tau}$-distributed dissection, with two independent copies of a $\mathbb{P}_{\cO^\omega, \tau/\phi(\tau)}$-distributed outerplanar map attached to its root, and one fresh independent copy attached to every other vertex, except for the marked vertex. Hence $\hat{\mO}_*$ looks like an infinite chain of these objects, with a single additional independent $\mathbb{P}_{\cO^\omega, \tau/\phi(\tau)}$-distributed map attached to the marked vertex of the first element in the spine.
\end{proof}


\begin{proof}[Proof of Theorem~\ref{te:convout}]
	In Lemma~\ref{le:coupling}, we constructed the enriched plane tree $(\cT_n, \beta_n)$ by first generating the random tree $\cT_n$, and then sampling for each vertex $v \in \cT_n$ an $\cR$-structure $\beta_n(v) \in \cR[d^+_{\cT_n}(v)]$  with probability proportional to its $\kappa$-weight. The labels of $\beta_n(v)$ correspond in a canonical way to the ordered set of offspring of the vertex $v$, which is why the tree $(\cT_n, \beta_n)$ may be interpreted as an enriched tree. The final $\cR$-enriched tree $\mA_n^\cR = (\mA_n, \alpha_n)$ is then obtained by relabelling through a uniformly at random drawn bijection.
	
	The precise way for identifying the offspring of an vertex $v$ with the atoms of the $\cR$-structure $\beta_n(v)$ does not affect the {\em distribution} of the resulting random unordered enriched tree $\mA_n^\cR$. We may match the offspring of $v$ and the atoms of $\beta_n(v)$ according to any rule, which only takes $\beta_n(v)$ into account. Different matchings may very well change the isomorphism type of the corresponding $\cR$-enriched tree, but its distribution does not change. This may verified by using the fact that $(\cT_n, \beta_n)$ is distributed like $(\cT,\beta)$ conditioned on having $n$ vertices, and that $(\cT, \beta)$ has this invariance property, as any offspring of a vertex $v\in \cT$ becomes the root of an independent copy of $(\cT, \beta)$, regardless to which atom of $\beta(v)$ it gets matched.
	
	In the case of random outerplanar maps, the species $\cR$ is given by $\cR^\kappa = \Seq \circ (\cD)^\gamma$. We may consider the random enriched plane tree $(\cT_n, \lambda_n)$ constructed from $(\cT_n, \beta_n)$ by matching for each vertex $v \in \cT_n$ the offspring of $v$ with the atoms of $\beta_n(v)$ by ordering them in the following way. We start by putting the neighbours of the $*$-vertices in the derived dissections in any canonical order, and then proceed with the vertices at distance $2$ from their respective $*$-vertices, and so  on.  We may construct an outerplanar map  out of $(\cT_n, \lambda_n)$ according to the bijection in Section~\ref{sec:decompouter}, which is distributed like the random outerplanar map $\mO_n^\omega$. Thus, in the following we assume that $\mO_n^\omega$ corresponds directly to $(\cT_n, \lambda_n)$.
	
	For any integer $m \ge 0$, consider the subset $V^{[m]} \subset \VHT$ of the vertices of the Ulam--Harris tree given by 
	\[
	V^{[m]} = \{ (i_1, \ldots, i_t) \mid t \le m, i_1, \ldots, i_\ell \le m\}.
	\]
	That is, we consider the first $m$ sons of the root, and for each of those again the first $m$ sons, and so on, until we reach generation $m$.
	The reason why we consider $(\cT_n, \lambda_n)$ is that for any integer $\ell \ge 0$ and any finite outerplanar map $O$ there exists a constant $m(\ell, O) \ge 0$, such that the event, that the $\ell$-neighbourhood $V_\ell(\mO_n^\omega)$ is identical to $O$ as (half-edge-rooted) planar maps, is already completely determined by the family $(\lambda_n(v))_{v \in V^{[m(\ell, O)]}}$. There are two reasons for this. First, vertices with distance at most $\ell$ from the origin of the root-edge in $\mO_n^\omega$ also have block-distance  at most $\ell$ from the origin, and hence height at most $\ell$ in $\cT_n$. Second, by the construction of $\lambda_n$, for any vertex $v$ the subset of its ordered sequence of sons, that still lies in the $\ell$-neighbourhood of $v_n$, is an initial segment in the ordered list. If $V_\ell(\mO_n^\omega) \simeq O$ as (half-edge-rooted) planar maps, then the length of this initial segment must be bounded by the number of vertices of $O$. Hence, if we take $m(\ell, O)$ large enough depending on $\ell$ and the size of the map $O$, then $(\lambda_n(v))_{v \in V^{[m(\ell, O)]}}$ contains all information necessary to decide whether $V_\ell(\mO_n^\omega) \simeq O$ as (half-edge-rooted) planar maps. 
	
	Janson~\cite[Sec. 20]{MR2908619} constructs a deterministic sequence $\Omega_n \to \infty$ and a modified Galton--Watson tree $\cT_{1n}$ that is obtained from $\hat{\cT}$ by sampling a random degree $\tilde{D}_n \ge \Omega_n$ independently from $\hat{\cT}$, and pruning $\hat{\cT}$ at its unique vertex $v^*$ with infinite degree, keeping only the first $\tilde{D}_n$ children of $v^*$. The distribution of $\tilde{D}_n$ is given in Equation (20.4) of Janson's survey \cite{MR2908619}, and the construction of $\Omega_n$ in \cite[Lem. 19.32]{MR2908619}. We will not require detailed knowledge of these, but will make use of his result \cite[Thm. 20.2]{MR2908619}, which states that for any fixed integer $m \ge 0$ it holds that
	\begin{align}
	\label{equ:conv1}
	\lim_{n \to \infty} d_{\textsc{TV}}( (d^+_{\cT_n}(v))_{v \in V^{[m]}}, (d^+_{\cT_{1n}}(v))_{v \in V^{[m]}}) = 0.
	\end{align}
	In other words, the tip of the spine in $\hat{\cT}$ corresponds to a vertex with large degree in $\cT_n$. 
	
	The tree $\cT_{1n}$ is almost surely finite, and we may turn it into an enriched plane tree $(\cT_{1n}, \beta_{1n})$, by sampling for each vertex $v$ an element $\beta_{1n}(v)$ from $\cR[d^+_{\cT_{1n}}(v)]$ with probability proportional to its weight. Again we may match the non-$*$-vertices of the set of derived blocks $\beta_{1n}(v)$ with the ordered offspring of $v$ according to their distance from their respective $*$-vertices in order to obtain an enriched tree $(\cT_{1n}, \lambda_{1n})$, in precisely the same way as we constructed $(\cT_n, \lambda_n)$ out of $(\cT_n, \beta_n)$. For any finite family of vertices $v_i$, $i \in I$, integers $d_i \ge 0$ and $\cR$-structures $R_i \in \cR[d_i]$, it holds that
	\begin{align*}
	\Pr{ \beta_n(v_i) = R_i, i \in I \mid d_{\cT_n}^+(v_i) = d_i, i \in I} &= \prod_{i \in I} \kappa(R_i) / |\cR[d_i]|_\kappa  \\
	&= \Pr{ \beta_{1n}(v_i) = R_i, i \in I \mid d_{\hat{\cT}_{1n}}^+(v_i) = d_i, i \in I}.
	\end{align*}
	So, Equation~\eqref{equ:conv1} already implies
	\begin{align}
	\label{equ:conv22}
	\lim_{n \to \infty} d_{\textsc{TV}}( (\lambda_n(v))_{v \in V^{[m]}}, (\lambda_{1n}(v))_{v \in V^{[m]}}) = 0.
	\end{align}
	Let $\mO_{1n}^\omega$ denote the outerplanar map corresponding to $(\cT_{1n}, \lambda_{1n})$ according to the bijection in Section~\ref{sec:decompouter}. Setting $m = m(\ell, O)$, it follows that
	\begin{align}
	\label{equ:reduct}
	\lim_{n \to \infty} |\Pr{V_\ell(\mO_n^\omega) \simeq O} - \Pr{V_\ell(\mO_{1n}^\omega) \simeq O} | = 0.
	\end{align}
	
	Thus, it remains to determine the limit probability for the event that $V_\ell(\mO^\omega_{1n})$ is equal to $O$ as half-edge-rooted planar map.  For any $k\ge 0$ with $|\cR[k]|_\kappa>0$ let $\mR_k$ denote a random $\cR = \Seq \circ \cD^\gamma$ structure sampled from $\cR[k]$ with probability proportional to its $\kappa$-weight.  Let $\hat{\mR}$ and $\hat{\mR}'$  denote two independent  Boltzmann distributed $\cR$-objects with parameter $\tau$. 
	
	If $|\hat{\mR}| + |\hat{\mR}'| \le k$ and $|\cD[k - |\hat{\mR}|]|_\gamma >0$, then we let $\hat{\mR}_k$ denote the outerplanar map obtained by identifying the root-vertices of $\hat{\mR}$, an independent random dissection $\hat{\mD}_k$ from $\cD[k - |\hat{\mR}| - |\hat{\mR}'|]$ sampled with probability proportional to its $\gamma$-weight, and $\hat{\mR}'$, such that in the resulting outerplanar map is rooted at the oriented root-edge of $\hat{\mR}$ and such that in the counter-clock-wise ordered list of edges incident to the origin of the root-edge, we encounter the roots of $\hat{\mR}$, $\hat{\mD}_k$ and $\hat{\mR}'$ in this order. If $|\hat{\mR}| + |\hat{\mR}'| > k$ or $|\cD[k - |\hat{\mR}|]|_\gamma =0$, then we set $\hat{\mR}_k = \diamond$ for some placeholder symbol $\diamond$.

	As $\hat{\mR}$ and $\hat{\mR}'$ are  almost surely finite, it follows that $\hat{\mR}_k \ne \diamond$ with high probability. By assumption, the composition $\Seq \circ \cD^\gamma$ has convergent type with parameter $\tau$, so $\mR_k$ behaves asymptotically like a Boltzmann-distributed $\Seq' \circ \cD^\gamma$ object with parameter $\tau$, plus a large dissection. It holds that $\Seq' \simeq \Seq \cdot \Seq$, since a derived sequence consists of an ordered initial segment, then the derived atom, and the ordered final segment. This means that as (unlabelled) half-edge-rooted planar maps, it holds that
	\begin{align}
	\label{equ:needaname}
	\lim_{k \to \infty} d_{\textsc{TV}}(\mR_k,\hat{\mR}_k) = 0.
	\end{align}
	If we let $u^*$ denote the vertex at the tip of the spine in $\cT_{1n}$, then the $\Seq \circ \cD^\gamma$ structure $\beta_{1n}(u^*)$ is distributed like $\mR_{D_n}$ with $D_n$ being a copy of $\tilde{D}_n$ that is independent of all other random variables considered so far. As $D_n \ge \Omega_n$ and $\Omega_n \to \infty$, Equation~\eqref{equ:needaname} implies that, up to labelling, 
	\begin{align}
	\label{equ:needaname2}
	\lim_{n \to \infty} d_{\textsc{TV}}(\lambda_{1n}(u^*),\hat{\mR}_{D_n}) = 0.
	\end{align}
	
	For any ordered list $R$ of $\cD$-objects let $O(R)$ denote  the outerplanar map  obtained by identifying the $*$-vertices of all dissections with each other, such that the result is rooted at the root-edge of the first dissection, and the remaining dissections are ordered counter-clockwise around the origin of the root-edge according to their order in $R$.
	
	Then, for any $r \ge 0$ the $r$-neighbourhood $V_r(O(R))$ is given by the union of the $r$-neighbourhoods of the $*$-vertices in the components. This may be expressed by
	\[
	V_r(O(R)) = O( (V_r(Q))_{Q \in R}).
	\]
	The sizes of $\hat{\mR}$ and $\hat{\mR}'$ are almost surely finite, hence with probability tending to one as $k$ becomes large it holds that $|\hat{\mD}_k| \ge k - \log k$. We assumed that random $n$-sized blocks sampled with probability proportional to its $\gamma$-weight converge in the Benjamini--Schramm sense toward a limit graph $\hat{\mD}$. It follows that $\hat{\mD}_k$ also converges in the local weak sense toward $\hat{\mD}$. 	Let $\hat{\mD}^\circ$ denote the result of declaring the origin of the root-edge of $\hat{\mD}$ to be a $*$-placeholder vertex. It holds that
	\[
	V_r(O(\hat{\mR}_k)) = O(V_r(\hat{\mR}),  V_r(\hat{\mD}_k), V_r(\hat{\mR}')).
	\]
	Hence, as $k$ becomes large,
	\begin{align}
	\label{equ:jjnn}
	V_r(O(\hat{\mR}_k)) \convdis O(V_r(\hat{\mR}),  V_r(\hat{\mD}^\circ), V_r(\hat{\mR}')) = V_r( O(\hat{\mR}, \hat{\mD}^\circ, \hat{\mR}')).
	\end{align}
	Since $D_n \ge \Omega_n$ and $\Omega_n \to \infty$, it follows that $O(\lambda_{1n}(u^*))$ converges in the local weak sense toward $O( \hat{\mR}, \hat{\mD}^\circ, \hat{\mR}')$. In order to decide whether $V_\ell(\mO_{1n}^\omega) \simeq O$ as half-edge-rooted planar maps, it is more than enough to know the $\ell$-neighbourhoods $V_\ell(O(\beta_{1n}(v)))$ for all $v \in V^{[m]}$. (It would also suffice to just consider the $(\ell-h_{\cT_{1n}}(v))$-neighbourhoods of the vertices $v$). The limit~\eqref{equ:jjnn} implies that
	\begin{align}
	\label{equ:jjmm}
	(V_\ell(O(\lambda_{1n}(v))))_{v \in V^{[m]}} \convdis (V_\ell(O(\hat{\lambda}(v))))_{v \in V^{[m]}},
	\end{align}
	where we let $(\hat{\cT}, \hat{\lambda})$ denote the limit enriched plane tree obtained from $(\hat{\cT}, \hat{\beta})$ by matching the offspring of any vertex $v$ with finite outdegree $d_{\hat{\cT}}^+(v) < \infty$ with the atoms of the set of derived blocks $\hat{\beta}(v)$ in the same way as we did for $(\cT_n, \lambda_n)$. For the unique vertex $v^*$ with $d_{\hat{\cT}}^+(v^*) = \infty$, we let $\lambda(v^*)$ be given by $(\hat{\mR}, \hat{\mD}^\circ, \hat{\mR}')$, where we also  match the countably infinite offspring of $\lambda(v^*)$ with the countably infinite number of non-$*$-vertices of $(\hat{\mR},\hat{\mD}^\circ, \hat{\mR}')$ in the same way. (It is easily verified that the random map $\hat{\mD}^\circ$ has countably infinite many vertices. For an upper bound, we only need the fact that it is locally finite, and the lower bound follows as it is the limit of a sequence of random graphs whose size deterministically tends to infinity.)
	
	The convergence in \eqref{equ:jjmm} implies that the random rooted graph $\hat{\mO}$ that corresponds to the $\cR$-enriched plane tree $(\hat{\cT}, \hat{\lambda})$ satisfies
	\[
	\lim_{n \to \infty} \Pr{V_\ell(\mO_n^\omega) \simeq O} = \Pr{V_\ell(\hat{\mO}) \simeq O}.
	\]
	As $O$ and $\ell$ were arbitrary, it follows that $\hat{\mO}$ is the local weak limit of the random map $\mO_n^\omega$.
	
	It remains to argue, that $\hat{\mO}$ is distributed as claimed in Theorem~\ref{te:convout}. Recall that we constructed $\hat{\mO}$ by concatenating the independent identically distributed dissections $(\mD_i^\bullet)_{1 \le i \le L}$, where $L$ follows a geometric distribution with parameter $\nu$, glue the limit $\hat{\mD}$ at the tip of this chain, and finally identify each vertex of this graph with the root of one or two independent copies of the Boltzmann-distributed random outerplanar map $\mO$.
	
	The height of the vertex $v^*$ in $\hat{\cT}$ is distributed like $L$, and the $\cR$-structures along the spine in $\hat{\cT}$ actually follow Boltzmann distributions of $(\Seq \circ \cD^\gamma)^\bullet$ with parameter $\tau$. As $\Seq' \simeq \Seq \cdot \Seq$ it follows that
	\[
	(\Set \circ \cD^\gamma)^\bullet \simeq (\Seq \circ \cD^\gamma) \cdot \cD^\bullet \cdot (\Seq \circ \cD^\gamma).
	\]
	The product rule in Section~\ref{sec:WeBoSa} implies that each of the blocks containing consecutive spine vertices actually follows a Boltzmann distributions for $(\cD^\bullet)^\gamma$ with parameter $\tau$, and the remainder of the corresponding $(\Seq \circ \cD^\gamma)^\bullet$-object is independent from this block is composed of two independent sequences of blocks following $\Seq \circ \cD^\gamma$-Boltzmann distributions with parameter $\tau$. The isomorphism $\cA_\cR^\omega \simeq \cX \cdot \cR^\kappa(\cA_\cR^\omega)$ and the composition rule in Section~\ref{sec:WeBoSa} imply that if we take a Boltzmann distributed $\Seq \circ \cD^\gamma$-structure with parameter $\tau$, glue the $*$-vertices together, and identify its vertices with the roots of independent copies of $\mO$, then the result follows a Boltzmann distribution for $\cO^\omega$ with parameter $\tau /\phi(\tau)$. So, summing up, the random graph $\hat{\mO}$ corresponding to $(\hat{\cT}, \hat{\beta})$ is distributed as described in Theorem~\ref{te:convout}.
\end{proof}

\begin{proof}[Proof of Lemma~\ref{le:light}]
	We start with Claim {\em (1)}. As $\phi(z) = 1/(1 - \cD^\gamma(z))$, it holds that
	\begin{align}
	\label{eq:need4speed}
	\psi(z) = z\phi'(z)/\phi(z) = z (\cD^\gamma)'(z) / (1-\cD^\gamma(z)).
	\end{align}
	We have to show that $\psi(x)$ tends to infinity, as $x$ tends from below to the radius of convergence $\rho_\phi$ of the series $\phi(z)$. Equation~\eqref{eq:need4speed} implies that
	\[
	\psi(z) = z (\cD^\gamma)'(z) \sum_{k \ge 0} (\cD^\gamma(z))^k
	\]
	is a power series with non-negative coefficients. Consequently, it suffices to show that the sum $\psi(\rho_\phi)$ is infinite.

	Set $\psi_\cD(z) = z \phi_\cD'(z)/\phi(z)$. We assumed that $\cD^\gamma$ has type I, hence there is a constant $\tau_1$ that is bounded by the radius of convergence $\rho_{\phi_\cD}$ of $\phi_\cD(z)$ and satisfies $\psi_\cD(\tau_1) = 1$. Let $\rho_{\cD}>0$ denote the radius of convergence of $\cD^\gamma(z)$. By Equation~\eqref{eq:z1} and \eqref{eq:z2} it holds that $\rho_\cD = \tau_1 / \phi_\cD(\tau_1)$ and $\cD^\gamma(\rho_\cD) = \tau_1$.  We treat the two cases $\tau_1 \le 1$ and $\tau_1 > 1$ separately.

	If $\tau_1\le 1$, then for all $0 \le x < \rho_\cD$ it holds that $\cD^\gamma(x) < \cD^\gamma(\rho_\cD)\le1$ and consequently also $\phi(x)<\infty$. Hence $\rho_\phi \ge \rho_\cD$. Equation~\eqref{eq:needmorenames} yields
	\[
	(\cD^\gamma)'(\rho_\cD) = \phi_\cD(\cD \gamma(\rho_\cD)) / (1 - \rho_\cD \phi_\cD'(\cD^\gamma(\rho_\cD))) = \phi_\cD(\tau_1) / (1 - \psi_\cD(\tau_1)) = \infty,
	\]
	since $\psi_\cD(\tau_1)=1$.
	The series $\psi(x)$ is non-decreasing in the interval $[0, \rho_\phi]$, hence Equation~\eqref{eq:need4speed} implies that $\psi(\rho_\phi) \ge \psi(\rho_\cD) = \infty$. Thus $\mO_n^\omega$ has type I$a$ with $\nu=\infty$.
	
	If $\tau_1 > 1$, then $\rho_\phi$ is the unique number with $\cD^\gamma(\rho_\phi)=1$. Equation~\eqref{eq:need4speed} immediately implies that
	\[
	\psi(\rho_\phi) = \rho_\phi (\cD^\gamma)'(\rho_\phi)/(1- \cD^\gamma(\rho_\phi)) = \infty.
	\]
	Hence $\nu=\infty$ and Claim {\em (1)} follows.
	
	As for Claim {\em (2)}, $\cD^\gamma$ having type II means that $0< \nu_\cD < 1$ and $\tau_\cD = \tau_1 = \rho_{\phi_\cD}>0$.  If $\tau_\cD < 1$, then $\cD^\gamma(\rho_\cD) = \tau_\cD$ implies that $\phi(z) = 1/(1 - \cD^\gamma(z))$ has radius of convergence $\rho_\phi = \rho_\cD$.
	It follows that 
	\[
	\nu = \psi(\rho_\phi) = \frac{\rho_\cD (\cD^\gamma)'(\rho_\cD)}{1 - \cD^\gamma(\rho_\cD)} = \frac{ \rho_\cD (\cD^\gamma)'(\rho_\cD) }{1 - \tau_\cD}.
	\]
	Since $\cD^\gamma(z) = z\phi_\cD(\cD^\gamma(z))$ it also holds that
	\[
	(\cD^\gamma)'(\rho_\cD) = \frac{\phi_\cD(\cD^\gamma(\rho_\cD))}{1- \rho_\cD \phi_\cD'(\cD^\gamma(\rho_\cD))} = \frac{\phi_\cD(\tau_\cD)}{1 - \psi_\cD(\tau_\cD)} = \frac{\tau_\cD}{\rho_\cD(1 - \nu_\cD)}.
	\]
	Hence
	\[
	\nu = \frac{\tau_\cD}{(1- \tau_\cD)(1 - \nu_\cD)} \in ]0, \infty[.
	\]
	
	If $\tau_\cD\ge 1$ then Equation~\eqref{eq:z2} implies that there is a number $t_0$ with $\cD^\gamma(t_0)=1$. Hence $\phi(z)$ has radius of convergence $\rho_\phi = t_0$, and
	\[
	\psi(\rho_\phi) = t_0 (\cD^\gamma)'(t_0)/(1- \cD^\gamma(t_0)) = \infty.
	\]
	This shows that in this case $\mO_n^\omega$ has type I$a$ with $\nu= \infty$.
	
	As for Claim {\em (3)}, it is clear by the discussion in Section~\ref{sec:types} that $\cD^\gamma$ has type III if and only if $\cD^\gamma(z)$ is not analytic (at the origin). In that case $\phi(z)$ is also not analytic and hence $\mO_n^\omega$ has type III.
\end{proof}

\begin{proof}[Proof of Theorem~\ref{te:facethm}]
	If $\mathbf{w}$ has type I, then we do not have to show anything at all.
	
	If $\mathbf{w}$ has type II, then the random dissection $\mD_n^\gamma$ also has type II by Lemma~\ref{le:light}, and converges by Theorem~\ref{te:bsconvdis} toward a limit graph $\hat{\mD}$. By Lemma~\ref{le:subexp} we know that the series $\cD^\gamma(z)/z$ belongs to the class $\mathcal{S}_d$ of subexponential sequences with span $d$ for some $d \ge 1$. {\em If} $d=1$, then it follows that $\cD^\gamma(z)$ belongs to $\mathcal{S}_1$ and  $\Seq \circ \cD^\gamma$ has convergent type by Lemma~\ref{le:gibbs}. So in this case, Theorem~\ref{te:convout} may be applied and readily yields local weak convergence of $\mO_n^\omega$. For $d \ge 2$, the situation is a bit more complicated, because then the composition $\Seq \circ \cD^\gamma$ does not have convergent type. Instead, by  Lemma~\ref{le:gibbs} we know that the limit behaviour  of the small fragments not contained in the giant component of a random element $\mR_k$ from $\Seq \circ \cD^\gamma[k]$ (drawn with probability proportional to its weight) depends along which of the lattices $a + d \ndN$, $0 \le a < d$  we let $k$ tend to infinity. But this is not really a problem. Although we cannot apply Theorem~\ref{te:convout} directly, its proof needs only a small modification to be adapted to this situation:
	
	Lemma~\ref{le:gibbs} states that if $\Seq_a$ denotes the restriction of $\Seq$ to sequences with length in $a + \ndZ$, and if $k$ satisfies $k \equiv a \mod d$, then $\mR_k$ takes values only in $(\Seq_a \circ \cD^\gamma)[k]$. If $\mR_k^s$ denotes the $\Seq' \circ \cD^\gamma$-object obtained by deleting the (or any single) largest component from $\mR_k$, then $\mR_k^s$ converges in total variation toward a $\mathbb{P}_{\Seq'_a \circ \cD^\gamma, \tau}$-distributed limit object $\hat{\mR}^a$. Since \[\Seq \simeq \Seq_0 + \ldots + \Seq_{d-1},\] it follows that \[ \Seq' \circ \cD^\gamma \simeq \Seq'_0 \circ \cD^\gamma + \ldots + \Seq'_{d-1} \circ \cD^\gamma.\] Hence the rules for weighted Boltzmann distributions given in Lemma~\ref{le:bole} imply that the $\mathbb{P}_{\Seq_a \circ \cD^\gamma, \tau}$-distributed Boltzmann object $\hat{\mR}$ is distributed like $\hat{\mR}^{a_0}$ for an independent random number $0 \le a_0 < d$ with distribution
	\[
		\Pr{a_0 = a} = (\Seq'_a \circ \cD^\gamma)(\tau) / (\Seq' \circ \cD^\gamma)(\tau).
	\]
	
	The idea of the proof of Theorem~\ref{te:convout} was that instead of showing convergence of the graph corresponding to the enriched tree $(\cT_n, \beta_n)$, it suffices to show convergence of the graph corresponding to a  modified enriched tree $(\cT_{1n}, \beta_{1n})$. This tree was obtained from $(\hat{\cT}, \hat{\beta})$ by taking a certain deterministic sequence $\Omega_n$, sampling a certain random degree $\tilde{D}_n \ge \Omega_n$ independently from $\hat{\cT}$, and pruning $\hat{\cT}$ at its unique vertex $v^*$ with infinite degree, keeping only the first $\tilde{D}_n$ children of $v^*$. The $\cR^\kappa = \Seq \circ \cD^\gamma$ object $\beta_{1n}(v^*)$ corresponding to $v^*$ is distributed like $\mR_{\tilde{D}_n}$. In the proof, the assumption that $ \lim_{k \to \infty} d_{\textsc{TV}}(\mR_k^s, \hat{\mR}) = 0$ (as unlabelled $\Seq' \circ \cD^\gamma$-objects) was combined with $\tilde{D}_n \ge \Omega_n \to \infty$ to deduce
	\begin{align}
		\label{eq:usain}
		\lim_{n \to \infty} d_{\textsc{TV}}( \mR_{\tilde{D}_n}^s, \hat{\mR}) = 0.
	\end{align}
	This is the only place where the assumption, that $\Seq \circ \cD^\gamma$ has convergent type, was used. So the only modification we need to make is how to deduce Equation~\eqref{eq:usain}.

	The construction of $\Omega_n$ was done in  \cite[Lem. 19.32]{MR2908619} such that if $N_k$ denotes the number of vertices with degree $k$ in $\cT_n$, then 
	\begin{align}
		\label{eq:bolt1}
		\sum_{k \le \Omega_n} k N_k = \nu n + o_p(n) \qquad \text{and} \qquad \sum_{k > \Omega_n} k N_k = (1 - \nu)n + o_p(n).
	\end{align}
	If $o$ denotes the root vertex of $\cT_n$, then
	\begin{align}
		\label{eq:bolt0}
		\Pr{d_{\cT_n}^+(o) = k} = \frac{n}{n-1} \Ex{ \frac{k N_k}{ n}},
	\end{align}
	by \cite[Lem. 15.7]{MR2908619}, so Equations~\eqref{eq:bolt1} may be reformulated as
	\begin{align}
		\label{eq:bolt2}
		\Pr{d_{\cT_n}^+(o) \le \Omega_n} = \nu + o(1) \qquad \text{and} \qquad \Pr{d_{\cT_n}^+(o) > \Omega_n} = 1 - \nu + o(1).
	\end{align}
	The distribution of $\tilde{D}_n$ is given in \cite[Equation (20.4)]{MR2908619} by $\Pr{\tilde{D}_n = k} = 0$ for $k \le \Omega_n$, and 
	\begin{align}
		\Pr{\tilde{D}_n = k} = \frac{ k\Ex{N_k}}{\sum_{\ell > \Omega_n} \ell \Ex{ N_\ell}} 
	\end{align}
	for all $k > \Omega_n$. By Equations~\eqref{eq:bolt0} and \eqref{eq:bolt2} this may be expressed by
	\begin{align}
		\label{eq:boltfin}
		\Pr{\tilde{D}_n = k} = \Pr{ d_{\cT_n}^+(o) = k \mid d_{\cT_n}^+(o) > \Omega_n}.
	\end{align}
	
	For any set $\cE$ of $\Seq' \circ \cD^\gamma$-objects we have
	\begin{align*}
		\Pr{ \mR^s_{\tilde{D}_n} \in \cE} &= \sum_{a=0}^{d-1}  \Pr{\tilde{D}_n \equiv a \mod d} \Pr{ \mR^s_{\tilde{D}_n} \in \cE \mid \tilde{D}_n \equiv a \mod d}.
	\end{align*} 
	It follows from Lemma~\ref{le:gibbs} and $\tilde{D}_n \ge \Omega_n \to \infty$ that uniformly for all $\cE$ 
	\[
		\lim_{n \to \infty} \Pr{ \mR^s_{\tilde{D}_n} \in \cE \mid \tilde{D}_n \equiv a \mod d} = \Pr{\hat{\mR}^a \in \cE}.
	\]
	So in order to verify \eqref{eq:usain}, it remains to check that
	\begin{align}
		\label{eq:yomei}
		\lim_{n \to \infty}  \Pr{\tilde{D}_n \equiv a \mod d} = \Pr{a_0 = a}.
	\end{align}
	The offspring distribution $\xi$ of the Galton--Watson tree $\cT$ is distributed like the size of a $\mathbb{P}_{\Seq \circ \cD^\gamma, \tau}$-distributed compound structure. That is,
	\[
		\Ex{z^\xi} = \Seq(\cD^\gamma(\tau z)) / \Seq(\cD^\gamma(\tau)).
	\]
	Consequently, the size-biased version $\hat{\xi}$ with $
	\Pr{\hat{\xi}=\infty} = 1 - \nu 
	$ and $\Pr{\hat{\xi}=k} = k \Pr{\xi=k}$ satisfies
	\[
		\Pr{\hat{\xi}=k} = [z^k] (z \frac{\text{d}}{\text{d}z} \Ex{z^\xi} ) = [z^k] \frac{ (\Seq \circ \cD^\gamma)^\bullet(\tau z)}{ (\Seq \circ \cD^\gamma)(\tau)} = [z^k] \frac{ (\Seq \circ \cD^\gamma)^\bullet(\tau z)}{ (\Seq \circ \cD^\gamma)^\bullet(\tau)} \Ex{\xi}.
	\]	
	Since $\Pr{\hat{\xi}< \infty} = \nu = \Ex{\xi}$, it follows that $(\hat{\xi} \mid \hat{\xi} < \infty)$ is distributed like the size of a $\mathbb{P}_{(\Seq \circ \cD^\gamma)^\bullet, \tau}$-distributed structure. As
	\[
		(\Seq \circ \cD^\gamma)^\bullet \simeq (\Seq' \circ \cD^\gamma )(\cD^\bullet)^\gamma,
	\]
	it follows from Lemma~\ref{le:bole} that the size of a $\mathbb{P}_{(\Seq \circ \cD^\gamma)^\bullet, \tau}$-distributed composite structure is distributed like the sum of the sizes of a $\mathbb{P}_{\Seq' \circ \cD^\gamma, \tau}$-distributed composite structure $X$ and an independent $\mathbb{P}_{(\cD^\bullet)^\gamma, \tau}$-distributed dissection $Y$. The latter may always be expressed as $1$ plus a multiple of $d$. Hence
	\[
		\Pr{\hat{\xi} \equiv a \mod d \mid \hat{\xi} < \infty} = \Pr{X + Y \equiv a \mod d} = \Pr{X \equiv a -1 \mod d}.
	\]
	Recall that $\Seq_a \circ \cD^\gamma$ are precisely the composite structures from $\Seq \circ \cD^\gamma$ with size in $a + d \ndZ$, so $\Seq_a' \circ \cD^\gamma$ are precisely the structures from $\Seq' \circ \cD^\gamma$ with size in $a-1 + d \ndZ$. Hence
	\begin{align*}
		\Pr{X \equiv a -1 \mod d}  &= \mathbb{P}_{\Seq' \circ \cD^\gamma, \tau}(\bigcup_{k} (\Seq_a' \circ \cD^\gamma)[k]) \\
		&= (\Seq_a' \circ \cD^\gamma)(\tau) / (\Seq' \circ \cD^\gamma)(\tau) \\
		&= \Pr{a_0 = a}.
	\end{align*}
	It follows by Equation~\eqref{eq:boltfin} that in order to verify \eqref{eq:yomei}, we need to check that
	\begin{align}
		\label{eq:solvemenow}
		\lim_{n \to \infty} \Pr{ d_{\cT_n}^+(o) \equiv a \mod d \mid d_{\cT_n}^+(o) > \Omega_n} = \Pr{\hat{\xi} \equiv a \mod d \mid \hat{\xi} < \infty}.
	\end{align}
	Note that Equation~\eqref{eq:bolt2} states that 
	\[
		\Pr{ d_{\cT_n}^+(o) \le \Omega_n} \sim \Pr{\hat{\xi} < \infty}
	\] and hence for any fixed integer $k$ it holds that as $n$ becomes large
	\begin{align*}
		\Pr{ d_{\cT_n}^+(o) = k \mid d_{\cT_n}^+(o) \le \Omega_n} &= \Pr{ d_{\cT_n}^+(o) = k } / (o(1) + \Pr{\hat{\xi} < \infty}) \\ &= \Pr{\hat{\xi} = k \mid \hat{\xi} < \infty} + o(1).
	\end{align*}
	This verifies
	\[
		(d_{\cT_n}^+(o) \mid d_{\cT_n}^+(o) \le \Omega_n) \convdis (\hat{\xi} \mid \hat{\xi} < \infty).
	\]
	Thus it follows by Equation~\eqref{eq:bolt2} that Equation~\eqref{eq:solvemenow} is actually equivalent to
	\begin{align}
		\label{eq:solvemenow2}
		\lim_{n \to \infty} \Pr{ d_{\cT_n}^+(o) \equiv a \mod d } = \Pr{\hat{\xi} \equiv a \mod d \mid \hat{\xi} < \infty}.
	\end{align}
	Note that
	\begin{align}
		\label{eq:heelp}
		\Pr{ d_{\cT_n}^+(o) \equiv a \mod d} = \frac{[z^n] z(\Seq_a \circ \cD^\gamma \circ \cZ)(z)}{ [z^n]\cZ(z)}
	\end{align}
	with $\cZ(z)$ being given by
	\[
		\cZ(z) = z \phi(\cZ(z)), \qquad \phi(z) = 1/(1- \cD^\gamma \circ \cZ (z)).
	\]
	This implies that 
	\[
		\bar{\cZ}(z) := \cZ(z) / z -1 = f( (\cD^\gamma \circ \cZ) (z)), \qquad \text{with} \qquad f(x) = \frac{x}{1-x}.
	\]
	Taking the inverse of $f$, it follows that
	\[
	(\cD^\gamma \circ \cZ)(z) = g(\bar{\cZ}(z)), \qquad \text{with} \qquad g(x) = \frac{x}{1+x}.
	\]
	The (sequence of coefficients of the) series $\bar{\cZ}(z)$ belongs to the class $\mathscr{S}_1$ of subexponential sequences by  Lemma~\ref{le:subexp}. The function $g(x)$ is analytic on $\ndC \setminus \{-1\}$. Hence we may apply Proposition~\ref{pro:sleeky} to obtain 
	\begin{align}
		\label{eq:yay}
	[z^{n-1}] (\cD^\gamma \circ \cZ)(z) = [z^{n-1}] g(\bar{\cZ}(z)) \sim g'( \bar{\cZ}(\rho_\cZ)) [z^n] \cZ(z),
	\end{align}
	with $\rho_\cZ$ denoting the radius of convergence of $\cZ(z)$. 
	But Equation~\eqref{eq:yay} also implies that $(\cD^\gamma \circ \cZ)(z)$ also belongs to the class $\mathscr{S}_1$ with radius of convergence $\rho_\cZ$. So we may apply Proposition~\ref{pro:sleeky} to obtain
	\begin{align*}
		[z^{n-1}] (\Seq_a \circ \cD^\gamma \circ \cZ)(z) &\sim \Seq_a'((\cD^\gamma \circ \cZ)(\rho_{\cZ})) [z^{n-1}](\cD^\gamma \circ \cZ)(z)
	\end{align*}
	and 
	\begin{align*}
		[z^{n-1}] (\Seq_a \circ \cD^\gamma \circ \cZ)(z) &\sim \Seq'((\cD^\gamma \circ \cZ)(\rho_{\cZ})) [z^{n-1}](\cD^\gamma \circ \cZ)(z).
	\end{align*}
	Thus it follows from Equation~\eqref{eq:heelp} that
	\begin{align*}
	\Pr{ d_{\cT_n}^+(o) \equiv a \mod d} \sim  \frac{\Seq_a'((\cD^\gamma \circ \cZ)(\rho_{\cZ}))}{\Seq'((\cD^\gamma \circ \cZ)(\rho_{\cZ}))}
	\end{align*}
	Note that in the type II setting it holds that $\cZ(\rho_\cZ) = \rho_\phi$, And  an elementary computation shows that
	\[
\frac{\Seq_a'(\cD^\gamma(\rho_\phi))}{\Seq'(\cD^\gamma(\rho_\phi))} ={\Pr{\hat\xi \equiv a \mod d}}/{\nu} = \Pr{\hat \xi \equiv a \mod d \mid \hat \xi < \infty}.
	\]
	This verifies Equations~\eqref{eq:solvemenow} and \eqref{eq:solvemenow2}, and thus concludes the proof for the case where the weight sequence $\mathbf{w}$ has type II.

	It remains to treat the case where  $\mathbf{w}$ has type III. Equation~\eqref{eq:outerschroeder} stated an isomorphism
	\begin{align}
	\cO^\omega \simeq \cX + (\cX \cdot \cD^\gamma)(\cO^\omega),
	\end{align}
	which represents face-weighted outerplanar maps as $\cX \cdot \cD^\gamma$-enriched parenthesizations. Thus Lemma~\ref{le:couplingschroeder} provides a coupling of $\mO_n^\omega$ with an enriched plane tree $(\tau_n, \delta_n)$, such that $\tau_n$ is a simply generated tree with leaves as atoms. 
	
	Recall that in Lemma~\ref{le:couplingschroeder}, we constructed the enriched plane tree $(\tau_n, \lambda_n)$ by first generating the random tree $\tau_n$, and then sampling for each vertex $v \in \cT_n$ an $\cN = \cX \circ \cD^\gamma$-structure $\delta_n(v) \in \cR[d^+_{\tau_n}(v)]$  with probability proportional to its weight. The labels of $\delta_n(v)$ correspond in a canonical way to the ordered set of offspring of the vertex $v$, which is why the tree $(\tau_n, \delta_n)$ may be interpreted as an enriched tree. The final enriched parenthesization $\mS_n^\cN$ is then obtained by relabelling through a uniformly at random drawn bijection.
		
	The precise way for identifying the offspring of an vertex $v$ with the atoms of the $\cN$-structure $\delta_n(v)$ does not affect the {\em distribution} of the resulting random unordered enriched parenthesization $\mS_n^\cN$. We may match the offspring of $v$ with the atoms of $\delta_n(v)$ according to any rule, which only takes $\delta_n(v)$ into account. Different matchings may very well change the corresponding $\cN$-enriched Schr\"oder parenthesization, but its distribution does not change. 

 
	 Thus we may consider the random enriched plane tree $(\tau_n, \lambda_n)$ constructed from $(\tau_n, \delta_n)$ by matching for each vertex $v \in \tau_n$ the offspring of $v$ with the atoms of the dissection $\tau_n(v)$ by ordering them in a canonically non-decreasing way according to their distance from the root-vertex of the dissection. This way, for any $\ell \ge 0$ the $\ell$-neighbourhood of the root in $\lambda_n(v)$ corresponds to an initial segment of the ordered offspring of $v$ in $\tau_n$.
	 
	 The outerplanar map corresponding to $(\tau_n, \lambda_n)$ is distributed like $\mO_n^\omega$, so we may assume without loss of generality that $\mO_n^\omega$ corresponds directly to $(\tau_n, \lambda_n)$. 
	 
	 By Lemma~\ref{le:types} and Lemma~\ref{le:convll} it holds that $\tau_n$ converges weakly in $\fmT$ toward an infinite star, that is, a tree consisting of a single root-vertex with infinitely many offspring, all of which are leaves. Hence there is a deterministic sequence  of integers $\Omega_n$ with $\Omega_n \to \infty$ such that with probability tending to $1$ as $n$ becomes large the root $o$ of $\tau_n$ has at least $\Omega_n$ children and the first $\Omega_n$ of them are all leaves.
	 
	 Since the weight-sequence $\mathbf{w}$ has type III, it follows that the random dissection $\mD_n^\gamma$ also has type III by Lemma~\ref{le:light}, and converges in the local weak sense by Theorem~\ref{te:gibbst3} toward a doubly infinite rooted path $P$. As $d_{\tau_n}^+(o) \ge \Omega_n$ with high probability, it follows that the dissection $\lambda_n(o)$ also converges in the local weak sense toward $P$.
	 
	 Let $\ell \ge 0$ be a fixed integer, it follows that with high probability the $\ell$-neighbourhood $V_\ell(\lambda_n(o))$ is a path with length $2 \ell$ that is rooted at its center vertex. All its $2 \ell +1$ vertices have with high probability also no children in $\tau_n$ (since they form an initial segment with fixed finite length of the offspring of the root). Hence, with high probability it holds that $V_\ell(\mO_n^\omega) \simeq V_\ell(\lambda_n^\omega) \simeq V_\ell(P)$. This confirms that the doubly infinite path $P$ is the local weak limit of $\mO_n^\omega$. This concludes the proof.
\end{proof}

\begin{proof}[Proof of Theorem~\ref{te:bsfaces}]
	The type I case is fully described by Theorem~\ref{te:convouter}, so it remains to treat the other two cases. Suppose that the weight-sequence $\mathbf{w}$ has type II or III.
	
	Let $v_0$ denote a uniformly at random selected vertex of the simply generated tree $\cT_n$.  Let $v_0, \ldots, v_h$ denote the path joining $v_0$ with the root of $\cT_n$, that is, the spine of the pointed tree $(\cT_n, v_0)$. Given the location of $v_0$ as a coordinate in $\VHT$, and the outdegrees $d_{\cT_n}^+(v_i)$ for all $1 \le i \le h$, the fringe subtrees at the non-spine offspring of the $v_i$ for $i \ge 1$ and at the vertex $v_0$ are exchangeable. This gives a certain degree of freedom in matching the vertices of the $\cR$-structures $\beta_n(v_i)$ to the offspring of $v_i$ in the tree $\cT_n$ without changing the \emph{distribution} of the unordered labelled enriched tree $\mA_n^\cR$, as long as we do not change which atom corresponds to $v_{i-1}$.

	For each $\cR = \Seq \circ \cD^\gamma$ structure $R$ let $G(R)$ denote the rooted outerplanar map obtained by gluing together the $*$-vertices of the dissections in $R$. Note that in the map $G(\beta_n(v_i))$ there is a priori no relation between the distance of a vertex to the distinguished vertex $v_{i-1}$, and its location in the linearly ordered list of offspring of $v_{i}$. This is not ideal, as we would like  to have siblings of $v_{i-1}$ that are close in the offspring list also to be close in the map $G(\beta_n(v_i))$.

	For each $1 \le i \le h$ let $a_i \in [d_{\cT_n}^+]$ denote the atom corresponding to the vertex $v_{i-1}$. Let $\sigma_i: [d_{\cT_n}^+(v_i)] \to [d_{\cT_n}^+(v_i)]$ denote the permutation that fixes the atom $a_i$, and permutes the remaining elements in a canonical way such that the $\cR$-structure $\lambda_n(v_i) := \cR[\sigma_i](\beta_n(v_i))$ has the property, that in the map $G(\lambda_n(v_i))$  it holds that the distance $d_{G(\lambda_n(v_i))}(a, a_i)$ is non-decreasing along  $a = a_i, a_i -1, \ldots$ and also non-decreasing along $a = a_i, a_i +1, \ldots$. This way, for each finite pointed map $G$ and each $\ell \ge 0$ there is a finite number $m$ such that whenever the $\ell$-neighbourhood $V_\ell(G(\lambda_n(v_i)))$ is equal to $G$ as rooted map, then all its vertices correspond to siblings of $v_{i-1}$ that lie at most $m$ to the left or $m$ to the right of $v_{i-1}$.
	
	Let $((\cT_n, \lambda_n), v_0)$ denote the pointed $\cR$-enriched plane tree obtained in this way, where we let $\lambda_n(v_i)$ for $1 \le i \le h$ be constructed as above, and for all other vertices $v$ we set $\lambda_n(v) = \beta_n(v)$. By the discussion above, the unordered pointed enriched tree corresponding to $((\cT_n, \lambda_n), v_0)$ is up to vertex labelling identically distributed as the random enriched tree $\mA_n^\cR$ pointed at a uniformly selected vertex.  As the unlabelled outerplanar map corresponding to an enriched tree does not depend on the vertex ordering or labelling, we may assume that $((\cT_n, \lambda_n), v_0)$ corresponds directly to the outerplanar map $\mO_n^\omega$ rooted at a uniformly at random selected vertex.
	
	Recall that in Subsection~\ref{sec:rndlimit} we defined a random pointed $\cR$-enriched plane tree $(\cT_n^*, \beta_n^*)$ having a finite spine $u_0, \ldots, u_k$,  with random length $k \ge 1$ and root-degree $d_{\cT_n^*}^*(u_k)$ given by a random variable $\tilde{D}_n$ defined in Equation~\eqref{eq:thed}. Again, by exchangeability, we may modify the matchings of the atoms of the $\cR$-structures with the offspring  vertices in the same way as we did for $( (\cT_n, \beta_n),v_0)$, without changing the distribution of the corresponding pointed outerplanar map. This yields a pointed enriched tree $(\cT^*_n, \lambda_n^*)$. Let $\hat{\mO}_n$ denote the corresponding outerplanar map.
	
	In Equation~\eqref{eq:om} we characterized a certain deterministic sequence $\Omega_n \to \infty$ that satisfies $\tilde{D}_n > \Omega_n$ for all $n$  by Equation~\eqref{eq:thed}.

	If the weight-sequence $\mathbf{w}$ has type II, then it follows by Lemma~\ref{le:light} that the species of dissections   $\cD^\gamma$  also has type II. Hence by Lemma~\ref{le:subexp} we know that $\cD^\gamma(z)/z$ belongs to the class $\mathscr{S}_d$ of subexponential series with span $d$ for some $d \ge 1$. It follows from Lemma~\ref{le:gibbs} and $\tilde{D}_n \ge \Omega_n \to \infty$ that the largest block in the $\Seq \circ \cD^\gamma$-structure $\lambda_n^*(u_k)$ has size $\tilde{D}_n + O_p(1)$ that converges in probability toward $\infty$ in the space $\bar{\ndN}_0$.
	
	If the weight-sequence $\mathbf{w}$ has type III, then by Lemma~\ref{le:light} the species of dissections $\cD^\gamma$ also has type III. By Theorem~\ref{te:need} and Lemma~\ref{le:supen} it follows that with high probability the largest dissection in $\lambda_n^*(u_k)$ has size $\tilde{D}_n + O_p(1)$ with the $O_p(1)$-term even admitting a deterministic upper bound with high probability. 
	
	Thus, regardless whether the weight-sequence $\mathbf{w}$ has type II or III, the vertex $u_{k-1}$ lies with probability tending to $1$ as $n$ becomes large in a large component $\mD_n$ of $\lambda_n^*(u_k)$, which up to relabelling is distributed like drawing a dissection from $\cD^\gamma[s_n]$ with probability proportional to its $\gamma$-weight for some random integer $s_n$ with $s_n \convp \infty$ in the space $\bar{\ndN}_0$. Moreover, given that $u_{k-1}$ lies in $\mD_n$, its location is uniformly distributed among all non-$*$-vertices of $\mD_n$. By Theorems~\ref{te:bsconvdis} and \ref{te:gibbst3} it follows that $\mD_n$ rooted at its $*$-vertex converges toward a limit dissection $\hat{\mD}$. In particular, for any fixed $\ell \ge 0$, the size of the $\ell$-neighbourhood of the $*$-vertex in $\hat{\mD}$ is stochastically bounded, and hence $u_{k-1}$ does not lie in there with probability tending to $1$ as $n$ becomes large. As this holds for arbitrary $\ell \ge 0$, this implies that $d_{\mD_n}(*, u_{k-1}) \convp \infty$ in the space $\bar{\ndN}_0$. The total variational distance between the location of $u_{k-1}$ and a uniformly at random chosen vertex of $\mD_n$ (including the $*$-vertex) tends to zero as $n$ becomes large. Since $\mD_n$ satisfies the rerooting invariance, it follows that $\mD_n$ rooted at $u{h-1}$ also converges toward the limit $\hat{\mD}$ in the local weak sense. As $d_{\mD_n}(*, u_{k-1}) \convp \infty$, we know that for any fixed $\ell$ the $\ell$-neighbourhood of $u_{k-1}$ in the map $G(\lambda_n^*(u_k))$ lies with high probability entirely in $\mD_n$ and does not contain the $*$-vertex. It follows that the random map $G(\lambda_n^*(u_k))$ rooted at $u_k$ converges in the local weak sense toward the dissection limit $\hat{\mD}$.
	
	Let $\hat{\mO}_*$ denote the rooted map obtained by taking the map corresponding to the pointed enriched fringe subtree of $(\cT_n^*, \lambda_n^*)$ at the vertex $u_{k-1}$, identifying the vertex $u_{k-1}$ with the root vertex of the dissection limit $\hat{\mD}$, and identifying the each non-root vertex of $\hat{\mD}$ with the root vertex of an independent copy of the map corresponding to the enriched tree $(\cT, \beta)$. Note that $\hat{\mO}_*$ does not depend on $n$, as the only part of $(\cT_n^*, \lambda_n^*)$ that does is the degree of the vertex $u_k$. By the discussion above, it follows that $\hat{\mO}_*$ is the local weak limit of the map $G(\cT_n^*, \lambda_n^*)$ centered at $u_0$, as every offspring of $u_k$ in $(\cT_n^*, \lambda_n^*)$ becomes the root of an independent copy of $(\cT, \beta)$.

	In Theorem~\ref{te:thmco} we stated that the pointed enriched fringe subtree of $((\cT_n, \beta_n), v_0)$ at the first ancestor of the vertex $v_0$ with degree at least $\Omega_n$ behaves like $(\cT_n^*, \beta_n^*)$, and hence the same holds for $((\cT_n, \lambda_n), v_0)$ with $(\cT_n^*, \lambda_n^*)$.  That is, if $\mH_{k_n}$ denotes the pointed enriched fringe subtree of $((\cT_n, \lambda_n), v_0)$ at the first ancestor $v_{k_n}$ of $v_0$ that has degree bigger than $\Omega_n$, then it holds that
	\begin{align}
		\label{eq:appro}
		d_{\textsc{TV}}(P_m(\mH_{k_n}), P_m(\cT_n^*, \lambda_n^*)) \to 0
	\end{align}
	as $n$ becomes large, with $P_m(\cdot)$ denoting the pruning operator defined in Section~\ref{sec:covic}. That is, roughly speaking, $P_m(\cdot)$ prunes away all offspring of sons of the root vertex that lie more than $m$ to the left or right from the unique spine offspring. 
	
	Let $G^\bullet$ be a given rooted outerplanar map and $\ell \ge 1$ a fixed integer. By the construction of  $\lambda_n$ it follows that for any fixed  $m > \max(\ell, |G^\bullet|)$  the pruned tree $P_m(\cT_n^*, \lambda_n^*)$ contains all information necessary to decide whether the $\ell$-neighbourhood $V_\ell(\hat{\mO}_n, u_0)$ is equal to $G^\bullet$ as rooted outerplanar map. Likewise, $P_m(\mH_{k_n})$ contains all information to decide whether $V_\ell( G(\mH_{k_n})) = G^\bullet$. The reason for this is that we constructed the $\lambda_n$ in such a way that vertices that siblings of $u_{k-1}$ that are close to $u_{k-1}$ in the map $G(\cT_n^*, \lambda_n^*)$ are also close to $u_{k-1}$ in the linear order of the offspring of $u_k$ in the tree $\cT_n^*$, and likewise for $P_m(\mH_{k_n})$.
	
	Since the random map $\hat{\mO}_n$ rooted at $u_0$ converges in the local weak sense toward the limit map~$\hat{\mO}_*$, the limit in \eqref{eq:appro} implies that  
	\[
		\Pr{V_\ell( G(\mH_{k_n}), v_0) = G^\bullet} \to \Pr{V_\ell(\hat{\mO}_n, u_0) = G^\bullet}
	\]
	as $n$ becomes large. As $\ell$ and $G^\bullet$ where arbitrary, it follows that $(G(\mH_{k_n}), v_0)$ converges in the local weak sense toward $\hat{\mO}_*$.

	Above we verified that $d_{\mD_n}(*, u_{k-1}) \convp \infty$ and consequently for any fixed $\ell \ge 1$ it holds that with high probability the $\ell$-neighbourhood $V_\ell( \hat{\mO}_n, u_0)$ does not contain the $*$-vertex of $\mD_n$, that is, the vertex $u_k$.  By the limit in \eqref{eq:appro} it follows that likewise the $\ell$-neighbourhood $V_\ell( G(\mH_{k_n}), v_0)$ does with high probability  not contain the  spine-vertex $v_{k_n}$. But this implies also that $V_\ell(\mO_n^\omega, v_0) = V_\ell(G(\mH_{k_n}), v_0)$ holds with probability tending to $1$. It follows that $(\mO_n^\omega, v_0)$ converges in the local weak sense toward the limit $\hat{\mO}_*$. 

	It remains to describe the distribution of $\hat{\mO}_*$. If the weight sequence $\mathbf{w}$ has type III, then $\hat{\mD}$ is a deterministic doubly-infinite path and the tree $\cT$ consists almost surely of a single vertex. Thus, in this case $\hat{\mO}_*$ is also a deterministic doubly infinite path.
	
	Suppose that the weight sequence $\mathbf{w}$ has type II. In order to describe the distribution of $\hat{\mO}_*$ we make use of the following observations.
	\begin{enumerate}
		\item The enriched tree $(\cT, \beta)$ corresponds to a $\mathbb{P}_{\cO^\omega, \tau/\phi(\tau)}$-distributed map.
		\item For each $1 \le i \le u_{k-1}$ the $\cR^\bullet$-structure $\lambda_n^*(u_i)$ (pointed at $u_{i-1}$) is distributed like $G(\mS_1, \mD^\bullet, \mS_2)$, with $\mS_1$ and $\mS_2$ being independent $\mathbb{P}_{\cR, \tau}$-distributed sequences of dissections, and $\mD^\bullet$ an independent $\mathbb{P}_{(\cD^\bullet)^\gamma, \tau}$-distributed pointed dissection.
		\item If we assign $\mathbb{P}_{\cR, \tau}$-distributed $\cR$-structure to a single vertex, and attach an independent copy of $(\cT, \beta)$ to each of its atoms, then the outerplanar map corresponding to this enriched tree follows a $\mathbb{P}_{\cO^\omega, \tau/\phi(\tau)}$ distribution.
		\item If we draw a random $\mathbb{P}_{(\cD^\bullet)^\gamma, \tau}$-distributed structure and then switch the $*$-vertex with the pointed vertex, the distribution does not change. (That is, the $*$-vertex becomes a regular vertex with the label of the pointed vertex, and the pointed vertex becomes a label-free  $*$-vertex.)
	\end{enumerate}

	The first claim follows from the fact that the size of $|\cT|$ is distributed like the size of a $\mathbb{P}_{\cO^\omega, \tau/\phi(\tau)}$-distributed map, and the map corresponding to the conditioned tree $(\cT_n, \beta_n)$ is drawn up to relabelling from $\cO^\omega[n]$ with probability proportional to its weight.
	
	As for the second claim, notice first that the $\cR$-structure $\lambda_n^*(u_i)$ with the distinguished atom $u_{i-1}$ is up to relabelling distributed like a $(\hat{\xi} \mid \hat{\xi} < \infty)$-sized random $\cR^\bullet$ structure drawn with probability proportional to its weight. As $(\hat{\xi} \mid \hat{\xi} < \infty)$ follows the size of a $\mathbb{P}_{\cR^\bullet, \tau}$-distributed object, it already follows that $\lambda_n^*(u_i)$ with the distinguished atom $u_{i-1}$ follows up to relabelling a $\mathbb{P}_{\cR^\bullet, \tau}$-distribution. Using the chain rule of Proposition~\ref{pro:chainprod} and the isomorphism 
	\[
		\Seq' \simeq \Seq^2
	\]
	it follows that
	\[
		\cR^\bullet \simeq (\Seq \circ \cD^\gamma) (\cD^\bullet)^\gamma (\Seq \circ \cD^\gamma).
	\]
	By the product rule of Lemma~\ref{le:bole} it follows that, up to relabelling, we may sample a $\mathbb{P}_{\cR^\bullet, \tau}$-distribution by taking two independent $\mathbb{P}_{\cR, \tau}$-distributed sequences of dissection and placing an independent $\mathbb{P}_{(\cD^\bullet)^\gamma, \tau}$ pointed dissection in the middle.

	In order to verify the third claim, notice that the isomorphism
	\[
		\cA_\cR^\omega \simeq \cX \cdot \cR \circ \cA_\cR^\omega
	\]
	combined with the product and composition rule of Lemma~\ref{le:bole} tells us that we may sample a $\mathbb{P}_{\cA_\cR^\omega, \tau/\phi(\tau)}$-distributed $\cR$-enriched tree by assigning a $\mathbb{P}_{\cR, \tau}$-distributed structure to a root vertex and identifying each of its atoms with an independent $\mathbb{P}_{\cA_\cR^\omega, \tau/\phi(\tau)}$-distributed enriched tree. Together with the first claim, this verifies the third claim.

	The fourth claim follows easily by symmetry.

	Thus, for $1 \le i \le k-1$ the pointed map corresponding to the pointed enriched fringe subtree of $(\cT_n, \lambda_n^*)$ at the spine vertex $u_i$ is distributed like the result of taking $G(\mO_1, \mD^\bullet, \mO_2)$, with $\mO_i$ denoting independent maps following a $\mathbb{P}_{\cO^\omega, \tau/\phi(\tau)}$-distribution, and identifying the pointed vertex of $\mD^\bullet$ with the root of the pointed map corresponding to the enriched fringe subtree of $(\cT_n, \lambda_n^*)$ at the vertex $u_{i-1}$.
	
	So, summing up, the map $\hat{\mO}^*$ may be sampled as follows.
	\begin{enumerate}
		\item 
		Let $(\mD_i^\bullet)_{1 \le i \le k-1}$ be a family of independent identically distributed $\cD^\bullet$-objects following a 
		$\mathbb{P}_{(\cD^\bullet)^\gamma, \tau}$-distribution.
		 Concatenate the $\mD_i^\bullet$ by identifying the pointed vertex of $\mD_i^\bullet$ with the root $*$-vertex of $\mD_{i-1}^\bullet$ for all $i \ge 2$. 
		Identify the $*$-vertex of $\mD_{k-1}$ with the root-vertex of the dissection limit $\hat{\mD}$. This vertex corresponds to $u_{k-1}$. Likewise, the $*$-vertex of $\mD_i$ corresponds to $u_i$ for $i \ge 1$ and the pointed vertex of $\mD_1$ to the vertex $u_0$. We let $C$ denote the result and mark the vertices $u_1, \ldots, u_{k-1}$ with the colour blue.
		\item  Each vertex of $C$ gets identified with the root vertex of an independent copy of a $\mathbb{P}_{\cO^\omega, \tau/\phi(\tau)}$-distributed outerplanar map that gets attached from outside, except for the blue vertices, which receive two such maps, one from each side.
	\end{enumerate}
	The Benjamini--Schramm limit $\hat{\mO}_*$ is similar to the local weak limit $\hat{\mO}$ of Theorem~\ref{te:facethm}. In the construction of $\hat{\mO}_*$, the roles of the $*$-vertex and pointed vertex of the pointed dissection are reversed compared to the construction of $\hat{\mO}$. But, as we argued in the fourth claim above, this makes no difference for the resulting distribution. Thus the only difference between $\hat{\mO}_*$ and $\hat{\mO}$ is that the root of $\hat{\mO}$ is identified with the root of a $\mathbb{P}_{(\cD^\bullet)^\gamma, \tau}$-distributed dissection and two $\mathbb{P}_{\cO^\omega, \tau/\phi(\tau)}$-distributed outerplanar maps, and the root of $\hat{\mO}_*$ on the other hand gets identified with the root of a $\mathbb{P}_{(\cD^\bullet)^\gamma, \tau}$-distributed dissection but only one $\mathbb{P}_{\cO^\omega, \tau/\phi(\tau)}$-distributed outerplanar map. Thus $\hat{\mO}$ is distributed like the result of taking the Benjamini--Schramm limit $\hat{\mO}_*$ and identifying its root vertex with the root of an independent  $\mathbb{P}_{\cO^\omega, \tau/\phi(\tau)}$-distributed outerplanar map.
\end{proof}

\begin{proof}[Proof of Corollary~\ref{co:gottaname}]
		We assumed that $\mO_n^\omega$ has type II, hence Theorem~\ref{te:facethm} implies that $\mD_n^\gamma$ has type II with $\Seq_{\ge1}^\iota$ having radius of convergence $\tau_\cD$ and $\Seq_{\ge 1}^\iota(\tau_\cD) <1$. We also assumed that
		\begin{align}
					\label{eq:outt1}
		[z^k] \Seq_{\ge 1}^\iota(z) = f(k) k^{-\beta} r^{-\beta}
		\end{align}
		for a constant $\beta>2$ and some constant $r>0$ which necessarily must be equal to $\tau_\cD$.
		It follows by Propositions~\ref{pro:sleeky} and \ref{pro:yehaa}  that the series $\phi_\cD(z) = \Seq \circ \Seq_{\ge 1}^\iota(z)$ satisfies
		\begin{align}
		[z^k] \phi_\cD(z)  \sim (1 - \Seq_{\ge 1}^\iota(\tau_\cD))^{-2} f(k) k^{-\beta} \tau_\cD^{-k}.
		\end{align}

		By a general result for the partition function of simply generated trees \cite[Eq. (14)]{MR3335012}, that is based on a local large-deviation theorem established  in \cite{MR2440928}, it follows that
		\begin{align}
		\label{eq:outt2}
		[z^k] \cD^\gamma(z) \sim \tau_\cD  (1 - \nu_\cD)^{-\beta} (1 - \Seq_{\ge 1}^\iota(\tau_\cD))^{-2} f(k) k^{-\beta} \rho_\cD^{-k}
		\end{align}
		with 
		\[
		\rho_\cD = \tau_\cD ( 1 - \Seq_{\ge 1}^\iota(\tau_\cD)).
		\]
		Since
		\[
		\cD^\gamma(\rho_\cD) = \tau_\cD < 1,
		\] we may apply Proposition~\ref{pro:sleeky} to obtain
		\begin{align*}
		[z^k] 1/(1 - \cD^\gamma(z)) \sim \hat{c}  f(k) k^{-\beta} \rho_\cD^{-k}, \qquad \hat{c} = \tau_\cD  (1 - \nu_\cD)^{-\beta} (1 - \tau_\cD)^{-2}(1 - \Seq_{\ge 1}^\iota(\tau_\cD))^{-2}.
		\end{align*}
		Let $Y_{(1)}$ and $Y_{(2)}$  denote the size of the largest and second-largest block of $\mO_n^\omega$.
		By \cite[Thm. 1]{MR3335012} it follows from \eqref{eq:outt2} that there is a slowly varying function $g_1$ such that $Y_{(2)} = O_p(g_1(n) n^{1/\alpha})$ and that 
		\begin{align}
		\frac{ (1-\nu)n - Y_{(1)}}{g_1(n) n^{1/\alpha}} \convdis X_\alpha,
		\end{align} where $X_\alpha$ is an $\alpha$-stable random variable with Laplace transform \[\Ex{e^{-tX_\alpha}} = \exp(\Gamma(-\alpha)t^\alpha), \quad \operatorname{Re}\, t \ge 0.\]
		For each $1 \le i \le n$ let $B^+_i$ denote the largest dissection in the $\Seq \circ \cD^\gamma$-structure $\beta_n(v_i)$ of the vertex $v_i$, and let $F^+_i$ denote the size of the largest face of $B^+_i$. Lemma~\ref{le:gibbs} and \eqref{eq:outt2} imply that
		\begin{align}
			\label{eq:squee}
			Y_{(1)} - |B^+_1| = O_p(1).
		\end{align}
		Consequently 
		\begin{align}
		\label{eq:sizeb}
		\frac{ (1-\nu)n - |B^+_1|}{g_1(n) n^{1/\alpha}} \convdis X_\alpha.
		\end{align}
		By Corollary~\ref{co:itsclear} and Equation~\eqref{eq:outt1} we know that the largest face $F_{(1)}(\mD_k^\gamma)$ and second largest face $F_{(2)}(\mD_k^\gamma)$ in a random dissection $\mD_k^\gamma$ of a $k$-gon satisfy $F_{(2)}(\mD_k^\gamma) = O_p(g_2(k)k^{1/\alpha})$ and
		\begin{align}
			\label{eq:connn}
		\frac{ (1-\nu_\cD)k - F_{(1)}(\mD_k^\gamma)}{g_2(k) k^{1/\alpha}} \convdis X_\alpha
		\end{align}
		as $k$ becomes large for some slowly varying function $g_2(k)$. For any $k$ it holds that
		\[
			(B_1^+ \mid |B_1^+| = k) \eqdist \mD_k^\gamma.
		\]
		By the limit in \eqref{eq:sizeb} it follows that
		\begin{align}
			\label{eq:plpl}
			\frac{ (1-\nu_\cD)|B_1^+| - F_1^+}{g_2(|B_1^+|) |B_1^+|^{1/\alpha}} \convdis X_\alpha.
		\end{align}
		Any slowly varying function $h$ has the property, that for any compact interval $[a,b]$ with $a>0$ it holds that
		\[
			\lim_{x \to \infty}	\sup_{t \in [a,b]} \left | \frac{h(tx)}{h(x)} -1 \right | = 0.
		\]
		(See for example Feller's book \cite{MR0270403} for this standard fact.) Hence \eqref{eq:sizeb} implies that
		\[
			g_2(|B_1^+|)=g_2(n)(1 + o_p(1)).
		\]
		Moreover, 
		\[
			|B_1^+|^{1/\alpha} = (1-\nu)^{1/\alpha}n^{1/\alpha}(1 + o_p(1)).
		\]
		Setting $g(n) = g_2(n)(1-\nu)^{1/\alpha}$, it follows from \eqref{eq:plpl} that
		\[
			\frac{(1- \nu)(1-\nu_\cD)n - F_1^+}{g(n)n^{1/\alpha}} \convdis X_\alpha.
		\]
		Equation~\eqref{eq:squee} implies that the second largest block  of $\beta_n(v_1)$ has size $O_p(1)$. Hence the size of the second largest face in $\beta_n(v_1)$ has order $O_p(g_2(n)n^{1/\alpha})$ and the size of the largest face outside of $\beta_n(v_1)$ is bounded by $Y_{(2)} = O_p(g_1(n)n^{1/\alpha})$. Thus \[
			F_{(2)} = O_p(\max(g_1(n), g_2(n)) n^{1/\alpha}).
		\] Any   slowly varying function $h$ satisfies $h(t)t^{-\epsilon} \to 0$ as $t$ becomes large for all $\epsilon > 0$ \cite{MR0270403}, and hence so does $\max(g_1(n), g_2(n))$.
\end{proof}

\subsection{Proofs of the applications to random weighted graphs in Section~\ref{sec:blockcl}}

\label{sec:proofgraph}

We are going to list the proof of the results from Section~\ref{sec:blockcl} roughly in order of their appearance, with the exception of the observations that were already sufficiently justified there.

\begin{proof}[Proof of Theorem~\ref{te:discon}]
	We need to show that the Gibbs partition  $\Seq \circ \cC^\omega$ admits a giant component with size $n + O_p(1)$. This was observed in Stufler~\cite[Thm. 4.2 and Section 5]{doi:10.1002/rsa.20771} for the case where the weight-sequence $\mathbf{w}$ has type I or II. 
	
	In the superexponential case, when $\mathbf{w}$ has type III, we know by Lemma~\ref{le:supen} that the coefficients
	\[
		a_n = [z^n] (\cC^\bullet)^\omega(z)
	\]
	satisfy
	\begin{align*}
	\sum_{\substack{i_1 + \ldots + i_k = n \\ 1 \le i_1, \ldots, i_k < n - (k-1)}} a_{i_1} \cdots a_{i_k} = o(a_{n-(k-1)}).
	\end{align*}
	for all $k \ge 2$. It follows that the coefficients of the unrooted graphs
	\[
		c_n = [z^n] \cC^\omega(z) = \frac{a_n}{n}
	\]
	satisfy
	\begin{align*}
		\sum_{\substack{i_1 + \ldots + i_k = n \\ 1 \le i_1, \ldots, i_k < n - (k-1)}} c_{i_1} \cdots c_{i_k} 
		&= \sum_{\substack{i_1 + \ldots + i_k \\  1 \le i_1, \ldots, i_k < n - (k-1)}} \frac{1}{i_1 \cdots i_k} a_{i_1} \cdots a_{i_k} \\
		&\le O (n^{-1}) \sum_{\substack{i_1 + \ldots + i_k = n \\ 1 \le i_1, \ldots, i_k < n - (k-1)}} a_{i_1} \cdots a_{i_k}  \\
		&= o(n a_{n-(k-1)}) \\
	    &= o(c_{n-(k-1)}).
	\end{align*}
	Hence we may apply Theorem~\ref{te:need} to obtain that there is a fixed integer $n_0 \ge 0$ such that the size $K_n$ of the largest connected component satisfies
	\begin{align}
		\label{eq:crossref}
		K_n \ge n-n_0
	\end{align}
	with probability tending to $1$ as $n$ becomes large.
\end{proof}

\begin{proof}[Proof of Corollary~\ref{co:discon}]
	Suppose that the weight-sequence $\mathbf{w}$ has type III. It was already observed in Equation~\eqref{eq:crossref} that there is a constant $n_0 \ge 0$ such that the largest connected component of the random graph $\mG_n^\upsilon$ has size at least $n \ge n_0$ with high probability.
	
	If the complete graph with $2$ vertices receives positive weight, then it holds that \[[z^2]\cC^\omega(z) >0\] and hence it follows from Corollary~\ref{co:ne} that $\mG_n^\upsilon$ is with high probability connected.
\end{proof}

Theorem~\ref{te:loclunl} and Corollary~\ref{co:prel} are direct applications of Theorems~\ref{te:local} and \ref{te:strengthend}.  The claims made in Example~\ref{ex:outerplanar} for random weighted outerplanar graphs may be justified by analogous arguments as for Theorem~\ref{te:facethm} and Lemma~\ref{le:light} for random weighted outerplanar maps, which we do not aim to repeat here.

Corollaries~\ref{co:blocksize} and \ref{cor:planarpath} are sufficiently justified by the explanations given in Section~\ref{sec:blockcl}, except for the fact that we need to check that they apply to uniform random planar graphs, which we do here:

\begin{lemma}
	If $\mC_n^\omega$ is the uniform $n$-vertex planar graph, then the weight-sequence $\mathbf{w}$ has type II and the corresponding probability weight-sequence $\pi_k$ given in Equation~\ref{eq:tt} satisfies $\pi_k \sim c k^{-5/2}$ as $k \to \infty$ for some constant $c>0$.
\end{lemma}
\begin{proof}
	By enumeration results given in \cite{MR2476775}, there are constants $c_1, \rho_\cB >0$ such that the number of $2$-connected planar graphs is asymptotically equivalent to $c_1 n^{-7/2} \rho_\cB^{-n} $. So the exponents of $(\cB')^\gamma(z)$ admit the same asymptotic expression, only with the exponent $-5/2$ instead of $-7/2$. It follows by general properties for functions of power-series, see for example Embrechts and Omney \cite[Sec. 2.2]{MR772907}, that there is a constant $c_2>0$ such that
	\[
		\omega_k = [z^k]\exp( (\cB')^\gamma(z)) \sim c_2 n^{-5/2} \rho_\cB^{-n}.
	\]
	Moreover, by \cite[Claim 1]{MR2476775} we know that \[\nu = \rho_{\cB} \cB''(\rho_{\cB}) < 1.\] (See Section~\ref{sec:pretree} for the definition of the parameter $\nu$.)  Hence the weight sequence $\mathbf{w} = (\omega_k)_k$ has type II and $\tau=\rho_\cB$. Thus $\pi_k \sim c_3 k^{-5/2}$ as $k \to \infty$. This concludes the proof.
\end{proof}

\begin{proof}[Proof of Remark~\ref{re:distribution}]
	Essentially, we have to show that the description of the limit graph follows the distribution of the limit enriched tree $(\hat{\cT}, \hat{\beta})$ interpreted as a graph according to the bijection in Section~\ref{sec:bijblock}.
	
	Here it is useful to note that the fringe subtree at its second spine vertex follows the same distribution as the whole tree. So we are going to interpret $(\hat{\cT}, \hat{\beta})$ without this fringe subtree as a graph, and then use this recursion. 
	
	The root $o$ of $\hat{\cT}$ receives an $\cR$-object $\hat{\beta}(o)$ of which a uniformly at random drawn atom forms the second spine vertex. Taking a $\hat{\xi}$-sized $\cR^\kappa = \Set \circ (\cB')^\gamma$ object with probability proportional to its $\kappa$-weight and marking a uniformly at random chosen non-$*$-vertex, is equivalent to taking a $\mathbb{P}_{(\cR^\bullet)^\kappa, \tau}$ object. So, $\hat{\beta}(o)$ with the marked atom given by the second spine-vertex follows a $\mathbb{P}_{(\cR^\bullet)^\kappa, \tau}$-distribution.
	
	We may apply the rules for operations on species in Section~\ref{sec:opspe} and the isomorphism $\Set' \simeq \Set$ to deduce that 
	\[
	(\cR^\bullet)^\kappa \simeq (\Set' \circ (\cB')^\gamma) \cdot (\cB'^\bullet)^\gamma \simeq  (\Set \circ (\cB')^\gamma) \cdot (\cB'^\bullet)^\gamma.
	\]
	The interpretation is that an unordered sequence of derived blocks where one atom is marked consists of a distinguished marked block and an unordered collection of unmarked blocks.  
	
	We apply the rules in Lemma~\ref{le:bole} governing the relation between weighted Boltzmann distributions and operations on species. This yields that the marked block and the collection of unmarked blocks are independent, and follow Boltzmann distributions  $\mathbb{P}_{(\cB'^\bullet)^\gamma, \tau}$ and $\mathbb{P}_{\Set \circ (\cB')^\gamma, \tau}$. In the tree $(\hat{\cT}, \hat{\beta})$, each of the non-marked atoms of the $\cR$-object $\beta(o)$ becomes the root of an independent copy of $(\cT, \beta)$. As graph, $(\cT, \beta)$ follows a Boltzmann distribution $\mathbb{P}_{(\cC^\bullet)^\omega, \tau/\phi(\tau)}$.
	
	The isomorphism
	\[
	(\cC^\bullet)^\omega \simeq \cX \cdot (\Set \circ (\cB'^\bullet)^\gamma)((\cC^\bullet)^\omega)
	\]
	and the rules in Lemma~\ref{le:bole} imply, that if we glue the $*$-vertices of a $\mathbb{P}_{\Set \circ (\cB')^\gamma, \tau}$-distributed collection of blocks together, and identify each non-$*$-vertex with the root of a fresh independent copy of a $\mathbb{P}_{(\cC^\bullet)^\omega, \tau/\phi(\tau)}$-distributed connected rooted graph, then the result again follows a $\mathbb{P}_{(\cC^\bullet)^\omega, \tau/\phi(\tau)}$ Boltzmann distribution.
	
	Thus, the graph corresponding to $(\hat{\cT}, \hat{\beta})$, without the fringe-subtree at the second spine vertex, corresponds to a pointed  $\mathbb{P}_{(\cB'^\bullet)^\gamma, \tau}$-distributed block, where every vertex except for the pointed vertex gets identified with the root of a fresh copy of a $\mathbb{P}_{(\cC^\bullet)^\omega, \tau/\phi(\tau)}$-distributed random connected graph.
	
	Since the  fringe subtree at the second spine vertex of $(\hat{\cT}, \hat{\beta})$ is distributed like $(\hat{\cT}, \hat{\beta})$ itself, it follows that the graph $\hat{\mC}$ corresponding to $(\hat{\cT}, \hat{\beta})$ is distributed like an infinite chain of such joint objects, where the marked vertex of any object is identified with the $*$-vertex of the subsequent one. Thus, $\hat{\mC}$ follows the distribution described in the remark.	
\end{proof}

\begin{proof}[Proof of Corollary~\ref{co:cntemb}]
	We have to check if $\mathbf{w}$ has type I$a$, then for arbitrarily large $p$ the root-degree of $\mC_n^\omega$ is bounded in $\mathbb{L}_p$ as $n$ becomes large. This implies arbitrarily high uniform integrability.
	
	In the coupling with the simply generated tree $\cT_n$, the root-degree in the graph is bounded by the root-degree $d_{\cT_n}^+(o)$ in the tree. So it suffices to consider the moments of $d_{\cT_n}^+(o)$. As $\mathbf{w}$ has type I$a$, the offspring distribution $\xi$ has finite exponential moments, and $\Pr{|\cT|=n}$ has order $n^{-3/2}$. That is, there are constants $C,c>0$ such that for all $x \ge 0$
	\[
		\Pr{d_{\cT_n}^+(o) \ge x} \le \Pr{|\cT| = n}^{-1} \Pr{\xi \ge x} \le C n^{3/2} \exp(-c x).
	\]
	In particular, 
	\[
		\lim_{n \to \infty} n^p \Pr{d_{\cT}^+(o) \ge \log^2 n} = 0.
	\]
	Equation~\eqref{eq:stuff} implies that
	\[
		\lim_{n \to \infty} \sup_{x \le \log^2 n} |\Pr{d_{\cT_n}^+(o) =x} / \Pr{d_{\hat{\cT}}^+(o) =x}| -1 | = 0,
	\]
	with $d_{\hat{\cT}}^+(o) \eqdist \hat{\xi}$ also having finite exponential moments. This yields
	\[
		\Ex{ (d_{\cT_n}^+(o))^p)} =  o(1) + \sum_{x=1}^{\lfloor \log^2 n \rfloor} x^p \Pr{\hat{\xi}=x}(1 + o(1)) = o(1) + \Ex{\hat{\xi}^p}.
	\]
\end{proof}

\begin{proof}[Proof of Theorem~\ref{te:convpl}]
	The proof is rather lengthy, hence we divide it into parts, starting with the overall strategy.
	
	\paragraph*{The proof strategy.} Our overall strategy is to follow these steps:
	\begin{enumerate}
	\item We first show that we may work with a modified version $(\cT_n, \lambda_n)$ obtained from the tree $(\cT_n, \lambda_n)$ by matching the vertices of the $\cR$-structures with the offspring sets in a more convenient way, such that $\ell$-neighbourhoods of the corresponding graph are determined \emph{initial segments} of the offspring sets. This is important as we are going to encounter a vertex with large degree for which we, by this trick, require information on the structures corresponding to the atoms of an initial segment of its offspring.
	
	\item The next step is that we to approximate the tree $(\cT_n, \lambda_n)$ by a tree $(\cT_{1n}, \lambda_{1n})$ obtained from the enriched tree $(\hat{\cT}, \hat{\beta})$ by replacing its tip of the spine by a large $\cR$-structure having an independent random size $\tilde{D}_n$, that has a deterministic lower bound which tends to infinity.  We have full information on the behaviour of the tree $(\cT_{1n}, \lambda_{1n})$, except for the Gibbs partition $\mS_n$ at the tip of its spine that gets sampled from $\cR[\tilde{D}_n]$ with probability proportional to its $\kappa$-weight. Studying its behaviour takes two steps. 
	
	\item First, we  show that the composition of $\Set$ with $(\cB')^\gamma \circ \cA_\cR^\omega$ has convergent type.  Roughly speaking, this states that in the $\cR$-structure corresponding to the root of $\cT_n$ there is a typically a distinguished block such that the union of the block, and all the fringe subtrees dangling from it, has size $n + O_p(1)$.
	
	\item We combine this fact together with properties of the coupling of $(\cT_n, \lambda_n)$ with $(\cT_{1n}, \lambda_{1n})$ to deduce that $\mS_n$ may be approximated in total variation by a large randomly sized block $\hat{\mB}_n$ and a $\mathbb{P}_{\Set \circ (\cB')^\gamma, \tau}$-distributed remainder $\hat{\mR}$. Thus we may approximate the enriched tree $(\cT_{1n}, \lambda_{1n})$ by a modified version $(\cT_{1n}^*, \lambda_{1n}^*)$, where the tip of the spine $v^*$ receives a copy of $\hat{\mR}$ and the block $\hat{\mB}_n$.
	
	\item The coupling that we construct has the property, that the $\ell$-neighbourhood of the random root of the graph $\mC_n^\omega$ is, as unlabelled rooted graph, with high probability equal to the $\ell$-neighbourhood of the graph $(\mC_{1n}^\omega, v_{1n})$ that corresponds to the tree $(\cT_{1n}^*, \lambda_{1n}^*)$. The only part in the construction of this enriched tree, that depends on $n$, is the randomly sized block $\hat{\mB}_n$. If the random block $\hat{\mB}_n$ converges in the local weak sense toward a limit $\hat{\mB}$, then Benjamini--Schramm convergence of $\mC_{1n}^\omega$ and hence also of $\mC_n^\omega$ follows.

	\item In the statement Theorem~\ref{te:convpl} we gave an explicit and simple description of the limit object. We check that its distribution coincides with the distribution of the graph corresponding to the limit enriched tree.
	
	\item We have verified that weak convergence of the randomly sized $2$-connected graph $\hat{\mB}_n$ implies convergence of the random connected graph $\mC_n^\omega$. Conversely, if we know that $\mC_n^\omega$ (and hence also $\mC_{1n}^\omega$) converges in the Benjamini--Schramm sense, then our coupling allows us to deduce local weak convergence of the randomly sized block $\hat{\mB}_n$. We do this in two steps. The first step is to verify that for all $\ell \ge 0$, the graph $G_n$, obtained by gluing the blocks of $\lambda_{1n}^*(v^*)$ together at their $*$-vertex, has the property that $V_\ell(G_n)$ converges weakly for all fixed $\ell$.
	
	\item The second step is to use the convergence of the neighbourhoods $V_\ell(G_n)$ to deduce that for each fixed $\ell$, the neighbourhood $V_\ell(\hat{\mB}_n)$ converges weakly toward a limit distribution $\mu_\ell$. The family $(\mu_\ell)_\ell$ is projective, hence we may deduce from this that the graph $\hat{\mB}_n$ converges in the local weak sense toward the projective limit of $(\mu_\ell)_\ell$.
	\end{enumerate}

	\paragraph*{1. Matching the vertices in a convenient manner:  the enriched tree $(\cT_n, \lambda_n)$.}
	Recall that in Lemma~\ref{le:coupling}, we construct the enriched plane tree $(\cT_n, \beta_n)$ by first generating the random tree $\cT_n$, and then sampling for each vertex $v \in \cT_n$ an $\cR$-structure $\beta_n(v) \in \cR[d^+_{\cT_n}(v)]$  with probability proportional to its $\kappa$-weight. The labels of $\beta_n(v)$ correspond in a canonical way to the ordered set of offspring of the vertex $v$, which is why the tree $(\cT_n, \beta_n)$ may be interpreted as an enriched tree. The final $\cR$-enriched tree $\mA_n^\cR = (\mA_n, \alpha_n)$ is then obtained by relabelling through a uniformly at random drawn bijection.
	
	As we already noted in the proof of Theorem~\ref{te:convout}, the precise way for identifying the offspring of an vertex $v$ with the atoms of the $\cR$-structure $\beta_n(v)$ does not affect the {\em distribution} of the resulting random unordered enriched tree $\mA_n^\cR$. We may match the offspring of $v$ and the atoms of $\beta_n(v)$ according to any rule, which only takes $\beta_n(v)$ into account. Different matchings may  change the isomorphism type of the corresponding $\cR$-enriched tree, but its distribution does not change. It is not hard to verify this by using the fact that $(\cT_n, \beta_n)$ is distributed like $(\cT,\beta)$ conditioned on having $n$ vertices, and that $(\cT, \beta)$ has this invariance property:  any offspring of a vertex $v\in \cT$ becomes the root of an independent copy of $(\cT, \beta)$, regardless to which atom of $\beta(v)$ it gets matched.
	
	In the case of random connected graphs, the species $\cR$ is given by $\cR^\kappa = \Set \circ (\cB')^\gamma$. We may consider the random enriched plane tree $(\cT_n, \lambda_n)$ constructed from $(\cT_n, \beta_n)$ by matching for each vertex $v \in \cT_n$ the offspring of $v$ with the atoms of $\beta_n(v)$ by ordering them following way. We start with the neighbours of the $*$-vertices in the derived block and order them in any canonical way, and then proceed with the vertices at distance $2$ from their respective $*$-vertices, and so  on.  We may construct a rooted graph  out of $(\cT_n, \lambda_n)$ according to the bijection in Section~\ref{sec:bijblock}, which is distributed like the random graph $\mC_n^\omega$ rooted at a uniformly at random drawn vertex $v_n$. So, in the following we assume that $(\mC_n^\omega, v_n)$ corresponds directly to $(\cT_n, \lambda_n)$.

	 The reason why we consider the modified tree $(\cT_n, \lambda_n)$ instead of working with $(\cT_n, \beta_n)$ directly is that for any integer $\ell \ge 0$ and any finite rooted connected graph $G^\bullet$ there exists a constant $m(\ell, G^\bullet) \ge 0$, such that the event, that the $\ell$-neighbourhood $V_\ell(\mC_n^\omega, v_n)$ is isomorphic to $G^\bullet$ as rooted graphs, is already completely determined by the family $(\lambda_n(v))_{v \in V^{[m(\ell, G^\bullet)]}}$. There are two reasons for this. First, vertices with distance at most $\ell$ from $v_n$ in $\mC_n^\omega$ also have block-distance  at most $\ell$ from $v_n$, and hence height at most $\ell$ in $\cT_n$. Second, by the construction of $\lambda_n$, for any vertex $v$ the subset of its ordered sequence of sons, that still lies in the $\ell$-neighbourhood of $v_n$, is an initial segment in the ordered list. If $V_\ell(\mC_n^\omega, v_n) \simeq G^\bullet$, then the length of this initial segment must be bounded by the number of vertices of $\mG^\bullet$. Hence, if we take $m(\ell, G^\bullet)$ large enough depending on $\ell$ and the size of $\mG^\bullet$, then $(\lambda_n(v))_{v \in V^{[m(\ell, G^\bullet)]}}$ contains all information necessary to decide whether $V_\ell(\mC_n^\omega, v_n) \simeq G^\bullet$. The same holds in general for arbitrary  $\cR$-enriched plane trees, as long as the matching of the $\cR$-structures with the offspring sets is done in the same way.

	\paragraph*{2. Coupling the tree $(\cT_n, \lambda_n)$ with the tree $(\cT_{1n}, \lambda_{1n})$}

	For any integer $m \ge 0$, we consider the subset $V^{[m]} \subset \VHT$ of the vertices of the Ulam--Harris tree given by 
	\[
		V^{[m]} = \{ (i_1, \ldots, i_t) \mid t \le m, i_1, \ldots, i_\ell \le m\}.
	\]
	Here we consider the first $m$ sons of the root, and for each of those again the first $m$ sons, and so on, until we reach generation $m$.

	Janson~\cite[Sec. 20]{MR2908619} constructs a deterministic sequence $\Omega_n \to \infty$ and a modified Galton--Watson tree $\cT_{1n}$ that is obtained from $\hat{\cT}$ by sampling a random degree $\tilde{D}_n \ge \Omega_n$ independently from $\hat{\cT}$, and pruning $\hat{\cT}$ at its unique vertex $v^*$ with infinite degree, keeping only the first $\tilde{D}_n$ children of $v^*$. In \cite[Thm. 20.2]{MR2908619} it is stated that for any fixed integer $m \ge 0$ it holds that
	\begin{align}
	\label{eq:conv1}
	\lim_{n \to \infty} d_{\textsc{TV}}( (d^+_{\cT_n}(v))_{v \in V^{[m]}}, (d^+_{\cT_{1n}}(v))_{v \in V^{[m]}}) = 0.
	\end{align}
	That is to say, the tip of the spine in $\hat{\cT}$ corresponds to a vertex with large degree $\tilde{D}_n$ in $\cT_n$.

	The construction of $\Omega_n$ is stated in  \cite[Lem. 19.32]{MR2908619} such that if $N_k$ denotes the number of vertices with degree $k$ in $\cT_n$, then 
	\begin{align}
	\label{eqq:bolt1}
	\sum_{k \le \Omega_n} k N_k = \nu n + o_p(n) \qquad \text{and} \qquad \sum_{k > \Omega_n} k N_k = (1 - \nu)n + o_p(n).
	\end{align}
	Letting $o$ denote the root vertex of $\cT_n$, it holds by \cite[Lem. 15.7]{MR2908619} that
	\begin{align}
	\label{eqq:bolt0}
	\Pr{d_{\cT_n}^+(o) = k} = \frac{n}{n-1} \Ex{ \frac{k N_k}{ n}}.
	\end{align}
	Hence Equations~\eqref{eqq:bolt1} may be rephrased by
	\begin{align}
	\label{eqq:bolt2}
	\Pr{d_{\cT_n}^+(o) \le \Omega_n} = \nu + o(1) \qquad \text{and} \qquad \Pr{d_{\cT_n}^+(o) > \Omega_n} = 1 - \nu + o(1).
	\end{align}
	The distribution of $\tilde{D}_n$ is stated in \cite[Equation (20.4)]{MR2908619} by $\Pr{\tilde{D}_n = k} = 0$ for $k \le \Omega_n$, and 
	\begin{align}
	\Pr{\tilde{D}_n = k} = \frac{ k\Ex{N_k}}{\sum_{\ell > \Omega_n} \ell \Ex{ N_\ell}} 
	\end{align}
	for all $k > \Omega_n$. By Equations~\eqref{eqq:bolt0} and \eqref{eqq:bolt2} it follows that
	\begin{align}
	\label{eqq:boltfin}
	\Pr{\tilde{D}_n = k} = \Pr{ d_{\cT_n}^+(o) = k \mid d_{\cT_n}^+(o) > \Omega_n}.
	\end{align}
	
	Since the tree $\cT_{1n}$ is almost surely finite,  we may turn it into an enriched plane tree $(\cT_{1n}, \beta_{1n})$, by sampling for each vertex $v$ an element $\beta_{1n}(v)$ from $\cR[d^+_{\cT_{1n}}(v)]$ with probability proportional to its weight. Again we may match the non-$*$-vertices of the set of derived blocks $\beta_{1n}(v)$ with the ordered offspring of $v$ according to their distance from their respective $*$-vertices in order to obtain an enriched tree $(\cT_{1n}, \lambda_{1n})$, in precisely the same way as we constructed $(\cT_n, \lambda_n)$ out of $(\cT_n, \beta_n)$. For any finite family of vertices $v_i$, $i \in I$, integers $d_i \ge 0$ and $\cR$-structures $R_i \in \cR[d_i]$, it holds that
	\begin{align*}
		\Pr{ \beta_n(v_i) = R_i, i \in I \mid d_{\cT_n}^+(v_i) = d_i, i \in I} &= \prod_{i \in I} \kappa(R_i) / |\cR[d_i]|_\kappa  \\
		&= \Pr{ \beta_{1n}(v_i) = R_i, i \in I \mid d_{\hat{\cT}_{1n}}^+(v_i) = d_i, i \in I}.
	\end{align*}
	So, Equation~\eqref{eq:conv1} already implies
	\begin{align}
		\label{eq:conv22}
		\lim_{n \to \infty} d_{\textsc{TV}}( (\lambda_n(v))_{v \in V^{[m]}}, (\lambda_{1n}(v))_{v \in V^{[m]}}) = 0.
	\end{align}
	
	\paragraph*{3. The randomly sized Gibbs partition at the tip of the spine - first part.}

The asymptotic behaviour of the graph corresponding to $(\cT_{1n}, \lambda_{1n})$ depends on the  behaviour of the randomly sized Gibbs partition $\lambda_{1n}(v^*)$, that gets sampled from $(\Set \circ (\cB')^\gamma)[\tilde{D}_n]$ with probability proportional to its weight. The problem is, that it does not need to hold that $\Set \circ (\cB')^\gamma$ has convergent type, so there may be no sensible limit for a  random element from $(\Set \circ (\cB')^\gamma)[k]$ as $k$ deterministically tends to infinity. However, we are dealing with a randomly sized Gibbs partition, and this makes all the difference.  We are going to verify that the composition of $\Set$ with $(\cB')^\gamma \circ \cA_\cR^\omega$ has convergent type, and then use the fact that $\tilde{\cD}_n$ is distributed like the root degree of $\cT_n$ conditioned to be large.
	
	The equation
	\[
		 z \exp( (\cB')^\gamma (\cA_\cR^\omega(z))) = \cA_\cR^\omega(z)
	\]
	may be rewritten by
	\[
		\exp( (\cB')^\gamma (\cA_\cR^\omega(z))) = 1 + a(z) 
	\]
	with $a(z)$ satisfying $a(0)=0$. Consequently,
	\[
		 (\cB')^\gamma (\cA_\cR^\omega(z)) = \log(1 + a(z)).
	\]
	Here $\log$ denotes the principal branch of the logarithm, which is holomorphic in the domain $\ndC \setminus ]-\infty, 0]$. The power series $a(z)$ has radius of convergence $\tau/\phi(\tau)<\infty$ and satisfies $a(\tau/\phi(\tau)) = \phi(\tau) -1 < \infty$. We may check that the set \[\{1 + a(z) \mid z \in \ndC, |z| \le \tau / \phi(\tau)\}\] is contained in the domain $\ndC \setminus ]-\infty,0]$. Indeed, 
	\[
	 \nu = \tau \phi'(\tau) / \phi(\tau) = \tau  (\cB'')^\gamma(\tau) \ge (\cB')^\gamma(\tau) = (\cB')^\gamma( \cA^\omega_\cR(\tau / \phi(\tau))).
	\]
	Since all coefficients of $(\cB')^\gamma( \cA^\omega_\cR(z))$ are non-negative, it follows that for all $z \in  \ndC$ with $|z| \le \tau / \phi(\tau)$ 
	\[
		|(\cB')^\gamma( \cA^\omega_\cR(z))| \le \nu.
	\]
	As $\mathbf{w}$ has type II, we know that $\nu < 1$. By basic properties of the complex exponential function, it follows that
	\[
		1 + a(z) = \exp((\cB')^\gamma( \cA^\omega_\cR(z))) \notin ]-\infty,0]
	\]
	whenever $|z| \le \tau/\phi(\tau)$. 
	
	The (sequence of coefficients of the) series $a(z)$ belongs to the class $\mathscr{S}_d$ with $d = \spa(\mathbf{w})$  by  Lemma~\ref{le:subexp}. Hence we may apply Proposition~\ref{pro:sleeky} to obtain, since we always assume that $n \equiv 1 \mod \spa(\mathbf{w})$,
	\begin{align*}
	[z^{n-1}] ((\cB')^\gamma \circ \cA_\cR^\omega)(z) &= [z^{n-1}] \log(1 + a(z))  \\ &\sim \frac{1}{ 1 + a(\tau/\phi(\tau))} [z^{n-1}]a(z),
	\end{align*}
	Since $a(z)$ belongs to the class $\mathscr{S}_d$, this also implies that $((\cB')^\gamma \circ \cA_\cR^\omega)(z)$ belongs to $\mathscr{S}_d$. Thus the composition of $\Set$ with $(\cB')^\gamma \circ \cA_\cR^\omega$ has convergent type. 
	
	\paragraph*{4. Randomly sized Gibbs partitions and the tree $(\cT_{1n}^*, \lambda_{1n}^*)$}
	
	Knowing that  the composition $\Set$ with $(\cB')^\gamma \circ \cA_\cR^\omega$ has convergent type will help us to determine the limit behaviour of the randomly sized Gibbs partition drawn from $(\Set \circ (\cB')^\gamma)[\tilde{D}_n]$ with probability proportional to its weight. We are going to show that it consists typically of a giant component with a stochastically bounded rest that converges in total variation toward a $\mathbb{P}_{\Set \circ (\cB')^\gamma, \tau}$-distribution. Hence, it behaves precisely as if the composition $\Set \circ (\cB')^\gamma$ had convergent type, although the latter need not hold at all.
	
	The isomorphism
	\[
		\cA_\cR^\omega \simeq \cX \cdot \Set \circ ( (\cB')^\gamma \circ \cA_\cR^\omega)
	\]
	allows us to view the forest obtained from $(\cT_n, \beta_n)$ by removing the root-vertex as a Gibbs partition corresponding to the composition of $\Set$ and $(\cB')^\gamma \circ \cA_\cR^\omega$. As this composition has convergent type, it follows that it exhibits a giant component with a stochastically bounded remainder, that converges in total variation to a $\mathbb{P}_{\Set' \circ (\cB')^\gamma \circ \cA_\cR^\omega, \tau/\phi(\tau)}$-distribution. Applying the rules for Boltzmann distributions in Lemma~\ref{le:bole}, we obtain that the collection of blocks of this limit, that correspond to the root of the tree, follow a $\mathbb{P}_{\Set' \circ (\cB')^\gamma, \tau}$-distribution. Note that $\Set'$ and $\Set$ are isomorphic species, so there is no real difference to a $\mathbb{P}_{\Set \circ (\cB')^\gamma, \tau}$-distribution.
	
	 The question is, how does the size of the root block of the giant $(\cB')^\gamma \circ \cA_\cR^\omega$-component behave? To answer this, note that since the limit \eqref{eq:conv22} holds for arbitrarily large $m$, there is a coupling of $(\cT_n, \lambda_n)$ and $(\cT_{1n}, \lambda_{1n})$ such that
	 \begin{align}
		 \label{eq:sklein}
		 (1 + \sup\{m \ge 0 \mid \lambda_n(v) = \lambda_{1n}(v) \text{ for all $v \in V^{[m]}$}\})^{-1} \convp 0.
	 \end{align}
	 We know that $\lambda_{n}(o)$ consists of a single block $B(n)$, whose asymptotic size we do not know yet, and a remainder $R(n)$ that converges in total variation toward a $\mathbb{P}_{\Set \circ (\cB')^\gamma, \tau}$-distributed limit. Moreover, we know that the total size of the union $S(n)$ of fringe subtrees dangling from the remainder $R(n)$ in $\cT_n$ is stochastically bounded. As $\lambda_{n}(o)=\lambda_{1n}(o)$ with probability tending to $1$ as $n$ becomes large, we may also express $\lambda_{1n}(o)$ as the disjoint union of a block $B(1n)$ and a remainder $R(1n)$ such that with high probability $B(1n)=B(n)$ and $R(1n) = R(n)$.  We are going to argue, that with high probability the vertex $v^*$ does not lie in the union $S(1n)$ of fringe subtrees dangling from $R(1n)$ in the tree $\cT_{1n}$. Indeed, given $\epsilon>0$ we may take $m$ large enough such that $v^* \in V^{[m]}$ with probability at least $1- \epsilon$ for all $n$. As with high probability $R(n) = R(1n)$ and $\lambda_{n}(v)=\lambda_{1n}(v)$ for all $v \in V^{[m]}$, it follows that whenever $v^*$ lies in $V^{[m]} \cap S(1n)$, then the size of $S(n)$ is at least $\Omega_n \to \infty$. As the size of $S(n)$ is stochastically bounded it follows that the probability for $v^*$ to lie in $S(1n)$ is  bounded from below by $1- 2\epsilon$ for all large enough $n$. As $\epsilon>0$ was arbitrary, it follows that with high probability $v^*$ does not lies in $S(1n)$. 	 
	 
	 Depending on the location of the vertex $v^*$, there are two possible behaviours for $B(1n)$ and $R(1n)$. If the tip of the spine $v^*$ has height at least $1$, then the $\Set \circ (\cB')^\gamma$-object $\lambda_{1n}(o)$ together with its unique atom that belongs to the spine of $\cT_{1n}$ gets drawn from $(\Set \circ (\cB')^\gamma)^\bullet[\hat{\xi}]$ with probability proportional to its weight. By the rules for Boltzmann samplers in Lemma~\ref{le:bole}, this means that it follows a $\mathbb{P}_{ (\cR^\bullet)^\kappa, \tau}$-distribution. As
	 \[
		 \cR^\bullet \simeq (\cB'^\bullet)^\gamma \cdot (\Set \circ (\cB')^\gamma),
	 \]
	 it follows that in this case, as unlabelled objects, the block containing the spine vertex follows a $\mathbb{P}_{ (\cB'^\bullet)^\gamma, \tau}$-distribution (and is with high probability equal to $B(1n)$), and is independent from its remainder, which follows a $\mathbb{P}_{ \Set \circ (\cB')^\gamma, \tau}$-distribution.
	 
	 If the tip $v^*$ is equal to the root $o$, then the corresponding $\Set \circ (\cB')^\gamma$-object gets drawn from $\Set \circ (\cB')^\gamma[\tilde{D}_n]$ with probability proportional to its weight. It holds uniformly for every set $\cE$ of unlabelled $\Set \circ (\cB)^\gamma$-objects that
	 \begin{align*}
		 \Pr{R(1n) \in \cE \mid v^* = o} &= \frac{ \Pr{R(1n) \in \cE} - \Pr{R(1n) \in \cE \mid v^* \ne o} \Pr{v^* \ne o} }{\Pr{v^* = 0}} \\
		 &\sim \frac{\mathbb{P}_{\cR^\kappa, \tau}(\cE) - \mathbb{P}_{\cR^\kappa, \tau}(\cE) \nu }{1- \nu} \\
		 &= \mathbb{P}_{\cR^\kappa, \tau}(\cE).
	 \end{align*}
	 As $\tilde{D}_n \ge \Omega_n \to \infty$ and the size of $R(1n)$ is stochastically bounded, it follows that in this case the block $B(1n)$ becomes large. Note that conditioned on having specific size $k$, the block $B(1n)$ gets drawn with probability proportional to its $\gamma$-weight among all elements from $\cB'[k]$. Let $\hat{\mR}$ denote a $\mathbb{P}_{\cR^\kappa, \tau}$-distributed collection of blocks, and $\hat{\mB}_n$ a block drawn from $\cB'[\tilde{D}_n - |\hat{\mR}|]$ with probability proportional to its weight. The latter is only possible if $\tilde{D}_n - |\hat{\mR}| >0$ and $|\cB'[\tilde{D}_n - |\hat{\mR}|]|_\gamma >0$, but this holds with high probability. We have shown that if we draw a composite structure $\mS_n$ from $\cR[\tilde{D}_n]$ with probability proportional to its weight, then
	 \begin{align}
		 \label{eq:jukname}
		 d_{\textsc{tv}}( \mS_n, \{\hat{\mB}_n\} \cup \hat{\mR}) \to 0
	 \end{align}
	 as $n$ becomes large. Let $(\cT_{1n}^*, \lambda_{1n}^*)$ denote the tree constructed by modifying the tree $(\hat{\cT}, \hat{\beta})$ in a similar way as we did for $(\cT_{1n}, \lambda_{1n})$, with the sole difference that instead of assigning a $\tilde{D}_n$-sized random $\cR$-object to the tip of the spine, we use the union $\{\hat{\mB}_n\} \cup \hat{\mR}$. Equation~\eqref{eq:jukname} shows that for all $m \ge 0$ it holds that
	 \[
		 d_{\textsc{TV}}( (\lambda_{1n}(v))_{v \in V^{[m]}}, (\lambda_{1n}^*(v))_{v \in V^{[m]}} \to 0.
	 \]
	 Hence, by the limit in~\eqref{eq:sklein} it follows that there is a coupling of  $(\cT_n, \lambda_n)$ and $(\cT_{1n}, \lambda_{1n})$ such that
	 \[
	 (1 + \sup\{m \ge 0 \mid \lambda_n(v) = \lambda_{1n}^*(v) \text{ for all $v \in V^{[m]}$}\})^{-1} \convp 0.
	 \]
	 
	 \paragraph*{5. Weak convergence of $\hat{\mB}_n$ implies Benjamini--Schramm convergence of $\mC_n^\omega$.}
	 
	Let $(\mC_{1n}^\omega, v_{1n})$ denote the rooted graph corresponding to $(\cT_{1n}^*, \lambda_{1n}^*)$ according to the bijection in Section~\ref{sec:bijblock}. Let $G^\bullet$ be an arbitrary finite rooted graph and let $\ell \ge1$ be an integer. As discussed in the first step, there is an integer $m(\ell, G^\bullet) \ge 0$ such that for any fixed $m \ge m(\ell, G^\bullet)$ the family $(\lambda_{1n}^*(v))_{v \in V^{[m]}}$ already contains all information necessary to decide whether $V_\ell(\mC_{1n}^\omega, v_{1n}) \simeq G^\bullet$ as rooted graphs.

	For any set $R$ of derived graphs let $G(R)$ denote  the derived graph obtained by identifying the $*$-vertices of all blocks with each other. Then, for any $r \ge 0$ the $r$-neighbourhood $V_r(G(R))$ is given by the union of the $r$-neighbourhoods of the $*$-vertices in the components. This may be expressed by
	\[
		V_r(G(R)) = G( \{ V_r(Q) \mid Q \in R\}).
	\]
	
	Suppose that the randomly sized derived  $2$-connected graph $\hat{\mB}_n$ converges in the local weak sense toward a limit $\hat{\mB}$. To unify notation, we treat the root-vertex of $\hat{\mB}$ like a $*$-placeholder vertex. It holds that
	\begin{align}
		\label{eq:jjnn}
		V_r(G(\lambda_{1n}^*(v^*))) = G(\{V_r(\hat{\mR}),  V_r(\hat{\mB}_n)\}) \convdis G(\{V_r(\hat{\mR}),  V_r(\hat{\mB})\}) = V_r( G(\hat{\mR} \cup \{\hat{\mB}\})).
	\end{align}
	 In order to decide whether $V_\ell(\mC_{1n}^\omega, v_{1n}) \simeq G^\bullet$, it is more than enough to know the $\ell$-neighbourhoods $V_\ell(G(\beta_{1n}(v)))$ for all $v \in V^{[m]}$. (It would also suffice to just consider the $(\ell-h_{\cT_{1n}}(v))$-neighbourhoods of the vertices $v$). The limit~\eqref{eq:jjnn} implies that
	\begin{align}
		\label{eq:jjmm}
		(V_\ell(G(\lambda_{1n}(v))))_{v \in V^{[m]}} \convdis (V_\ell(G(\hat{\lambda}(v))))_{v \in V^{[m]}},
	\end{align}
	where we let $(\hat{\cT}, \hat{\lambda})$ denote the limit enriched plane tree obtained from $(\hat{\cT}, \hat{\beta})$ by matching the offspring of any vertex $v$ with finite outdegree $d_{\hat{\cT}}^+(v) < \infty$ with the atoms of the set of derived blocks $\hat{\beta}(v)$ in the same way as we did for $(\cT_n, \lambda_n)$. For the unique vertex $u^*$ with $d_{\hat{\cT}}^+(u^*) = \infty$, we let $\lambda(u^*)$ be given by $\hat{\mR} \cup \{\hat{\mB}\}$, where we also  match the countably infinite offspring of $\lambda(v^*)$ with the countably infinite number of non-$*$-vertices of $\hat{\mR} \cup \{\hat{\mB}^\circ\}$ in the same way. This is possible, since it is easily verified that the random graph $\hat{\mB}$ has countably infinite many vertices. For an upper bound, we only need the fact that it is locally finite, and the lower bound follows as it is the limit of a sequence of random graphs whose size deterministically tends to infinity.
	
	The convergence in \eqref{eq:jjmm} implies that the random rooted graph $\hat{\mG}^\bullet$ that corresponds to the $\cR$-enriched plane tree $(\hat{\cT}, \hat{\lambda})$ satisfies
	\[
		\lim_{n \to \infty} \Pr{V_\ell(\mC_n^\omega, v_n) \simeq G^\bullet} = \Pr{V_\ell(\hat{\mG}^\bullet) \simeq G^\bullet}.
	\]
	As $G^\bullet$ and $\ell$ where arbitrary, it follows that $\hat{\mG}^\bullet$ is the Benjamini--Schramm limit of the random graph $\mC_{1n}^\omega$ and thus also of $\mC_n^\omega$. 
	
	\paragraph*{6. Distribution of the limit}
	We are going to argue, that $\hat{\mG}^\bullet$ is distributed like the graph $\hat{\mC}$ described in Theorem~\ref{te:convpl}. Recall that we constructed $\hat{\mC}$ by concatenating the independent identically distributed blocks $(\mB_i'^\bullet)_{1 \le i \le K}$, where $K$ follows a geometric distribution with parameter $\nu$, glue the limit $\hat{\mB}$ at the tip of this chain, and finally identify each vertex of this graph with the root of an independent copy of the Boltzmann-distributed random graph $\mC^\bullet$. The height of the vertex $u^*$ in $\hat{\cT}$ is distributed like $K$, and the $\cR$-structures along the spine in $\hat{\cT}$ actually follow Boltzmann distributions of $(\Set \circ \cB')^\bullet$ with parameter $\tau$. As
	\[
		(\Set \circ (\cB')^\gamma)^\bullet \simeq (\Set \circ (\cB')^\gamma) \cdot \cB'^\bullet,
	\]
	the product rule in Section~\ref{sec:WeBoSa} implies that each of the blocks containing consecutive spine vertices actually follows a Boltzmann distributions for $(\cB'^\bullet)^\gamma$ with parameter $\tau$, and the remainder of the corresponding $(\Set \circ (\cB')^\gamma)^\bullet$-object is independent from this block and follows a $\Set \circ (\cB')^\gamma$ distribution with parameter $\tau$. The isomorphism $\cA_\cR^\omega \simeq \cX \cdot \cR^\kappa(\cA_\cR^\omega)$ and the composition rule in Section~\ref{sec:WeBoSa} imply that if we take a Boltzmann distributed $\Set \circ \cB'$-structure with parameter $\tau$, glue the $*$-vertices together, and identify its vertices with the roots of independent copies of $\mC^\bullet$, then the result follows a Boltzmann distribution for $(\cC^\bullet)^\omega$ with parameter $\tau /\phi(\tau)$. So, summing up, the random graph $\hat{\mG}^\bullet$ corresponding to $(\hat{\cT}, \hat{\beta})$ is distributed like $\hat{\mC}$. 
	
	\paragraph*{7. Convergence of $\mC_n^\omega$ implies convergence of $\hat{\mB}_n$ - first part.} 
	We have shown that 
	if $\mC_n^\omega$ converges in the Benjamini--Schramm sense, then so does the random graph $\hat{\mB}_n$. 
	
	Indeed, suppose that $\mC_n^\omega$ admits a distributional limit $\hat{\mC}^\circ$. It follows that this graph is also the distributional limit of the rooted graph $(\mC_{1n}, v_{1n})$ that corresponds to the enriched tree $(\cT_{1n}^*, \lambda_{1n}^*)$. 
	
	We are first going to show that the derived graph $G_n := G(\lambda_{1n}^*(v^*))$ has the property, that for each $\ell \ge 1$ the neighbourhood $V_\ell(G_n)$ converges weakly in the countable discrete set of unlabelled rooted graph with height at most $\ell$. Let us start with the root-degree. For any $k$ it holds that
	\begin{align}
		\label{eq:sum1}
		\Pr{d_{\mC_{1n}}(v_{1n}) = k}  &= \Pr{d_{\mC_{1n}}(v_{1n}) = k \mid v^* = o} (1-\nu) + \Pr{d_{\mC_{1n}}(v_{1n}) = k \mid v^* \ne o} \nu \nonumber \\
		&= \Pr{ d(G_n) = k}(1-\nu) + \Pr{d(\hat{\mR}^\bullet) = k} \nu,
	\end{align}
	with $\hat{\mR}^\bullet$ denoting a $\mathbb{P}_{(\cR^\bullet)^\kappa, \tau}$-distributed random object. Since 
	\[
		\Pr{d_{\mC_{1n}}(v_{1n}) = k} \to \Pr{d(\mC^\circ) = k},
	\]
	it follows that
	\[
		\Pr{ d(G_n) = k} \to d_k
	\]
	for some $d_k \ge 0$.  Equation~\eqref{eq:sum1} also implies that the limits $(d_k)_{k \ge 0}$ satisfy $\sum_{k \ge 0} d_k = 1$. In other words, the $1$-neighbourhood $V_1(G_n)$ converges weakly.

	We proceed to show weak convergence of the $\ell$-neighbourhood $V_\ell(G_n)$ by induction on $\ell$. So assume that $V_{i}(G_n)$ admits a weak limit for all $i< \ell$. Let $G^\bullet$ denote a finite graph. Recall that there is a $m = m(G^\bullet, \ell)$ such that the family $(\lambda_{1n}^*(v))_{v \in V^{[m]}}$ contains all informations to decide whether $V_\ell(\mC_{1n}) \simeq G^\bullet$. A bit more precise, we only require knowledge of the neighbourhoods $V_{\ell - \he_{\cT_{1n}^*}(v)}(G(\lambda_{1n}^*(v)))$ for all $v \in V^{[m]}$ with $\he_{\cT_{1n}^*}(v) < \ell$. Let $V(\ell) \subset V^{[m]}$ denote the subset of vertices with height less than $\ell$. Thus, there is a finite set $\scM$ of families $(G_v)_{v \in V(\ell)}$ of rooted graphs such that 
	\[
	V_\ell(\mC_{1n}) \simeq G^\bullet \qquad \text{if and only if} \qquad  \left(V_{\ell - \he_{\cT_{1n}^*}(v)}(G(\lambda_{1n}^*(v)))\right)_{v \in V(\ell)} \in \scM.
	\]
	Thus
	\begin{align}
	\label{eq:suum}
	\Pr{V_\ell(\mC_{1n}) \simeq G^\bullet} &= \sum_{h = 0}^\ell \Pb{V_{\ell - \he_{\cT_{1n}^*}(v)}(G(\lambda_{1n}^*(v)))_{v \in V(\ell)} \in \scM \mid \he_{\cT_{1n}^*}(v^*) = h} \nu^{h}(1-\nu).
	\end{align}
	Note that the left-hand side of this equation converges, since we assumed that $\mC_{1n}$ has a weak limit. As for the right hand-side, all summands with $h \ge 1$ converge by induction, as they depend on the neighbourhood $V_i(\lambda_{1n}^*(v^*))$  only for $i < \ell$. Let
	\[
		\scM_0 = \{ (G_v)_{v \in V^(\ell)} \in \scM \mid G_o \simeq G^\bullet\}
	\]
	denote the subset that corresponds to the event that the $\ell$-neighbourhood at the root vertex $o$ is already isomorphic to $G^\bullet$. Note that all elements $(G_v)_{v \in V^(\ell)} \in \scM \setminus \scM_0$ satisfy $|G_o|<|G^\bullet|$.	
	
	If $v^* = o$, then any non-root vertex that belongs to the tree $\cT_{1n}^*$ receives an $\cR$-structure according to a $\mathbb{P}_{\Set \circ (\cB')^\gamma, \tau}$-distribution, which assumes a zero-sized object with positive probability. Hence
	\begin{align}
		\label{eq:hard}
		\Pr{V_{\ell - \he_{\cT_{1n}^*}(v)}(G(\lambda_{1n}^*(v)))_{v \in V(\ell)} \in \scM_0 \mid \he_{\cT_{1n}^*}(v^*) = 0} = \Pr{ V_\ell(G_n) \simeq G^\bullet} C(G^\bullet)
	\end{align}
	for some constant $C(G^\bullet)>0$. In fact, we may set
	\begin{align}
		\label{eq:cbul}
		C(G^\bullet) = p^{|V_{\ell-1}(G^\bullet)|}
	\end{align}
	with $p>0$ denoting the probability, that a $\mathbb{P}_{\Set \circ (\cB')^\gamma, \tau}$-distributed structure has size $0$. To justify this, note that when $\Pr{V_\ell(G_n) \simeq G^\bullet}=0$, then both sides of Equation~\eqref{eq:hard} are always zero. If $\Pr{V_\ell(G_n) \simeq G^\bullet}>0$, then it holds that
	\begin{multline*}
	\Pr{V_{\ell - \he_{\cT_{1n}^*}(v)}(G(\lambda_{1n}^*(v)))_{v \in V(\ell)} \in \scM_0 \mid \he_{\cT_{1n}^*}(v^*) = 0} \Pr{ V_\ell(G_n) \simeq G^\bullet}^{-1} = \\ \Pr{V_\ell(\mC_{1n}) \simeq G^\bullet \mid v^* = o, V_\ell(G(\lambda_{1n}^*(o))) \simeq G^\bullet}.
	\end{multline*}
	 Conditional on $v^* = o$ and $V_\ell(G(\lambda_{1n}^*(o))) \simeq G^\bullet$, the event $V_\ell(\mC_{1n}) \simeq G^\bullet$ takes place if and only if each offspring vertex $v$ of the root that corresponds to a vertex in $V_{\ell-1}(G(\lambda_{1n}^*(o))) \simeq V_{\ell -1}(G^\bullet)$ satisfies $|\lambda_{1n}^*(v)|=0$. There are precisely $V_{\ell -1}(G^\bullet)$ many such vertices, and each receives an independent $\mathbb{P}_{\Set \circ (\cB')^\gamma, \tau}$-distributed structure. Hence this conditional probability is equal to $ p^{|V_{\ell-1}(G^\bullet)|}$.
	
Having verified that Equation~\eqref{eq:cbul}, it follows from Equation~\eqref{eq:suum} that
	\begin{align}
	\label{eq:compli}
	\Pr{V_\ell(\mC^\circ) \simeq G^\bullet} + o(1) &= \sum_{h = 1}^\ell \Pb{V_{\ell - \he_{\cT_{1n}^*}(v)}(G(\lambda_{1n}^*(v)))_{v \in V(\ell)} \in \scM \mid \he_{\cT_{1n}^*}(v^*) = h} \nu^{h}(1-\nu) \nonumber \\
	&+ 	\Pb{V_{\ell - \he_{\cT_{1n}^*}(v)}(G(\lambda_{1n}^*(v)))_{v \in V(\ell)} \in \scM \setminus \scM_0 \mid \he_{\cT_{1n}^*}(v^*) = 0}(1-\nu) \nonumber \\
	&+ \Pr{ V_\ell(G_n) \simeq G^\bullet} C(G^\bullet) (1-\nu).
	\end{align}
	On the right-hand side, the first summand converges, since we assumed that $V_i(G_n)$ converges for $i < \ell$. The second summand contains only conditions of the form $V_\ell(G_n) \simeq G'$ for graphs $G'$ with size $|G'| < |G^\bullet|$. Since $C(G^\bullet)>0$, it follows by induction on $|G^\bullet|$ that \[
	\Pr{V_\ell(G_n) \simeq G^\bullet} \to p_\ell(G^\bullet)
	\]
	for some $p_\ell(G^\bullet) \ge 0$ as $n$ becomes large.
	
	In order to deduce weak convergence of $V_\ell(G_n)$, it remains to verify that $\sum_{G^\bullet} p_\ell(G^\bullet) = 1$, with the sum index $G^\bullet$ ranging over all unlabelled rooted graphs. Suppose that this does not hold, that is, $\sum_{G^\bullet} p_\ell(G^\bullet) = 1 - \epsilon$ for some $\epsilon>0$.  Then for any fixed $s \ge 1$ 
	\[
		\Pr{|V_\ell(G_n)| \ge s} = 1 - \Pr{|V_\ell(G_n)| < s} \to 1 - \sum_{G^\bullet, |G^\bullet|<s} p_\ell(G^\bullet) \ge \epsilon.
	\]
	Thus, there is a sequence $s_n \to \infty$ such that 
	\[
		 \Pr{|V_\ell(G_n)| \ge s_n } \ge \epsilon/2
	\]
	We know that $|V_{\ell-1}(G_n)|$ is stochastically bounded, hence there is a constant $S>0$ such that
	\[
		\Pr{|V_{\ell-1}(G_n)| \le S} \ge 1 - \epsilon/4
	\]
	for all $n$. Consequently,
	\[
		\Pr{|V_\ell(G_n)| \ge s_n, |V_{\ell-1}(G_n)| \le S} \ge \epsilon/4
	\]
	for all $n$. Using the third term in Equation~\eqref{eq:compli} as a lower bound, it follows that
	\begin{align*}
		\Pr{|V_\ell(C_{1n})| \ge s_n } &\ge \sum_{\substack{ G^\bullet, |V_\ell(G^\bullet)| \ge s_n \\ |V_{\ell-1}(G_n)| \le S}} \Pr{V_\ell(G_n) \simeq G^\bullet} p^{|V_{\ell-1}(G^\bullet)|} (1 - \nu) \\
		& \ge p^S(1-\nu)(1-\nu) \Pr{|V_\ell(G_n)| \ge s_n, |V_{\ell-1}(G_n)| \le S} \\
		&\ge p^S(1-\nu)(1-\nu) \epsilon/4
	\end{align*}
	for all $n$. But clearly $\Pr{|V_\ell(C_{1n})| \ge s_n }$ tends to zero. Hence it must hold that \[\sum_{G^\bullet}p_\ell(G^\bullet) = 1.\] In other words $V_\ell(G_n)$ converges weakly.
	
	\paragraph*{8. Convergence of $\mC_n^\omega$ implies convergence of $\hat{\mB}_n$ - second part.}
	Recall that $\lambda_{1n}^*(v^*)$ is the union of the graph $\hat{\mB}_n$ and the Boltzmann-distributed object $\hat{\mR}$. Let $\ell \ge 1$ be given, and let $G^\bullet$ denote a rooted (connected) graph with height at most $\ell$. Then it holds that
	\begin{align}
		\label{eq:y4p}
		\Pr{V_\ell(G_n) \simeq G^\bullet} = \sum_{(H_1, H_2) \vdash G^\bullet} \Pr{V_\ell(B_n) \simeq H_1} \Pr{V_\ell(G(\hat{\mR})) \simeq H_2},
	\end{align}
	with the index $(H_1, H_2)$ ranging over all ordered pairs of unlabelled connected graphs that are rooted at a $*$-placeholder vertex, such that the graph obtained by identifying the roots of $H_1$ and $H_2$ is equal to $G^\bullet$. It follows that
	\[
		\Pr{V_\ell(B_n) \simeq G^\bullet} =  p^{-1} \left( \Pr{V_\ell(G_n) \simeq G^\bullet} - \sum_{\substack{(H_1, H_2) \vdash G^\bullet \\ H_2 \ne *}} \Pr{V_\ell(B_n) \simeq H_1} \Pr{V_\ell(G(\hat{\mR})) \simeq H_2}  \right)
	\]
	with $p = \Pr{|\hat{\mR}|=0}>0$. Note that in this sum it always holds that $|H_1| < |G^\bullet|$. Since  $\Pr{V_\ell(G_n) \simeq G^\bullet}$ converges toward $p_\ell(G^\bullet)$, it follows by induction on $|G^\bullet|$ that 
	\[
	\Pr{V_\ell(B_n) \simeq G^\bullet} \to q_\ell(G^\bullet)
	\]
	for some constant $q_\ell(G^\bullet)\ge0$. Equation~\eqref{eq:y4p} implies that
	\begin{align*}
		1 &= \sum_{G^\bullet} \sum_{(H_1, H_2) \vdash G^\bullet} q_\ell(H_1) \Pr{V_\ell(G(\hat{\mR})) \simeq H_2} \\
		&= \sum_H q_\ell(H) \sum_{H'} \Pr{V_\ell(G(\hat{\mR})) \simeq H'} \\
		&= \sum_H q_\ell(H),
	\end{align*}
	with the sum indices $H$ and $H'$ ranging over all rooted unlabelled (connected) graphs. This shows that $V_\ell(\hat{\mB}_n)$ converges weakly toward a random graph $\mQ_\ell$ for all $\ell \ge 1$. Clearly it holds that
	\[
		\mQ_\ell \simeq V_\ell(\mQ_k)
	\]
	for all $k \ge \ell$, since
	\[
		V_\ell(\hat{\mB}_n) = V_\ell(V_k(\hat{\mB}_n))
	\]
	for all $n$ and $V_\ell(\cdot)$ is a continuous map from the Polish space   of locally finite graphs, equipped with the metric from Equation \eqref{eq:thebsdistance}, to itself. This means that the distributions $\mu_\ell = \cL(\mQ_\ell)$ form a projective family $(\mu_\ell)_{\ell \in \ndN}$ with projections $f_{i,j} = V_i(\cdot)$ for all $i \le j$. Note that if we form the projective limit of the copies of $ (\mathbb{B}, d_{\textsc{BS}})$ with respect to the projections $(f_{i,j})_{i \le j}$, then each of its element may be interpreted as a locally finite graph. By Lemma~\ref{le:project}, it follows that there exists a random locally finite rooted graph $\hat{\mB}$ such that for all $\ell \ge 1$ it holds that \[
	\cL(V_\ell(\hat{\mB})) = \mu_\ell.
	\]
	Thus $\hat{\mB}$ is the local weak limit of the sequence $\hat{\mB}_n$ of randomly sized $2$-connected graphs.
\end{proof}

\subsection{Proofs of the applications to random weighted dissections in Section~\ref{sec:diss22}}

\label{sec:proofsdis}

As discussed in Section~\ref{sec:diss}, the species $\cD^\omega$ of  dissections, where each face with degree $k$ receives weight $\gamma_k$, satisfies an isomorphism of the form
\[
	\cD^\omega \simeq \cX + \Seq_{\ge 2}^{\gamma'} \circ \cD^\omega, \qquad \Seq_{\ge 2}^{\gamma'}(z) = \sum_{k=2}^\infty \gamma_{k+1} z^k. 
\]
The species $\cT_\ell^\omega$ of plane trees with leaves as atoms, where each vertex with out-degree $k \ge 2$ receives weight $\gamma_{k+1}$, satisfies by the discussion in Section~\ref{sec:lgwt} also an isomorphism
\begin{align}
	\label{eq:uselater}
	\cT_\ell^\omega \simeq \cX + \Seq_{\ge 2}^{\gamma'} \circ \cT_\ell^\omega.
\end{align}
We may use this to construct a bijection $\tau$ from the set of unlabelled $\cD$-objects to the set of locally finite plane trees that have no vertex with outdegree $1$, such that dissections with $n$ non-root vertices correspond to trees with $n$ leaves and the weight $\omega(D)$ of a dissection $D$ equals the weight $\omega(\tau(D))$ of the corresponding tree. This implies  that the random tree
\begin{align}
\label{eq:tautree}
\tau_n := \tau(\mD_n^\omega)
\end{align} gets drawn from $\cT_\ell^\omega[n]$ with probability proportional to its weight.

\begin{figure}[t]
	\centering
	\begin{minipage}{1.0\textwidth}
		\centering
		\includegraphics[width=0.5\textwidth]{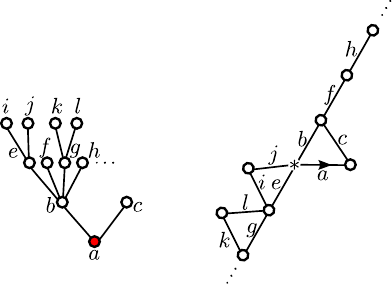}
		\caption{Correspondence of trees to "dissections with faces of infinite degrees". The letters show how the vertices of the tree correspond to the edges of the dissection.}
		\label{fi:dbij}
	\end{minipage}
\end{figure}

Given a dissection $D$ we let $\tau(D)$ denote the rooted plane tree whose vertices correspond to the edges of $D$, and whose root is given by the root-edge of $D$. The offspring of the root is given by the non-root edges incident to the root-face, in a canonical order. It will be convenient to  order the edges in a monotonically increasing way according to their proximity to the root-edge, by starting with the first edge to the left of the root-edge, then the first to the right, then the second to the left of the root-edge, then the second to the right, and so on. Each of these edges may be interpreted as the root-edge of the dissection attached to the root-face as in Figure~\ref{fi:disspol}. Hence we may recurse to complete the construction of $\tau(D)$. Note that the leaves of $\tau(D)$ correspond to non-root edges incident to the outer face of $D$. Hence we may say that the origin of the root edge of $D$ corresponds to the root vertex of $\tau(D)$, and each leaf of $\tau(D)$ corresponds to a unique non-$*$-vertex of $D$.

In Section~\ref{se:loconv} we discussed the compact Polish space $\fmT$ of plane trees that may have vertices with infinite degree. We let the $\fmT_\ell \subset \fmT$ denote the compact subspace of trees that have no vertex with outdegree $1$. We let $\mathfrak{D}$ denote the set of all unlabelled $\cD$-objects where we additionally allow "faces of infinite degree", given by doubly infinite paths. This allows us to extend the bijection $\tau$ to a bijection $\bar{\tau}$ from $\mathfrak{D}$ to $\fmT_\ell$ as illustrated in Figure~\ref{fi:dbij}. The set $\mathfrak{D}$ is a collection of rooted plane graphs that we may equip with the metric
\[
d_{\mathfrak{D}}(D_1, D_2) = 1 / (1 + \sup \{ \ell \ge 0 \mid V_{\ell}(D_1) = V_{\ell}(D_2) \text{ as edge-rooted planar maps} \}).
\]

The following observation is analogous to a lemma by Bj\"ornberg and Stef\'ansson \cite[Lem. 2.1]{MR3317412}, who studied infinite discrete looptrees.

\begin{lemma}[Correspondence between dissections and trees]
	\label{le:cordt}
	The bijection $\bar\tau: \mathfrak{D} \to \fmT_\ell$ is a homeomorphism.
\end{lemma}
\begin{proof}
	The space $\fmT_\ell$ is compact, hence it suffices to show that the inverse $\bar{\tau}^{-1}$ is continuous. To this end, let $(T_i)_{i \ge 1}$ be a sequence in $\fmT_\ell$ that converges toward a limit tree $T_0$, and set $D_i=\bar{\tau}^{-1}(T_i)$ for all $i \ge 0$. We need to show that the $D_i$ converges toward $D_0$ in the space $\mathfrak{D}$ as $i$ becomes large. 
	
	Let $\ell \ge 1$ be an arbitrarily large integer. The neighbourhood $V_\ell(D_0)$ contains only finitely many edges, since $D_0$ is locally finite. Consequently, there is a finite subset $V \subset \VHT$ of the vertex set of the Ulam--Harris tree described in Section~\ref{se:loconv}, and a family of subsets $(\cE(v))_{v \in V}$ with
	\[
		\cE(v) \in \{ \{0\}, \{1\}, \ldots\} \cup \{ \{0, 1, \ldots, \infty\}, \{1, 2, \ldots, \infty\}, \ldots\}
	\]
	for all $v$, such that any dissection $D \in \mathfrak{D}$ satisfies $V_\ell(D) = V_\ell(D_0)$ if and only if the plane tree $T=\bar{\tau}(D)$ satisfies $d^+_T(v) \in \cE(v)$ for all $v \in V$. Since $\cE(v)$ is an open subset of the space $\ndN_0 \cup \{\infty\}$ equipped with the one-point compactification topology, it follows that
	\[
		U = \{T \in \fmT_\ell \mid d^+_T(v) \in \cE(v) \text{ for all } v \in V\}
	\]
	is an open neighbourhood of $T_0$ in $\fmT_\ell$. Since $\lim_{i \to \infty} T_i = T_0$, this implies that for sufficiently large $i$ it holds that $T_i \in U$, and hence $V_{\ell}(D_i) = V_{\ell}(D_0)$ and
	\[
		d_{\mathfrak{D}}(D_i, D_0) \le 1/(1 + \ell).
	\]
	As $\ell \ge 1$ was arbitrary, it follows that  $D_0 = \lim_{i \to \infty} D_i$ in the space $\mathfrak{D}$.
\end{proof}

Recall that any tree $T \in \cT_\ell^\omega[n]$ receives the weight 
\[
	\omega(T) = \prod_{v \in T} p_{d^+_T(v)},
\]
with $p_0=1$ and $p_k = \gamma_{k+1}$ for $k \ge 1$. In Section~\ref{sec:schroeder} we associated a probability distribution $p^{(t_0)}$ to $(p_k)_k$ and showed in Lemma~\ref{le:convll} that the modified Galton--Watson tree from Section~\ref{se:modgwt}, that corresponds to the offspring distribution $p^{(t_0)}$, is the weak limit of $\tau_n$ in the space $\fmT$. Let $\hat{\tau}$ denote this limit tree.

We are now in the position to provide the proofs of the results stated in Section~\ref{sec:diss22}. Lemma~\ref{le:convll} and Lemma~\ref{le:cordt} prove the Benjamini--Schramm convergence of the random dissection $\mD_n^\omega$ for arbitrary weights.  Verifying  the convergence of $o(\sqrt{n})$-neighbourhoods in the I$a$ regime requires some additional arguments, for which we build on the work by Curien, Haas and Kortchemski \cite{MR3382675}.

\begin{proof}[Proof of Theorem \ref{te:bsconvdist1} and Remark \ref{re:disdist}]
	Suppose that the weight-sequence $\mathbf{w}$ has type I.  Note that since $\tau=t_0$ by Lemma~\ref{le:types}, the distribution of the number $\mF$ given in Equation~\ref{eq:distsbf} corresponds to the number of non-root edges in a dissection corresponding to a tree with a root and size-biased $p^{(t_0)}$ offspring. Moreover, a $\mathbb{P}_{\cD^\omega, \tau/\phi(\tau)}$-distributed dissection corresponds to a Galton--Watson tree with offspring distribution $p^{(t_0)}$. Hence the description of the limit in Remark~\ref{re:disdist} is identical to the distribution of the dissection corresponding to the modified Galton--Watson tree $\hat{\tau}$. This verifies that the graph described in Remark~\ref{re:disdist} is the Benjamini--Schramm limit of $\mD_n^\omega$.
	
	The infinite $\cR^\kappa = \Seq \circ \Seq_{\ge 1}^\gamma$ enriched tree $(\hat{\cT}, \hat{\beta})$ corresponds to a locally finite rooted plane graph $\hat{\mD}$. We now verify  $\hat{\mD}$  is also the Benjamini--Schramm limit of $\mD_n^\omega$, which implies $\hat{\mD} \eqdist \tau^{-1}(\hat{\tau})$ and hence verifies Remark~\ref{re:disdist}. For any $\ell \ge 1$ and any rooted graph $G$ there is an integer $h \ge 1$ and a set $\cE$ of trimmed $\cR$-enriched trees such that any $\cR$-enriched tree $(T,\alpha)$ corresponding to a rooted dissection $D$, the $\ell$-graph-distance neighbourhood $V_\ell(D)$ is isomorphic as plane graph to $G$ if and only if the trimming $(T, \alpha)^{[h]}$ at height $h$ belongs to the set $\cE$. Consequently, Theorem~\ref{te:local} implies that 
	\[
	\Pr{V_{\ell}(\mD_n^\omega)  \simeq G} = \Pr{ (\cT_n, \beta_n)^{[h]} \in \cE} \to \Pr{ (\hat{\cT}, \hat{\beta})^{[h]} \in \cE} = \Pr{V_{\ell}(\hat{\mD}) \simeq G},
	\]
	as $n \equiv 1 \mod \spa(\mathbf{w})$ becomes large. This establishes $\hat{\mD}$ as the Benjamini--Schramm limit of $\mD_n^\omega$, and consequently it must hold that $\hat{\mD} \eqdist \tau^{-1}(\hat{\tau})$.

	In order to conclude the proof of Theorem~\ref{te:bsconvdist1}, we need to show the convergence of $o(\sqrt{n})$-neighbourhoods in the I$a$ setting. Let $t_n = o(\sqrt{n})$ be sequence of non-negative integers. We need to show that
	\begin{align}
	\label{eq:rainbow}
	\lim_{n \to \infty}	d_{\textsc{TV}}(V_{t_n}(\mD_n^\omega) , V_{t_n}(\hat{\mD})) = 0. 
	\end{align}
	Without loss of generality we may assume that $t_n$ tends to infinity as $n$ becomes large. If the weight-sequence $\mathbf{w}$ has type I$a$, then Theorem~\ref{te:strengthend} says that for any sequence $h_n = o(\sqrt{n})$ it holds that
	\[
	\lim_{n \to \infty}	d_{\textsc{TV}}((\cT_n, \beta_n)^{[h_n]}, (\hat{\cT}, \hat{\beta})^{[h_n]}) = 0. 
	\]
	Hence in order to verify \eqref{eq:rainbow}, it suffices to show that there is a sequence $h_n = o(\sqrt{n})$ such that with high probability all vertices in $V_{t_n}(\mD_n^\omega)$ have height most $h_n$ in $(\cT_n, \beta_n)$.
	
	It was shown by Curien, Haas and Kortchemski in \cite[Lem. 9, Eq. (8)]{MR3382675} that there is a constant $C>0$ such that for all $\epsilon>0$ it holds with high probability that all vertices $v \in \mD_n^\omega$ satisfy
	\[
		\he_{\tau_n}(v) \le C \he_{\mD_n^\omega}(v) + C \epsilon \max(\Di(\tau_n), \sqrt{n}).
	\]
	The rescaled tree $(\tau_n, n^{-1/2} d_{\tau_n})$ converges toward the CRT in the Gromov--Hausdorff topology by results of Kortchemski~\cite{MR2946438}, consequently the rescaled diameter $\Di(\tau_n) / \sqrt{n})$ is tight. It follows that there is a sequence $h_n = o(\sqrt{n})$ such that with high probability all vertices $v$ with $\he_{\mD_n^\omega(v)} \le t_n$ satisfy $h_{\tau_n}(v) \le h_n$. It is easy to check that $h_{\cT_n}(u) \le h_{\tau_n}(u)$ for all vertices $u$ of $\mD_n^\omega$, so all vertices in $V_{t_n}(\mD_n^\omega)$ have with high probability height at most $h_n$ in $(\cT_n, \beta_n)$. This verifies \eqref{eq:rainbow}.
\end{proof}

If the weight-sequence $\mathbf{w}$ has type II, then analogously to the type I case the limit $\hat{\tau}$ corresponds to the dissection $\hat{\mD}$ of Theorem~\ref{te:bsconvdis}. Hence Lemma~\ref{le:convll} and Lemma~\ref{le:cordt} prove Theorem~\ref{te:bsconvdis}. If $\mathbf{w}$ has type III, then $\hat{\tau}$ is the infinite star that corresponds to a dissection with a single face of infinite degree, in other words, a doubly infinite path. This verifies Theorem~\ref{te:gibbst3}.

\subsection{Proofs of the applications to random $k$-trees in Section~\ref{sec:ktrees}}

\begin{proof}[Proof of Theorem~\ref{te:anamesname}]
	It was established in \cite[Lem. 6]{2016arXiv160505191D} that there are constants $\mathfrak{m}_k, C,c>0$ such that with probability at least $1 - C n^{- \log^c n}$ any two vertices $x,y \in \cT_n$  satisfy
	\begin{align}
	\label{eq:ldev}
	|d_{\mK_{1,n}^\circ}(x,y) - \mathfrak{m}_k d_{\cT_n}(x,y)| \le \max( d_{\cT_n}(x,y)^{3/4}, \log^{3}(n)).
	\end{align}
	Note that the vertices of $\cT_n$ correspond to the vertices of $\mK_{1,n}^\circ$ outside of the $k$-element root-front. The behaviour around a uniform random vertex of $\mK_{1,n}^\circ$ is asymptotically identical to uniform random vertex $u_n$ of $\cT_n$, because the probability to hit the root-front tends to zero as $n$ becomes large.
	
	Without loss of generality we may assume that $t_n \ge n^{1/4}$. It follows that with probability tending to one as $n$ becomes large the $t_n$ graph-distance neighbourhood $V_{t_n}(\mK_{1,n}^\circ, u_n)$ is a subset of the union of the root-front of $\mK_{1,n}$ and the  $s_n = t_n + t_n^{3/4} = o(\sqrt{n})$ tree-distance neighbourhood of $u_n$ in $\cT_n$. But $s_n = o(\sqrt{n})$ implies that $\he_{\cT_n}(u_n) > s_n$ with probability tending to one, so $V_{t_n}(\mK_{1,n}^\circ, u_n)$ does not contain the root of $\cT_n$ and hence no vertex of the root-front of $\mK_{1,n}^\circ$ at all.
	
	In particular, if $(\mH_i^n)_{i \ge 0}$ denotes the growing enriched fringe subtree representation of $(\cT_n, \beta_n)$ at $u_n$ (that is, $\mH_i^n$ is the  enriched fringe subtree at the $i$th ancestor of $u_n$), then with high probability the vector $(\mH_{1}^n, \ldots, \mH_{s_n}^n)$ contains all information about $V_{t_n}(\mK_{1,n}^\circ, u_n)$. Theorem~\ref{te:thmben} ensures that
	\[
	\lim_{n \to \infty} d_{\textsc{TV}}( (\mH_1^n, \ldots, \mH_{s_n}^n), (\hat{\mH}_1, \ldots, \hat{\mH}_{s_n}^n)) = 0,
	\]
	with $(\hat{\mH}_i)_{i \ge 0}$ denoting the growing enriched fringe tree representation of $(\cT^*, \beta^*)$ along its backwards growing spine. This completes the proof.
\end{proof}

\begin{proof}[Proof of Remark~\ref{re:limktreedis}]
	Let $(\hat{\mH}_i)_{i \ge 0}$ denote the growing enriched fringe tree representation of $(\cT^*, \beta^*)$ along its backwards growing spine $u_0, u_1, \ldots$ (That is, $\hat{\mH}_i$ is the enriched fringe subtree at the vertex $u_i$). Then $\hat{\mH}_0$ is distributed like the $\cR = \Set^k$ enriched tree $(\cT, \beta)$ and the corresponding $k$-tree follows a Boltzmann distribution $\mathbb{P}_{{\cK_1^\circ}, \frac{1}{ek} }$.  The decomposition $\cK_{1}^\circ \simeq \cX \cdot \Set^k(\cK_{1}^\circ)$ together with the rules of Lemma~\ref{le:bole} then show that $\hat{\mH}_0$ corresponds to a  root-front consisting of placeholder $*$-vertices, such that $u_0$ is connected to each vertex of the root-front, and each of the $k$ fronts incident to $u_0$ is identified with the root-front of an independent copy of $\mK^\circ$.
	
	For any $i \ge1$, the $\cR$-object of the offspring of $u_i$ in $\hat{\mH}_i$ follows a $\mathbb{P}_{\cR^\bullet, \cK_1^\circ(\frac{1}{ek})}$ distribution. It holds that
	\begin{align*}
	\cR^\bullet &\simeq \Set^\bullet \cdot \Set^{k-1} + \Set \cdot \Set^\bullet \cdot \Set^{k-2}  + \ldots + \Set^{k-1} \Set^\bullet \simeq \sum_{i =1}^k \cX \cdot \Set^k,
	\end{align*}
	since $\Set^\bullet \simeq \cX \cdot \Set' \simeq \cX \cdot \Set$. This may be interpreted by stating that an $\cR^\bullet$-object consists of an $\cR$-object, with one additional hedron attached to any of the $k$ possible locations.
	
	Hence Lemma~\ref{le:bole} yields that $\hat{\mH}_i$ is distributed like taking a hedra consisting of $u_i$ and a root-front of $k$ distinct $*$-placeholder vertices, identifying each of the $k$ fronts incident to $u_i$ with the root-front of a fresh independent copy of $\mK^\circ$, and then selecting one of the $k$ fronts uniformly at random and identifying it with the root-front of $\hat{\mH}_{i-1}$ in any canonical way. Permuting the vertices of the root-front of $\hat{\mH}_{i-1}$ by any fixed permutation does not change its {\em distribution}, so it does not matter which matching of the front-vertices we choose.
	
	Letting $i$ formally tend to infinity yields the limit $\hat{\mK}$. Now, instead of attaching the independent copies of $\mK^\circ$ in each step, we may just as well procrastinate and do that after having glued together the infinitely many hedra containing the root-fronts. This yields the description with the random walker in Remark~\ref{re:limktreedis}.
\end{proof}

\begin{proof}[Proof of Remark~\ref{re:ktreesrem}]
	In the description of the limit $\hat{\mK}$ in Remark~\ref{re:limktreedis}, the vertex $u_0$ is at first incident to precisely $k$ fronts, and then each front gets identified the root-front of an independent copy of $\mK^\circ$. So 
	\[
	d_{\hat{\mK}}(u_0) \eqdist k + \sum_{i=1}^k d_i
	\]
	with $d_i$ being independent copies of the number $d$ of non-root-front vertices incident to a fixed root-front vertex in $\mK^\circ$. As $\mK^\circ$ follows a $\mathbb{P}_{\Set \circ \cK_1^\circ,\frac{1}{ek}}$-distribution with $\cK_1^\circ(\frac{1}{ek}) = 1/k$, it follows from Lemma~\ref{le:bole} and the decomposition $\cK_1^\circ \simeq \cX \cdot \Set^k(\cK_1^\circ)$, that $d$ is distributed like the sum of $Z$ and $(k-1)Z$ many independent copies of $d$. This recursion yields the description of $d$ in terms of mono-type vertices in a $2$-type Galton--Watson tree.
	
	In particular, we may stochastically bound the number $d_{\hat{\mK}}(u_0)$ by the size of a $\text{Poisson}((k-1)/k)$-Galton--Watson tree, by considering the tree obtained from the $2$-type tree by identifying for each vertex the $Z$ type $A$ offspring vertices with any $Z$ type $B$ offspring. This verifies that $d_{\hat{\mK}}(u_0)$ has finite exponential moments.
\end{proof}

\subsection{Proofs of the applications to random weighted planar maps in Section~\ref{sec:appmaps}}

\begin{proof}[Proof of Theorem~\ref{te:plant1}]
	As for Claim~{\em (1)}, let $k$ be a positive integer and $M$ an edge-rooted planar map. We have to be careful as the $k$-neighbourhood $V_k(\mM_n^\omega)$ may contain vertices of arbitrarily large height in the corresponding tree $(\cT_{2n+1}, \beta_{2n+1})$. However, as $M$ has only finitely many edges, there is an integer $K$ and a (possibly infinite) set $\cE_K$ of $\cQ$-enriched planed trees trimmed at height $K$ such that any $\cQ$-enriched plane tree $(T,\alpha)$ that corresponds to a map $N$ satisfies $V_{k}(N) \simeq M$ if and only if $(T, \alpha)^{[K]} \in \cE_K$. Theorem~\ref{te:local} implies that $\Pr{ (\cT_{2n+1}, \beta_{2n+1})^{[K]} \in \cE_K}$ converges toward $\Pr{ (\hat{\cT}, \hat{\beta})^{[K]} \in \cE_K}$, and consequently 
	\[
	\lim_{n \to \infty} \Pr{ V_k(\mM_n^\omega) \simeq M} = \Pr{ V_k(\hat{\mM}) \simeq M}.
	\]
	For Claim~{\em (2)}, we are faced with the same problem on how to relate  graph-metric neighbourhoods to trimming heights. In Theorem~\ref{te:main1} we establish a scaling limit for a general model of semi-metric spaces based on random enriched trees, which as described in Theorem~\ref{te:convplan} yields a scaling limit for type I$a$ block-weighted planar maps with respect to the first-passage percolation metric. The idea there is to relate the metric on $\mM_n$ metric to a semi-metric on the set of corners of $\mM_n$, and then in turn this metric to the tree-distance in the coupled enriched plane tree. As $k_n=o(\sqrt{n})$, it follows by Equations~\eqref{eq:sca1} and \eqref{eq:sca2} in the proof of Theorem~\ref{te:main1} that there is another, larger sequence $s_n = o(\sqrt{n})$ such that with probability tending to one as $n$ becomes all corners $c$ incident to a vertex $v$ with height $\he_{\mM_n^\omega}(v) \le k_n$ satisfy $\he_{\cT_{2n+1}}(c) \le s_n$. Whenever this event takes place, the trimmed tree $(\cT_{2n+1}, \beta_{2n+1})^{[s_n]}$ already contains all information on the graph metric neighbourhood $V_{k_n}(\mM_n^\omega)$. Theorem~\ref{te:strengthend} ensures that
	\[
	\lim_{n \to \infty}d_{\textsc{TV}}((\cT_{2n+1}, \beta_{2n+1})^{[s_n]}, (\hat{\cT}, \hat{\beta})^{[s_n]}) = 0,
	\]
	since $s_n = o(\sqrt{n})$. Consequently,
	\[
	\lim_{n \to \infty} d_{\textsc{TV}}( V_{k_n}(\mM_n^\omega), V_{k_n}(\hat{\mM})) =0.
	\]
\end{proof}

\begin{proof}[Proof of Theorem~\ref{te:thmplan2}]
	The proof is analogous to the proof of Theorem~\ref{te:convout} for the local weak limit of random outerplanar maps. In fact, it is much simpler, as $\cR=\cQ$ is not a compound structure and hence we do not have to implement the limits of convergent type Gibbs partitions.
	
		Similar as in the proof of Theorem~\ref{te:convout} (and as justified in detail there), we may modify the matching of the $2$-connected maps $(\beta_n(v))_{v \in \cT_n}$ to the offspring sets in $\cT_n$ in a canonical way to create an enriched tree $(\cT_n, \lambda_n)$ such that for each $v \in \cT_n$ with offspring $v_1,\ldots, v_k$ the sequence of heights (that is, distance from the origin of the root-edge) $h_{\lambda_n(v)}(v_1), \ldots, h_{\lambda_n(v)}(v_k)$  in the map $\lambda_n(v)$ is non-decreasing. (By abuse of notation, we identify here the height of a corner with the height of the unique incident vertex.) So $(\cT_n, \beta_n)$ and $(\cT_n, \lambda_n)$ correspond to different maps, but they follow the same {\em distribution}, and we may assume that $\mM_n^\omega$ corresponds directly to $(\cT_n, \lambda_n)$. Likewise we may construct an enriched tree $(\hat{\cT}, \hat{\lambda})$ out of $(\hat{\cT}, \hat{\beta})$ in the same way without changing the distribution of the corresponding map, and let $\hat{\mM}$ denote the map that corresponds directly to $(\hat{\cT}, \hat{\lambda})$ where the offspring of the unique vertex with infinite degree is identified with the corners of the limit $\hat{\mQ}$ of non-separable maps in the same height-preserving way.
	
		As the ordering on the offspring sets respects the heights, it follows that for each integer $\ell\ge1$ and each finite planar map $M$ (considered as rooted at the origin of its root-edge) there is a constant integer $m = m(\ell, M) \ge 1$ that does not depend on $n$, such that we may decide whether $V_\ell(\mM_n^\omega) = M$ (as unlabelled edge-rooted planar maps) by only looking at the family $(\lambda_n(v))_{v \in V^{[m]}}$, with
	\[
	V^{[m]} = \{ (i_1, \ldots, i_t) \mid t \le m, i_1, \ldots, i_\ell \le m\} \subset \VHT
	\]
	a left-ball subset of the Ulam--Harris tree. Likewise, the event $V_\ell(\hat{\mM})=M$ is entirely determined by the family $(\hat{\lambda}(v))_{v \in V^{[m]}}$.
	
		Janson~\cite[Sec. 20]{MR2908619} constructs a deterministic sequence $\Omega_n \to \infty$ and a modified Galton--Watson tree $\cT_{1n}$ that is obtained from $\hat{\cT}$ by sampling a random degree $\tilde{D}_n \ge \Omega_n$ independently from $\hat{\cT}$, and pruning $\hat{\cT}$ at its unique vertex $v^*$ with infinite degree, keeping only the first $\tilde{D}_n$ children of $v^*$.
		The tree $\cT_{1n}$ is almost surely finite, and we may turn it into an enriched plane tree $(\cT_{1n}, \lambda_{1n})$, by sampling for each vertex $v\in \cT_{1n}$ an element from $\cQ[d^+_{\cT_{1n}}(v)]$ with probability proportional to its $\kappa$-weight, and matching its vertices with the offspring of $v$ in a canonical way that respects their height, in the same way as in the construction of $(\cT_n, \lambda_n)$ out of $(\cT_n, \beta_n)$. Janson's result \cite[Thm. 20.2]{MR2908619} ensures that	
		\begin{align*}
		\lim_{n \to \infty} d_{\textsc{TV}}( (d^+_{\cT_n}(v))_{v \in V^{[m]}}, (d^+_{\cT_{1n}}(v))_{v \in V^{[m]}}) = 0.
		\end{align*}
		For each family $(k_v)_{v \in V^{[m]}}$ of non-negative integers  with  $\Pr{d^+_{\cT_n}(v) = k_v \text{ for all }v \in V^{[m]}} >0$ it holds that
		\begin{multline*}
			( (\lambda_n(v))_{v \in V^{[m]}} \mid d^+_{\cT_n}(v) = k_v \text{ for all $v \in V^{[m]}$}) \\ \eqdist ( (\lambda_{1n}(v))_{v \in V^{[m]}} \mid d^+_{\cT_{1n}}(v) = k_v \text{ for all $v \in V^{[m]}$}).
		\end{multline*}
		This is easily verified, as for each $v \in V^{[m]}$ the conditional distribution of $\lambda(v)$ given $d_{\cT_n}^+(v) = k_v$ samples a random $k_v$-sized $\cR$-object from $\cR[k_v]$ with probability proportional to its weight, and likewise for $\lambda_{1n}(v)$. It follows that
		\begin{align*}
		\lim_{n \to \infty} d_{\textsc{TV}}( (\lambda_n(v))_{v \in V^{[m]}}, (\lambda_{1n}(v))_{v \in V^{[m]}}) = 0.
		\end{align*}
		By the construction of $m = m(\ell, M)$, this implies
				\begin{align*}
				\lim_{n \to \infty} |\Pr{V_\ell(\mM_n^\omega) = M} - \Pr{V_\ell(\mM_{1n}^\omega) = M} | = 0.
				\end{align*}
		with $\mM_{1n}^\omega$ denoting the planar map corresponding to $(\cT_{1n}, \beta_{1n})$.
		
		The only part of $(\cT_{1n}, \beta_{1n})$ that actually depends on $n$ is the offspring of the vertex~$v_*$. 
		
		Suppose that the random non-separable map $\mQ_{\tilde{D}_n}^\kappa$ has a distributional limit $\hat{\mQ}$. Hence $\lambda_{1n}(v_*)$  converges toward $\hat{\mQ}$ in the local weak sense. As $\hat{\mM}$ corresponds to $(\hat{\cT}, \hat{\lambda})$, and $(\hat{\cT}, \hat{\lambda})$ is distributed like $(\cT_{1n}, \beta_{1n})$ if we would replace the offspring of $v_*$ by the limit $\hat{\mQ}$, it follows that
		\[
			 \lim_{n \to \infty} \Pr{V_\ell(\mM_n^\omega) = M} = \Pr{V_\ell(\hat{\mM}) = M}.
		\]
		That is, the map $\hat{\mM}$ is the local weak limit $\mM_n^\omega$.
		
		Conversely, if $\mM_n^\omega$ converges in the local weak sense, it follows by analogous arguments as in step 7 and 8 in the proof of Theorem~\ref{te:thmplan2} that $\mQ_{\tilde{D}_n}^\kappa$ does as well.
	
	It remains to explain the distribution of the limit map. The tree $\hat{\cT}$ is composed of normal vertices, that receive offspring according to an independent copy of $\xi$, and special vertices, that receive offspring according to an independent copy of $\hat{\xi}$. The root is special, if a special vertex has finitely many offspring one selected uniformly at random and declared special. Thus the length of the spine of $\hat{\cT}$ follows a geometric distribution with parameter $\nu$, and the tip of the spine is a vertex having an infinite number of normal offspring. The distribution of the limit map is then easily deduced from the fact that drawing a $\cQ$-structure from $\cQ[\xi]$ with probability proportional to its weight follows a  $\mathbb{P}_{\cQ^\bullet, \tau}$-Boltzmann distribution, and sampling a $(\hat{\xi} \mid \hat{\xi} < \infty)$-sized $\cQ$-structure with probability proportional to its weight follows a $\mathbb{P}_{(\cQ^\bullet)^\kappa, \tau}$-distribution.
\end{proof}

\subsection{Proofs of the scaling limits and diameter tail bounds in Section~\ref{sec:partA}}
\label{sec:propartA}

\begin{proof}[Proof of Theorem~\ref{te:main1}]
Consider the coupling of the random enriched tree $\mA_n^\cR$ with the enriched plane tree $(\cT_n, \beta_n)$ given in Lemma~\ref{le:coupling}. By assumption, the weight sequence $\mathbf{w}$ has type I$a$. Hence by the discussion in Section~\ref{sec:pretree} the simply generated tree  $\cT_n$ is distributed like a  critical Galton--Watson tree $\cT$ conditioned on having size $n$, such that the offspring distribution $(\pi_k)_k$ has finite exponential moments. In particular, as $n \equiv 1 \mod \spa(\mathbf{w})$ tends to infinity, we have that $\sigma \cT_n / (2 \sqrt{n})$ converges towards the CRT $\CRT$ with $\sigma^2$ denoting the variance of the offspring distribution. 

For any finite set $U$ and $\cR$-structure $R \in \cR[U]$ let $\eta_R$ denote the $\delta_R$-distance from the $*_U$ point to a uniformly at random chosen label from $U$. Moreover, let $\eta$ denote the random number given by choosing an integer $k$ according to the distribution $(k\pi_k)_k$, choosing a random $\cR$-structure $\mR$ from $\cR[k]$ with probability proportional to its $\kappa$-weight and setting $\eta = \eta_\mR$. By assumption, for any $\cR$-structure $R$ the random variable $\eta_R$ is bounded by the sum of $|R|$ copies of a random variable $\chi \ge 0$ having finite exponential moments. Since the distribution $(k \pi_k)_k$ also has finite exponential moments, it follows that $\eta$ has finite exponential moments.

We are going to show that the Gromov--Hausdorff distance between $\Ex{\eta} \cT_n/ \sqrt{n}$ and $\mX_n/\sqrt{n}$ converges in probability to zero. This immediately implies that \[\sigma \mX_n / (2 \Ex{\eta} \sqrt{n}) \convdis \CRT\] and we are done.

Let $s>1$ and $t>0$ be arbitrary constants and set $s_n = \log(n)^s$ and $t_n = n^t$. Let $\epsilon >0$ be given and let $\cE_1$ denote the event that there exists a vertex $v \in \cT_n$ and an ancestor $u$ of $v$ with the property that \[d_{\cT_n}(u,v) \ge s_n \quad \text{and} \quad d_{\mX_n}(u,v) \notin (1 \pm \epsilon) \Ex{\eta} d_{\cT_n}(u,v).\] Likewise, let $\cE_2$ denote the event that there exists a vertex $v$ and an ancestor $u$ of $v$ with \[d_{\cT_n}(u,v) \le s_n \quad \text{and} \quad d_{\mX_n}(u, v) \ge t_n.\] We are going to show that with high probability none of the events $\cE_1$ and $\cE_2$ takes place.

This suffices to show the claim: Take $s=2$ and $t=1/4$ and suppose that the complementary events $\cE_1^{c}$ and $\cE_2^{c}$ hold. Given vertices $a \ne b$ let $x$ denote their lowest common ancestor in the tree $\cT_n$. If $x \in \{a,b\}$ then we have \[d_{\mX_n}(a,b) = d_{\mX_n}(a,x) + d_{\mX_n}(b,x).\] If $x \ne a,b$, then let $a'$ denote the offspring of $x$ that lies on the $\cT_n$-path joining $a$ and $x$ and likewise $b'$ the offspring of $x$ lying on the path joining $x$ and $b$. Hence we have that \[d_{\mX_n}(a,b) = d_{\mX_n}(a,x) + d_{\mX_n}(b,x) + R \quad \text{with} \quad R = d_{\mX_n}(a,a') - d_{\mX_n}(a',x) - d_{\mX_n}(b',x).\] By property $\cE_2^c$ and the triangle inequality it follows that $|R| = -R \le 2 n^{1/4}$. Thus, regardless whether $x \in \{a,b\}$, it holds that
\[
d_{\mX_n}(a,b) = d_{\mX_n}(a,x) + d_{\mX_n}(b,x) + O(n^{1/4}).
\]
Moreover, if $d_{\cT_n}(a,x) \ge \log(n)^2$, then it follows by property $\cE_1^c$ that \begin{align} \label{eq:sca1}
d_{\mX_n}(a,x) \in (1 \pm \epsilon)\Ex{\eta} d_{\cT_n}(a,x).
\end{align}
 Otherwise, if $d_{\cT_n}(a,x) < \log(n)^2$ then it follows by property $\cE_2^c$ that $d_{\mX_n}(a,x) \le n^{1/4}$ and thus \begin{align}
 \label{eq:sca2}
 |d_{\mX_n}(a,x) - \Ex{\eta} d_{\cT_n}(a,x)| \le Cn^{1/4}\end{align}
  for a fixed constant $C$ that does not depend on $n$ or the points $a$ and $x$. It follows that 
\[
	|d_{\mX_n}(a,b)/\sqrt{n} - \Ex{\eta}d_{\cT_n}(a,b)/\sqrt{n}| \le \epsilon \Di(\cT_n)/\sqrt{n} + o(1),
\]
with $\Di(\cT_n)$ denoting the diameter.
Thus  \[d_{\text{GH}}(\mX_n, \Ex{\eta}\cT_n) \le \epsilon \Di(\cT_n)/\sqrt{n} + o(1)\] holds with high probability. Since we may choose $\epsilon$ arbitrarily small, and $\Di(\cT_n)/\sqrt{n}$ converges in distribution (to a multiple of the diameter of the CRT),  it follows that $d_{\text{GH}}(\mX_n, \Ex{\eta}\cT_n) \to 0$ in probability and we are done.

It remains to show that the events $\cE_i^c$ hold with high probability. To this end, recall the construction of the modified Galton--Watson tree $\hat{\cT}$ from Section~\ref{sec:pretree} that corresponds to the distribution $(\pi_k)_k$. Given $\ell \ge 0$  we may consider the truncated version $\hat{\cT}^{(\ell)}$ which has a finite spine of length $\ell$. At the top of the spine the special node becomes normal and reproduces normally. We call this vertex the outer root, i.e. $\hat{\cT}^{(\ell)}$ is a random pointed plane tree. This construction was introduced in \cite[Ch. 3]{MR3077536} and has the property, that for any plane tree $T$ with a distinguished vertex $r$  that has height $\ell = h_T(r)$ we have that \[\Pr{\hat{\cT}^{(\ell)} = (T,r)} = \Pr{\cT = T}.\] This is due to the fact that the probability, that a special node has offspring of size $k$ and precisely the $i$th child is declared special, is given by $k \pi_k / k = \pi_k$. For any vertices $v$ of $\cT$ and $u$ of $\hat{\cT}^{(\ell)}$ we choose random $\cR$-structures $\beta(v) \in \cR[d_{\cT}(v)]$ and $\hat{\beta}^{(\ell)}(u) \in \cR[d_{\hat{\cT}^{(\ell)}}^+(u)]$, each with probability proportional to its $\kappa$-weight. In particular, $(\cT_n, \beta_n)$ is distributed like $(\cT, \beta)$ conditioned on having $n$ vertices. For any pointed enriched plane tree $((T,r), \gamma)$ we have that \[\Pr{(\hat{\cT}^{(\ell)}, \hat{\beta}^{(\ell)}) = ((T,r),\gamma)} = \Pr{(\cT,\beta) = (T, \gamma)}.\]

For each $R \in \bigcup_{n \ge 0} \cR[n]$ let $(\delta_R^i)_{i \in \ndN_0}$ be a family of independent copies of $\delta_R$. Given an $\cR$-enriched plane $S = (T, \gamma)$ we may form the family $(\delta^S(v))_{v \in T}$ of random metrics by traversing bijectively the vertices of $T$ in a fixed order, let's say in depth-first-search order, and assigning to each vertex $v$ the "leftmost" unused copy from the list $(\delta_{\gamma(v)}^1, \delta_{\gamma(v)}^2, \ldots)$. The metrics can be patched together to a metric $d^S$ on the vertex set of the plane tree $T$ just as described in Section~\ref{sec:partA}.

We may assume that all random variables considered so far are defined on the same probability space and that the metric $d_{\mX_n}$ of $\mX_n$ coincides with the metric $d^{(\cT_n, \beta_n)}$. Given $(\delta^i_R)_{R,i}$ let $\cH$ denote the finite set of $\cR$-enriched plane trees of size $n$ such that the event $\cE_1$ takes place if and only if $(\cT_n, \beta_n) \in \cH$. By the definition of the event $\cE_1$ for any $H=(T, \gamma) \in \cH$ we may fix a vertex $v_H$ of $H=(T, \gamma)$ having the property that there exists an ancestor $u$ with \[d_T(u,v_H) \ge s_n \quad \text{and} \quad d^H(u, v_H) \notin (1 \pm \epsilon)\Ex{\eta}d_T(u,v).\] Let $\ell_H$ denote the height $h_T(v_H)$. The probability for the critical Galton--Watson tree $\cT$ to have size $n$ is $\Theta(n^{-3/2})$ and hence the conditional distribution of the event $\cE_1$ given $(\delta^i_R)_{R,i}$ equals
\[
\sum_{H \in \cH} \Pr{(\cT_n, \beta_n) = H \mid (\delta^i_R)_{R,i}}  = \Theta(n^{3/2}) \sum_{(T, \gamma) \in \cH} \Pr{(\hat{\cT}^{(\ell_H)}, \hat{\beta}) = ((T, v_H), \gamma) \mid (\delta^i_R)_{R,i}}
.\]
Let $v_0, \ldots, v_{\ell}$ denote the spine of $\hat{\cT}^{(\ell)}$, i.e. $v_{\ell}$ is the outer root, $v_0$ is the inner root, and $(v_0, \ldots, v_{\ell})$ is the directed path connecting the roots. It follows that the probability for the event $\cE_1$ is bounded by
\[
\tag{$*$} \Theta(n^{3/2}) \sum_{\ell=1}^n \sum_{k=0}^{\ell - s_n} \Pr{d^{(\hat{\cT}^{(\ell)}, \hat{\beta})}(v_k, v_{\ell}) \notin (1 \pm \epsilon) \Ex{\eta} (\ell-k)}
\]
But the $d^{(\hat{\cT}^{(\ell)}, \hat{\beta})}$-distance between two spine vertices $v_i$ and $v_j$ is distributed like the sum $\eta_1 + \ldots + \eta_{|i-j|}$ of independent copies $(\eta_i)_i$ of $\eta$. We know that $\eta$ has finite exponential moments and hence, by the deviation inequality in Lemma~\ref{le:deviation}, the bound $(*)$ converges to zero as $n \equiv 1 \mod \spa(\mathbf{w})$ tends to infinity. Thus the event $\cE_1$ holds with high probability. By the same arguments we may bound the probability for the event $\cE_2$ by \[\Theta(n^{3/2}) \sum_{\ell = 1}^{n} \sum_{k=1}^{\min(s_n,\ell)} \Pr{\eta_1 + \ldots + \eta_{k} \ge t_n}\]
which also converges to zero. This concludes the proof.
\end{proof}

\begin{proof}[Proof of Lemma~\ref{le:tail1}]
It suffices to show that there are constants $C,c,N>0$ such that for all $n \ge N$ and $h \ge \sqrt{n}$ we have that \[\Pr{\He(\mX_n) \ge h} \le C (\exp(-c h^2/n) + \exp(-c h)).\]

Recall the coupling in Lemma~\ref{le:coupling} of the random graph $\mA_n^\cR$ with a simply generated tree $\cT_n$ sharing the same vertex set. The weight sequence $\mathbf{w}$ has type I$\alpha$ by assumption, hence $\cT_n$ is distributed like a critical Galton--Watson tree conditioned on having $n$ vertices with the offspring distribution having finite non-zero variance.

For any vertex $v$ set $\ell(v) = \sum_u d_{\cT_n}^+(u)$ with the sum index $u$ ranging over all ancestors (not equal to $v$) of the vertex $v$ in the plane tree $\cT_n$. Let $s>r>0$ be constants. Given $h \ge \sqrt{n}$ let $\cE^h$ denote the event that $\He(\mX_n) \ge h$. Clearly \[\cE^h \subset \cE_1^h \cup \cE_2^h \cup \cE_3^h\] with $\cE_1^h$ the event that $\He(\cT_n) \ge rh$, $\cE_2^h$ the event that $\He(\cT_n) \le rh$, and there exists a vertex $v$ with $\ell(v) \ge s h$ and $\cE_3^h$ the event that $\He(\cT_n) \le rh$, $\ell(v) \le sh$ for all vertices $v$ and $\He(\mX_n) \ge h$.

We are going to show that if we choose $r$ and $s$ sufficiently small then each of these events is sufficiently unlikely. By the tail bound (\ref{eq:gwttail}) for Galton--Watson trees it follows that there are constants $C_1, c_1 > 0$ such that \[\Pr{\cE_1^h} \le C_1 \exp(-c_1 h^2 / (r^2 n)).\]

In order to bound the probability for the event $\cE_2^h$ suppose that we are given a vertex $v$ with $\hei_{\cT_n}(v) \le rh$, and $\ell(v) \ge s h$. Let $(Q_i)_i$ and $(Q'_i)_i$ denote the DFS and reverse DFS queues as defined in Section~\ref{sec:dfs} and let $j$ and $k$ denote the indices corresponding to the vertex $v$ in the lexicographic ordering of $\cT_n$ and its mirror image, respectively. By Equality~(\ref{eq:queues}) it follows that \[Q_j + Q'_k = 2 + \ell(v) - \hei_{\cT_n}(v) \ge (s-r)h\] and hence $Q_j \ge (s-r)/2$ or $Q'_k \ge (s-r)/2$. It follows by Inequality~\ref{eq:dfsbound} that \[\Pr{\cE_2^h} \le C_2 \exp(-c_2 h^2/n)\] for some constants $C_2, c_2 > 0$ that do not depend on $n$ or $h$.

We assumed that there is a real-valued random variable $\chi \ge 0$ such that for any $\cR$-structure $R$ the diameter of the metric $\delta_R$ is bounded by the sum of $|R|$ independent copies $\chi_1^R, \ldots, \chi_{|R|}^R$ of $\chi$. In particular, for any vertex $v$ of $\mX_n$ the height $\hei_{\mX_n}(v)$ is bounded by the sum of $\ell(v)$ independent copies of $\chi$. It follows that \[\Pr{\cE_3^h} \le n \Pr{\chi_1 + \ldots + \chi_{\lfloor sh \rfloor} \ge h}\] with $(\chi_i)_{i \in \ndN}$ a family of independent copies of $\chi$. By the inequality in Lemma~\ref{le:deviation} it follows that there are constants $\lambda, c>0$ such that 
\[ \Pr{\cE^h_3} \le 2 \exp(\log(n) + \lfloor sh \rfloor(c \lambda^2 + \lambda \Ex{\chi}) - \lambda h).\] We assumed that $h \ge \sqrt{n}$ and $n \ge N$, hence we may take $s$ sufficiently small and $N$ sufficiently large (depending only on $\lambda$ and $c$ and thus not depending on $h$ or $n$) such that there are constants $C_3, c_3$ with \[\Pr{\cE^h_3} \le C_3 \exp(-c_3 h)\] for all $h$ and $n \ge N$. Hence \[\Pr{\cE^h} \le \Pr{\cE_1^h} + \Pr{\cE_2^h} + \Pr{\cE_3^h} \le C_4 (\exp(-c_4 h^2/n) + \exp(-c_4 h))\] for some constants $C_4, c_4>0$ and we are done.
\end{proof}

\bibliographystyle{plain}
\bibliography{entrees}

\end{document}